\documentclass[reqno,11pt]{amsart}

\usepackage{amsmath,amssymb,amsthm,amsfonts}
\usepackage{mathrsfs}
\usepackage{bbm}
\usepackage{hyperref}

\usepackage{color}

\newtheorem{lemma}{Lemma}[section]
\newtheorem{theorem}{Theorem}[section]

\newtheorem{proposition}{Proposition}[section]

\newtheorem{corollary}{Corollary}[section]

\numberwithin{equation}{section}

\newcommand{\dis}{\displaystyle}

\newcommand{\R}{\mathbb{R}}

\renewcommand{\S}{\mathbb{S}}
\newcommand{\T}{\mathbb{T}}

\newcommand{\FF}{\mathbf{F}}
\newcommand{\FG}{\mathbf{G}}
\newcommand{\FH}{\mathbf{H}}
\newcommand{\FM}{\mathbf{M}}
\newcommand{\FP}{\mathbf{P}}
\newcommand{\FQ}{\mathbf{Q}}
\newcommand{\FL}{\mathbf{L}}

\newcommand{\FR}{\mathbf{R}}

\newcommand{\CA}{\mathcal{A}}
\newcommand{\CB}{\mathcal{B}}

\newcommand{\CE}{\mathcal{E}}

\newcommand{\CI}{\mathcal{I}}
\newcommand{\CJ}{\mathcal{J}}
\newcommand{\CK}{\mathcal{K}}

\newcommand{\CN}{\mathcal{N}}
\newcommand{\CM}{\mathcal{M}}

\newcommand{\al}{\alpha}
\newcommand{\be}{\beta}
\newcommand{\ga}{\gamma}

\newcommand{\la}{\lambda}
\newcommand{\de}{\delta}
\newcommand{\si}{\sigma}
\newcommand{\pa}{\partial}
\newcommand{\ka}{\kappa}
\newcommand{\eps}{\epsilon}

\newcommand{\ta}{\theta}

\newcommand{\De}{\Delta}

\begin{document}
\title[The two-component Vlasov-Poissonn-Boltzmann system]{
%The %two-species
%Vlasov-Poisson-Boltzmann System near local bi-Maxwellians with rarefaction waves
%for a disparate mass binary mixture
The Vlasov-Poisson-Boltzmann system for a disparate mass binary mixture
%: Stability near local bi-Maxwellians with rarefaction waves
}

\author[R.-J. Duan]{Renjun Duan}
%\thanks{}
\address[RJD]{Department of Mathematics, The Chinese University of Hong Kong,
Shatin, Hong Kong, P.R.~China}
\email{rjduan@math.cuhk.edu.hk}

\author[S.-Q. Liu]{Shuangqian Liu}
\address[SQL]{Department of Mathematics, Jinan University, Guangzhou 510632, P.R.~China}
\email{tsqliu@jnu.edu.cn}

%\date{\today}

\begin{abstract}
The Vlasov-Poisson-Boltzmann system is often used to govern the motion of plasmas consisting of electrons and ions with disparate masses when collisions of charged particles are described by the two-component Boltzmann collision operator. The perturbation theory of the system around global Maxwellians recently has been well established in \cite{G}. It should be more interesting to further study the existence and stability of nontrivial large time asymptotic profiles for the system even with slab symmetry in space, particularly understanding the effect of the self-consistent potential on the non-trivial long-term dynamics of the binary system. In the paper, we consider the problem in the setting of rarefaction waves. The analytical tool is based on the macro-micro decomposition introduced in \cite{LYY} that we can be able to develop into the case for the two-component Boltzmann equations around local bi-Maxwellians. Our focus is to explore how the disparate masses and charges of particles play a role in the analysis of the approach of the complex coupling system time-asymptotically toward a non-constant equilibrium state whose macroscopic quantities satisfy the quasineutral nonisentropic Euler system.
\end{abstract}



\maketitle
\thispagestyle{empty}

\tableofcontents

%\newpage
\section{Introduction}

\subsection{Presentation of the problem}
In the paper we are concerned with the nontrivial long-time dynamics of the Vlasov-Poisson-Boltzmann  (VPB for short) system used for describing the motion of charged particles in plasma (e.g., ions and electrons) when collisions between particles are taken into account, cf.~\cite{CC,KT}. Compared to the close-to-constant-equilibrium framework (cf.~\cite{G}), the perturbation theory around the non-constant equilibrium state would be more interesting and difficult due to the appearance of disparate masses and charges for gas mixtures, cf.~\cite{AAP,ABT,Ta,TaAo}. In the case of three space dimensions with slab symmetry, the governing equations take the form of
%\begin{equation}
%\label{2svpb}
%\begin{split}
%\pa_t F_i +\xi_1\pa_x F_i -\frac{1}{m_i}\pa_x\phi \pa_{\xi_1}F_i &=Q_{ii}(F_i,F_i)+Q_{ie}(F_i,F_e),\\
%\pa_t F_e+\xi_1\pa_x F_e +\frac{1}{m_e}\pa_x\phi \pa_{\xi_1}F_e &=Q_{ee}(F_e,F_e)+Q_{ei}(F_e,F_i).
%\end{split}
%\end{equation}
\begin{eqnarray}\label{VPB}
%\left\{
\begin{array}{rcccccccc}
%\begin{split}
\dis \pa_t F_{i} &+&\xi_1\pa_x F_{i}&-&\dis\frac{q_i}{m_i}\pa_x\phi\,\pa_{\xi_1}F_i &=&Q_{ii}(F_i,F_i)&+&Q_{ie}(F_i,F_e),\\[3mm]
\dis \pa_t F_{e} &+&\xi_1\pa_x F_{e}&-&\dis\frac{q_e}{m_e}\pa_x\phi\,\pa_{\xi_1}F_e &=&Q_{ee}(F_e,F_e)&+&Q_{ei}(F_e,F_i).
%\\-\pa_x^2\phi&=q_i\dis{\int_{\R^3}F_i}\,d\xi+q_e\dis{\int_{\R^3}F_e}\,d\xi.
%, \   \ x\in\R,\ \xi\in\R^3,\ t>0.
%\end{split}
\end{array}
%\right.
\end{eqnarray}
The self-consistent potential function $\phi=\phi(t,x)$ satisfies the Poisson equation
%\begin{equation}
%\label{2svpb-f}
%-\pa_x^2\phi=\int_{\R^3}F_i\,d\xi-\int_{\R^3}F_e\,d\xi.
%\end{equation}
\begin{equation}
\label{PE}
-\pa_x^2\phi=q_i\dis{\int_{\R^3}F_i}\,d\xi+q_e\dis{\int_{\R^3}F_e}\,d\xi.
\end{equation}
Here $F_{i}(t,x,\xi)$ and  $F_{i}(t,x,\xi)$ stand for the nonnegative number distribution functions for ions and electrons which have position $x\in \R$ and velocity $\xi=(\xi_1,\xi_2,\xi_3)\in \R^3$ at time $t\geq 0$.  Ions and electrons are assumed to have masses $m_{i}>0$, $m_e>0$ and charges $q_i>0$, $q_e<0$, respectively. Without loss of generality we suppose $m_i\geq m_e$ which is consistent with the physical situation where ions are much heavier than electrons. 

Regarding the binary collisions between like or unlike particles on the right-hand side of \eqref{VPB}, we assume that they are described by the Boltzmann operator for the hard-sphere model whose exact form reads
\begin{equation}\label{g.cop}
%\begin{split}
Q_{\mathcal {A}\mathcal {B}}(F_\CA,F_\CB)=\int_{\R^3\times \S^2} B_{\CA\CB}(|\xi-\xi_\ast|,\omega)[F_\CA(\xi')F_\CB(\xi'_\ast)-F_\CA(\xi)F_\CB(\xi_\ast)]\, d\xi_\ast d\omega ,
%\\
%=& Q_{\CA\CB}^{\rm{gain}}(F_\CA,F_\CB) + Q_{\CA\CB}^{\rm{loss}}(F_\CA,F_\CB),
%\end{split}
\end{equation}
for $\CA,\CB\in\{i,e\}$. Here  $\S^2$ is the unit sphere of $\R^3$. The collision kernel is given by
%$\S_+^2=\{\omega\in\S^2: (\xi-\xi_\ast)\cdot\omega\geq0\}$,
\begin{equation}\notag
%\label{ }
B_{\CA\CB}=\frac{(\si_\CA+\si_\CB)^2}{4}|(\xi-\xi_\ast)\cdot\omega|,
\end{equation}
with
$\si_\CA>0$ denoting the dimeter of particles of $\CA$ species, and through the paper we always take $\si_\CA=\si_\CB=\si$ without loss of generality. The pre-collisional velocity pair $(\xi,\xi_{\ast})$ and the post-collisional velocity pair $(\xi',\xi_{\ast}')$ corresponding to the integrand of \eqref{g.cop} satisfy the relationship
\begin{equation}\notag
%\label{ }
\begin{array}{rccccc}
   \xi'&=&\xi&-&\dis \frac{2m_\CB}{m_\CA+m_\CB}&[(\xi-\xi_\ast)\cdot\omega]\,\omega,\\[5mm]%\quad
\xi_\ast'&=&\xi_\ast&+&\dis \frac{2m_\CA}{m_\CA+m_\CB}&[(\xi-\xi_\ast)\cdot\omega]\,\omega,
\end{array}
\end{equation}
%\begin{equation}
%\begin{split}
%\xi'&=\xi-\frac{2m_\CB}{m_\CA+m_\CB}[(\xi-\xi_\ast)\cdot\omega]\omega,\\%\quad
%\xi_\ast'&=\xi_\ast+\frac{2m_\CA}{m_\CA+m_\CB}[(\xi-\xi_\ast)\cdot\omega]\omega,
%\end{split}
%\end{equation}
which follows from conversation of momentum and energy
\begin{equation}\notag
%\label{ }
\begin{split}
m_\CA\xi+m_\CB\xi_{\ast}&=m_\CA\xi'+m_\CB\xi'_{\ast},\\ %\quad
m_\CA|\xi|^2+m_\CB|\xi_{\ast}|^2&=m_\CA|\xi'|^2+m_\CB|\xi'_{\ast}|^2,
\end{split}
\end{equation}
for two colliding particles $\CA$ and $\CB$. Note that collisions between particles in plasma physics are often modelled by the long-range collision operator, for instance, the Boltzmann operator for soft potentials or the Landau operator for the Coulomb potential, cf.~\cite{Vi}.
%Moreover, $ |\xi-\xi_{\ast}| =|\xi'-\xi'_{\ast}|$ holds.
One may expect that the techniques of analysis developed in the paper together with the ones in \cite{DLYZ, DYZ-h,DYZ-s,Guo-VPL} could also be applied to those more physical situations.

For notational convenience,  as in \cite{Guo3}, we denote in the sequel
$$
\FF(t,x,\xi)=\left[\begin{array}{rll}
&F_i(t,x,\xi)\\[3mm]
&F_e(t,x,\xi)
\end{array}\right].
$$
The system \eqref{VPB}, \eqref{PE} is supplemented with initial data
\begin{eqnarray}\label{id.be}
\FF(0,x,\xi)=\FF_0(x,\xi)
%=\left[\begin{array}{rll}
%&F_{i}(0,x,\xi)\\[3mm]
%&F_{e}(0,x,\xi)
%\end{array}\right]
=\left[\begin{array}{rll}
&F_{0i}(x,\xi)\\[3mm]
&F_{0e}(x,\xi)
\end{array}\right],
\end{eqnarray}
and with boundary data at far fields
\begin{eqnarray}
\label{bd.be}
\lim\limits_{x\rightarrow\pm\infty}\FF_0(x,\xi)=\FF_{0\pm\infty}(\xi),
%\FM_{\pm\infty}=
%\left[
%\begin{array}{rll}
%      &M_{[\frac{-q_e}{q_i}n_{\pm},u_{\pm},\theta_\pm;m_i]}(\xi)\\[3mm]
%      &M_{[n_{\pm},u_{\pm},\theta_\pm;m_e]}(\xi)
%\end{array}\right],
%=\FM_{[\rho_\pm,u_\pm,\theta_\pm]}(\xi),\quad u_\pm=[u_{1\pm}, 0, 0],
\end{eqnarray}
and
\begin{equation}\label{bd.phi}
\lim\limits_{x\rightarrow\pm\infty}\phi(t,x)=\phi_\pm.
\end{equation}
Through the paper, due to the basic property of the two-component Boltzmann collision operator as discussed in the next section, we assume that $\FF_{0\pm\infty}(\xi)$ are the spatially homogeneous bi-Maxwellians whose exact definition will be introduced in \eqref{def.bim} and \eqref{def.sgm}, that is, 
\begin{equation}
\label{id.be.bim}
\FF_{0\pm\infty}=\FM_{\pm\infty}=
\left[
\begin{array}{rll}
      &M_{[n_{i\pm},u_{\pm},\theta_\pm;m_i]}(\xi)\\[3mm]
      &M_{[n_{e\pm},u_{\pm},\theta_\pm;m_e]}(\xi)
\end{array}\right],
\end{equation}
where $n_{i\pm}>0$, $n_{e\pm}$, $u_\pm=(u_{1\pm,0,0})$, $\theta_\pm>0$ are given constants, with the quasineutral assumption
\begin{equation}\notag
%\label{ }
q_in_{i\pm}+q_en_{e\pm}=0.
\end{equation}
For later use, for brevity  we always take 
$$
n_{e\pm}=n_{\pm},\quad n_{i\pm}=-\frac{q_e}{q_i}n_{\pm},
$$ 
with given constants $n_\pm>0$. 

A general question is to investigate the existence, uniqueness, regularity and large-time behavior of solutions to the Cauchy problem on the above VPB system in terms of given initial data with general far fields. Note that the far-field data at $x=\pm\infty$ could be distinct, and hence the long-term dynamics could be nontrivial with spatial variation along the direction of $x$ variable.

\subsection{Literature and background}

In what follows we review some relevant literature. First of all, in general settings for large initial data,  the Cauchy problem or the IBVP on the VPB system related to its one-dimensional version \eqref{VPB}, \eqref{PE} has been studied by many people. Among them, we would only mention Desvillettes-Dolbeault \cite{DD} for the long time asymptotics of the system,  Bernis-Desvillettes \cite{BD} for the propagation of regularity of solutions, Mischler \cite{Mis} for the initial boundary value problem, Bostan-Gamba-Goudon-Vasseur \cite{BGGV} for the stationary problem on the bounded domain, and Guo \cite{G-va} for global existence of classical solutions near vacuum. Note that the existence of renormalised solutions of the much more complex Vlasov-Maxwell-Boltzmann system with a defect measure has been recently studied in Arsenio-Saint-Raymond \cite{AS}. 

In perturbation regime around global Maxwellians on the spatially periodic domain $\T^3$, a number of  progresses have been made by Guo   \cite{Guo3,G, Guo-VPL}.  
%where even the magnetic field %and the relativistic effect can be taken into account. 
His approach is based on the robust energy method through constructing the appropriate energy functional and energy dissipation rate functional so that the nonlinear collision terms can be controlled along the linearised dynamics under smallness assumption, where the mathematical analysis strongly relies on both the structure of the system and the dissipative property of the linearised operator. A general technique in the proof is to design good velocity weight functions for closing the a priori estimates.  In the case of the whole space, the Poincar\'e inequality fails to capture the dissipation of solutions over the low-frequency domain, and hence the energy method is also extent to further study the local stability and convergence rates of solutions around global Maxwellians in $\R^3$, for instance, Strain \cite{S}, Duan-Strain \cite{DS}, Strain-Zhu \cite{SZ}, Wang \cite{Wa}, Duan-Liu \cite{DL-CMP}, Duan-Lei-Yang-Zhao \cite{DLYZ}, and many references therein. Recently, the decay structure of the linearized system is characterized by the spectral analysis in  Li-Yang-Zhong \cite{LYZ} and Huang \cite{Hu} following the classical works by  Ellis-Pinsky \cite{EP} and Ukai \cite{U74}; see also Glassey-Struass \cite{GlS} for an early study of spectrum of the VPB system.
%and it would be of interest to further pursue the pointwise kinetic-fluid structure by the construction of Green's functions developed by Liu-Yu \cite{LY}, which turns out to be a more challenging problem.
%
We should point out that the appearance of the self-consistent electric field or the magnetic field makes the dissipative structure of  system more complicated.

A common feature in most of works in perturbative regime mentioned above is that the large-time behavior of solutions to the VPB system is trivial, namely, $F_{i,e}(t,x,\xi)$ are global Maxwellians and $\phi(t,x)$ is a constant. Unfortunately, this property may not be true in the general situation where regarding the VPB system \eqref{VPB} and \eqref{PE}, either initial data $F_{0\al}(x,\xi)$ with $\al=i,e$ tend to two distinct global Maxwellians
\begin{equation*}
%\label{ }
M_{[n_{\al\pm},u_\pm,\theta_\pm;m_\al]}(\xi)=
n_{\al\pm} \left(\frac{m_\al}{2\pi k_B\theta_\pm}\right)^{3/2}\exp\left\{-\frac{m_\al (|\xi_1-u_\pm|^2+|\xi_2|^2+|\xi_3|^2)}{2k_B\theta_\pm}\right\},
\end{equation*}
for a binary gas-mixture or $\phi(t,x)$ tends to two distinct constant states $\phi_\pm$, as $x$ goes to $\pm\infty$, where $k_B>0$ is the Boltzmann constant. Here, as pointed out before, the fact that two Maxwellians have the same bulk velocities and the same temperatures is due to the Boltzmann's  $H$-theorem in the two-component situation; see details in the next section. In such cases, from the local macroscopic balance laws, $F_\al(t,x,\xi)$ and $\phi(t,x)$ are no longer global Maxwellians and constant in large time, respectively. This is the situation considered in the paper, and our main objective is to construct the non-trivial rarefaction wave profile under certain compatibility conditions on far-field data, and further show the local time-asymptotic stability. As a byproduct, those results in the case of the constant-equilibrium state (cf.~\cite{G}) can be recovered when the strength of rarefaction wave reduces to zero.

We further recall  a few literatures for the existence and stability of wave patterns in the content of the pure Boltzmann equation without any force as one may expect to extend them to the VPB system under consideration. These include the shock wave (cf., Caflisch-Nicolaenko \cite{CN}, Liu-Yu \cite{LY}, Yu \cite{Yu2}, Liu-Yu \cite{LY-im}), rarefaction wave (cf., Liu-Yang-Yu-Zhao \cite{LYYZ}, Xin-Yang-Yu \cite{XYY}), contact discontinuity (cf., Huang-Xin-Yang \cite{HXY}); see also many other references therein. Note that the construction of solutions with a general BV data corresponding to the celebrated work Bianchini-Bressan \cite{BB} on the finite-dimensional conservation laws at the fluid level is a big open problem, cf.~\cite{SR}. Regarding the rarefaction wave of the pure Boltzmann equation, one can take it as a local Maxwellian with the macroscopic fluid quantities solving the Riemann problem on the corresponding Euler system with initial data given by both far-field global Maxwellians. For \eqref{VPB}, \eqref{PE} we will explain later on how to construct the rarefaction wave through the quasineutral Euler equations.  To study the local stability of such local Maxwellian, another type of energy method is proposed in Liu-Yu \cite{LY} and developed by Liu-Yang-Yu \cite{LYY}. Here, different from the previous approach by setting perturbations around global Maxwellians, the key idea in \cite{LY,LYY} is to make the macro-micro decomposition for the single-component Boltzmann equation
$$
F(t,x,\xi)=M(t,x,\xi)+G(t,x,\xi),
$$ 
with the local Maxwellian $M(t,x,\xi)$ determined by the solution $F(t,x,\xi)$ itself through conservation laws of mass, momentum and energy, and hence write the Boltzmann equation in the form of the compressible Navier-Stokes equations coupling to high-order moments of the microscopic part $G(t,x,\xi)$. A priori estimates can be made by a combination of the stability analysis of fluid dynamic equations and the kinetic dissipation of $G(t,x,\xi)$ from the $H$-theorem. We note that the nonlinear stability  in large time of wave patterns for the viscous compressible fluid on the whole line has been well studied, for instance, Goodman \cite{Go}, Matsumura-Nishihara \cite{MN86,MN92,MN-S},
%Liu-Xin \cite{LX},
Huang-Xin-Yang \cite{HXY}, see also the monograph \cite{CF} for the general theory. Moreover, hydrodynamic limits of the Boltzmann equation to the classical Euler or Navier-Stokes equations have been also extensively studied by many people in different settings, for instance, see the recent works \cite{GJ,HWWY} in perturbation framework and the monograph \cite{SR} in non perturbation framework.
%Nishida \cite{Ni}, Caflisch \cite{Caf}, Ukai-Asano \cite{UA}, Huang-Wang-Wang-Yang \cite{HWWY}, Bardos-Golse-Levermore \cite{}, Golse-Saint Raymond, Levermore-Masmoudi, and reference therein.

When there is a self-consistent force, few results are known on the stability of wave patterns for the kinetic equation. A natural starting point is to look at the corresponding fluid dynamic approximate equations. In what follows, let us mainly focus on the rarefaction wave; the issue on the shock wave or contact discontinuity, even only regarding the existence, should be a completely different problem; see the Sone's book  \cite{Sone} and reference therein. In \cite{DY}, Duan-Yang proposed to study the following two-fluid system in the isothermal case
%\begin{equation}\label{nsp}
%\left\{\begin{array}{l}
%\pa_tn_i+\pa_x(n_i u_i)=0,\\
%m_in_i(\pa_t u_i+u_i\pa_x u_i)+T_i\pa_x n_i+n_i\pa_x\phi=\mu_i \pa_x^2u_i,\\
%\pa_tn_e+\pa_x(n_e u_e)=0,\\
%m_en_e(\pa_t u_e+u_e\pa_x u_e)+T_e\pa_x n_e-n_e\pa_x\phi=\mu_e \pa_x^2u_e,\\
%-\pa_x^2\phi=n_i-n_e,
%\end{array}
%\right.
%\end{equation}
\begin{equation}\label{nsp}
\left\{\begin{array}{l}
\pa_tn_\al+\pa_x(n_\al u_\al)=0,\\[3mm]
m_\al n_\al (\pa_t u_\al+u_\al\pa_x u_\al)+T_\al\pa_x n_\al+q_\al n_\al \pa_x\phi=\mu_\al \pa_x^2u_\al,\quad \al=i,e,\\[3mm]
%\pa_tn_e+\pa_x(n_e u_e)=0,\\
%m_en_e(\pa_t u_e+u_e\pa_x u_e)+T_e\pa_x n_e-n_e\pa_x\phi=\mu_e \pa_x^2u_e,\\
-\pa_x^2\phi=q_in_i+q_en_e,
\end{array}
\right.
\end{equation}
which is called the Navier-Stokes-Poisson system due to the appearance of diffusion terms. Here $T_\al>0$, $\mu_\al>0$ are constant temperatures and viscosity coefficients, respectively. Note, as pointed out in \cite{Ch}, that for a collisionless fluid plasma, the Euler-Poisson system is enough to describe the monition of charged particles, and the global existence of classical solutions close to constant steady state has been recently proved in Guo-Ionescu-Pausader \cite{GIP} in the case of the whole space $\R^3$. Since we are interested in the study of \eqref{VPB}, \eqref{PE} in the context of collisional plasma, it could be a good way to make use of the theory of the viscous compressible fluid with self-consistent  forces. We established in \cite{DY} the global-in-time stability of the rarefaction wave and the boundary layer for the outflow problem on  \eqref{nsp} on the half line. A drawback of the result is that the large-time behavior of the electric field is zero, due to an artificial choice of physical constants, namely, 
$$
m_i=m_e,\ T_i=T_e,\ \mu_i=\mu_e,\ q_i+q_e=0,
$$
and hence the dynamics of the two-fluid NSP system is the same as the one of the single NS system. However, we recovered a good dissipative property of the electric field, that is, although $\pa_x\phi$ is not time-space integrable, it can be true for $(\pa_xu^r)^{1/2}\pa_x\phi$ by using the two-fluid coupling property, where $\pa_xu^r>0$ has a good sign.

Recently,  we  removed in \cite{DL} the restrictions on those physical constants. Particularly, it is found that as long as initial data satisfy some compatibility conditions related to the construction of the rarefaction wave, the dynamics of system \eqref{nsp} can be described in large time by the
corresponding quasineutral Euler system
\begin{equation}\notag%\label{nsp}
\left\{\begin{array}{l}
\pa_tn_\al+\pa_x(n_\al u_\al)=0,\\[3mm]
m_\al n_\al (\pa_t u+u\pa_x u)+T_\al\pa_x n_\al+q_\al n_\al \pa_x\phi=0,\quad \al=i,e,\\[3mm]
%\pa_tn_e+\pa_x(n_e u_e)=0,\\
%m_en_e(\pa_t u_e+u_e\pa_x u_e)+T_e\pa_x n_e-n_e\pa_x\phi=\mu_e \pa_x^2u_e,\\
q_in_i+q_en_e=0,
\end{array}
\right.
\end{equation}
by formally assuming $u_i=u=u_e$ and ignoring all the second-order derivative terms. Note that by letting $n_e=n$ and $n_i=-\frac{q_e}{q_i}n$, the above quasineutral Euler system can further reduce to the form of
\begin{equation}%\label{Euler2}
\left\{\begin{array}{l}
\dis \pa_tn+\pa_x(n u)=0,\\[3mm]
\dis n(\pa_t u+u\pa_x u)+\frac{T_i |q_e|+T_e|q_i|}{m_i |q_e|+m_e|q_i|}\pa_x n=0,
\end{array}
\right.\notag
\end{equation}
with the potential function $\phi$ given by
$$
\phi=\frac{m_iT_e-m_e T_i}{m_i|q_e|+m_e|q_i|} \ln n.
$$
The similar result has been also extent to the non-isentropic two-fluid case in \cite{DLiuYZ} with some technical restriction on the ratio of masses of two fluids. Furthermore, in a parallel work \cite{DL-VPB} we also make use of the same idea to further have studied the stability of the rarefaction wave of the VPB system for ions' dynamics governed by the model of the form
\begin{equation}%\label{VPB.ions}
\left\{\begin{array}{l}
\dis \pa_t F_{i} +\xi_1\pa_x F_{i}-\dis\frac{q_i}{m_i}\pa_x\phi\,\pa_{\xi_1}F_i =Q_{ii}(F_i,F_i),\\[3mm]
\dis -\pa_x^2\phi=q_i\dis{\int_{\R^3}F_i}\,d\xi+q_en_e,
\end{array}
\right.\notag
\end{equation}
where compared to the two-component VPB system \eqref{VPB}, the dynamical equation of electrons and the ions-electrons collisions have been omitted, and the number density $n_e=\int_{\R^3} f_e\,d\xi $ has been replaced by an analogue of the classical Boltzmann relation
$$
n_e=\exp\{\phi/T_e\},
$$
or  a general function depending on the potential function $\phi$. We remark that the Boltzmann relation has been recently extensively visited in a lot of studies of kinetic and related fluid dynamic equations, for instance,  \cite{CDPS,GP,GPu,HK-VP,HK,Su}.

Inspired by our previous works \cite{DL,DL-VPB,DLiuYZ}, we expect in the paper to further consider the  much more physical two-component VPB system, particularly extending the results in \cite{G,Guo-VPL} to the case of perturbations of the non-constant equilibrium state. In fact, besides its own importance in physics, the two-component collisional kinetic system enjoys  more complex  dissipation structure compared to either the modelling system studied in \cite{DLiuYZ} or the single-component kinetic system, cf.~\cite{AAP,CDMS,DL1,DL2,EGM-1,EGM-2,GS,SoY}. For the numerical and mathematical investigations on non-trivial profiles of a gas mixture with the Boltzmann collision, we would mention \cite{ABT,BG,Br,KAT,Ta,TaAo} and reference therein; see also discussions in \cite{BiDo,BiDe,Do,JM} on the limit of the gas mixture kinetic equations to the fluid dynamical equations. 

For the two-component VPB system \eqref{VPB}, \eqref{PE} under consideration, disparate masses and charges play a key role in the stability analysis of non-constant time-asymptotic profiles, which is different from the one for considering perturbations around constant equilibrium states where all physical constants can be normalised to be one without loss of generality, cf.~\cite{Guo3}. Moreover, as discussed in \cite{GS},  the dissipation by the two-component  Boltzmann collisions behaves in a complex way, and the approach to equilibrium can be divided roughly into two processes: one is called the Maxwellization which occurs due to either self-collisions alone, or cross-collisions, or a combination of both, and the other is called  equilibration of two species, i.e., the vanishing of differences in velocity and temperature in the species. In the paper, we expect to provide an analytical view to this issue by further developing the macro-micro decomposition in the two-component case.

\subsection{Main result}

We now begin to state the main result of the paper. Before doing that, we first introduce some notations. Let  $[n^R(x/t),u^R(x/t),\ta^R(x/t)]$ and $[n^r(t,x),u^r(t,x),\ta^r(t,x)]$ with $u^R=(u^R_1,0,0)$ and $u^r=(u_1^r,0,0)$, where the far-field data  at $x=\pm\infty$ are given by $[n_\pm,u_{1\pm},\theta_\pm]$, be the $3$-family centred rarefaction wave and the corresponding smooth rarefaction wave, respectively,
%constructed in \eqref{Org.RW.} and \eqref{1-RW.def.}, respectively,
in connection with the quasineutral Euler system
%\eqref{cons.law.f.qe} or \eqref{cons.law.f.as}.
\begin{eqnarray}\label{qne}
\left\{
\begin{array}{clll}
&\dis \pa_tn+\pa_x (nu_1)=0,\\[3mm]
&\dis \pa_t u_1+u_1\pa_xu_1+\frac{2(|q_i|+|q_e|)}{3(m_e|q_i|+m_i|q_e|)}\frac{1}{n}\pa_x (n\theta)=0,\\[3mm]
&\dis \pa_t \theta+u_1\pa_x \theta+\frac{2}{3} \theta \pa_x u_1=0.
\end{array}
\right.
\end{eqnarray}
See Section \ref{sec.rw} for more details to the derivation of the system \eqref{qne}. As in \cite{XZ}, we define
%$$\FM_*(\xi)=\FM_{[1/v_*,u_*,\ta_*]}(\xi)
%$$
\begin{eqnarray}\label{def.mb}
%\FM_*(\xi)=
\left[\begin{array}{cc}
      M_{\ast i}    \\[3mm]
      M_{\ast e}
\end{array}\right]
=\left[\begin{array}{cc}
      M_{[n_{*i},u_{*},\ta_{*};m_i]}(\xi)    \\[3mm]
      M_{[n_{*e},u_{*},\ta_{*};m_e]}(\xi)
\end{array}\right],
\end{eqnarray}
with constants $n_{*i}$, $n_{*e}$,  $u_*=[u_{*1},0,0]$, $\ta_*$ satisfying
\begin{eqnarray}\label{def.mbp}
\left\{\begin{array}{rll}
\begin{split}
&\frac{1}{2}\sup\limits_{(t,x)\in \R_+\times \R}\ta^r(t,x)<\ta_*<\sup\limits_{(t,x)\in \R_+\times \R}\ta^r(t,x),\\[3mm]
&\sup\limits_{(t,x)\in \R_+\times \R}\left\{\left|-\frac{q_e}{q_i}n^r(t,x)-n_{*i}\right|+|n^r(t,x)-n_{*e}|+|u^r(t,x)-u_*|+|\ta^r(t,x)-\ta_*|\right\}<\eta_0,
\end{split}
\end{array}\right.
\end{eqnarray}
for a constant $\eta_0>0$ which is suitably small.

Then the main result of this paper can be stated as follows. More notations will be explained later on.

\begin{theorem}\label{main.Res.}
Consider the Cauchy problem on the VPB system \eqref{VPB}, \eqref{PE}, \eqref{id.be}, \eqref{bd.be}, \eqref{bd.phi}, \eqref{id.be.bim}. Assume $[n_+,u_{1+},\theta_+]\in R_3(n_-,u_{1-},\theta_-)$, $\phi_+=\phi_-$, and $q_i\leq 9 |q_e|$, where $R_3(n_-,u_{1-},\theta_-)$ is defined in \eqref{def.r3} denoting the set of right constant states connected with the left constant state $[n_-,u_{1-},\theta_-]$ through the $3$-family rarefaction wave of the quasineutral Euler system \eqref{qne}.
%$\frac{q_i}{|q_e|}\leq 9$.
Let
\begin{equation}
\label{def.rws}
\delta_r=|n_+-n_-|+|u_{1+}-u_{1-}|+|\ta_+-\ta_-|,
\end{equation}
be the wave strength which is suitably small. There are constants $\eps_0>0$, $0<\eta_0\leq\de_r$ and $C_0>0$,
%$$
%\delta_r=|v_+-v_-|+|u_+-u_-|+|\ta_+-\ta_-|<\eta_0,\ \sup\limits_{(t,x)\in\R^+\times\R}\ta^r(t,x)<2\inf\limits_{(t,x)\in\R^+\times\R}\ta^r(t,x),
%$$
such that if $F_{0i}(x,\xi)\geq 0$, $F_{0e}(x,\xi)\geq 0$,  and
\begin{equation*}%\label{ID.VPB-E}
\begin{split}
&\sum\limits_{|\al|+|\be|\leq2\atop{0\leq \al_0\leq1}}\left\|\pa^{\al}\pa^\be\left(F_{0i}(x,\xi)
-M_{[-\frac{q_e}{q_i}n^r,u^r,\ta^r;m_i](0,x)}(\xi)\right)\right\|^2_{L_x^2\left(L^2_\xi\left(\frac{1}{\sqrt{M_{*i}(\xi)}}\right)\right)}
\\&\quad+\sum\limits_{|\al|+|\be|\leq2\atop{0\leq \al_0\leq1}}\left\|\pa^{\al}\pa^\be\left(F_{0e}(x,\xi)
-M_{[n^r,u^r,\ta^r;m_e](0,x)}(\xi)\right)\right\|^2_{L_x^2\left(L^2_\xi\left(\frac{1}{\sqrt{M_{*e}(\xi)}}\right)\right)}
\\&\quad+\sum\limits_{|\al|\leq1}\left\|\pa^\al\pa_x\phi(0,x)\right\|_{H^1}^2
+\de_r\\
&\leq\eps^2_0,
\end{split}
\end{equation*}
then %the VPB system \eqref{VPB} with  \eqref{bd.be} and \eqref{bd.phi}
the  Cauchy problem admits a unique global solution $[F_i(t,x,\xi)\geq 0,F_e(t,x,\xi)\geq0, \phi(t,x)]$ satisfying
\begin{equation}\label{VPB.sol.}
\begin{split}
&\sup\limits_{t\geq0}\sum\limits_{|\al|+|\be|\leq2\atop{0\leq \al_0\leq1}}\left\|\pa^\al \pa^\be\left(F_i(t,x,\xi)-M_{[-\frac{q_e}{q_i}n^r,u^r,\ta^r;m_i](t,x)}(\xi)\right)
\right\|^2_{L_x^2\left(L^2_\xi\left(\frac{1}{\sqrt{M_{*i}(\xi)}}\right)\right)}
\\&\quad+\sup\limits_{t\geq0}\sum\limits_{|\al|+|\be|\leq2\atop{0\leq \al_0\leq1}}\left\|\pa^\al \pa^\be\left(F_e(t,x,\xi)-M_{[n^r,u^r,\ta^r;m_e](t,x)}(\xi)\right)
\right\|^2_{L_x^2\left(L^2_\xi\left(\frac{1}{\sqrt{M_{*e}(\xi)}}\right)\right)}
\\&\quad+\sup\limits_{t\geq0}\sum\limits_{|\al|\leq1}\|\pa^\al\pa_x\phi(t,x)\|^2_{H^1}\\
&\leq C_0\eps_0^{2}.
\end{split}
\end{equation}
Moreover, it holds that
\begin{multline}\label{sol.Lab}
\sup\limits_{t\rightarrow +\infty}\sup\limits_{x\in\R}\Bigg\{\left\|F_i(t,x,\xi)-M_{[-\frac{q_e}{q_i}n^R,u^R,\ta^R](x/t)}(\xi)
\right\|_{L^2_\xi\left(\frac{1}{\sqrt{M_{*i}(\xi)}}\right)}
\\+\left\|F_e(t,x,\xi)-M_{[n^R,u^R,\ta^R](x/t)}(\xi)
\right\|_{L^2_\xi\left(\frac{1}{\sqrt{M_{*e}(\xi)}}\right)}\Bigg\}
=0.
\end{multline}
%\begin{equation}\label{sol.Lab}
%\begin{split}
%\sup\limits_{t\rightarrow +\infty}\sup\limits_{x\in\R}&\Bigg\{\left\|F_i(t,x,\xi)-M_{[-\frac{q_e}{q_i}n^R,u^R,\ta^R](x/t)}(\xi)
%\right\|_{L^2_\xi\left(\frac{1}{\sqrt{\FM_{*i}(\xi)}}\right)}
%\\&\qquad\qquad+\left\|F_e(t,x,\xi)-M_{[n^R,u^R,\ta^R](x/t)}(\xi)
%\right\|_{L^2_\xi\left(\frac{1}{\sqrt{\FM_{*e}(\xi)}}\right)}\Bigg\}
%=0.
%\end{split}
%\end{equation}
\end{theorem}

We give a few remarks on the above theorem. The estimate  \eqref{sol.Lab} indeed shows the convergence of the two-component VPB system \eqref{VPB}, \eqref{PE} to the quasineural Euler system \eqref{def.r3} in the setting of rarefaction waves for well-prepared small and smooth initial data. Thus, the long-term dynamcis of the VPB system can be a non-trivial time-asymptotic profile connecting two distinct constant equilibrium states. As seen in \eqref{qne},  disparate masses and charges of particles also enter into the asymptotic profile and hence they can take the effect on the nontrivial large time behavior of the complex VPB system. The obtained result may be regarded as a generalisation of the existing perturbation theory for the VPB system in the cases either for initial data around constant equilibrium states in \cite{G} or for the single-component Boltzmann collision in \cite{LYYZ} and \cite{DL-VPB}. More importantly, although we may only provide a preliminary understanding of the stability of the rarefaction wave profile for the VPB system, it is expected that the analysis developed in the paper could be adopted to treat many other relevant problems in connection with those fluid-type systems derived in Section \ref{sec2}, cf.~\cite{Gr}. 

In the end we point out that the condition $q_i\leq 9|q_e|$ is only a technical assumption used in the proof of the zero-order energy estimate; see \eqref{def.matrM} for its positivity in Section \ref{sec6.1}. On the other hand, the condition $\phi_+=\phi_-$ is essentially required in the proof of the theorem, and it is indeed unknown how to construct a non-trivial large-time profile of the potential function $\phi$ associated with $\phi_+\neq \phi_-$ as we did in \cite{DL,DL-VPB,DLiuYZ}.

\subsection{Outline and key points of the proof}

%\Red{To be added}

The proof of Theorem \ref{main.Res.} is based on the two-component decomposition as well as the refined energy method. First of all, $H$-theorem of the two-component Boltzmann equations implies that the large-time behavior of the VPB system should be in connection with a bi-Maxwellian $\FM$ determined by six local fluid quantities $n_i$, $n_e$, $u=(u_1,u_2,u_3)$, and $\theta$. This induces one to define $\FM$ in terms of $\FF$ such that they have the same average values with respect to all six two-component collision invariants, namely, 
%through 
$$
\int_{\R^3}\psi_j(\xi)\cdot \FF\,d\xi=\int_{\R^3}\psi_j(\xi)\cdot \FM\,d\xi,\quad j=1,2,...,6.
$$
%where $\psi_j(\xi)$ $(1\leq j\leq 6)$ are the two-component collision invariants. 
Therefore, the energy dissipation of the non-fluid part $\FG:=\FF-\FM$ can be obtained by the linearised $H$-theorem. See the coercivity estimate \eqref{co.est00.} in Lemma \ref{co.est0.} whose proof is based on the compactness argument as in \cite{Guo3}. In most applications of \eqref{co.est00.}, one has to vary the weight function such that the modified macroscopic quantities are sufficiently close to those of the background bi-Maxwellian, and this has been done in Lemma \ref{co.est.}. Moreover, as in \cite{XZ}, it  can not be direct to make the zero-order energy estimate on $\FG$, because $\FM^{-1/2}\FG$ is not  integrable in $L^2_{t,x,\xi}$. To treat this trouble, one has to construct a background non-fluid function $\overline{\FG}$ in terms of the time-asymptotic fluid profile, see \eqref{def.ng} for the exact formula. We would emphasise that as the large-time profile of $\phi(t,x)$ under the assumption $\phi_+=\phi_-$ is expected to be constant, $\overline{\FG}$ does not involve any term of the potential function, which is quite different from the previous work \cite{DL-VPB} in the single-component case. Due to this technique, it seems impossible for us to construct a non-trivial large-time potential function $\phi^r(t,x)$ accounting for some distinct far-field data $\phi_\pm$ similar to the two-fluid models considered in \cite{DL,DLiuYZ}. 

The a priori estimates on the fluid part $\FM$ of the solution $\FF$ is much more technical. The key point is to find out the appropriate viscous fluid-type equations of the macroscopic quantities of $\FM$ such that the energy estimates on the fluid part can be controled in terms of the non-fluid part in a good way; see Proposition \ref{mac.eng.lem.}. Considering the two-component moment equations with respect to all collision invariants, cf.~\eqref{co.i} and \eqref{co.e}, 
%$$
%\dis \int_{\R^3}\psi_{j\CA}\left(\pa_t F_\CA+\xi_1\pa_x F_\CA-\frac{q_\CA}{m_\CA}\pa_x\phi \pa_{\xi_1}F_\CA\right)\,d\xi=\int_{\R^3}\psi_{j\CA} Q_\CA(\FF,\FF)\,d\xi,
%$$  
%with $j=1,2,...,6$ and $\CA=i,e$, 
and using the two-component macro-micro decomposition, it is straightforward to obtain the two-fluid Euler-Poisson type system \eqref{cons.law.i}, \eqref{cons.law.e}, \eqref{co.po}. To capture the diffusion and heat-conductivity, we essentially have used the dissipation effect of like-particle collisions. In fact, using the decomposition, one can rewrite $Q_{\CA}(\FF,\FF)$ in the way on the right-hand side of \eqref{mmBE}, and hence get the representation \eqref{PGA}, where we note that the right-hand first term is exactly responsible for diffusion and heat-conductivity and the remaining term $\overline{R}_\CA$ does not involve any linear term in $\phi(t,x)$. Therefore, by plugging \eqref{PGA} into the two-fluid Euler-Poisson type system, one can further obtain the two-fluid Navier-Stokes-Poisson type system \eqref{BE-NSi}, \eqref{BE-NSe} and \eqref{co.po}, which becomes the key step for making the energy estimates on the fluid part as in \cite{DLiuYZ}.   
%can write that for the like-particle collisions,
%$$
%Q_{\CA\CA} (F_\CA,F_\CA)=Q_{\CA\CA} (M_\CA,G_\CA)+Q_{\CA\CA}(G_\CA,M_\CA)+Q_{\CA\CA}(G_\CA,G_\CA),
%$$
%and for the unlike-particle collisions,

%\vspace{3cm}

\medskip

The rest of the paper is arranged as follows. In the following three sections we make some preparations for the proof of the main result. Particularly, in Section \ref{sec2}, we introduce the macro-micro decomposition for the two-component Boltzmann equation with disparate masses. In terms of the decomposition, we derive the zero-order and first-order approximate fluid-type systems, which is a crucial step for both the construction of large-time rarefaction wave profiles and the energy estimates on the stability of profiles. Note that we also make use of the single-component projections to find out the diffusion and heat-conductivity of the fluid part. In Section \ref{sec.rw}, we deduce the quasineutral Euler system as the time-asympotic equations of the VPB system, and further construct the corresponding rarefaction wave profile and study the basic properties of the profile. In Section \ref{sec4}, we consider the two-component Boltzmann collision operator and provide estimates on dissipation of the linearised operator and also upper bound estimates on the nonlinear term both with respect to some local bi-Maxwellians. In Section \ref{sec5}, we give a sketch of the proof of Theorem \ref{main.Res.} basing on two main propositions whose proofs are postponed to  Section \ref{sec.est.f} and Section \ref{sec.est.nf}, respectively.

\medskip

\noindent {\it Notations.} Throughout the paper,  $C$ denotes some generic positive (generally large) constant and $\la$ denotes some generic positive (generally small) constant, where both $C$ and $\la$ may take different values in different places. $D\lesssim E$ means that  there is a generic constant $C>0$
such that $D\leq CE$. $D\sim E$
means $D\lesssim E$ and $E\lesssim D$. $\|\cdot\|_{L^p}$ $(1\leq p\leq+\infty)$
stands for the $L_x^p-$norm. Sometimes, for convenience, we use $\|\cdot\|$ to denote $L_x^2$-norm, and use $(\cdot,\cdot)$ to denote the inner product in $L^2_x$ or $L^2_{x,\xi}$.
We also use $H^{k}$ $(k\geq0)$ to denote the usual Sobolev space with respect to $x$ variable. We denote  $\pa^\al\pa^\be=\pa_t^{\al_0}\pa_x^{\al_1}\pa_\xi^{\be}$ and $\pa_\xi^\be=\pa_{\xi_1}^{\be_1}\pa_{\xi_2}^{\be_2}\pa_{\xi_3}^{\be_3}$, with $|\al|=\al_0+\al_1$ and  $|\be|=\be_1+\be_2+\be_3.$
We call $\be'\leq \be$ if each
component of $\be'$ is not greater than that of $\be$. %, we denote the condition by $\be'\leq \be$.
We also call $\be'<\be$ if $\be'\leq \be$ and $|\be'|<|\be|$.
 For $\be'\leq \be$, we also use $C^{\be}_{\be'}$ to denote the usual binomial coefficient.
% $\begin{pmatrix}
%      \al     \\
%      \al'
%\end{pmatrix}$.
The same notations  also apply to $\al$ and $\al'$. For the notational simplicity, we use $\FM_\ast^{-1}$ to denote the $2\times2$ diagonal matrix
${\rm diag}(1/M_{\ast i}, 1/M_{\ast e})$. Similarly, the $2\times2$ diagonal matrix
${\rm diag}(1/\sqrt{M_{\ast i}}, 1/\sqrt{M_{\ast e}})$ is denoted by $\FM_\ast^{-1/2}$. The same notations  also apply to all bi-Maxwellians used in the paper, for instance, $\FM$, $\FM_\sharp$ and $\widehat{\FM}$, etc.

%\newpage

%\subsection{The problem}

\section{Two-component macro-micro decomposition}\label{sec2}

In this section we introduce the two-component macro-micro decomposition. First of all, we list the elementary properties of the collision operator, including the local equilibrium state, an identity, collision invariants, and the entropy inequality. An important and interesting concept is the bi-Maxwellian, cf.~\cite{ABT}. After that, we introduce the fluid quantities for a disparate mass binary mixture, define the macro-micro decomposition of the solution, and derive the zero-order macroscopic balance laws. In the end, we discuss how to capture the velocity diffusion and heat conductivity.

\subsection{Elementary properties of collisions}
%
%\Red{\begin{itemize}
%  \item Identity
%  \item Collision invariants
%  \item Balance laws
%  \item Boltzmann inequality, entropy, and bi-Maxwellian
%\end{itemize}}

%\Blue{\bf [P0].} Equilibrium state of single $\CA$-species ($\CA=i$ or $e$):

In what follows we list a few elementary properties of the two-component Boltzmann collision operator without any proof. Interested readers may refer to \cite{ABT,CC}. To the end, we always denote
\begin{equation}
\label{def.sgm}
M_{[n(t,x),u(t,x),\theta(t,x);m]}(\xi)=n(t,x) \left(\frac{m}{2\pi k_B \theta(t,x)}\right)^{3/2} \exp \left(-\frac{m |\xi-u(t,x)|^2}{2k_B\theta(t,x)}\right),
\end{equation}
to be a local Maxwellian with the fluid density $n(t,x)$, bulk velocity $u(t,x)$, and temperature $\theta(t,x)$ as well as the particle mass $m>0$.

\bigskip

\noindent {\bf [P1].} For the like-particles collision  ($\CA=i$ or $e$),
$$
Q_{\CA\CA}(F_\CA,F_\CA)=0\ \ \text{iff}\ \ F_\CA=\overline{M}_\CA,
$$
where
%$\overline{M}_\CA$ is a general local Maxwellian of $\CA$-species, defined by
\begin{equation}%\label{a.M}
 \overline{M}_\CA:=M_{[n_\CA(t,x),u_\CA(t,x),\theta_\CA(t,x);m_\CA]}(\xi),\notag
% \\
% &=& n_\CA(t,x) \left(\frac{m_\CA}{2\pi k_B \theta_\CA(t,x)}\right)^{3/2} \exp \left(-\frac{m_\CA |\xi-u_\CA(t,x)|^2}{2k_B\theta_\CA(t,x)}\right).
\end{equation}
is a general local Maxwellian of $\CA$-species.
For later use it is also convenient to rewrite $\overline{M}_\CA$  as
\begin{equation}\notag
%\label{ }
 \overline{M}_\CA=\frac{n_\CA(t,x)}{\big(2\pi k_\CA \theta_\CA(t,x)\big)^{3/2}} \exp \left\{-\frac{|\xi-u_\CA(t,x)|^2}{2k_\CA\theta_\CA(t,x)}\right\},
\end{equation}
with $k_\CA:=\frac{k_B}{m_\CA}$, and for brevity we always take $k_B=\frac{2}{3}$. For the unlike-particles collision ($\CA\neq \CB$),
\begin{equation*}
%\label{ }
Q_{\CA\CB}(\overline{M}_\CA,\overline{M}_\CB)=0, \ \ \text{provided that}\ u_\CA=u_\CB \ \text{and}\  \theta_\CA=\theta_\CB.
\end{equation*}

%To get all kinds of equilibrium states for multi-species particles, we need  \underline{the elementary identity}:\\[6mm]

\medskip
\noindent {\bf [P2].}  For $\FF=[F_i,F_e]^T$,  we set
$$
\FQ(\FF,\FF):=
\left[\begin{array}{l}
\FQ_i(\FF,\FF)\\[3mm]
\FQ_e(\FF,\FF)\end{array}\right]
=\left[\begin{array}{c}
Q_{ii}(F_i,F_i)+Q_{ie}(F_i,F_e)\\[3mm]
Q_{ee}(F_e,F_e)+Q_{ei}(F_e,F_i)\end{array}\right].
$$
Then, for $\FG=[G_i,G_e]^T$, one has
\begin{eqnarray}
&\dis \big(\FQ(\FF,\FF),\FG\big)_{L^2_\xi\times L^2_\xi}=-\frac{1}{4}I_{ii}(F_i,G_i)-\frac{1}{2}I_{ie}(F_i,F_e,G_i,G_e)-\frac{1}{4}I_{ee}(F_e,G_e),\notag
\end{eqnarray}
with
\begin{eqnarray*}
I_{ii}(F_i,G_i)  &=&  \int_{\R^3\times \R^3\times \S^2} [F_i(\xi_\ast')F_i(\xi')-F_i(\xi_\ast)F_i(\xi)]\\
&&\qquad\qquad\quad \times [G_i(\xi_\ast')+G_i(\xi')-G_i(\xi_\ast)-G_i(\xi)] B(|\xi-\xi_\ast|,\omega)\, d \xi d\xi_\ast d\omega, \\
 I_{ee}(F_e,G_e)  &=&  \int_{\R^3\times \R^3\times \S^2}  [F_e(\xi_\ast')F_e(\xi')-F_e(\xi_\ast)F_e(\xi)]\\
&&\qquad\qquad\quad \times   [G_e(\xi_\ast')+G_e(\xi')-G_e(\xi_\ast)-G_e(\xi)] B(|\xi-\xi_\ast|,\omega)\,d \xi d\xi_\ast d\omega, \\
  I_{ie}(F_i,F_e,G_i,G_e)& =&\int_{\R^3\times \R^3\times \S^2}  [F_i(\xi_\ast')F_e(\xi')-F_i(\xi_\ast)F_e(\xi)]\\
&&\qquad\qquad\quad \times  [G_i(\xi_\ast')+G_e(\xi')-G_i(\xi_\ast)-G_e(\xi)] B(|\xi-\xi_\ast|,\omega)\, d \xi d\xi_\ast d\omega.
\end{eqnarray*}

\medskip
\noindent {\bf [P3].} Two-component Boltzmann collision operator $\FQ$ has six collision invariants:
%$\psi_{j}$ ($j=1,2,3,4,5,6$) are the six collision invariants
\begin{eqnarray*}%\label{six.in}
\psi_{1}=\left[\begin{array}{cc}
      m_i    \\
   0
\end{array}\right],\
\psi_{2}=\left[\begin{array}{cc}
      0    \\
      m_e
\end{array}\right],\
\psi_{j}=\left[\begin{array}{cc}
      m_i \xi_j    \\
      m_e\xi_j
\end{array}\right],\ j=3,4,5,\
\psi_{6} =\left[\begin{array}{cc}
      \frac{1}{2}m_i|\xi|^2    \\[3mm]
      \frac{1}{2}m_e|\xi|^2
\end{array}\right],
\end{eqnarray*}
satisfying
$$
\int_{\R^3}\psi_j\cdot \FQ(\FF,\FF)\,d\xi=0,\quad j=1,2,...,6.
$$
Specifically, it holds that
$$
\int_{{\R}^3} \psi_{1i}Q_{\CA\CB}(F_\CA, F_{\CB})\,d\xi=\int_{{\R}^3} \psi_{2e}Q_{\CA\CB}(F_\CA, F_{\CB})\,d\xi=0,\ \
\textrm{for} \ \CA\in\{i,e\},
$$
$$
\int_{{\R}^3}\psi_{j\CA}
Q_{\CA\CA}(F_\CA, F_{\CA})  \,d\xi=0,\ j=3,4,5,6,\ \ \textrm{for} \ \CA,\CB\in\{i,e\},
$$
and
$$
\int_{{\R}^3}\psi_{j\CA}
Q_{\CA\CB}(F_\CA, F_{\CB})  \,d\xi+\int_{{\R}^3}\psi_{j\CB}
Q_{\CB\CA}(F_\CB, F_{\CA})  \,d\xi=0,\ \ \textrm{for}\  \CA\neq\CB.
$$
Here for $1\leq j\leq6$, $\psi_{ji}$ and $\psi_{je}$ stand for the first and second component of the vector-valued function $\psi_{j}$.

\medskip
\noindent {\bf [P4].} For any $\FF=[F_i,F_e]^T$,
$$
\big(\FQ(\FF,\FF),\ln \FF\big)_{L^2_\xi\times L^2_\xi}:=\left(\FQ(\FF,\FF),\left[\begin{array}{l}
\ln F_i\\
\ln F_e
\end{array}\right]\right)_{L^2_\xi\times L^2_\xi}=\sum_{\CA=i,e}\int_{\R^3}Q_\CA(\FF,\FF)\ln F_\CA\,d\xi\leq 0,
$$
and ``$=$" holds iff $\FF=\FM$ is a bi-Maxwellian defined by
\begin{eqnarray}
\label{def.bim}
\FM=\left[\begin{array}{cc}
      M_i    \\[3mm]
      M_e
\end{array}\right]
=\left[\begin{array}{cc}
      M_{[n_i,u,\theta;m_i]}(\xi)    \\[3mm]
      M_{[n_e,u,\theta;m_e]}(\xi)
\end{array}\right].
\end{eqnarray}
Particularly, if $\FQ(\FF,\FF)=0$ then $\FF$ is a bi-Maxwellian. Here, we emphasise that for $\CA=i$ or $e$, $M_\CA$ is different from $\overline{M}_\CA$. In fact, the bi-Maxwellian $\FM$ is a two-component equilibrium state, with $M_i$, $M_e$ being the first and second component of  $\FM$, and  $\overline{M}_i$, $\overline{M}_e$ are  Maxwellians of $i$-species and $e$-species, respectively, which are single-species equilibrium states.

%\subsection{Macro-micro decomposition around local bi-Maxwellian}
\subsection{Decomposition around local bi-Maxwellian}

As in the single-component case \cite{LYY}, we introduce the two-component macro-micro decomposition around local bi-Maxwellians in the following way. Let $\FF=\FF(t,x,\xi)$ be a function satisfying the two-component VPB system \eqref{VPB}. We decompose it as
\begin{equation}\label{dec.2}
\FF(t,x,\xi)=\FM(t,x,\xi)+\FG(t,x,\xi).
\end{equation}
Here $\FM=\FM(t,x,\xi)$ is the macroscopic (or fluid) part
represented
by the local bi-Maxwellian
\begin{eqnarray}\notag
%\label{bi.M}
\FM=\left[\begin{array}{cc}
      M_i    \\[3mm]
      M_e
\end{array}\right]
=\left[\begin{array}{cc}
      M_{[n_i(t,x),u(t,x),\theta(t,x);m_i]}(\xi)    \\[3mm]
      M_{[n_e(t,x),u(t,x),\theta(t,x);m_e]}(\xi)
\end{array}\right],
\end{eqnarray}
such that for all fix collision invariants,
\begin{equation}\notag
%\label{ }
\int_{\R^3} \psi_j (\xi)\cdot [\FF(t,x,\xi)-\FM(t,x,\xi)]\,d\xi=0,\quad j=1,2,...,6.
\end{equation}
Note that $\FM(t,x,\xi)$ involves the exact six macroscopic quantities
$$
\Big[n_i(t,x),n_e(t,x),u(t,x)=\big(u_1(t,x),u_2(t,x),u_3(t,x)\big),\theta(t,x)\Big],
$$
which can be determined by
\begin{eqnarray*}%\label{macroscopic-component}
\left\{
\begin{array}{rll}
\begin{split}
m_in_i &\equiv \int_{{\R}^3}\psi_{1}\cdot \FF(t,x,\xi)\,d\xi,\\
 m_en_e &\equiv \int_{{\R}^3}\psi_{2}\cdot \FF(t,x,\xi)\,d\xi, \\
(m_{i}n_i+m_en_e)u_j &\equiv\int_{{\R}^3}
\psi_{j}\cdot \FF(t,x,\xi)\,d\xi,\ \ j=3,4,5,\\
n_i\left(\ta+\frac{1}{2}m_i|u(t,x)|^2\right)+n_e\left(\ta+\frac{1}{2}m_e|u(t,x)|^2\right) &\equiv
\int_{{\R}^3}\psi_{6}\cdot \FF(t,x,\xi)\,d\xi.
\end{split}
\end{array}\right.
\end{eqnarray*}
Therefore, $\FM$ is well defined, and then $\FG:=\FF-\FM$ is the microscopic (or non-fluid) part denoted by
\begin{equation*}%\label{bG}
\FG=\FG(t,x,\xi)=\left[\begin{array}{cc}
      G_i (t,x,\xi)\\[3mm]
      G_e (t,x,\xi)
\end{array}\right].
\end{equation*}
We remark that  $\FF$ also enjoys another kind of the decomposition with each component being around the single-species local equilibrium state
 \begin{equation}
\label{dec.1}
\FF=\left[\begin{array}{cc}
      M_{i}^{(1)}+G_{i}^{(1)}    \\[3mm]
      M_{e}^{(1)}+G_{e}^{(1)}
\end{array}\right],
\end{equation}
where for $\CA=i$ or $e$, $M_{\CA}^{(1)}:=M_{[n_\CA(t,x),u_\CA(t,x),\theta_\CA(t,x);m_\CA]}(\xi)$ is the local  equilibrium state of collisions of like-particles with the fluid quantities determined by
\begin{eqnarray*}
\left\{
\begin{array}{cll}
\begin{split}
n_\CA(t,x) &:=  \int_{\R^3} F_\CA (t,x,\xi)\,d\xi,\\
u_{\CA j}(t,x)  &:=   \frac{1}{n_\CA(t,x)}\int_{\R^3} \xi_j F_\CA (t,x,\xi)\,d\xi,\quad j=1,2,3,\\
\theta_\CA(t,x) &:=  \frac{1}{3 k_\CA n_\CA}\int_{\R^3}|\xi-u_\CA(t,x)|^2F_\CA (t,x,\xi)\,d\xi.
\end{split}
\end{array}\right.
\end{eqnarray*}
It should be pointed out that \eqref{dec.1} is different from \eqref{dec.2}. One can also obtain the link of the two-component fluid quantities $[n_i,n_e,u,\theta]$ and the single-component fluid quantities $[n_\CA, u_\CA,\theta_\CA]$ $(\CA=i,e)$ in the way that
\begin{eqnarray}
u&=&\frac{m_in_iu_i + m_en_eu_e}{m_in_i+m_en_e},\notag\\
\theta&=& \frac{n_i\theta_i+n_e\theta_e}{n_i+n_e} +\frac{m_im_en_in_e}{3k_B (n_i+n_e)(m_in_i+m_en_e)}|u_i-u_e|^2. \notag
\end{eqnarray}
We further remark that though the single-component fluid quantities
$
u_i,u_e,\theta_i,\theta_e
$
are not macroscopic in the two-component sense, the differences of them with the corresponding two-component fluid quantities
$
u,\theta
$
turn out to be microscopic in the two-component sense. Namely, after direct computations, for $\CA=i,e$,
\begin{eqnarray}
u_\CA -u&= & \frac{1}{n_\CA}\int_{\R^3}\xi G_\CA\,d\xi,\notag \\
\theta_\CA -\theta& = & \frac{1}{3k_\CA} |u-u_\CA|^2 +\frac{1}{3k_\CA n_\CA}\int_{\R^3} |\xi-u_\CA|^2 G_\CA\,d\xi.\notag
\end{eqnarray}
This observation is a key point for understanding the dissipation of macroscopic quantities of the two-component system.

We begin to introduce the two-component projection operators $\FP_0^\FM$ and $\FP_1^\FM$. For this purpose, one has to first introduce an orthonormal basis related to an arbitrary bi-Maxwellian
$$
\widehat{\FM}=\left[\begin{array}{rll}\widehat{M}_i\\[3mm]
\widehat{M}_e\end{array}\right].
$$
Associated with $\widehat{\FM}$, we define an inner
product in $\xi$ variable as
$$
\langle \FF,\FH\rangle_{\widehat{\FM}}\equiv\int_{{\R}^3}\frac{F_{i}(\xi)H_{i}(\xi)}{\widehat{M}_i}\,d\xi
+\int_{{\R}^3}\frac{F_{e}(\xi)H_{e}(\xi)}{\widehat{M}_e}\,d\xi,
$$
for functions $\FF=[F_{i},F_{e}]^{\rm{T}}$ and $\FH=[H_{i},H_{e}]^{\rm{T}}$ such that the integrals above is well defined.
Using the  inner product with respect to the bi-Maxwellian $\widehat{\FM}$, the
following functions spanning the macroscopic subspace are pairwise orthogonal:
\begin{eqnarray*}
\chi^{\widehat{\FM}}_{1}\left(\xi;\widehat{n}_{i},\widehat{u},\widehat{\theta}\right)&\equiv& \left[\begin{array}{cc}
\begin{split}\frac{1}{\sqrt{\widehat{n}_i}}M_i\\
0
\end{split}\end{array}\right],\\[3mm]
\chi^{\widehat{\FM}}_{2}\left(\xi;\widehat{n}_{e},\widehat{u},\widehat{\theta}\right)&\equiv& \left[\begin{array}{cc}
\begin{split}0\\
\frac{1}{\sqrt{\widehat{n}_e}}M_e
\end{split}\end{array}\right],
\\[3mm]
\chi^{\widehat{{\bf
M}}}_{j}\left(\xi;\widehat{n}_i,\widehat{n}_e,\widehat{u},\widehat{\theta}\right)&\equiv&
\left[\begin{array}{cc}
\begin{split}\frac{\sqrt{m_i}}{\sqrt{m_i\widehat{n}_i+m_e\widehat{n}_e}}\frac{\xi_j-\widehat{u}_j}
{\sqrt{k_i\widehat{\theta}}}\widehat{M}_i\\
\frac{\sqrt{m_e}}{\sqrt{m_i\widehat{n}_i+m_e\widehat{n}_e}}\frac{\xi_j-\widehat{u}_j}
{\sqrt{k_e\widehat{\theta}}}\widehat{M}_e
\end{split}
\end{array}\right],\ \ j=3,4,5,\\[3mm]
\chi^{\widehat{{\bf
M}}}_{6}\left(\xi;\widehat{n}_i,\widehat{n}_e,\widehat{u},\widehat{\theta}\right)&\equiv&
\left[\begin{array}{cc}
\begin{split}\frac{1}{\sqrt{6(\widehat{n}_i+\widehat{n}_e)}}\left(\frac{|\xi-\widehat{u}|^2}
{k_i\widehat{\theta}}-3\right)\widehat{M}_i\\
\frac{1}{\sqrt{6(\widehat{n}_i+\widehat{n}_e)}}\left(\frac{|\xi-\widehat{u}|^2}
{k_e\widehat{\theta}}-3\right)\widehat{M}_e
\end{split}
\end{array}\right],\\[3mm]
\left\langle\chi^{\widehat{\FM}}_j,\chi^{\widehat{{\bf
M}}}_k\right\rangle_{\widehat{\FM}}&=&\delta_{jk},\ \ \text{for}\ \ j,k=1,2,\cdots,6,
\end{eqnarray*}
where $\delta_{jk}$ is the Kronecker delta.  With the above orthonormal basis, the two-component macroscopic projection ${\bf P}^{\widehat{\FM}}_0$ and the two-component  microscopic
projection ${\bf P}^{\widehat{\FM}}_1$ can be defined as
$$
\left\{
\begin{array}{rll}{\bf P}^{\widehat{{\bf
M}}}_0\FF&\equiv&\sum\limits_{j=1}^6\left\langle \FF,\chi^{\widehat{{\bf
M}}}_j\right\rangle_{\widehat{\FM}}\chi^{\widehat{\FM}}_j,\\[3mm]
{\bf P}^{\widehat{\FM}}_1\FF&\equiv&\FF-{\bf P}^{\widehat{\FM}}_0\FF.
\end{array} \right.
$$
Notice that the operators ${\bf P}^{\widehat{\FM}}_0$ and ${\bf
P}^{\widehat{\FM}}_1$ are orthogonal (and thus self-adjoint)
projections with respect to the inner product
$\langle\cdot,\cdot\rangle_{\widehat{\FM}}$, i.e.,
$$
{\bf P}^{\widehat{\FM}}_0{\bf P}^{\widehat{\FM}}_0={\bf
P}^{\widehat{\FM}}_0,\ \ {\bf P}^{\widehat{\FM}}_1{\bf
P}^{\widehat{\FM}}_1={\bf P}^{\widehat{\FM}}_1,\ \ {\bf
P}^{\widehat{\FM}}_0{\bf P}^{\widehat{\FM}}_1={\bf
P}^{\widehat{\FM}}_1{\bf P}^{\widehat{\FM}}_0=0.
$$
Moreover, it is straightforward to check that
$$
\left\langle{\bf P}^{\widehat{\FM}}_0\FF,
{\bf P}^{\widehat{\widehat{\bf M}}}_1\FF\right\rangle_{\widehat{\FM}}=\left\langle{\bf P}^{\widehat{\widehat{\bf
M}}}_0\FF, {\bf P}^{\widehat{\FM}}_1\FF\right\rangle_{\widehat{\widehat{\bf M}}}=0
$$
for any two bi-Maxwellians $\widehat{\FM}$ and $\widehat{\widehat{\bf M}}$. Finally we remark that due to the definitions of ${\bf P}^{\widehat{\FM}}_0$ and ${\bf
P}^{\widehat{\FM}}_1$, one has
$$
{\bf P}^{\FM}_0\FF=\FM,\ \ {\bf P}^{\FM}_1\FF=\FG,
$$
whenever the bi-Maxwellian $\widehat{\FM}=\FM$ is  the macroscopic part of $\FF$.

%\subsection{Macro-micro decomposition}

With the two-component macro-micro decomposition of the solution $\FF$ to the VPB system \eqref{VPB}, \eqref{PE}, one may derive the dynamical equations of the fluid part $\FM$ and the non-fluid part $\FG$. For this, in the sequel we denote
\begin{eqnarray*}
\FQ(\FF,\FH)=\left[\begin{array}{c}
%\begin{split}
Q_i(\FF,\FH)\\[3mm]
Q_e(\FF,\FH)
%\end{split}
\end{array}
\right]=\left[\begin{array}{c}
%\begin{split}
Q_{ii}(F_i,H_i)+Q_{ie}(F_i,H_e)\\[3mm]
Q_{ee}(F_e,H_e)+Q_{ei}(F_e,H_i)
%\end{split}
\end{array}
\right].
\end{eqnarray*}
For convenience, we rewrite \eqref{VPB} as the following vector form
\begin{equation}\label{v.F}
\pa_t\FF+\xi_1\pa_x\FF+q_0\pa_x\phi\pa_{\xi_1}\FF=\FQ(\FF,\FF),
\end{equation}
where $q_0$ denotes the $2 \times2$ diagonal matrix
${\rm diag} (-q_i/m_i, -q_e/m_e)$.
%is denoted by $q_0$ and $\FQ(\FF,\FF)$ is in the form of
%\begin{equation*}\label{def.Q}
%\FQ(\FF,\FF)=\left[\begin{array}{cc}
%Q_{i}(\FF,\FF)
%\\[3mm]
%Q_{e}(\FF,\FF)
%\end{array}\right]=\left[\begin{array}{cc}
%Q_{ii}(F_i,F_e)+Q_{ie}(F_i,F_e)
%\\[3mm]
%Q_{ee}(F_e,F_e)+Q_{ei}(F_e,F_i)
%\end{array}\right].
%\end{equation*}
Upon using the macro-micro decomposition \eqref{dec.2}, the VPB system \eqref{v.F}
can be further rewritten as
\begin{eqnarray}%\label{mmBE}
\pa_t(\FM+\FG)+\xi_1
\pa_x(\FM+\FG)+
q_0\pa_x\phi
(\FM+\FG)\notag
=\FL_{\FM}\FG+\FQ(\FG,\FG).
\end{eqnarray}
Here,
$\FL_{\FM}$ is the two-component linearized Boltzmann collision operator given by
%\begin{multline}
\begin{eqnarray}
\FL_{\FM}\FG&=&\left[\begin{array}{cc}
Q_{i}(\FM,\FG)+Q_{i}(\FG,\FM)
\\[3mm]
Q_{e}(\FM,\FG)+Q_{e}(\FG,\FM)
\end{array}\right]\notag\\
&=&\left[\begin{array}{cc}
Q_{ii}(M_i,G_i)+Q_{ii}(G_i,M_i)+Q_{ie}(M_i,G_e)+Q_{ie}(G_i,M_e)
\\[3mm]
Q_{ee}(M_e,G_e)+Q_{ee}(G_e,M_e)+Q_{ei}(M_e,G_i)+Q_{ei}(G_e,M_i)
\end{array}\right],\label{def.L}
\end{eqnarray}
%\end{multline}
%with
%\begin{eqnarray}
%\label{bi.M.LC}
%\FM=\left[\begin{array}{cc}
%      M_i    \\[3mm]
%      M_e
%\end{array}\right]
%=\left[\begin{array}{cc}
%      M_{[\frac{1}{Zv_i(t,x)},u(t,x),\theta(t,x);m_i]}(\xi)    \\[3mm]
%%      M_{[\frac{1}{v_e(t,x)},u(t,x),\theta(t,x);m_e]}(\xi)
%\end{array}\right],
%\end{eqnarray}
and the nonlinear part
$\FQ(\FG,\FG)$ is defined as
\begin{equation*}%\label{def.Q}
\FQ(\FG,\FG)=\left[\begin{array}{cc}
Q_{i}(\FG,\FG)
\\[3mm]
Q_{e}(\FG,\FG)
\end{array}\right]=\left[\begin{array}{cc}
Q_{ii}(G_i,G_i)+Q_{ie}(G_i,G_e)
\\[3mm]
Q_{ee}(G_e,G_e)+Q_{ei}(G_e,G_i)
\end{array}\right].
\end{equation*}
%It is more convenient to modify $\FL_{\FM}\FG$ as
%\begin{eqnarray}\label{md.L}
%\FL_{\FM}\FG=\left[\begin{array}{cc}
%Q_{ii}(M_i,G_i)+Q_{ii}(G_i,M_i)+Q_{ie}(M_i,G_e)+Q_{ie}(G_i,M_e)
%\\[3mm]
%Q_{ee}(M_e,G_e)+Q_{ee}(G_e,M_e)+Q_{ei}(M_e,G_i)+Q_{ei}(G_e,M_i)
%\end{array}\right].
%\end{eqnarray}
%Noticing that ${\bf P}^{\FM}_0\left(q_0\frac{\pa_x\phi}{v}\pa_{\xi_1}{\bf M}\right)=0$,
Applying ${\bf P}^{\FM}_0$ and ${\bf P}^{\FM}_1$ to
\eqref{v.F} respectively, one has
\begin{equation*}%\label{macBE}
\begin{split}
\pa_t\FM&+{\bf P}^{\FM}_0\left(\xi_1\pa_x{\bf M}\right)
+{\bf P}^{\FM}_0\left(\xi_1\pa_x\FG\right)
+{\bf P}^{\FM}_0\left(q_0\pa_x\phi\pa_{\xi_1}{\bf M}\right)
+{\bf P}^{\FM}_0\left(q_0\pa_x\phi\pa_{\xi_1}{\bf G}\right)=0,
\end{split}
\end{equation*}
and
\begin{eqnarray}
\pa_t\FG+{\bf P}^{\FM}_1(\xi_1\pa_x{\bf
M})
&+&{\bf P}^{\FM}_1\left(\xi_1\pa_x{\bf
G}\right)
+{\bf P}^{\FM}_1\left(q_0 \pa_x\phi\pa_{\xi_1}{\bf
M}\right)\notag\\
&+&{\bf P}^{\FM}_1\left(q_0\pa_x\phi\pa_{\xi_1}{\bf
G}\right)
=\FL_{\FM}\FG+\FQ(\FG,\FG).\label{micBE}
\end{eqnarray}

Moreover, one also may derive the fluid-type system of the macroscopic quantities of the fluid part $\FM$ by using six two-component collision invariants $\psi_j(\xi)$ $(1\leq j\leq 6)$. For later use, we start from two component equations of \eqref{VPB}. Taking the inner product of equations of $\CA=i$ and $\CA=e$ with  $\psi_{j\CA}$ over $\xi\in \R^3$ respectively, it follows that
\begin{equation}\label{co.i}
\int_{{\R}^3}\psi_{ji}\left(\pa_t F_{i} +\xi_1\pa_x F_{i}
-\frac{q_i\pa_x\phi}{m_i}\pa_{\xi_1}F_{i}\right)d\xi=\int_{{\R}^3}\psi_{ji}Q_i(\FF,\FF)\,d\xi,\quad j=1,3,4,5,6,
\end{equation}
and
\begin{equation}\label{co.e}
\int_{{\R}^3}\psi_{je}\left(\pa_t F_{e} +\xi_1\pa_x F_{e}
-\frac{q_e\pa_x\phi}{m_e}\pa_{\xi_1}F_{e}\right)d\xi=\int_{{\R}^3}\psi_{je}Q_e(\FF,\FF)\,d\xi,\quad j=2,3,4,5,6.
\end{equation}
Applying the component forms $F_\CA=M_\CA+G_\CA$ $(\CA=i,e)$ of the macro-micro decomposition $\FF=\FM+\FG$ as well as the definition of the bi-Maxwellian $\FM$, one further deduces
\begin{eqnarray}\label{cons.law.i}
\left\{
\begin{array}{clll}
%\begin{split}
&\dis\pa_tn_i+\pa_x (n_iu_1)=-\int_{{\R}^3}\xi_1\pa_xG_i \,d\xi,\\[3mm]
%&\pa_tv_e-\frac{v_e}{v}\pa_x u_1=\frac{v_e^2}{v}\int_{{\R}^3}\xi_1\pa_xG_e d\xi,\\[3mm]
&\dis m_in_i\pa_t u_1+m_in_iu_1\pa_xu_1+\frac{2}{3}\pa_x \left(n_i\ta\right)+q_in_i\pa_x\phi
\\[3mm]&\dis \quad=-\int_{{\R}^3}\psi_{3i}\pa_tG_i \,d\xi-\int_{{\R}^3}\psi_{3i}\xi_1\pa_x G_i \,d\xi
%\left(P_1^{M_i}G_i\right) d\xi
%-\int_{{\R}^3}\psi_{3i}\xi_1\pa_x \left(P_0^{M_i}G_i\right) d\xi
%\\[3mm]
%&\qquad\quad
+\int_{{\R}^3}\psi_{3i}Q_i(\FF,\FF)\,d\xi+u_1\int_{{\R}^3}m_i\xi_1\pa_xG_i \,d\xi,\\[3mm]
&\dis m_in_i\pa_t u_j+m_in_iu_1\pa_xu_j=-\int_{{\R}^3}\psi_{(j+2)i}\pa_tG_i \,d\xi-\int_{{\R}^3}\psi_{(j+2)i}\xi_1\pa_x G_i \,d\xi
%\left(P_1^{M_i}G_i\right) d\xi
%-\int_{{\R}^3}\psi_{(j+2)i}\xi_1\pa_x\left(P_0^{M_i}G_i\right) d\xi
\\[3mm]
&\dis \qquad\qquad\qquad\qquad\qquad\qquad+\int_{{\R}^3}\psi_{(j+2)i}Q_i(\FF,\FF)\,d\xi
+u_j\int_{{\R}^3}m_i\xi_1\pa_xG_i \,d\xi,\
\ j=2,3,\\[3mm]
&\dis n_i\pa_t \theta+n_iu_1\pa_x\ta+\frac{2n_i\ta}{3}\pa_x u_1 \\[3mm]
&\dis\quad
=-\int_{{\R}^3}\psi_{6i}\pa_tG_i \,d\xi
-\int_{{\R}^3}\psi_{6i}\xi_1\pa_x G_i \,d\xi
%\left(P_1^{M_i}G_i\right) d\xi-\int_{{\R}^3}\psi_{6i}\xi_1\pa_x
%\left(P_0^{M_i}G_i\right) d\xi
+\int_{{\R}^3}\psi_{6i}Q_i(\FF,\FF)\,d\xi
\\[3mm]&\dis \qquad+\sum\limits_{j=1}^3u_j\int_{{\R}^3}\psi_{(j+2)i}\xi_1\pa_xG_i \,d\xi
-\frac{1}{2}\sum\limits_{j=1}^3u^2_j\int_{{\R}^3}m_i\xi_1\pa_xG_i \,d\xi
+\sum\limits_{j=1}^3u_j\int_{{\R}^3}\psi_{(j+2)i}\pa_tG_i \,d\xi
\\[3mm]
&\dis \qquad-\sum\limits_{j=1}^3u_j\int_{{\R}^3}\psi_{(j+2)i}Q_i(\FF,\FF)\,d\xi
+\ta\int_{{\R}^3}\xi_1\pa_xG_i \,d\xi
+\frac{q_i\pa_x\phi}{m_i}\int_{{\R}^3}\psi_{6i}\pa_{\xi_1}G_i\,d\xi,
%\end{split}
\end{array}
\right.
\end{eqnarray}
and
\begin{eqnarray}\label{cons.law.e}
\left\{
\begin{array}{clll}
%\begin{split}
&\dis \pa_tn_e+\pa_x (n_eu_1)=-\int_{{\R}^3}\xi_1\pa_xG_e \,d\xi,\\[3mm]
%&\pa_tv_e-\frac{v_e}{v}\pa_x u_1=\frac{v_e^2}{v}\int_{{\R}^3}\xi_1\pa_xG_e d\xi,\\[3mm]
&\dis m_en_e\pa_t u_1+m_en_eu_1\pa_xu_1+\frac{2}{3}\pa_x \left(n_e\ta\right)+q_en_e\pa_x\phi
\\[3mm]&\dis \quad=-\int_{{\R}^3}\psi_{3e}\pa_tG_e \,d\xi
-\int_{{\R}^3}\psi_{3e}\xi_1\pa_x G_e \,d\xi
%\left(P_1^{M_e}G_e\right) d\xi
%-\int_{{\R}^3}\psi_{3e}\xi_1\pa_x\left(P_0^{M_e}G_e\right) d\xi
%\\&\qquad\quad
+\int_{{\R}^3}\psi_{3e}Q_e(\FF,\FF)\,d\xi+u_1\int_{{\R}^3}m_e\xi_1\pa_xG_e \,d\xi,\\[3mm]
&\dis m_en_e\pa_t u_j+m_en_eu_1\pa_xu_j=-\int_{{\R}^3}\psi_{(j+2)e}\pa_tG_e \,d\xi
-\int_{{\R}^3}\psi_{(j+2)e}\xi_1\pa_x G_e \,d\xi
%\left(P_1^{M_e}G_e\right) d\xi
%-\int_{{\R}^3}\psi_{(j+2)e}\xi_1\pa_x\left(P_0^{M_e}G_e\right) d\xi
\\[3mm]
&\dis \qquad\qquad\qquad\qquad\qquad\qquad+\int_{{\R}^3}\psi_{(j+2)e}Q_e(\FF,\FF)d\xi
+u_j\int_{{\R}^3}m_e\xi_1\pa_xG_e \,d\xi,\
\ j=2,3,\\[3mm]
&\dis n_e\pa_t \theta+n_eu_1\pa_x\ta+\frac{2n_e\ta}{3}\pa_x u_1 \\[3mm]
&\dis \quad
=-\int_{{\R}^3}\psi_{6e}\pa_tG_e \,d\xi
-\int_{{\R}^3}\psi_{6e}\xi_1\pa_x G_e \,d\xi
%\left(P_1^{M_e}G_e\right)d\xi
%-\int_{{\R}^3}\psi_{6e}\xi_1\pa_x\left(P_0^{M_e}G_e\right)d\xi
+\int_{{\R}^3}\psi_{6e}Q_e(\FF,\FF)\,d\xi
\\[3mm]
&\dis \qquad+\sum\limits_{j=1}^3u_j\int_{{\R}^3}\psi_{(j+2)e}\xi_1\pa_xG_e \,d\xi
-\frac{1}{2}\sum\limits_{j=1}^3u^2_j\int_{{\R}^3}m_e\xi_1\pa_xG_e \,d\xi
+\sum\limits_{j=1}^3u_j\int_{{\R}^3}\psi_{(j+2)e}\pa_tG_e \,d\xi
\\[3mm]
&\dis \qquad-\sum\limits_{j=1}^3u_j\int_{{\R}^3}\psi_{(j+2)e}Q_e(\FF,\FF)\,d\xi
+\ta\int_{{\R}^3}\xi_1\pa_xG_e \,d\xi
+\frac{q_e\pa_x\phi}{m_e}\int_{{\R}^3}\psi_{6e}\pa_{\xi_1}G_e\,d\xi.
%\end{split}
\end{array}
\right.
\end{eqnarray}
Here the self-consistent potential $\phi$ satisfies the Poisson equation
\begin{equation}%\notag
\label{co.po}
-\pa_x^2\phi=q_in_i+q_en_e.
\end{equation}
Note that if one only considers the macroscopic balance laws of \eqref{v.F} in terms of six collision invariants, one can obtain six dynamical equations of fluid quantities $n_i,n_e,u_1,u_2,u_3,\theta$ which correspond to the above two systems \eqref{cons.law.i} and \eqref{cons.law.e} after both the equations of momentums and the equations of temperatures are taken summation, respectively. Namely, one has
\begin{eqnarray}\label{cons.law.s}
\left\{
\begin{array}{clll}
%\begin{split}
%&\pa_tn_i+\pa_x (n_iu_1)=-\int_{{\R}^3}\xi_1\pa_xG_i d\xi,\\[3mm]
%&\pa_tv_e-\frac{v_e}{v}\pa_x u_1=\frac{v_e^2}{v}\int_{{\R}^3}\xi_1\pa_xG_e d\xi,\\[3mm]
&\dis (m_in_i+m_en_e)(\pa_t u_1+u_1\pa_xu_1)+\frac{2}{3}\pa_x \left[(n_i+n_e)\ta\right]+(q_in_i+q_en_e)\pa_x\phi\\[3mm]
&\dis \qquad=-\pa_x \int_{{\R}^3}\xi_1\psi_{3}\cdot \FG\,d\xi,
\\[3mm]
&\dis (m_in_i+m_en_e)(\pa_t u_j+u_1\pa_xu_j)=-\pa_x \int_{{\R}^3}\xi_1\psi_{(j+2)}\cdot \FG\, d\xi, \ \ j=2,3,\\[3mm]
&\dis (n_i+n_e)(\pa_t \theta+u_1\pa_x\ta)+\frac{2}{3}(n_i+n_e)\ta\pa_x u_1 \\[3mm]&\quad
=\dis
%-\int_{{\R}^3}\psi_{6i}\pa_tG_i d\xi
-\pa_x\int_{{\R}^3}\xi_1\psi_{6}\cdot \FG\, d\xi
+\sum\limits_{j=1}^3u_j\pa_x \int_{{\R}^3}\xi_1\psi_{(j+2)}\cdot \FG\, d\xi
%-\frac{1}{2}\sum\limits_{j=1}^3u^2_j\int_{{\R}^3}m_i\xi_1\pa_xG_i d\xi
%+\sum\limits_{j=1}^3u_j\int_{{\R}^3}\psi_{(j+2)i}\pa_tG_i d\xi
\\[3mm]
&\dis \qquad
%-\sum\limits_{j=1}^3u_j\int_{{\R}^3}\psi_{(j+2)i}Q_i(\FF,\FF)d\xi
+\ta\int_{{\R}^3}[\xi_1,\xi_1]^{\rm T}\cdot \pa_x\FG\, d\xi
+\pa_x\phi \int_{{\R}^3}\frac{|\xi|^2}{2}[q_i,q_e]^{\rm T}\cdot \pa_{\xi_1}\FG\,d\xi.
%\end{split}
\end{array}
\right.
\end{eqnarray}
Moreover, if one further ignores those terms involving the non-fluid part $\FG$, one has the closed fluid-type system of six knowns $n_i,n_e,u_1,u_2,u_3,\theta$:
%\begin{eqnarray}
%&& \pa_tn_i+\pa_x (n_iu_1)=0,\\
%&&\pa_tn_e+\pa_x (n_eu_1)=0,\\
%&&(m_in_i+m_en_e)(\pa_t u_1+u_1\pa_xu_1)+\frac{2}{3}\pa_x \left[(n_i+n_e)\ta\right]+(q_in_i+q_en_e)\pa_x\phi=0,\\
%&&(m_in_i+m_en_e)(\pa_t u_j+u_1\pa_xu_j)=0,\quad j=2,3,\\
%&&(n_i+n_e)(\pa_t \theta+u_1\pa_x\ta)+\frac{2(n_i+n_e)\ta}{3}\pa_x u_1=0.
%\end{eqnarray}
\begin{eqnarray}\label{cons.law.f}
\left\{
\begin{array}{clll}
&\dis \pa_tn_i+\pa_x (n_iu_1)=0,\\[3mm]
&\dis \pa_tn_e+\pa_x (n_eu_1)=0,\\[3mm]
&\dis (m_in_i+m_en_e)(\pa_t u_1+u_1\pa_xu_1)+\frac{2}{3}\pa_x \left[(n_i+n_e)\ta\right]+(q_in_i+q_en_e)\pa_x\phi=0,\\[3mm]
&\dis (m_in_i+m_en_e)(\pa_t u_j+u_1\pa_xu_j)=0,\quad j=2,3,\\[3mm]
&\dis (n_i+n_e)(\pa_t \theta+u_1\pa_x\ta)+\frac{2}{3}(n_i+n_e)\ta\pa_x u_1=0,\\[3mm]
&\dis -\pa_x^2\phi=q_in_i+q_en_e.
\end{array}
\right.
\end{eqnarray}
Note that \eqref{cons.law.f} could be thought to be the zero-order fluid dynamic approximation of the VPB system \eqref{VPB}, \eqref{PE}.

%\subsection{Macro-micro decomposition around local single Maxwellian}

\subsection{Diffusion and heat-conductivity}

As in \cite{LYY}, in order to further consider the first-order fluid dynamic approximation of the VPB system, one has to find out diffusion and heat-conductivity corresponding to velocity function $u$ and temperature function $\theta$, respectively. One way for that is to formally solve $\FG$ through the microscopic equation \eqref{micBE} as
\begin{equation*}
%\label{micBEc}
\FG=\FL^{-1}_{\FM}\left\{{\bf P}^{\FM}_1\left(\xi_1\pa_x{\bf
M}\right)\right\}+\FR,
\end{equation*}
with
\begin{equation*}%\label{CR.def}
\begin{split}
\FR=&\FL^{-1}_{\FM}\left\{\pa_t\FG
+{\bf P}^{\FM}_1\left(\xi_1\pa_x{\bf
G}\right)+{\bf P}^{\FM}_1\left(q_0\pa_x\phi\pa_{\xi_1}{\bf
G}\right)-\FQ(\FG,\FG)\right\}
\\&+\underbrace{\FL^{-1}_{\FM}\left\{{\bf P}^{\FM}_1\left(q_0\pa_x\phi\pa_{\xi_1}{\bf
M}\right)\right\}}_{\FR_\phi},
\end{split}
\end{equation*}
and then plug it into \eqref{cons.law.s} so that those terms related to diffusion and heat-conductivity could be obtained by computing
\begin{equation}\notag
%\label{ }
-\pa_x \int_{{\R}^3}\xi_1\psi_{j}\cdot \FL^{-1}_{\FM}\left\{{\bf P}^{\FM}_1\left(\xi_1\pa_x{\bf
M}\right)\right\}\,d\xi,\quad 3\leq j\leq 6.
\end{equation}
We remark that such treatment may not be a good way because it is unknown whether or not the above integrals with $\FL^{-1}_{\FM}\left\{{\bf P}^{\FM}_1\left(\xi_1\pa_x{\bf
M}\right)\right\}$ replaced by $\FR_\phi$ are vanishing, and thus the right-hand terms of \eqref{cons.law.s} could involve $\phi$ in a linear way which should give much trouble to estimates on $\phi$.

Therefore we turn to another way for obtaining the effect of diffusion and heat-conductivity on the basis of two single-component equations of the VPB system \eqref{VPB}. The key point is to introduce single-component projection operators $P_0^{M_\CA}$ and  $P_1^{M_\CA}$ for $\CA=i$ and $e$, where we recall that $M_i$, $M_e$ are the component functions of the bi-Maxwellian $\FM$ defined in the two-component macro-micro decomposition \eqref{dec.2}.

To do so, similarly as before, for any given local Maxwellian
$$
\widehat{M}_\CA=M_{\left[\widehat{n}_\CA,\widehat{u}_\CA,\widehat{\theta}_\CA\right]},
$$
we define an inner
product in $\xi\in{\R}^3$ as
$$
\langle f,g\rangle_{\widehat{M}_{\CA}}\equiv\int_{{\R}^3}\frac{f(\xi)g(\xi)}{\widehat{M}_{\CA}}d\xi,
$$
for two scalar functions $f$ and $g$ such that the integral on the right is well defined. Applying the above inner product with respect to the single Maxwellian $\widehat{M}_{\CA}$, the
following five functions are also orthonormal:
\begin{eqnarray*}
\chi^{\widehat{M}_{\CA}}_0
\left(\xi;\widehat{n}_{\CA},\widehat{u}_{\CA},\widehat{\theta}_{\CA}\right)&\equiv& \frac{1}{\sqrt{\widehat{n}_{\CA}}}\widehat{M}_{\CA},\\[3mm]
\chi^{\widehat{M}_{\CA}}_{i}\left(\xi;\widehat{n}_{\CA},\widehat{u}_{\CA},\widehat{\theta}_{\CA}\right)&\equiv&\frac{\xi_j-\widehat{u}_j}
{\sqrt{k_\CA\widehat{n}_\CA\,\widehat{\theta}_\CA}}\widehat{M}_{\CA},\ \ j=1,2,3,\\[3mm]
\chi^{\widehat{M}_{\CA}}_{4}\left(\xi;\widehat{n}_{\CA},\widehat{u}_{\CA},\widehat{\theta}_{\CA}\right)
&\equiv&\frac{1}{\sqrt{6\widehat{n}_{\CA}}}\left(\frac{|\xi-\widehat{u}|^2}
{k_\CA\widehat{\theta}_{\CA}}-3\right)\widehat{M}_{\CA},\\[3mm]
\left\langle\chi^{\widehat{M}_{\CA}}_j,\chi^{\widehat{M}_{\CA}}_k\right\rangle_{\widehat{M}_{\CA}}&=&\delta_{jk},\ \ \text{for}\ \ j,k=0,1,2,3,4.
\end{eqnarray*}
%where $\delta_{ij}$ is the Kronecker delta.
With the above orthonormal set, we can also define the
macroscopic projection $P^{\widehat{M}_{\CA}}_0$ and the microscopic
projection $P^{\widehat{M}_{\CA}}_1$ as follows
$$
\left\{
\begin{array}{rll}P^{\widehat{M}_{\CA}}_0h&\equiv&\sum\limits_{j=0}^4
\left\langle h,\chi^{\widehat{M}_{\CA}}_j\right\rangle_{\widehat{M}_{\CA}}\chi^{\widehat{M}_{\CA}}_j,\\[3mm]
P^{\widehat{M}_{\CA}}_1h&\equiv&h-P^{\widehat{M}_{\CA}}_0h.
\end{array} \right.
$$
Note that the operators $P^{\widehat{M}_{\CA}}_0$ and
$P^{\widehat{M}_{\CA}}_1$ enjoy the similar properties as ${\bf P}^{\widehat{\FM}}_0$ and ${\bf P}^{\widehat{\FM}}_1$ given in the previous subsection.
%\subsection{Reformulation of the Cauchy problem}

Using notations above and recalling the decomposition \eqref{dec.2}, the solution $F_{\CA}(t,x,\xi)$ $(\CA=i,e)$ of \eqref{VPB}
satisfies
$$
P^{M_{\CA}}_0F_{\CA}=M_\CA+P^{M_{\CA}}_0G_\CA,\quad \ P^{M_{\CA}}_1F_{\CA}=P^{M_{\CA}}_1G_{\CA}.
$$
Noticing that
$$
P^{M_{\CA}}_1\left\{\frac{q_\CA\pa_x\phi}{m_{\CA}}\pa_{\xi_1}M_\CA\right\}=0.
$$
Acting  $P^{M_{i}}_1$ and $P^{M_{e}}_1$  to two equations of \eqref{VPB} respectively, one has that for $\CA=i,e$,
\begin{equation}\label{mmBE}
\begin{split}
P^{M_{\CA}}_1&\pa_tG_{\CA}+P^{M_{\CA}}_1\left\{\xi_1\pa_xM_\CA\right\}+P^{M_{\CA}}_1\left\{\xi_1\pa_xG_\CA\right\}
-P^{M_{\CA}}_1\left\{\frac{q_\CA\pa_x\phi}{m_{\CA}}\pa_{\xi_1}G_\CA\right\}
\\&=L_{M_{\CA}}P^{M_{\CA}}_1G_{\CA}+P^{M_{\CA}}_1\overline{Q}_{\CA}(\FG,\FG),
\end{split}
\end{equation}
where for the $\CA$-component, we have defined  the linearized Boltzmann collision operator around the local Maxwellian $M_{\CA}$ by
$$
L_{M_{\CA}}P^{M_{\CA}}_1G_{\CA}=L_{M_{\CA}}G_{\CA}=Q_{\CA\CA}(M_{\CA},G_{\CA})+Q_{\CA\CA}(G_{\CA},M_{\CA}),
$$
and the remaining term by
$$
\overline{Q}_{\CA}(\FG,\FG)=Q_{\CA\CB}(M_\CA,G_\CB)+Q_{\CA\CB}(G_\CA,M_\CB)+Q_{\CA\CA}(G_\CA,G_\CA)+Q_{\CA\CB}(G_\CA,G_\CB),
$$
with $\CA\neq \CB$. Moreover, from \eqref{mmBE}, it follows that
\begin{equation}\label{PGA}
P^{M_{\CA}}_1G_{\CA}=L_{M_{\CA}}^{-1}\left\{P^{M_{\CA}}_1\left\{\xi_1\pa_xM_\CA\right\}\right\}+\overline{R}_{\CA},
\end{equation}
with
\begin{eqnarray}
\overline{R}_{\CA}&=&L_{M_{\CA}}^{-1}\left \{P^{M_{\CA}}_1\pa_tG_{\CA}
+P^{M_{\CA}}_1\left\{\xi_1\pa_xG_\CA\right\}
-P^{M_{\CA}}_1\left\{\frac{q_\CA\pa_x\phi}{m_{\CA}}\pa_{\xi_1}G_\CA\right\}\right\}\notag\\
&&-L_{M_{\CA}}^{-1}\left\{P^{M_{\CA}}_1\overline{Q}_{\CA}(\FG,\FG)\right\}.\label{def.rbar}
\end{eqnarray}
Back to \eqref{cons.law.i} and \eqref{cons.law.e}, we rewrite $G_\CA$ in the right-hand second terms of momentum and temperature equations as
\begin{equation}\notag
%\label{ }
G_\CA=P_0^{M_\CA}G_\CA+P_1^{M_\CA}G_\CA,
\end{equation}
and then use \eqref{PGA} to replace $P_1^{M_\CA}G_\CA$ so as to obtain by some further calculations:
\begin{eqnarray}\label{BE-NSi}
\left\{
\begin{array}{clll}
%\begin{split}
&\dis \pa_tn_i+\pa_x (n_iu_1)=-\int_{{\R}^3}\xi_1\pa_xG_i \,d\xi,\\[3mm]
&\dis m_in_i(\pa_t u_1+u_1\pa_xu_1)+\frac{2}{3}\pa_x (n_i\ta)+q_in_i\pa_x\phi
\\[3mm]&\dis \quad=3\pa_x\left(\mu_i(\theta)\pa_x u_1\right)-\int_{{\R}^3}\psi_{3i}\pa_tG_i \,d\xi
-\int_{{\R}^3}\psi_{3i}\xi_1\pa_x (P_0^{M_i}G_i) \,d\xi
\\[3mm]
&\dis \qquad+\int_{{\R}^3}\psi_{3i}Q_i(\FF,\FF)\,d\xi+u_1\int_{{\R}^3}m_i\xi_1\pa_xG_i \,d\xi,
-\int_{{\R}^3}\psi_{3i}\xi_1\pa_x\overline{R}_i \,d\xi,\\[3mm]
&\dis m_in_i(\pa_t u_j+u_1\pa_xu_j)\\[3mm]
&\dis \quad =\pa_x\left(\mu_i(\theta)\pa_x u_j\right)-\int_{{\R}^3}\psi_{(j+2)i}\pa_tG_i \,d\xi
-\int_{{\R}^3}\psi_{(j+2)i}\xi_1\pa_x (P_0^{M_i}G_i) \,d\xi
\\[3mm]
&\dis \qquad%\qquad\qquad\qquad\qquad
+\int_{{\R}^3}\psi_{(j+2)i}Q_i(\FF,\FF)\,d\xi+u_j\int_{{\R}^3}m_i\xi_1\pa_xG_i \,d\xi,
-\int_{{\R}^3}\psi_{(j+2)i}\xi_1\pa_x\overline{R}_i \,d\xi,\ j=2,3,\\[3mm]
&\dis n_i\pa_t \theta+n_iu_1\pa_x\ta+\frac{2n_i\ta}{3}\pa_x u_1
\\&\dis \quad=\pa_x\left(\kappa_i(\theta)\pa_x \theta\right)+3\mu_i(\theta)(\pa_x u_1)^2+
\sum\limits_{j=2}^3\mu_i(\theta)(\pa_x u_j)^2
-\int_{{\R}^3}\xi_1(\psi_{6i}-\sum\limits_{j=1}^3u_j \psi_{(j+2)i})\pa_x\overline{R}_i \,d\xi
\\[3mm]&\dis \qquad
-\int_{{\R}^3}\psi_{6i}\pa_tG_i \,d\xi-\int_{{\R}^3}\psi_{6i}\xi_1\pa_x(P_0^{M_i}G_i) \,d\xi
+\sum\limits_{j=1}^3u_j\int_{{\R}^3}\psi_{(j+2)i}\xi_1\pa_x (P_0^{M_i}G_i) \,d\xi
\\[3mm]&\dis \qquad
-\frac{1}{2}\sum\limits_{j=1}^3u^2_j\int_{{\R}^3}m_i\xi_1\pa_xG_i \,d\xi
+\sum\limits_{j=1}^3u_j\int_{{\R}^3}\psi_{(j+2)i}\pa_tG_i \,d\xi
+\int_{{\R}^3}\psi_{6i}Q_i(\FF,\FF)\,d\xi
\\[3mm]&\dis \qquad
-\sum\limits_{j=1}^3u_j\int_{{\R}^3}\psi_{(j+2)i}Q_i(\FF,\FF)\,d\xi
+\ta\int_{{\R}^3}\xi_1\pa_xG_i \,d\xi
+q_i\pa_x\phi\int_{{\R}^3}\frac{|\xi|^2}{2}\pa_{\xi_1}G_i \,d\xi,
%\end{split}
\end{array}
\right.
\end{eqnarray}
and
\begin{eqnarray}\label{BE-NSe}
\left\{
\begin{array}{clll}
%\begin{split}
&\dis \pa_tn_e+\pa_x (n_eu_1)=-\int_{{\R}^3}\xi_1\pa_xG_e \,d\xi,\\[3mm]
&\dis m_en_e(\pa_t u_1+u_1\pa_xu_1)+\frac{2}{3}\pa_x (n_e\ta)+q_en_e\pa_x\phi
\\[3mm]&\dis \quad=3\pa_x\left(\mu_e(\theta)\pa_x u_1\right)-\int_{{\R}^3}\psi_{3e}\pa_tG_e \,d\xi
-\int_{{\R}^3}\psi_{3e}\xi_1\pa_x (P_0^{M_e}G_e) \,d\xi
\\[3mm]
&\dis \qquad+\int_{{\R}^3}\psi_{3e}Q_e(\FF,\FF)\,d\xi+u_1\int_{{\R}^3}m_e\xi_1\pa_xG_e \,d\xi,
-\int_{{\R}^3}\psi_{3e}\xi_1\pa_x\overline{R}_e \,d\xi,\\[3mm]
&\dis m_en_e(\pa_t u_j+u_1\pa_xu_j)\\[3mm]
&\dis \quad =\pa_x\left(\mu_e(\theta)\pa_x u_j\right)-\int_{{\R}^3}\psi_{(j+2)e}\pa_tG_e \,d\xi
-\int_{{\R}^3}\psi_{(j+2)e}\xi_1\pa_x (P_0^{M_e}G_e) \,d\xi
\\[3mm]
&\dis \qquad%\qquad\qquad\qquad\qquad
+\int_{{\R}^3}\psi_{(j+2)e}Q_e(\FF,\FF)\,d\xi+u_j\int_{{\R}^3}m_i\xi_1\pa_xG_e \,d\xi,
-\int_{{\R}^3}\psi_{(j+2)e}\xi_1\pa_x\overline{R}_e \,d\xi,\ j=2,3,\\[3mm]
&\dis n_e\pa_t \theta+n_eu_1\pa_x\ta+\frac{2n_e\ta}{3}\pa_x u_1
\\[3mm]
&\dis \quad=\pa_x\left(\kappa_e(\theta)\pa_x \theta\right)+3\mu_e(\theta)(\pa_x u_1)^2+
\sum\limits_{j=2}^3\mu_e(\theta)(\pa_x u_j)^2
-\int_{{\R}^3}\xi_1(\psi_{6e}-\sum\limits_{j=1}^3u_j \psi_{(j+2)e})\pa_x\overline{R}_e \,d\xi
\\[3mm]&\dis \qquad
-\int_{{\R}^3}\psi_{6e}\pa_tG_e \,d\xi-\int_{{\R}^3}\psi_{6e}\xi_1\pa_x(P_0^{M_e}G_e) \,d\xi
+\sum\limits_{j=1}^3u_j\int_{{\R}^3}\psi_{(j+2)e}\xi_1\pa_x (P_0^{M_e}G_e) \,d\xi
\\[3mm]&\qquad
-\frac{1}{2}\sum\limits_{j=1}^3u^2_j\int_{{\R}^3}m_e\xi_1\pa_xG_e \,d\xi
+\sum\limits_{j=1}^3u_j\int_{{\R}^3}\psi_{(j+2)e}\pa_tG_e \,d\xi
+\int_{{\R}^3}\psi_{6e}Q_e(\FF,\FF)\,d\xi
\\[3mm]&\dis \qquad
-\sum\limits_{j=1}^3u_j\int_{{\R}^3}\psi_{(j+2)i}Q_e(\FF,\FF)\,d\xi
+\ta\int_{{\R}^3}\xi_1\pa_xG_e \,d\xi
+q_e\pa_x\phi\int_{{\R}^3}\frac{|\xi|^2}{2}\pa_{\xi_1}G_e \,d\xi,
%\end{split}
\end{array}
\right.
\end{eqnarray}
where for $\CA=i$ and $e$ the viscosity coefficient $\mu_{\CA}(\theta)$ and
heat-conductivity coefficient $\kappa_{\CA}(\theta)$ are represented by
\begin{eqnarray}
\mu_{\CA}(\ta)&=&-\frac{1}{3k_{\CA}\theta}\int_{{\R}^3}m_{\CA}\xi_1^2L^{-1}_{M_{[1,u,\ta;m_\CA]}}
\left(m_{\CA}\xi_1^2M_{[1,u,\ta;m_{\CA}]}\right)\,d\xi\notag\\
&=&-\frac{1}{k_\CA\theta}\int_{{\R}^3}m_\CA\xi_1\xi_jL^{-1}_{M_{[1,u,\ta;m_{\CA}]}}
\left(m_{\CA}\xi_1\xi_jM_{[1,u,\ta;m_{\CA}]}\right)\,d\xi>0,\
j=2, 3,\notag
\end{eqnarray}
%\begin{eqnarray}\label{Bur.fun.}
%\left\{
%\begin{array}{rllll}
%\begin{split}
%\mu_{\CA}(\ta)&=-\frac{1}{3k_{\CA}\theta}\int_{{\R}^3}m_{\CA}\xi_1^2L^{-1}_{M_{[1,u,\ta;m_\CA]}}
%\left(m_{\CA}\xi_1^2M_{[1,u,\ta;m_{\CA}]}\right)\,d\xi\\
%&=-\frac{1}{k_\CA\theta}\int_{{\R}^3}m_\CA\xi_1\xi_jL^{-1}_{M_{[1,u,\ta;m_{\CA}]}}
%\left(m_{\CA}\xi_1\xi_jM_{[1,u,\ta;m_{\CA}]}\right)\,d\xi>0,\
%j=2, 3,\\
%\ka_{\CA}(\ta)&=-\frac{1}{4k_{\CA}\ta^2}\int_{{\R}^3}m_{\CA}|\xi-u|^2\xi_jL^{-1}_{M_{[1,u,\ta;m_{\CA}]}}
%\left(m_{\CA}|\xi-u|^2\xi_jM_{[1,u,\ta;m_{\CA}]}\right)\,d\xi>0,\
%j=1, 2, 3,
%\end{split}
%\end{array}
%\right.
%\end{eqnarray}
and
\begin{equation}\notag
\ka_{\CA}(\ta)=-\frac{1}{4k_{\CA}\ta^2}\int_{{\R}^3}m_{\CA}|\xi-u|^2\xi_jL^{-1}_{M_{[1,u,\ta;m_{\CA}]}}
\left(m_{\CA}|\xi-u|^2\xi_jM_{[1,u,\ta;m_{\CA}]}\right)\,d\xi>0,\
j=1, 2, 3,
\end{equation}
respectively. Here $L^{-1}_{M_{[1,u,\ta;m_{\CA}]}}$ is defined in the same way as $L^{-1}_{M_{\CA}}$.

Similarly for obtaining \eqref{cons.law.s}, from \eqref{BE-NSi} and \eqref{BE-NSe}, one has the equations of momentum for $u=(u_1,u_2,u_3)$:
\begin{eqnarray}\label{BE-NS-m}
\left\{
\begin{array}{clll}
%\begin{split}
%&\pa_t(m_in_i+m_en_e)+\pa_x ((m_in_i+m_en_e)u_1)=0,\\[3mm]
%&\pa_tv_e-\frac{v_e}{v}\pa_x u_1=\frac{v_e^2}{v}\int_{{\R}^3}\xi_1\pa_xG_e d\xi,\\[3mm]
&\dis (m_in_i+m_en_e)(\pa_t u_1+u_1\pa_xu_1)+\frac{2}{3}\pa_x [(n_i+n_e)\ta]+\left(q_in_i+q_en_e\right)\pa_x\phi
\\[3mm]
&\dis \quad=3\pa_x\left[(\mu_i(\theta)+\mu_e(\theta))\pa_x u_1\right]\\[3mm]
&\dis \qquad-\int_{{\R}^3}\psi_{3i}\xi_1\pa_x \left(P_0^{M_i}G_i\right) d\xi
-\int_{{\R}^3}\psi_{3e}\xi_1\pa_x \left(P_0^{M_e}G_e\right) d\xi
-\int_{{\R}^3}\xi_1\psi_{3}\cdot\pa_x\overline{\FR} \,d\xi,\\[3mm]
&(m_in_i+m_en_e)(\pa_t u_j+u_1\pa_xu_j)
\\[3mm]
&\dis \quad=\pa_x\left[(\mu_i(\theta)+\mu_e(\theta))\pa_x u_j\right]-\int_{{\R}^3}\psi_{(j+2)i}\xi_1\pa_x \left(P_0^{M_i}G_i\right) \,d\xi\\[3mm]
&\dis\qquad
-\int_{{\R}^3}\psi_{(j+2)e}\xi_1\pa_x \left(P_0^{M_e}G_e\right) d\xi
-\int_{{\R}^3}\xi_1\psi_{j+2}\cdot\pa_x\overline{\FR} \,d\xi,\
\ j=2,3,
%\\[3mm]
%&(n_i+n_e)\pa_t \theta+(n_i+n_e)u_1\pa_x\ta+P\pa_x u_1
%\\[3mm]&\qquad=\pa_x\left((\kappa_i(\theta)+\kappa_e(\theta))\pa_x \theta\right)+
%3(\mu_i(\theta)+\mu_e(\theta))(\pa_x u_1)^2+\sum\limits_{j=2}^3(\mu_i(\theta)+\mu_e(\theta))(\pa_x u_j)^2
%\\[3mm]&\qquad\quad-\int_{{\R}^3}\xi_1\left(\psi_{6}-\sum\limits_{j=1}^3u_j \psi_{j+2}\right)\cdot\pa_x\overline{\FR} d\xi
%-\int_{{\R}^3}\psi_{6i}\xi_1\pa_x\left(P_0^{M_i}G_i\right) d\xi-\int_{{\R}^3}\psi_{6e}\xi_1\pa_x\left(P_0^{M_e}G_e\right) d\xi
%\\[3mm]&\qquad\quad+\sum\limits_{j=1}^3u_j\int_{{\R}^3}\psi_{(j+2)i}\xi_1\pa_x \left(P_0^{M_i}G_i\right) d\xi
%+\sum\limits_{j=1}^3u_j\int_{{\R}^3}\psi_{(j+2)e}\xi_1\pa_x \left(P_0^{M_e}G_e\right) d\xi
%\\[3mm]&\qquad\quad+\ta\int_{{\R}^3}[\xi_1,\xi_1]^{\rm T}\cdot\pa_x\FG d\xi
%+\pa_x\phi\int_{{\R}^3}\frac{|\xi|^2}{2}\left[q_i,q_e\right]^{\rm T}\cdot\pa_{\xi_1}\FG d\xi
%,\\[3mm]
%&-\pa_x^2\phi=q_in_i+q_en_e.
%\end{split}
\end{array}
\right.
\end{eqnarray}
and the equation of temperature for $\theta$:
\begin{eqnarray}\label{BE-NS-t}
%\left\{
\begin{array}{clll}
%\begin{split}
%&\pa_t(m_in_i+m_en_e)+\pa_x ((m_in_i+m_en_e)u_1)=0,\\[3mm]
%&\pa_tv_e-\frac{v_e}{v}\pa_x u_1=\frac{v_e^2}{v}\int_{{\R}^3}\xi_1\pa_xG_e d\xi,\\[3mm]
%&(m_in_i+m_en_e)\pa_t u_1+(m_in_i+m_en_e)u_1\pa_xu_1+\pa_x P+\left(q_in_i+q_en_e\right)\pa_x\phi
%\\[3mm]&\quad=3\pa_x\left((\mu_i(\theta)+\mu_e(\theta))\pa_x u_1\right)
%-\int_{{\R}^3}\psi_{3i}\xi_1\pa_x \left(P_0^{M_i}G_i\right) d\xi
%-\int_{{\R}^3}\psi_{3e}\xi_1\pa_x \left(P_0^{M_e}G_e\right) d\xi
%-\int_{{\R}^3}\xi_1\psi_{3}\cdot\pa_x\overline{\FR} d\xi,\\[3mm]
%&(m_in_i+m_en_e)\pa_t u_j+(m_in_i+m_en_e)u_1\pa_xu_j
%\\[3mm]&\qquad=\pa_x\left((\mu_i(\theta)+\mu_e(\theta))\pa_x u_j\right)-\int_{{\R}^3}\psi_{(j+2)i}\xi_1\pa_x \left(P_0^{M_i}G_i\right) d\xi
%-\int_{{\R}^3}\psi_{(j+2)e}\xi_1\pa_x \left(P_0^{M_e}G_e\right) d\xi
%\\[3mm]&\qquad\quad
%-\int_{{\R}^3}\xi_1\psi_{j+2}\cdot\pa_x\overline{\FR} d\xi,\
%\ j=2,3,\\[3mm]
&\dis (n_i+n_e)(\pa_t \theta+u_1\pa_x\ta)+\frac{2}{3}\left(n_i+n_e\right)\ta\pa_x u_1
\\[3mm]
&\dis =\pa_x\left((\kappa_i(\theta)+\kappa_e(\theta))\pa_x \theta\right)+
3(\mu_i(\theta)+\mu_e(\theta))(\pa_x u_1)^2+\sum\limits_{j=2}^3(\mu_i(\theta)+\mu_e(\theta))(\pa_x u_j)^2
\\[3mm]
&\quad\dis-\int_{{\R}^3}\xi_1\left(\psi_{6}-\sum\limits_{j=1}^3u_j \psi_{j+2}\right)\cdot\pa_x\overline{\FR}\,d\xi
-\int_{{\R}^3}\psi_{6i}\xi_1\pa_x\left(P_0^{M_i}G_i\right)\,d\xi\\[3mm]
&\quad\dis-\int_{{\R}^3}\psi_{6e}\xi_1\pa_x\left(P_0^{M_e}G_e\right) d\xi
+\sum\limits_{j=1}^3u_j\int_{{\R}^3}\psi_{(j+2)i}\xi_1\pa_x \left(P_0^{M_i}G_i\right)\,d\xi\\[3mm]
&\quad\dis+\sum\limits_{j=1}^3u_j\int_{{\R}^3}\psi_{(j+2)e}\xi_1\pa_x \left(P_0^{M_e}G_e\right)\, d\xi
+\ta\int_{{\R}^3}[\xi_1,\xi_1]^{\rm T}\cdot\pa_x\FG\,d\xi\\[3mm]
&\quad\dis+\pa_x\phi\int_{{\R}^3}\frac{1}{2}|\xi|^2\left[q_i,q_e\right]^{\rm T}\cdot\pa_{\xi_1}\FG\,d\xi,
%\\[3mm]
%&-\pa_x^2\phi=q_in_i+q_en_e.
%\end{split}
\end{array}
%\right.
\end{eqnarray}
%Here $P=\frac{2}{3}\left(n_i+n_e\right)\ta$,
where we have denoted $\overline{\FR}=[\overline{R}_i,\overline{R}_e]^{\rm T}$. Note that $n_i,n_e$ satisfy equations of mass conservation:
\begin{align}\label{BE-NS-mass}
\begin{aligned}
\pa_tn_i+\pa_x (n_iu_1)&=-\int_{{\R}^3}\xi_1\pa_xG_i \,d\xi,\\
\pa_tn_e+\pa_x (n_eu_1)&=-\int_{{\R}^3}\xi_1\pa_xG_e \,d\xi.
\end{aligned}
\end{align}
Moreover, as for considering \eqref{cons.law.f}, if one further ignores those terms involving the non-fluid part $\FG$, one has the closed viscous fluid-type system of six knowns $n_i,n_e,u_1,u_2,u_3,\theta$:
\begin{eqnarray}\label{cons.law.fv}
\left\{
\begin{array}{clll}
& \pa_tn_i+\pa_x (n_iu_1)=0,\\[3mm]
& \pa_tn_e+\pa_x (n_eu_1)=0,\\[3mm]
& (m_in_i+m_en_e)(\pa_t u_1+u_1\pa_xu_1)+\frac{2}{3}\pa_x \left[(n_i+n_e)\ta\right]\\[3mm]
& \qquad\qquad\qquad\qquad+(q_in_i+q_en_e)\pa_x\phi=3\pa_x\left[(\mu_i(\theta)+\mu_e(\theta))\pa_x u_1\right],\\[3mm]
& (m_in_i+m_en_e)(\pa_t u_j+u_1\pa_xu_j)=\pa_x\left[(\mu_i(\theta)+\mu_e(\theta))\pa_x u_j\right],\quad j=2,3,\\[3mm]
& (n_i+n_e)(\pa_t \theta+u_1\pa_x\ta)+\frac{2}{3}(n_i+n_e)\ta\pa_x u_1 =\pa_x\left((\kappa_i(\theta)+\kappa_e(\theta))\pa_x \theta\right)\\[3mm]
&\qquad\qquad\qquad\qquad\qquad\qquad+
3(\mu_i(\theta)+\mu_e(\theta))(\pa_x u_1)^2+\sum\limits_{j=2}^3(\mu_i(\theta)+\mu_e(\theta))(\pa_x u_j)^2,\\[3mm]
& -\pa_x^2\phi=q_in_i+q_en_e.
\end{array}
\right.
\end{eqnarray}
Note that \eqref{cons.law.fv} could be thought to be the first-order fluid dynamic approximation of the VPB system \eqref{VPB}, \eqref{PE}.

%\Red{Add some words for how to formally obtain the two-fluid system with diffusions.}

%\subsection{Equation of entropy}

%It should be pointed out that $\FP_1^{\FM}\left\{{\rm diag}(m_iA_{ij},m_eA_{ej})\FM\right\}\neq0$ for any given $C_i$ and $C_e$.
%and $\mu_j(\ta)$, $\ka (\ta)$ may dependent on $n_i$ and $n_e$.

%Note that $\eqref{BE-NS}_1$ and $\eqref{BE-NS}_2$ imply
%\begin{equation}\label{v.eq}
%\pa_tv-\pa_xu_1=0.
%\end{equation}

For later use, we also introduce the entropy quantity and the corresponding equation. For given densities $n_i$, $n_e$  and temperature $\ta$, we define the entropy $S$ by
%\Blue{(Here, S may depend on the mass, for brevity, I use the simple form, SQ)}
\begin{equation}\label{BE.enp.}
S=-\frac{2}{3}\ln (n_i+n_e)+\ln \left(\frac{4\pi}{3}  \theta\right)+1.
\end{equation}
%Then from \eqref{v.eq} and $\eqref{BE-NS}_5$, we see that $S_i$ and $S_e$ satisfy
According to \eqref{BE-NS-m} and \eqref{BE-NS-t} as well as \eqref{BE-NS-mass}, one deduces that $S$ satisfies
\begin{equation*}%\label{S}
\begin{split}
\pa_t S+u_1\pa_xS=&\frac{1}{\overline{n}\ta}\pa_x\left((\kappa_i(\theta)+\kappa_e(\theta))\pa_x \theta\right)+\frac{3(\mu_i(\theta)+\mu_e(\theta))}{\overline{n}\ta}(\pa_x u_j)^2+
\sum\limits_{j=2}^3\frac{(\mu_i(\theta)+\mu_e(\theta))}{\overline{n}\ta}(\pa_x u_j)^2
\\&+\left(1-\frac{2}{3\overline{n}}\right)\int_{{\R}^3}[\xi_1,\xi_1]^{\rm T}\cdot\pa_x\FG  \,d\xi
-\frac{1}{\overline{n}\ta}\int_{{\R}^3}\xi_1\left(\psi_{6}-\sum\limits_{j=1}^3u_j \psi_{j+2}\right)\cdot\pa_x\overline{\FR} \,d\xi
\\&-\frac{1}{\overline{n}\ta}\int_{{\R}^3}\psi_{6i}\xi_1\pa_x\left(P_0^{M_i}G_i\right) \,d\xi-\frac{1}{\overline{n}\ta}\int_{{\R}^3}\psi_{6e}\xi_1\pa_x\left(P_0^{M_e}G_e\right) \,d\xi
\\&+\frac{1}{\overline{n}\ta}\sum\limits_{j=1}^3u_j\int_{{\R}^3}\psi_{(j+2)i}\xi_1\pa_x \left(P_0^{M_i}G_i\right) \,d\xi
+\frac{1}{\overline{n}\ta}\sum\limits_{j=1}^3u_j\int_{{\R}^3}\psi_{(j+2)e}\xi_1\pa_x \left(P_0^{M_e}G_e\right) \,d\xi
\\&+\frac{\pa_x\phi}{\overline{n}\ta}\int_{{\R}^3}\frac{|\xi|^2}{2}\left[q_i,q_e\right]^{\rm T}\cdot\pa_{\xi_1}\FG \,d\xi,
\end{split}
\end{equation*}
where  we have denoted $\overline{n}=n_i+n_e$. For later use, we also introduce the pressure function $P$ by
\begin{equation}\notag
%\label{ }
P=\frac{2}{3}\overline{n}\ta.
\end{equation}
Note that from \eqref{BE.enp.}, one has
$$
\theta=\frac{3}{2}k e^S\overline{n}^{2/3},\ \ P=\frac{2}{3}\overline{n}\ta=ke^S\overline{n}^{5/3},
$$
with the constant $k$ given by $k:= \frac{1}{2\pi e}$.
%according to $\eqref{BE-NS}_1$ and $\eqref{BE-NS}_4$. %Here $\overline{v}=\left(\frac{1}{Zv_i}+\frac{1}{v_e}\right)v$.
%we also define the following two entropy quantities $S_i$ and $S_e$
%as
%\begin{equation}\label{BE.enpi.}
%S_i\eqdef \frac{2}{3}\ln (Zv_i)+\ln \left(\frac{4\pi}{3Zm_i}  \theta\right)+1,
%\end{equation}
%and
%\begin{equation}\label{BE.enpe.}
%S_e\eqdef \frac{2}{3}\ln v_e+\ln \left(\frac{4\pi}{3 m_e} \theta\right)+1,
%\end{equation}
%respectively.
%Then from $\eqref{BE-NS}_1$, $\eqref{BE-NS}_2$ and $\eqref{BE-NS}_5$, we see that $S_i$ and $S_e$ satisfy
%\begin{equation*}%\label{S}
%\begin{split}
%\pa_t S=&\frac{1}{\overline{v}\ta}\pa_x\left(\frac{\kappa(\theta)}{v}\pa_x \theta\right)+\frac{4\mu(\theta)}{3v\overline{v}\ta}(\pa_x u_1)^2+
%\sum\limits_{i=2}^3\frac{\mu(\theta)}{v\overline{v}\ta}(\pa_x u_i)^2
%\\&-\frac{1}{\overline{v}\ta}\int_{{\R}^3}(|\xi|^2/2-u\cdot\xi)\xi_1[Zm_i,m_e]^{\rm T}\cdot\pa_x\FR d\xi
%+\frac{1}{\overline{v}}\int_{{\R}^3}[\xi_1,\xi_1]^{\rm T}\cdot\pa_x\FG  d\xi
%\\&+\frac{e\pa_x}{\overline{v}\ta}\phi\int_{{\R}^3}\frac{|\xi|^2}{2}\left[1,-1\right]^{\rm T}\cdot\pa_{\xi_1}\FG d\xi
%+\frac{Zv_i}{v}\int_{{\R}^3}\xi_1\pa_xG_i d\xi,
%\end{split}
%\end{equation*}

%Also note that in terms of $S_\al$ and $n_\al$,
%\begin{equation*}
%\label{ }
%p_\al=\frac{m_\al}{2\pi e} \exp\left(\frac{\ga-1}{R} S_\al\right) n_\al^\ga,\quad \al=i,e.
%\end{equation*}

%so that the pressure can be written as
%$$
%P=\frac{2}{3}\rho\theta=k e^{S}\rho^{5/3}.
%$$

%\subsection{Quasineutral Euler equations and rarefaction waves}\label{rarefaction.w}
\section{Quasineutral Euler equations and rarefaction waves}\label{sec.rw}

Recall that \eqref{cons.law.f} and \eqref{cons.law.fv} are thought to be the zero-order and first-order fluid dynamic approximation of the VPB system \eqref{VPB}, \eqref{PE}, respectively, if the two-component non-fluid part $\FG$ is ignored. Inspired by this, one may expect to justify in a rigorous way the large-time asymptotics of the VPB system  \eqref{VPB}, \eqref{PE} toward   \eqref{cons.law.f} or \eqref{cons.law.fv}. The goal of this paper is to treat this in the setting of rarefaction waves. Instead of directly using \eqref{cons.law.f} and \eqref{cons.law.fv}, the expected large-time asymptotic system is the quasineutral Euler system in the form of
\begin{eqnarray}\label{cons.law.f.qe}
\left\{
\begin{array}{clll}
&\dis \pa_tn_i+\pa_x (n_iu_1)=0,\\[3mm]
&\dis \pa_tn_e+\pa_x (n_eu_1)=0,\\[3mm]
&\dis (m_in_i+m_en_e)(\pa_t u_1+u_1\pa_xu_1)+\frac{2}{3}\pa_x \left[(n_i+n_e)\ta\right]
%+(q_in_i+q_en_e)\pa_x\phi
=0,\\[3mm]
%&\dis (m_in_i+m_en_e)(\pa_t u_j+u_1\pa_xu_j)=0,\quad j=2,3,\\[3mm]
&\dis (n_i+n_e)(\pa_t \theta+u_1\pa_x\ta)+\frac{2}{3}(n_i+n_e)\ta\pa_x u_1=0,\\[3mm]
&\dis q_in_i+q_en_e=0.
\end{array}
\right.
\end{eqnarray}
For simplicity, by letting
\begin{equation}\notag
%\label{ }
n_e=n,\quad n_i=-\frac{q_i}{q_e} n_e=-\frac{q_i}{q_e}n,
\end{equation}
in terms of the quasineutral assumption, \eqref{cons.law.f.qe} reduces to
\begin{eqnarray}\label{cons.law.f.as}
\left\{
\begin{array}{clll}
&\dis \pa_tn+\pa_x (nu_1)=0,\\[3mm]
&\dis \pa_t u_1+u_1\pa_xu_1+\frac{1}{n}\pa_xP(n,S)=0\\[3mm]
&\dis \pa_t S+u_1\pa_x S=0,
\end{array}
\right.
\end{eqnarray}
where
\begin{equation}\notag
%\label{ }
P(n,S)=\frac{2(q_i-q_e)}{3(m_eq_i-m_iq_e)}n\theta,\quad S=-\frac{2}{3}\ln \left(\frac{q_i-q_e}{q_i}n\right)+\ln\left(\frac{4\pi
R}{3}\theta\right)+1.
\end{equation}

To construct the large-time asymptotic rarefaction wave of the VPB system through \eqref{cons.law.f.qe} or \eqref{cons.law.f.as}, one has to assign some appropriate far-field data from \eqref{bd.be}, \eqref{bd.phi} and \eqref{id.be.bim}.
%Through the paper, we assume
%\begin{equation}
%\label{id.be.bim}
%\FF_{0\pm\infty}=\FM_{\pm\infty}=
%\left[
%\begin{array}{rll}
%      &M_{[n_{i\pm},u_{\pm},\theta_\pm;m_i]}(\xi)\\[2mm]
%      &M_{[n_{e\pm},u_{\pm},\theta_\pm;m_e]}(\xi)
%\end{array}\right],
%\end{equation}
%where $n_{i\pm}$, $n_{e\pm}$, $u_\pm=(u_{1\pm,0,0})$, $\theta_\pm$ are given constants, with
%\begin{equation}
%\label{ }
%q_in_{i\pm}+q_en_{e\pm}=0.
%\end{equation}
Recall that we have set $n_{e\pm}=n_{\pm}$ and hence $n_{i\pm}=-\frac{q_e}{q_i}n_{\pm}$.
In terms of $n_\pm$ and $\theta_{\pm}$, recalling \eqref{BE.enp.}, we define constants $S_\pm$ by
\begin{eqnarray*}
S_{\pm}=-\frac{2}{3}\ln \left(n_{i\pm}+n_{e\pm}\right)+\ln\left(\frac{4\pi
R}{3}\theta_{\pm}\right)+1,
\end{eqnarray*}
that is,
\begin{eqnarray*}%\label{S.fdc}
S_{\pm}=-\frac{2}{3}\ln \left(\frac{q_i-q_e}{q_i}n_{\pm}\right)+\ln\left(\frac{4\pi
R}{3}\theta_{\pm}\right)+1,
\end{eqnarray*}
% For later use, under the assumption \eqref{bd.be},
%%and \eqref{dd-ff1}, for $\CA=i,e$
%we  set  $S_{\pm}$ to be the values of $S$ at both far fields in terms of $n_\pm$ and $\theta_{\pm}$, that is
%
%at both far fields. It is natural to
To the end, we assume $S_+=S_-$
%when constructing the rarefaction wave of the system \eqref{BE-NS}, which means
or equivalently
\begin{equation*}%\label{Stn.com}
\frac{\theta_+}{n_+^{2/3}}=\frac{\theta_-}{n_-^{2/3}}=\frac{3}{2}ke^{S_\pm}\left(\frac{q_i-q_e}{q_i}\right)^{2/3}:=A.
\end{equation*}
%\begin{equation*}\label{Stn.com}
%\frac{\theta_{+}}{\theta_{-}}=\left(\frac{n_{+}}{n_{-}}\right)^{2/3}.
%\end{equation*}
%must be satisfied.

Under the above settings on the far-field values of initial data \eqref{id.be} for the VPB system \eqref{VPB}, \eqref{PE},  we then expect that
%
%
%
%To construct the time-asymptotic rarefaction wave of the Cauchy problem \eqref{VPB}, \eqref{id.be}, \eqref{bd.be} and
%\eqref{bd.phi}, we first make the quasi-neutral assumption
%$$
%q_in_i+q_en_e=0,\ \ {\rm i.e.}\ \ n_i=-\frac{q_e}{q_i}n_e=n,\ {\rm as}\ t\rightarrow\infty,
%$$
%and expect that
the solution $\FF(t,x,\xi)$ to the Cauchy problem tends time-asymptotically to the local bi-Maxwellian
\begin{equation*}
%\label{bi-M}
\FM_{R}
=\left[\begin{array}{cc}
      M_{Ri}(\xi)    \\[3mm]
      M_{Re}(\xi)
\end{array}\right]
=\left[\begin{array}{c}
      M_{[-\frac{q_e}{q_i}n^R,u^{R},\theta^{R};m_i]}(\xi)    \\[3mm]
      M_{[n^{R},u^{R},\theta^{R};m_e]}(\xi)
\end{array}\right],
\end{equation*}
where $[n^R,u^R,\theta^R]$ with $u^{R}=[u_1^{R},0,0]$ is the rarefaction wave solution of the Riemann problem on the  quasi-neutral Euler system \eqref{cons.law.f.as}
%\begin{eqnarray}\label{Euler1}
%\left\{
%\begin{array}{clll}
%\begin{split}
%&\pa_tn+\pa_x (nu_1)=0,\\[3mm]
%&\pa_t u_1+u_1\pa_xu_1+\frac{q_i}{(m_eq_i-m_iq_e)n}\pa_xP(n,S_-)=0,\\[3mm]
%%&\pa_t\ta+u_1\pa_x\ta+\frac{2}{3}\ta\pa_xu_1=0,\\[3mm]
%%&S=S_+=S_-,\\[3mm]
%%&P(n,S_-):=\frac{2}{3} \frac{q_i-q_e}{q_i}n\ta=\frac{2}{3}A\frac{q_i-q_e}{q_i}n^{5/3},\ \ \ta=An^{2/3},\ \ A:=\frac{3}{2}ke^{S_\pm}\left(\frac{q_i-q_e}{q_i}\right)^{2/3},%\ \ S_+=S_-,
%\end{split}
%\end{array}
%\right.
%\end{eqnarray}
with Riemann initial data given by
\begin{equation}
\label{MEtid}
[n, u_1,\theta](0,x)=\left[n^{R}_0,u^{R}_{1,0},\theta_0\right](x):= \left\{\begin{array}{rll}[n_-,u_{1-},\theta_-],&\ \ x<0,\\[3mm]
[n_+,u_{1+},\theta_+],&\ \ x>0.
\end{array}
\right.
\end{equation}
The Riemann problem can be solved in the usual way (cf.~\cite{LX}). Indeed, the quasineutral Euler system \eqref{cons.law.f.as}
has three characteristics
\begin{eqnarray}%\label{chara.}
\left\{\begin{array}{rll}
\la_1&=&\la_1(n,u_1,S):= u_1-\sqrt{\pa_n P (n,S)},\\[3mm]
\la_2&=&\la_2(n,u_1,S):= u_1,\\[3mm]
\la_3&=&\la_3(n,u_1,S):= u_1+\sqrt{\pa_n P (n,S)}.
\end{array}\right.\notag
\end{eqnarray}
In terms of  two Riemann invariants of the third eigenvalue $\la_3(n,u_1,S)$,
%regarding the original quasineutral Euler system \eqref{cons.law.f.qe} of $[\rho,u_1,\theta]$,
we define the set of right constant states  $[n_+,u_{1+},\theta_+]$ to which a given left constant state $[n_-,u_{1-},\theta_-]$ with $n_->0$ and $\theta_->0$ is connected  through the $3$-rarefaction wave to be
\begin{multline}\label{def.r3}
R_3(n_-,u_{1-}, \theta_-)\equiv\bigg\{[n,u_1,\theta]\in\R_+\times\R \times\R_+\ \Big|\
%\left(\frac{\rho}{\rho_-}\right)^{\frac{2}{3}}=\frac{\theta}{\theta_-},
\frac{n^{2/3}}{\theta}=\frac{n_-^{2/3}}{\theta_-},
\\
u_1-u_{1-}=\int_{n_-}^n
%\sqrt{\frac{5}{3} A_i \varrho^{-\frac{4}{3}}+\varrho^{-1}\left(\frac{d}{d\rho} (\rho_e^{-1})\right)(\varrho)}
\frac{\sqrt{\pa_n P (\eta,S_-) }}{\eta}
d\eta,
%u_2=0,\ {u}_3=0,
\ n>n_-,\ {u}_1>u_{1-}\bigg\}.
\end{multline}
%Here and in the sequel %$A_i=ke^{S_i}$.
%$S_i\eqdef -\frac{2}{3}\ln \rho_-+\ln (\frac{4}{3}\pi\theta_-) +1$ is a constant.
%
%The quasi-neutral Euler system \eqref{Euler1} has two distinct eigenvalues given by
%\begin{equation*}\label{ev}
%\la_1(n,u_1,S_-)=u_1-\sqrt{\frac{q_i\pa_nP(n,S_-)}{m_eq_i-m_iq_e}},\ \ \ \la_2(n,u_1,S_-)=u_1+\sqrt{\frac{q_i\pa_nP(n,S_-)}{m_eq_i-m_iq_e}}.
%\end{equation*}
%Both of $\la_1$ and $\la_2$ are genuinely nonlinear and can generate $1$-rarefaction wave and $2$-rarefaction wave solution  of \eqref{Euler1}.
%For brevity, we restrict our discussion to the 2-rarefaction. That is, regarding the quasineutral Euler system \eqref{Euler1} and \eqref{ta.p} of variables $[n,u_1,\ta]$, we define the set of right constant states  $[n_+,u_{1+},\ta_+]$ to which a given left constant state $[n_-,u_{1-},\ta_-]$ with $n_->0$ and $\ta_->0$ is connected through the $2$-rarefaction wave to be
%\begin{multline*}\label{def.r3}
%R_2(n_-,u_{1-}, \ta_-)\equiv\Big\{[n,u_1,\ta]\in\R_+\times\R \times\R_+ \ \Big|\
%\frac{\ta}{n^{2/3}}=\frac{\ta_-}{n_-^{2/3}},\\
%u_1=u_{1-}+\int_{n_-}^n
%s^{-1}\sqrt{\frac{q_i\pa_sP(s,S_-)}{m_eq_i-m_iq_e}}~ds,\
%\ n>n_-,\ {u}_1>u_{1-}\Big\}.
%\end{multline*}
%We emphasize that $\frac{\ta}{n^{2/3}}=\frac{\ta_-}{n_-^{2/3}}$ follows from
%$S=S_-=S_+$ and $S_\pm=-\frac{2}{3}\ln \left(\frac{q_i-q_e}{q_i}n_{\pm}\right)+\ln \frac{4\pi}{3}\ta_\pm+1$.
%
Now, letting $[n_+,u_{1+},\theta_+]\in R_3(n_-,u_{1-}, \theta_-)$,
the Riemann problem \eqref{cons.law.f.as}, \eqref{MEtid} admits a self-similar solution, the $3$-rarefaction wave
$\left[n^R,u_1^R,\ta^R\right](z)$ with $z=x/t\in \R$, explicitly defined by
\begin{eqnarray}%\label{Org.RW.}
\begin{split}
\left\{\begin{array}{rll}
&\la_3\left(n^R(z),u_1^R(z),S_-\right)=
\left\{\begin{array}{ll}
\la_3(n_-,u_{1-},S_-)&\quad \mbox{for}\  z<\la_3(n_-,u_{1-},S_-),\\[3mm]
z&\quad \mbox{for}\ \la_3(n_-,u_{1-},S_-)\leq z\leq  \la_3(n_+,u_{1+},S_-),\\[3mm]
\la_3(n_+,u_{1+},S_-) &\quad \mbox{for}\ z> \la_3(n_+,u_{1+},S_-),
\end{array}\right.\\[10mm]
&u^R_{1}(z)=u_{1-}+
\dis{\int_{n_-}^{n^R(z)}} \frac{\sqrt{\pa_n P (\eta,S_-) }}{\eta}
d\eta,
%\ u^R_2=0,\ u^R_3=0,
\\[5mm]
&%S^R=S_i,\
\ta^R(z)=A(n^R(z))^{2/3},\ \ A=\frac{3}{2}ke^{S_-}\left(\frac{q_i-q_e}{q_i}\right)^{2/3}.
\end{array}
\right.
\end{split}\notag
\end{eqnarray}

Since $[n^{R},u_1^{R},\ta^{R}]$ is a weak solution of the Riemann problem \eqref{cons.law.f.as} and \eqref{MEtid} and lack of regularity, one has establish a smooth approximation to the rarefaction wave $[n^{R},u_1^{R},\ta^{R}]$. To do this, in the usual way,  the smooth rarefaction wave $[n^r,u^r,\theta^r](t,x)$ with $u^r(t,x)=[u_1^r(t,x),0,0]$ is defined by
\begin{eqnarray}\label{1-RW.def.}
\begin{split}
\left\{\begin{array}{rll}
&%S^r(t,x)=S_i,\ \ta^r(t,x)=\frac{3A_i}{2}(\rho^r)^{2/3}(t,x)e^{S_i},\
\la_3(n^r(t,x),u^r_{1},S_-)=w(t,x),\\[3mm]
&u^r_1(t,x)=u_{1-}
+\dis{\int_{n_-}^{n^r(t,x)}}  \frac{\sqrt{\pa_n P (\eta,S_-) }}{\eta}
d\eta,%\ \ u^r_2=u^r_3=0,\ \ u^r=[u_1^r,u_2^r,u_3^r],
\\[4mm]
&\ta^r(t,x)=A(n^r(t,x))^{2/3},%\ \ A=\frac{3}{2}ke^{S_-}\left(\frac{q_i-q_e}{q_i}\right)^{2/3},
\\[3mm]
&\lim\limits_{x\to \pm\infty}[n^r,u_1^r,\theta^r](t,x)=[n_\pm,u_{1\pm},\theta_\pm],\quad
[n_+,u_{1+},\theta_+]\in R_3(n_-,u_{1-}, \theta_-),
\end{array}
\right.
\end{split}
\end{eqnarray}
with $w=w(t,x)$ being the solution to the Burgers' equation
\begin{equation}\label{cl.Re.cons.in}
\left\{\begin{array}{l}
\dis \pa_tw+w\pa_xw=0,\\[3mm]
w(0,x)=w_{0}(x):=\frac{1}{2}(w_{+}+w_{-})+\frac{1}{2}(w_{+}-w_{-})\tanh x,\quad w_\pm:= \la_3(n_\pm,u_{1\pm},S_-).
\end{array}\right.
\end{equation}
%Here $\eps>0$ is a constant to be chosen later on.
We remark that by letting $n_e^r=n_r$ and
%To make our presentation more easier to read, we now denote
$n_i^r=-\frac{q_e}{q_i}n^r$, in view of the construction of the smooth rarefaction wave above,  %the smooth approximation profile of $\phi$ can be constructed as follows
$[n_i^{r},n_e^{r},u_1^{r},\ta^{r}]$  satisfies
\begin{eqnarray}\label{Euler3}
\left\{
\begin{array}{clll}
\begin{split}
&\pa_tn_{i}^r+\pa_x (n_{i}^ru^r_1)=0,\\
&\pa_tn_{i}^r+\pa_x (n_{i}^ru^r_1)=0,\\
%&-\frac{q_e}{q_i}m_in^r\pa_t u^r_1-\frac{q_e}{q_i}m_in^ru^r_1\pa_xu_1^r-\frac{2q_e}{3q_i}\pa_x(n^r\ta^r)-q_en^r\pa_x\phi^r=0,\\[3mm]
&(m_in_i^r+m_en_e^r)(\pa_t u^r_1+u_1^r\pa_xu_1^r)+\frac{2}{3}\pa_x((n_i^r+n_e^r)\ta^r)=0,\\
&(q_im_in_i^r+q_em_en_e^r)(\pa_t u^r_1+u_1^r\pa_xu_1^r)
\\&\qquad=-\frac{2\ta^r}{3}\frac{q_im_in_i^r+q_em_en_e^r}{m_in^r_i+m_en^r_e}\pa_x(n_i^r+n_e^r)
-\frac{2}{3}\pa_x\ta^r\frac{q_im_in_i^r+q_em_en_e^r}{m_in^r_i+m_en^r_e}(n_i^r+n_e^r),\\[3mm]
&(n_i^r+n_e^r)(\pa_t \ta^r+u^r_1\pa_x\ta^r)+P^r\pa_xu^r_1=0.\\
\end{split}
\end{array}
\right.
\end{eqnarray}
Here $P^r=\frac{2}{3} (n_i^r+n_e^r)\ta^r=\frac{2}{3}A\frac{q_i-q_e}{q_i}(n^r)^{5/3}.$

The next lemma is devoted to the study of the properties of the smooth rarefaction wave $[n^r,u_1^r,\ta^r]$ constructed in \eqref{1-RW.def.} and \eqref{cl.Re.cons.in}.
\begin{lemma}\label{cl.RwRe}
%The approximation solution $[\rho^{r}, u_1^{r},\ta^r]$  given by \eqref{1-RW.def.}
%satisfies
It holds that

\noindent$(i)$ $\pa_xu_1^{r}(t,x)>0$ and $n_-<n^{r}(t,x)<n_+$,
$u_{1-}<u_1^{r}(t,x)<u_{1+}$ for $x\in\R$ and
$t\geq0$.

\noindent$(ii)$ For any $1\leq p\leq+\infty$, there exists a constant $C_p$ such that for $t>0$,
$$
\left\|\pa_x\left[n^{r}, u_1^{r},\ta^r\right]\right\|_{L^p}\leq C_p\min\left\{\delta_r, \delta_r^{1/p}t^{-1+1/p}\right\},
$$
$$
\left\|\pa^j_x\left[n^{r}, u_1^{r},\ta^r\right]\right\|_{L^p}\leq C_p\min\left\{\delta_r, t^{-1}\right\},\ \ j\geq2, 
$$
where we recall that $\de_r=|n_+-n_-|+|u_{1+}-u_{1-}|+|\ta_+-\ta_-|$ is the wave strength.

\noindent$(iii)$ $\lim\limits_{t\rightarrow+\infty}\sup\limits_{x\in\R}\left|\left[n^{r}, u_1^{r},\ta^r\right](t,x)-\left[n^R,u_1^R,\ta^R\right](x/t)\right|=0$.

\end{lemma}

%Here $\eqref{Euler3}_3$ is motivated from the summation of $q_i\eqref{cons.law.i}_2$ and $q_e\eqref{cons.law.e}_2$.

%From \eqref{Euler3}, one sees that the smooth approximation profile $\phi^r$ can be given by
%\begin{equation*}\label{p.phi}
%\phi^r(t,x)=\frac{(q_i-q_e)(m_i-m_e)}{6q_e(q_im_e-q_em_i)}\ta^r=\frac{A(q_i-q_e)(m_i-m_e)}{6q_e(q_im_e-q_em_i)}(n^r)^{2/3}.
%\end{equation*}

%\subsection{Main results}
%\section{The main result}

%Our main purpose in this paper is to
%show that the global solution of VPB system \eqref{v.F} exists and convergences to the local
%thermo-equilibrium determined by the given rarefaction wave as $t\rightarrow\infty$.

\section{Preliminary estimates on two-component collision operator}\label{sec4}

In this section, we list some basic inequalities on the two-component collision operator for later use. The first
lemma is concerned with the nonlinear
 collision operators  $Q_{\CA\CB}(\cdot, \cdot)$, whose proof can be found in \cite{GPS} when masses of the particles are normalised to be one.

\begin{lemma}\label{est.nonop}Let $\CA,\CB\in\{i,e\}$.
There exists a positive constant $C>0$ such
that
\begin{equation}\label{est.nonop.ine.}
\begin{split}
{\displaystyle\int_{{\R}^3}}&\frac{(1+|\xi|)^{-1}
\left|Q_{\CA\CB}(F_\CA,F_\CB)\right|^2}{\widehat{M}_\CA}\,d\xi
\\[3mm] \leq& C\left\{{\displaystyle\int_{{\R}^3}}
\frac{(1+|\xi|)F_{\CA}^2}{\widehat{ M}_\CA} \,d\xi\cdot
{\displaystyle\int_{{\R}^3}}\frac{F_{\CB}^2}{\widehat{ M}_\CB} \,d\xi
+{\displaystyle\int_{{\R}^3}}\frac{F_{\CA}^2}{\widehat{ M}_\CA}
\,d\xi\cdot{\displaystyle\int_{{\R}^3}}\frac{(1+|\xi|)F_{\CB}^2}{\widehat{M}_\CB} \,d\xi\right\},
\end{split}
\end{equation}
where we have defined
$$
\left[\widehat{M}_i, \widehat{M}_e\right]^{\rm{T}}\equiv\left[M_{[\widehat{n}_i,\widehat{u},\widehat{\theta};m_i]}(\xi),
M_{[\widehat{n}_e,\widehat{u},\widehat{\theta};m_e]}(\xi)\right]^{\rm{T}},
$$
to be any bi-Maxwellian such that the above
integrals are well defined.
\end{lemma}

\begin{proof}
Note that one can rewrite \eqref{g.cop} as
\begin{equation}
\label{Q.gl}
Q_{\CA\CB}(F_\CA,F_\CB)=Q^{\rm gain}_{\CA\CB}(F_\CA,F_\CB)+Q^{\rm loss}_{\CA\CB}(F_\CA,F_\CB),
\end{equation}
with the normal meaning for the gain part and the loss part.
To prove \eqref{est.nonop.ine.}, we first consider the gain part and thus compute
\begin{eqnarray}
&&\int_{{\R}^3}\frac{(1+|\xi|)^{-1}
\left|Q^{\rm gain}_{\CA\CB}(F_\CA,F_\CB)\right|^2}{\widehat{M}_\CA}\,d\xi\notag\\
&&=\si^2\int_{\R^3}(1+|\xi|)^{-1}\widehat{M}_\CA^{-1}
\left(\int_{{\R}^3}|(\xi-\xi_\ast)\cdot\omega|F_{\CA}(\xi')F_{\CB}(\xi_\ast')d\xi_\ast\right)^2\,d\xi.\notag
\end{eqnarray}
We now set
\begin{equation}
\label{def.fab}
F_\CA=\sqrt{\widehat{M}_\CA} f_\CA,\quad F_\CB=\sqrt{\widehat{M}_\CB} f_\CB,
\end{equation}
and use the identity
$$
\widehat{M}_\CA(\xi)\widehat{M}_\CB(\xi_\ast)=\widehat{M}_\CA(\xi')\widehat{M}_\CA(\xi'_\ast),
$$
so as to derive
\begin{equation*}
\begin{split}
&\int_{\R^3}(1+|\xi|)^{-1}\widehat{M}_\CA^{-1}
\left(\int_{{\R}^3}|(\xi-\xi_\ast)\cdot\omega|F_{\CA}(\xi')F_{\CB}(\xi_\ast')d\xi_\ast\right)^2d\xi
\\&=\int_{\R^3}(1+|\xi|)^{-1}\widehat{M}_\CA^{-1}
\left(\int_{{\R}^3}\sqrt{\widehat{M}_\CA(\xi)\widehat{M}_\CB(\xi_\ast)}
|(\xi-\xi_\ast)\cdot\omega|f_{\CA}(\xi')f_{\CB}(\xi_\ast')d\xi_\ast\right)^2d\xi
\\&\leq C\int_{\R^3}(1+|\xi|)^{-1}\left\{\int_{{\R}^3}\widehat{M}_\CB(\xi_\ast)
|(\xi-\xi_\ast)\cdot\omega|d\xi_\ast\int_{{\R}^3}|(\xi-\xi_\ast)\cdot\omega|\left|f_{\CA}(\xi')f_{\CB}(\xi_\ast')\right|^2d\xi_\ast\right\} d\xi
\\&\leq C\int_{\R^3\times \R^3}|\xi-\xi_\ast|\left|f_{\CA}(\xi')f_{\CB}(\xi_\ast')\right|^2d\xi_\ast d\xi,
\end{split}
\end{equation*}
where we have used the H\"{o}lder's inequality to obtain the first inequality above. In view of $|\xi-\xi_\ast|=|\xi'-\xi'_\ast|$
and by a change of variables
$(\xi,\xi_\ast)\rightarrow (\xi',\xi_\ast')$, one further has
\begin{equation}\label{gop.es}
\begin{split}
&\int_{\R^3\times \R^3}|\xi-\xi_\ast|\left|f_{\CA}(\xi')f_{\CB}(\xi_\ast')\right|^2\,d\xi_\ast d\xi
\\ &\leq C\int_{\R^3\times \R^3}(1+|\xi|+|\xi_\ast|)|f_{\CA}(\xi)f_{\CB}(\xi_\ast)|^2\,d\xi_\ast d\xi
\\& \leq C\int_{\R^3}(1+|\xi|)|f_{\CA}(\xi)|^2\,d\xi \int_{\R^3}|f_{\CB}(\xi_\ast)|^2\,d\xi_\ast+
C\int_{\R^3}|f_{\CA}(\xi)|^2\,d\xi \int_{\R^3}(1+|\xi_\ast|)|f_{\CB}(\xi_\ast)|^2\,d\xi_\ast,
\end{split}
\end{equation}
where the fact that
$$
\left|\frac{\pa (\xi',\xi'_\ast)}{\pa (\xi,\xi_\ast)}\right|=1,
$$
has been used. Rewriting \eqref{gop.es} in terms of  \eqref{def.fab} gives \eqref{est.nonop.ine.} for the contribution from the gain part in \eqref{Q.gl}.
As to the loss part in \eqref{Q.gl}, the proof is similar and details are omitted for brevity. This then completes the proof of Lemma \ref{est.nonop}.
\end{proof}

In order to obtain the energy estimates for the
Boltzmann equation \eqref{v.F}, for ${\bf P}_1^{\FM}\FF$ which means the
microscopic projection of its solution $\FF(t,x,\xi)$ with respect
to a given bi-Maxwellian
$$\FM=[M_i,M_e]^{\rm T}=[M_{[n_i(t,x),u(t,x),\theta(t,x);m_i]}(\xi),
 M_{[n_e(t,x),u(t,x),\theta(t,x);m_e]}(\xi)]^{\rm T},$$
one need to find out its dissipative effect through
the microscopic $H$-theorem. Like the single-component case, the
microscopic $H$-theorem states that the linearized collision
operator $\FL_{\FM}$ around a fixed bi-Mawellian $\FM$ is also negative definite on the non-fluid element ${\bf
P}_1^{\FM}\FF$, cf.~\cite{BGPS}.

\begin{lemma}\label{co.est0.}
It holds that
\begin{equation}\label{co.est00.}
 -\int_{{\R}^3}{\bf P}_1^{\FM}\FF\cdot\left\{\FM^{-1}\left(\FL_{\FM}{\bf
P}_1^{\FM}\FF\right)\right\}d\xi \geq \de\int_{{\R}^3}(1+|\xi|
\left|\FM^{-1/2}{\bf P}_1^{\FM}\FF \right|^2d\xi,
\end{equation}
for a positive constant $\de>0$ depending on $[n_i,n_e,u,\ta]$. In fact, $\de$ also depends on $m_i$ and $m_e$, and in what follows we shall omit pointing out such dependence for brevity.
\end{lemma}
\begin{proof}
Recall \eqref{dec.2}. We denote
\begin{equation*}
{\bf P}_1^{\FM}\FF=\FG=\FG(t,x,\xi)=\left[\begin{array}{cc}
      G_i (\xi)\\[3mm]
      G_e (\xi)
\end{array}\right].
\end{equation*}
Further recall the definitions \eqref{def.L} and \eqref{g.cop}.  Let us decompose $\FL_{\FM}\FG$ as
\begin{equation*}%\label{L.dec}
\FL_{\FM}\FG =-{\bf \nu} \FG+{\bf K}\FG,
\end{equation*}
with
\begin{eqnarray*}%\label{L.dec2}
-{\bf \nu} \FG
=\left[
\begin{array}{ccc}
\nu_{i}G_i\\[3mm] %\ \ K_{i} = K_{i}^{1} + K_{i}^{2} + K_{i}^{3} + K_{i}^{4} \\
\nu_{e}G_e  %\ K_{e} = K_{e}^{1} + K_{e}^{2} + K_{e}^{3} + K_{e}^{4}
\end{array}\right]
,\ \ {\bf K}\FG=\left[
\begin{array}{ccc}
\sqrt{M_{i}}K_{i}\FG\\[3mm] %\ \ K_{i} = K_{i}^{1} + K_{i}^{2} + K_{i}^{3} + K_{i}^{4} \\
\sqrt{M_{e}}K_{e}\FG%\ \ K_{e} = K_{e}^{1} + K_{e}^{2} + K_{e}^{3} + K_{e}^{4}
\end{array}\right]
=\left[
\begin{array}{ccc}
\sqrt{M_{i}}K_{i}(\FM^{-1/2}\FG)\\[3mm] %\ \ K_{i} = K_{i}^{1} + K_{i}^{2} + K_{i}^{3} + K_{i}^{4} \\
\sqrt{M_{e}}K_{e}(\FM^{-1/2}\FG)%\ \ K_{e} = K_{e}^{1} + K_{e}^{2} + K_{e}^{3} + K_{e}^{4}
\end{array}\right],
\end{eqnarray*}
and
$$
\nu_\CA=-\sum\limits_{\CB\in\{\CA,\CB\}}Q_{\CA\CB}^{\rm{loss}}(1,M_\CB)
= \int_{\R^3\times\S^2_+} B_{\CA\CA} M_{\CA}(\xi_{\ast}) \,d\xi_{\ast} d\omega
 +\int_{\R^3\times\S^2_+} B_{\CA\CB} M_{\CB}(\xi_{\ast}) \,d\xi_{\ast} d\omega,
$$
\begin{equation}\label{K.def1}
K_\CA=K_\CA^1+K_\CA^2+K_\CA^3+K_\CA^4,
\end{equation}
and
\begin{eqnarray} \label{K.def2}
\left\{\begin{array}{rlll}
\begin{split}
K_{\CA}^{1} \FG =& M^{-\frac{1}{2}}_{\CA}\sum\limits_{\CB\in\{\CA,\CB\}}Q_{\CA\CB}^{\rm{loss}}(M_{\CA}, G_{\CB}) \\%+ Q_{ie}^{loss}(\FM_{i},
            =& - \int_{\R^3\times\S^2_+} B_{\CA\CA} \sqrt{M_{\CA}(\xi)} \sqrt{M_{\CA}(\xi_{\ast})} \left(\frac{G_{\CA}}{\sqrt{M_{\CA}}}\right)(\xi_{\ast}) \,d\xi_{\ast} d\omega
             \\&- \int_{\R^3\times\S^2_+} B_{\CA\CB} \sqrt{M_{\CA}(\xi)} \sqrt{M_{\CB}(\xi_{\ast})}
             \left(\frac{G_{\CB}}{\sqrt{M_{\CB}}}\right)(\xi_{\ast}) \,d\xi_{\ast} d\omega,\ \ \CA\neq\CB,\\
K_{\CA}^{2} \FG =& M^{-\frac{1}{2}}_{\CA}\left\{Q_{\CA\CA}^{\rm {gain}}(M_{\CA}, G_{\CA})
+ Q_{\CA\CA}^{\rm{gain}}(G_{\CA}, M_{\CA}\right\} \\
            =& \int_{\R^3\times\S^2_+} B_{\CA\CA}  \sqrt{M_{\CA}(\xi_{\ast})}
               \left[\sqrt{M_{\CA}(\xi')} \left(\frac{G_{\CA}}{\sqrt{M_{\CA}}}\right)(\xi'_{\ast}) + \sqrt{M_{\CA}(\xi'_{\ast})}
             \left(\frac{G_{\CA}}{\sqrt{M_{\CA}}}\right)(\xi')\right] \,d\xi_{\ast} d\omega,\\
K_{\CA}^{3} \FG =& M^{-\frac{1}{2}}_{\CA}Q_{\CA\CB}^{\rm{gain}}(M_{\CA}, G_{\CB}) \\
            =& \int_{\R^3\times\S^2_+} B_{\CA\CB} \sqrt{M_{\CB}(\xi_{\ast})} \sqrt{M_{\CA}(\xi')}
             \left(\frac{G_{\CB}}{\sqrt{M_{\CB}}}\right)(\xi'_{\ast}) \,d\xi_{\ast} d\omega,\ \ \CA\neq\CB,\\
K_{\CA}^{4} \FG =& M^{-\frac{1}{2}}_{\CA}Q_{\CA\CB}^{\rm{gain}}(G_{\CA}, M_{\CB}) \\
            =& \int_{\R^3\times\S^2_+} B_{\CA\CB}  \sqrt{M_{\CB}(\xi_{\ast})} \sqrt{M_{\CB}(\xi'_{\ast})}
             \frac{G_{\CA}}{\sqrt{M_{\CA}}}(\xi') \,d\xi_{\ast} d\omega,\ \ \CA\neq\CB.
\end{split}
\end{array}\right.
\end{eqnarray}
One one hand, by performing the similar calculations as \cite{BGPS}, one can see that $K_i$ and $K_e$ defined by \eqref{K.def1} and \eqref{K.def2} are compact from
$$
L_\xi^2(\frac{1}{\sqrt{M_i}})\times L_\xi^2(\frac{1}{\sqrt{M_e}})
$$
to itself. %$L_\xi^2(\frac{1}{\sqrt{M_i}})\times L_\xi^2(\frac{1}{\sqrt{M_e}})$.
One the other hand, as $B_{\CA\CB}=B_{\CB\CA}=\si^2|(\xi-\xi_\ast)\cdot\omega|$, one can be able to show
$$
\nu_i\thicksim (1+|\xi|),\quad \nu_e\thicksim (1+|\xi|).
$$
Then the coercivity estimate \eqref{co.est00.} follows from the standard argument as \cite{CIP,Car,G}; see also \cite{ABT}. This ends the proof of Lemma \ref{co.est0.}.
\end{proof}

Furthermore, one can vary the background for the linearisation and the weight function. In fact, basing on Lemma
\ref{est.nonop} as well as its proof, we also have the following result, cf.~\cite{LYYZ}.

\begin{lemma}\label{co.est.}
Let $\frac{\ta}{2}<\widehat{\ta}$. Then there exist two positive constants
$\de=\de(n_i,n_e, u,\theta;\widehat{n}_i,\widehat{n}_e,\widehat{u},\widehat{\theta})$
and $\eta_0=\eta_0(n_i,n_e,u,\theta;\widehat{n}_i,\widehat{n}_e,\widehat{u},\widehat{\theta})$ such that
if
$$
|n_i-\widehat{n}_i|+|n_e-\widehat{n}_e|+|u-\widehat{u}|+|\theta-\widehat{\theta}|<\eta_0,
$$
it holds  that for
$\FH(\xi)=[H_i(\xi),H_e(\xi)]^{\rm{T}}\in {\mathcal{N}}^\bot$,
\begin{eqnarray}\label{vb.coL}
-\int_{{\R}^3}\FH\cdot \left\{\widehat{\bf M}^{-1} (\FL_{\bf M} \FH) \right\}d\xi \geq
\de\int_{{\R}^3}(1+|\xi|)\left|\widehat{\bf M}^{-1/2} \FH\right|^2d\xi,
\end{eqnarray}
where  we have denoted
\begin{eqnarray}
&&{\bf M}\equiv \left[M_{[n_i ,u,\theta;m_i]}(\xi),M_{[n_e,u,\theta;m_e]}(\xi)\right]^{\rm{T}},\notag\\
&&\widehat{\bf
M}\equiv \left[\widehat{M}_i, \widehat{M}_e\right]\equiv\left[M_{[\widehat{n}_i,\widehat{u},\widehat{\theta};m_i]}(\xi),
M_{[\widehat{n}_e,\widehat{u},\widehat{\theta};m_e]}(\xi)\right]^{\rm{T}},\notag\\
&&{\mathcal{N}}^\bot=\left\{ \FH(\xi):\ \ \int_{{\R}^3}\psi_j(\xi)\cdot\FH(\xi)d\xi=0,\ j=1,2,\cdots,6\right\}.\notag
\end{eqnarray}
%
%
%
%
%
%${\bf M}\equiv \left[M_{[n_i ,u,\theta;m_i]}(\xi),M_{[n_e,u,\theta;m_e]}(\xi)\right]^{\rm{T}}$, $ \widehat{\bf
%M}\equiv \left[\widehat{M}_i, \widehat{M}_e\right]\equiv\left[M_{[\widehat{n}_i,\widehat{u},\widehat{\theta};m_i]}(\xi),
%M_{[\widehat{n}_e,\widehat{u},\widehat{\theta};m_e]}(\xi)\right]^{\rm{T}}$
%and
%$$
%{\mathcal{N}}^\bot=\left\{ \FH(\xi):\ \ \int_{{\R}^3}\psi_j(\xi)\cdot\FH(\xi)d\xi=0,\ j=1,2,\cdots,6,\right\}.
%$$
\end{lemma}
\begin{proof}
We first write
\begin{equation}\label{df.L}
\begin{split}
-\int_{{\R}^3}\FH\cdot \left\{\widehat{\bf M}^{-1}\FL_{\bf M} \FH \right\}d\xi
=-\int_{{\R}^3}\FH\cdot \left\{\widehat{\FM}^{-1} \FL_{\widehat{\bf M}} \FH \right\}d\xi
-\int_{{\R}^3}\FH\cdot \left\{\widehat{\bf M}^{-1} \FL_{\FM-\widehat{\FM}}\FH\right\}d\xi.
\end{split}
\end{equation}
In light of Lemma \ref{co.est0.} , one has
\begin{equation}\label{df.L1}
\begin{split}
-\int_{{\R}^3}\FH\cdot \left\{\widehat{\FM}^{-1}\FL_{\widehat{\bf M}} \FH \right\}d\xi
\geq \de(\widehat{n}_i,\widehat{n}_e,\widehat{u},\widehat{\theta})\int_{{\R}^3}(1+|\xi|
\left|\widehat{\FM}^{-1/2}\FH \right|^2d\xi.
\end{split}
\end{equation}
%the second inequality follows from the fact that $\widehat{\ta}<\ta$.
For the second term on the right hand side of \eqref{df.L}, noticing
\begin{equation*}
\FL_{\FM-\widehat{\FM}}\FH=\left[\begin{array}{cc}
Q_{ii}(M_i-\widehat{M}_i,H_i)+Q_{ii}(H_i,M_i-\widehat{M}_i)+Q_{ie}(M_i-\widehat{M}_i,H_e)+Q_{ie}(H_i,M_e-\widehat{M}_e)
\\[3mm]
Q_{ee}(M_e-\widehat{M}_e,H_e)+Q_{ee}(H_e,M_e-\widehat{M}_e)+Q_{ei}(M_e-\widehat{M}_e,H_i)+Q_{ei}(H_e,M_i-\widehat{M}_i)
\end{array}\right],
\end{equation*}
it follows from Cauchy-Schwarz inequality  and Lemma \ref{est.nonop} that
\begin{eqnarray}\label{df.L2}
\begin{split}
&\left|-\int_{{\R}^3}\FH\cdot \left\{\widehat{\bf M}^{-1}\FL_{\FM-\widehat{\FM}}\FH \right\}d\xi\right|
\\&\leq \frac{\de(\widehat{n}_i,\widehat{n}_e,\widehat{u},\widehat{\theta})}{4}\int_{{\R}^3}(1+|\xi|)\left|\widehat{\bf M}^{-1/2}\FH\right|^2d\xi\\
&\quad+\frac{4}{\de(\widehat{n}_i,\widehat{n}_e,\widehat{u},\widehat{\theta})}
\int_{{\R}^3}(1+|\xi|)^{-1}\left|\widehat{\bf M}^{-1/2}\FL_{\FM-\widehat{\FM}}\FH\right|^2d\xi
\\&\leq \frac{\de(\widehat{n}_i,\widehat{n}_e,\widehat{u},\widehat{\theta})}{4}\int_{{\R}^3}(1+|\xi|)\left|\widehat{\bf M}^{-1/2}\FH\right|^2d\xi
\\&\quad+\frac{C_1}{\de(\widehat{n}_i,\widehat{n}_e,\widehat{u},\widehat{\theta})}\int_{{\R}^3}(1+|\xi|)\left|\widehat{\bf M}^{-1/2}\FH\right|^2d\xi
\int_{{\R}^3}(1+|\xi|)\left|\widehat{\bf M}^{-1/2}(\FM-\widehat{\FM})\right|^2d\xi.
\end{split}
\end{eqnarray}
To treat the integral
$$
\dis{\int_{{\R}^3}}(1+|\xi|)\left|\widehat{\bf M}^{-1/2}(\FM-\widehat{\FM})\right|^2d\xi,
$$
we use $\frac{\ta}{2}<\widehat{\ta}$ and choose a large positive constant $C_2=C_2(n_i,n_e, u,\theta;\widehat{n}_i,\widehat{n}_e,\widehat{u},\widehat{\theta})$ such that
\begin{equation}\label{df.L3}
\begin{split}
\int_{|\xi|\geq C_2}(1+|\xi|)\left|\widehat{\bf M}^{-1/2}(\FM-\widehat{\FM})\right|^2d\xi
 \leq& C_3\int_{|\xi|\geq C_2}(1+|\xi|)\left(\frac{M^2_i+\widehat{M}^2_i}{\widehat{M}_i}
+\frac{M^2_e+\widehat{M}^2_e}{\widehat{M}_e}\right)
d\xi
\\ \leq&
\frac{\de^2(\widehat{n}_i,\widehat{n}_e,\widehat{u},\widehat{\theta})}{16C_1}.
\end{split}
\end{equation}
For the integral in the remaining domain, it follows that
\begin{equation}\label{df.L4}
\begin{split}
\int_{|\xi|<C_2}(1+|\xi|)\left|\widehat{\bf M}^{-1/2}(\FM-\widehat{\FM})\right|^2d\xi
 \leq \frac{C_4}{4C_1}
\left(|n_i-\widehat{n}_i|+|n_e-\widehat{n}_e|+|u-\widehat{u}|+|\theta-\widehat{\theta}|\right)^2,
\end{split}
\end{equation}
for some constant $C_4=C_4(n_i,n_e, u,\theta;\widehat{n}_i,\widehat{n}_e,\widehat{u},\widehat{\theta})$.
Finally, by letting
$$
\eta_0=\frac{\de(\widehat{n}_i,\widehat{n}_e,\widehat{u},\widehat{\theta})}{2C_4(n_i,n_e, u,\theta;\widehat{n}_i,\widehat{n}_e,\widehat{u},\widehat{\theta})},
$$
and inserting \eqref{df.L1},
\eqref{df.L2}, \eqref{df.L3} and \eqref{df.L4}
into \eqref{df.L}, one sees that \eqref{vb.coL} holds true. This completes the proof of Lemma \ref{co.est.}.
\end{proof}

%\begin{remark}\label{small.gM.}
%The constant $\eta_0$ in Lemma \ref{co.est.} is some positive constant depending on the
%the first non-zero eigenvalue of the linearized operator $L_{\bf
%M}$. Note that $\eta_0$ is not necessary to be small, cf.
%\cite{LYYZ}.
%\end{remark}
A direct consequence of Lemma \ref{co.est.} with the help of the Cauchy inequality
is the following corollary, cf. \cite{LYYZ}.

\begin{corollary}\label{inv.L.} Under the assumptions in Lemma \ref{co.est.}, it holds that
for  $\FH(\xi)\in {\mathcal{N}}^\bot$,
\begin{eqnarray*}%\label{inv.L.ine.}
{\displaystyle\int_{{\R}^3}}(1+|\xi|)\left|\widehat{\bf
M}^{-1/2}\FL^{-1}_{\bf M}\FH\right|^2d\xi\leq
\de^{-2}{\displaystyle\int_{{\R}^3}}(1+|\xi|)^{-1}|\widehat{\bf M}^{-1/2}\FH|^2(\xi)\,d\xi.
\end{eqnarray*}
\end{corollary}

\section{Proof of the main result}\label{sec5}

With preparations in the previous sections, we begin to give the proof of Theorem \ref{main.Res.}. For later use we first introduce some notations. Recall that $[n^r,u_1^r,\theta^r]$ is the smooth $3$-family rarefaction wave to the quasineutral Euler system \eqref{cons.law.f.as} with far-field data $[n_\pm,u_{1\pm},\theta_\pm]$ connected by $[n_+,u_{1+},\theta_+]\in R_3(n_-,u_{1-},\theta_-)$. We define the local bi-Maxwellian:
%\begin{equation}\label{BarM}
%\FM_r=M_{[1/(v^r(t,x)),u^r(t,x),\ta^r(t,x)]}(\xi)
%=\frac{1}{v^r(t,x)\sqrt{\left(\frac{4}{3}\pi\ta^r(t,x)\right)^3}}
%\exp\left(-\frac{|\xi-u^r(t,x)|^2}{\frac{4}{3}\ta^r(t,x)}\right).
%\end{equation}
\begin{eqnarray*}
\FM_r=\left[\begin{array}{cc}
      M_{ri}    \\[3mm]
      M_{re}
\end{array}\right]
=\left[\begin{array}{cc}
      M_{[n_i^r(t,x),u^r(t,x),\theta^r(t,x);m_i]}(\xi)    \\[3mm]
      M_{[n_e^r(t,x),u^r(t,x),\theta^r(t,x);m_e]}(\xi)
\end{array}\right].
\end{eqnarray*}
with $n_e^r(t,x)=n^r(t,x)$ and  $n_i^{r}(t,x)=-\frac{q_e}{q_i}n_e^{r}(t,x)=-\frac{q_e}{q_i}n^r(t,x)$. In terms of \eqref{def.mb} and \eqref{def.mbp},  we also define the global bi-Maxwellian:
%$$\FM_*(\xi)=\FM_{[1/v_*,u_*,\ta_*]}(\xi)
%$$
\begin{eqnarray*}
\FM_*=\left[\begin{array}{cc}
      M_{\ast i}    \\[3mm]
      M_{\ast e}
\end{array}\right]
=\left[\begin{array}{cc}
      M_{[n_{*i},u_{*},\ta_{*};m_i]}(\xi)    \\[3mm]
      M_{[n_{*e},u_{*},\ta_{*};m_e]}(\xi)
\end{array}\right].
\end{eqnarray*}
%to denote a global bi-Maxwellian satisfying
%\begin{eqnarray*}\label{global.M.}
%\left\{\begin{array}{rll}
%\begin{split}
%&\frac{1}{2}\sup\limits_{(t,x)\in \R_+\times \R}\ta^r(t,x)<\ta_*<\sup\limits_{(t,x)\in \R_+\times \R}\ta^r(t,x),\\[3mm]
%&\sup\limits_{(t,x)\in \R_+\times \R}\left\{\left|-\frac{q_e}{q_i}n^r(t,x)-n_{*i}\right|+|n^r(t,x)-n_{*e}|+|u^r(t,x)-u_*|+|\ta^r(t,x)-\ta_*|\right\}<\eta_0,
%\end{split}
%\end{array}\right.
%\end{eqnarray*}
%for a constant $\eta_0>0$ which is suitably small, where $u_*=[u_{*1},0,0]$.
For a vector-valued function $\FH=[H_i,H_e]^{\rm T}$, we write that
$$
\FH\in L^2_{\xi}(\frac{1}{\sqrt{\FM_{*}}}),\quad
\text{if}\ \frac{H_i}{\sqrt{M_{*i}}}\in L^2_\xi\ \ \text{and}\ \ \frac{H_e}{\sqrt{M_{*e}}}\in L^2_\xi.
$$
Now we define the function space in which we seek the solutions of the VPB system \eqref{VPB}, \eqref{PE}. For given $T\in (0,+\infty]$,
%and given function $\FH(t,x,\xi)=[H_i(t,x,\xi),H_e(t,x,\xi)]^{\rm T}$,
we set
\begin{multline}\notag
\widetilde{\CE}([0,T])=\{ \FH(t,x,\xi)|
\frac{\partial^{\al}\pa^\be
H_{\CA}(t,x,\xi)}{\sqrt{M_{*\CA}(\xi)}}
\in  C\left([0,T]; L^2_{x,\xi}({\R}\times{\R}^3)\right) \\
\textrm{for}\ |\al|+|\be|\leq 2, \al_0\leq1,\ \CA=i,e\},
\end{multline}
%\begin{equation*}
%\widetilde{\CE}([0,T])=\left\{ \FH(t,x,\xi)\left|
%\begin{array}{rlll}
%\begin{split}
%\frac{\partial^{\al}\pa^\be
%H_{\CA}(t,x,\xi)}{\sqrt{M_{*\CA}(\xi)}}
%\in  C\left([0,T]; L^2_{x,\xi}({\R}\times{\R}^3)\right)\
%\textrm{for}\ |\al|+|\be|\leq 2, \al_0\leq1,\ {\rm and}\ \CA\in\{i,e\}
%\end{split}
%\end{array}
%\right.\right\},
%\end{equation*}
associated with the norm $\widetilde{\CE}_T(\cdot)$ defined by
%\marginpar{\Blue{RJ: Yes, I agree with you. SQ }}%For each $T>0$, the corresponding norm
$$
\widetilde{\CE}_T(\FH)\equiv \sum\limits_{|\al|+|\beta|\leq 2\atop{0\leq \al_0\leq1}}
\left\{\sup\limits_{0\leq t\leq T}\int_{{\R}\times{\R}^3}\frac{\left|\partial^\al\partial^\beta H_i(t,x,\xi)\right|^2}{M_{*i}}\,d\xi dx
+ \sup\limits_{0\leq t\leq T}\int_{{\R}\times{\R}^3}\frac{\left|\partial^\al\partial^\beta H_e(t,x,\xi)\right|^2}{M_{*e}}\,d\xi dx\right\}.
$$
%$$
%\widetilde{\CE}_T(\FH)\equiv \sup\limits_{0\leq t\leq T}\sum\limits_{|\al|+|\beta|\leq 2\atop{\al_0\leq1}}
%\left\{\int_{{\R}\times{\R}^3}\frac{\left|\partial^\al\partial^\beta H_i(t,x,\xi)\right|^2}{M_{*i}}d\xi dx
%+ \int_{{\R}\times{\R}^3}\frac{\left|\partial^\al\partial^\beta H_e(t,x,\xi)\right|^2}{M_{*e}}d\xi dx\right\},
%$$
%where $\pa^\al\pa^\be=\pa_t^{\al_0}\pa_x^{\al_1}\pa_\xi^{\be}$, $\pa_\xi^\be=\pa_{\xi_1}^{\be_1}\pa_{\xi_2}^{\be_2}\pa_{\xi_3}^{\be_3}$, $|\al|=\al_0+\al_1$ and  $|\be|=\be_1+\be_2+\be_3.$

The proof of Theorem \ref{main.Res.} is based on the energy estimates on both the fluid and non-fluid part of the solution $\FF(t,x,\xi)$. We first consider the fluid part. Recall that the macro quantities $[n_i,n_e,u,\theta]$ of the fluid part $\FM(t,x,\xi)$ satisfy the two-fluid Navier-Stokes-Poisson-type system \eqref{BE-NSi} and \eqref{BE-NSe}, and the macro quantities $[n_i^r,n_e^r,u^r,\theta^r]$ of the corresponding smooth approximate profile $\FM^r(t,x,\xi)$  satisfy \eqref{Euler3}.
We now define the perturbation
$$
\left[\widetilde{n}_i,\widetilde{n}_e,\widetilde{u},\widetilde{\ta}\right](t,x)
=\left[n_i-n_i^r,n_e-n_e^r,u-\left[u_1^r,0,0\right],\ta-\ta^r\right](t,x).
$$
%and also denote the perturbation of $v$ by $\widetilde{v}=v-v^r.$
%where $\phi^{r}=\ln v^{r}.$
%Then %$\left[\widetilde{n}_i,\widetilde{n}_e, \widetilde{u}, \widetilde{\ta},\phi \right](t,x)$ satisfies
Then one can deduce the perturbed equations for $\left[\widetilde{n}_i,\widetilde{n}_e,\widetilde{u},\widetilde{\ta}\right]$ through \eqref{BE-NSi}, \eqref{BE-NSe}, \eqref{cons.law.s}  and \eqref{Euler3} in the following way. For number densities $\widetilde{n}_i$ and $\widetilde{n}_e$, one has
\begin{eqnarray}
&&\pa_t\widetilde{n}_i+\pa_x\left(n_iu_1-n_i^ru_1^r\right)=-\int_{{\R}^3}\xi_1\pa_xG_i \,d\xi,\label{tvi}\\
&&\pa_t\widetilde{n}_e+\pa_x\left(n_eu_1-n_e^ru_1^r\right)=-\int_{{\R}^3}\xi_1\pa_xG_e \,d\xi.\label{tve}
\end{eqnarray}
For the momentum $\widetilde{u}=[\widetilde{u}_1,\widetilde{u}_2,\widetilde{u}_3]$, one has
\begin{eqnarray}
&&(m_in_i+m_en_e)(\pa_t \widetilde{u}_1+u_1\pa_x\widetilde{u}_1+\widetilde{u}_1\pa_xu_1^{r})+\pa_x (P-P^r)\notag\\
&&\quad+\left(1-\frac{m_in_i+m_en_e}{m_in_i^r+m_en_e^r}\right)\pa_xP^r+(q_in_i+q_en_e)\pa_x\phi\notag\\
&&=3\pa_x\left((\mu_i(\theta)+\mu_e(\theta))\pa_x \widetilde{u}_1\right)
+\pa_x\left((\mu_i(\theta)+\mu_e(\theta))\pa_x u_1^r\right)
\notag
\\
&&\quad-\int_{{\R}^3}\xi_1\psi_{3}\cdot\pa_x\overline{\FR} \,d\xi-\int_{{\R}^3}\psi_{3i}\xi_1\pa_x \left(P_0^{M_i}G_i\right) d\xi-\int_{{\R}^3}\psi_{3e}\xi_1\pa_x \left(P_0^{M_e}G_e\right) d\xi,\label{tu1}
\end{eqnarray}
and
\begin{multline}
(m_in_i+m_en_e)(\pa_t \widetilde{u}_j+u_1\pa_x\widetilde{u}_j)=\pa_x\left((\mu_i(\theta)+\mu_e(\theta))\pa_x \widetilde{u}_j\right)-\int_{{\R}^3}\xi_1\psi_{j+2}\cdot\pa_x\overline{\FR} \,d\xi
\\
-\int_{{\R}^3}\psi_{(j+2)i}\xi_1\pa_x \left(P_0^{M_i}G_i\right) d\xi-\int_{{\R}^3}\psi_{(j+2)e}\xi_1\pa_x \left(P_0^{M_e}G_e\right) d\xi,\
\ j=2,3.\label{tuj}
\end{multline}
For the equation of temperature $\widetilde{\ta}$, one has
\begin{eqnarray}
&&(n_i+n_e)(\pa_t\widetilde{\theta}+u_1\pa_x\widetilde{\ta}+\widetilde{u}_1\pa_x\ta^r)
+P\pa_x u_1-P^r\pa_x u^r+\left(1-\frac{n_i+n_e}{n_i^r+n_e^r}\right)P^r\pa_x u^r
\notag\\
&&=
\pa_x\left((\kappa_i(\theta)+\kappa_e(\theta))\pa_x \widetilde{\theta}\right)
+\pa_x\left((\kappa_i(\theta)+\kappa_e(\theta))\pa_x\ta^r\right)\notag
\\
&&\quad+3(\mu_i(\theta)+\mu_e(\theta))(\pa_x u_1)^2+
\sum\limits_{j=2}^3(\mu_i(\theta)+\mu_e(\theta))(\pa_x u_j)^2
\notag\\
&&\quad
-\int_{{\R}^3}\xi_1\left(\psi_{6}-\sum\limits_{j=1}^3u_j \psi_{j+2}\right)\cdot\pa_x\overline{\FR} d\xi-\int_{{\R}^3}\psi_{6i}\xi_1\pa_x\left(P_0^{M_i}G_i\right) d\xi-\int_{{\R}^3}\psi_{6e}\xi_1\pa_x\left(P_0^{M_e}G_e\right) d\xi
\notag\\
&&\quad+\sum\limits_{j=1}^3u_j\int_{{\R}^3}\psi_{(j+2)i}\xi_1\pa_x \left(P_0^{M_i}G_i\right) d\xi
+\sum\limits_{j=1}^3u_j\int_{{\R}^3}\psi_{(j+2)e}\xi_1\pa_x \left(P_0^{M_e}G_e\right) d\xi
\notag\\
&&\quad+\ta\int_{{\R}^3}[\xi_1,\xi_1]^{\rm T}\cdot\pa_x\FG \,d\xi
+\pa_x\phi\int_{{\R}^3}\frac{|\xi|^2}{2}\left[q_i,q_e\right]^{\rm T}\cdot\pa_{\xi_1}\FG \,d\xi.\label{tta}
\end{eqnarray}
Note that $\phi$ is coupled to the Poisson equation
\begin{equation}
-\pa^2_x\phi =q_in_i+q_en_e=q_i\widetilde{n}_i+q_e\widetilde{n}_e.\label{tphy}
\end{equation}
The above reformulated Cauchy problem on $\left[\widetilde{n}_i,\widetilde{n}_e,\widetilde{u},\widetilde{\ta}\right]$ is supplemented with initial data
\begin{multline}
\left[\widetilde{n}_i,\widetilde{n}_e,\widetilde{u}, \widetilde{\ta}\right](0,x)=\left[\widetilde{n}_{i0},\widetilde{n}_{e0},\widetilde{u}_0, \widetilde{\ta}_0\right](x)\\
=\big[n_{i0}(x)-n^r_{i0}(x),n_{e0}(x)-n^r_{e0}(x),u_0(x)-[u^r_{1,0}(x),0,0],\ta_0(x)-\ta^r_0(x)\big].\label{tid}
\end{multline}
Here we recall that $\overline{\FR}$ is defined in \eqref{def.rbar}. We also note that $\widetilde{u}_j=u_j$ for $j=2, 3$. As $\phi_+=\phi_-$, we further assume $\phi_+=0=\phi_-$ without loss of generality and let $\phi (t,x)$ be determined by
the elliptic equation \eqref{tphy} under the boundary condition that $\phi (t,x)\rightarrow 0$ as $x\rightarrow\pm\infty$.
%We also point out that the structural identity \eqref{tphy} will be of extremal importance for our proof.

For the non-fluid part $\FG(t,x,\xi)$, as in \cite{XZ}, we note that
$$
\left\|\frac{G_\CA}{\sqrt{M_\CA}}\right\|^2_{L^2_{x,\xi}}
$$
is not integrable with respect to the time variable, and hence it is necessary to consider the following perturbation
\begin{equation*}%\label{newde}
\widetilde{\FG}=\left[\widetilde{G}_i,\widetilde{G}_e\right]^{\rm{T}}
=\left[G_i-\overline{G}_{i},G_e-\overline{G}_{e}\right]^{\rm{T}},
\end{equation*}
where
\begin{equation}\label{def.ng}
\begin{split}
&\overline{\FG}=\left[\overline{G}_i,\overline{G}_e\right]^{\rm{T}}\\
&=\frac{3}{2\ta}\FL^{-1}_{\FM}
\left\{{\bf P}_1^{\FM}\left[\left[m_iM_i,m_eM_e\right]^{\rm T}
\xi_1\left(\xi_1\pa_xu_1^r+\frac{|\xi-u|^2}{2\ta}\pa_x\ta^r\right)\right]\right\}\\
&\quad +\FL^{-1}_{\FM}
\left\{{\bf P}_1^{\FM}\left[\left[n^{-1}_iM_i\pa_xn_i^r,n^{-1}_eM_e\pa_xn_e^r\right]^{\rm T}
\xi_1\right]\right\}
-\frac{3}{2\ta}\FL^{-1}_{\FM}
\left\{{\bf P}_1^{\FM}\left[\left[M_i,M_e\right]^{\rm T}
\xi_1\pa_x\ta^r\right]\right\}
\\&=\frac{3}{2\ta}\pa_xu_1^r\FL^{-1}_{\FM}
\left\{\left[m_iM_i,m_eM_e\right]^{\rm T}|\xi_1-u_1|^2
-\frac{1}{3}(|\xi-u|^2-3)\left[M_i,M_e\right]^{\rm T}\right\}\\
&\quad+\frac{3}{2\ta}\pa_x\ta^r\FL^{-1}_{\FM}
\left\{(\xi_1-u_1)\left[\left[m_iM_i,m_eM_e\right]^{\rm T}
\left(\frac{|\xi-u|^2}{2\ta}-\frac{5}{3}\frac{n_i+n_e}{m_in_i+m_en_e}\right)\right]\right\}\\
&\quad+\FL^{-1}_{\FM}
\left\{(\xi_1-u_1)\left[\left[n^{-1}_iM_i\pa_xn_i^r,n^{-1}_eM_e\pa_xn_e^r\right]^{\rm T}
-\left[m_iM_i,m_eM_e\right]^{\rm T}\frac{\pa_x(n_i^r+n_e^r)}{m_in_i+m_en_e}\right]\right\}
\\&\quad-\frac{3}{2\ta}\pa_x\ta^r\FL^{-1}_{\FM}
\left\{(\xi_1-u_1)\left[\left[M_i,M_e\right]^{\rm T}
-\left[m_iM_i,m_eM_e\right]^{\rm T}\frac{n_i+n_e}{m_in_i+m_en_e}\right]\right\}.
%\\&+\frac{3}{2\ta}\FL^{-1}_{\FM}
%\left\{ \frac{\pa_x\phi^r(\xi_1-u_1)(q_im_e-q_em_i)}{(m_in_i+m_en_e)}\left[n_eM_i,-n_iM_e\right]^{\rm T}\right\}.
\end{split}
\end{equation}

Now, to prove Theorem \ref{main.Res.}, the key point is to deduce the a priori energy estimates on the macroscopic part $\left[\widetilde{n}_i,\widetilde{n}_e, \widetilde{u}, \widetilde{\ta}, \phi \right]$ and the microscopic part $\FG$ and $\widetilde{\FG}$ based on the following a priori assumption:
\begin{equation}
\label{aps}
N^2(T)+\de_r\leq \eps_0^2,
\end{equation}
for an arbitrary positive time $T>0$,
where $\de_r$ is the wave strength of the rarefaction wave given in  \eqref{def.rws}, and $N(T)$ is defined by
\begin{eqnarray*}
N^2(T)&:=&\sup\limits_{0\leq t\leq T}
\left\|\left[\widetilde{n}_i,\widetilde{n}_e, \widetilde{u}, \widetilde{\ta}\right](t,x)\right\|^2
+\sup\limits_{0\leq t\leq T}\sum\limits_{|\al|\leq1}
\left\|\pa^\al\pa_x[n_i, n_e, u, \ta](t,x)\right\|^2
\\&&+\sup\limits_{0\leq t\leq T} \sum\limits_{1\leq|\al| \leq
2\atop{\al_0\leq1}} {\displaystyle\int_{\R\times{\R}^3}}\left|\FM_\ast^{-1/2}\partial^\al \FF(t,x,\xi)\right|^2d\xi dx
+\sup\limits_{0\leq t\leq T} {\displaystyle\int_{\R\times{\R}^3}}\left|\FM_\ast^{-1/2}\widetilde{\FG}(t,x,\xi)\right|^2d\xi dx
\\&&+\sup\limits_{0\leq t\leq T} \sum\limits_{|\al|+|\be|\leq
2\atop{|\be|\geq1}} {\displaystyle\int_{\R\times{\R}^3}}\left|\FM_\ast^{-1/2}\pa^\al \partial^\be\widetilde{\FG}(t,x,\xi)\right|^2d\xi dx
+\sup\limits_{0\leq t\leq T}\sum\limits_{|\al|\leq1}\left\|\pa^\al\pa_x\phi\right\|_{H^1}^2.
\end{eqnarray*}
%\begin{equation}
%\begin{split}
%N^2(T):=&\sup\limits_{0\leq t\leq T}
%\left\|\left[\widetilde{n}_i,\widetilde{n}_e, \widetilde{u}, \widetilde{\ta}\right](t,x)\right\|^2
%+\sup\limits_{0\leq t\leq T}\sum\limits_{|\al|\leq1}
%\left\|\pa^\al\pa_x[n_i, n_e, u, \ta](t,x)\right\|^2
%\\&+\sup\limits_{0\leq t\leq T} \sum\limits_{1\leq|\al| \leq
%2\atop{\al_0\leq1}} {\displaystyle\int_{\R\times{\R}^3}}\left|\FM_\ast^{-1/2}\partial^\al \FF(t,x,\xi)\right|^2d\xi dx
%+\sup\limits_{0\leq t\leq T} {\displaystyle\int_{\R\times{\R}^3}}\left|\FM_\ast^{-1/2}\widetilde{\FG}(t,x,\xi)\right|^2d\xi dx
%\\&+\sup\limits_{0\leq t\leq T} \sum\limits_{|\al|+|\be|\leq
%2\atop{|\be|\geq1}} {\displaystyle\int_{\R\times{\R}^3}}\left|\FM_\ast^{-1/2}\pa^\al \partial^\be\widetilde{\FG}(t,x,\xi)\right|^2d\xi dx
%+\sup\limits_{0\leq t\leq T}\sum\limits_{|\al|\leq1}\left\|\pa^\al\pa_x\phi\right\|_{H^1}^2
%\end{split}
%\end{equation}
 Here we first claim that the a priori bound of $N(T)$ immediately yields
\begin{equation}\label{aps4}
\begin{split}
 \sup\limits_{0\leq t\leq T}\sum\limits_{|\al|=
2\atop{\al_0\leq1}}\left\|\pa^\al[n_i,n_e,u,\ta]\right\|^2
+\sup\limits_{0\leq t\leq T}  \sum\limits_{|\al|=
2\atop{\al_0\leq1}}{\displaystyle\int_{\R\times{\R}^3}}\left|\FM_\ast^{-1/2}\partial^\al\FG(t,x,\xi)\right|^2d\xi dx
\leq C\eps_0^2,
\end{split}
\end{equation}
for a generic constant $C>0$. Indeed, due to the decomposition $\FF=\FM+\FG$, one may  notice
\begin{multline}
\sup\limits_{0\leq t\leq T}\sum\limits_{|\al|=
2\atop{\al_0\leq1}}
{\displaystyle\int_{\R\times{\R}^3}}\left|\FM^{-1/2}\partial^{\al}\FM\right|^2d\xi dx
+\sup\limits_{0\leq t\leq T} \sum\limits_{|\al|=
2\atop{\al_0\leq1}} {\displaystyle\int_{\R\times{\R}^3}}\left|\FM^{-1/2}\partial^{\al}\FG\right|^2d\xi dx\notag\\
\leq 2\sup\limits_{0\leq t\leq T} \sum\limits_{|\al|=
2\atop{\al_0\leq1}} \left|{\displaystyle\int_{\R\times{\R}^3}}(\FM^{-1}\pa^{\al}\FM)\cdot \partial^\al\FG \,d\xi dx\right|
+\sup\limits_{0\leq t\leq T}  \sum\limits_{|\al|=
2\atop{\al_0\leq1}}{\displaystyle\int_{\R\times{\R}^3}}\left|\FM^{-1/2}\partial^{\al} \FF\right|^2d\xi dx,
\end{multline}
where the right-hand first term is further bounded by
\begin{equation*}
%\label{ }
C\sup\limits_{0\leq t\leq T} \sum\limits_{|\al'|=1,|\al''|=1} \left(\int_{{\R}}\left|\pa^{\al'}[u,\ta]\right|^2\left|\pa^{\al''}[u,\ta]\right|^2dx\right)^{1/2}
\sum\limits_{|\al|=
2\atop{\al_0\leq1}}\left(\int_{\R\times{\R}^3}|\FM^{-1/2}\partial^{\al}\FG(t,x,\xi)|^2d\xi dx\right)^{1/2}.
\end{equation*}
It then follows that
\begin{eqnarray}
&&\sup\limits_{0\leq t\leq T}\sum\limits_{|\al|=
2\atop{\al_0\leq1}}
{\displaystyle\int_{\R\times{\R}^3}}\left|\FM^{-1/2}\partial^{\al}\FM\right|^2d\xi dx
+\sup\limits_{0\leq t\leq T} \sum\limits_{|\al|=
2\atop{\al_0\leq1}} {\displaystyle\int_{\R\times{\R}^3}}\left|\FM^{-1/2}\partial^{\al}\FG\right|^2d\xi dx\notag\\
%&&\leq 2\sup\limits_{0\leq t\leq T} \sum\limits_{|\al|=
%2\atop{\al_0\leq1}} \left|{\displaystyle\int_{\R\times{\R}^3}}(\FM^{-1}\pa^{\al}\FM)\cdot \partial^\al\FG d\xi dx\right|
%+\sup\limits_{0\leq t\leq T}  \sum\limits_{|\al|=
%2\atop{\al_0\leq1}}{\displaystyle\int_{\R\times{\R}^3}}\left|\FM_\ast^{-1/2}\partial^{\al} \FF\right|^2d\xi dx\notag\\
%&&\leq C\sup\limits_{0\leq t\leq T} \sum\limits_{|\al'|=1,|\al''|=1} \left(\int_{{\R}}\left|\pa^{\al'}[u,\ta]\right|^2\left|\pa^{\al''}[u,\ta]\right|^2dx\right)^{1/2}
%\sum\limits_{|\al|=
%2\atop{\al_0\leq1}}\left(\int_{\R\times{\R}^3}|\FM^{-1/2}\partial^{\al}\FG(t,x,\xi)|^2d\xi dx\right)^{1/2}
%\notag\\
%&&\quad +\sup\limits_{0\leq t\leq T} \sum\limits_{|\al|=
%2\atop{\al_0\leq1}} {\displaystyle\int_{\R\times{\R}^3}}\left|\FM_\ast^{-1/2}\partial^{\al} \FF\right|^2d\xi dx
%\notag\\
&&\leq C\eps_0\sup\limits_{0\leq t\leq T}\sum\limits_{1\leq|\al|\leq
2\atop{\al_0\leq1}}\left\|\pa^{\al}[u,\ta]\right\|^2
+C\eps_0\sum\limits_{|\al|=
2\atop{\al_0\leq1}}\int_{\R\times{\R}^3}|\FM^{-1/2}\partial^{\al}\FG|^2d\xi dx
\notag\\
&&\quad +\sup\limits_{0\leq t\leq T}\sum\limits_{|\al|=
2\atop{\al_0\leq1}}  {\displaystyle\int_{\R\times{\R}^3}}\left|\FM_\ast^{-1/2}\partial^{\al} \FF\right|^2d\xi dx.\label{aps2}
\end{eqnarray}
Moreover,
\begin{equation}\label{aps3}
\begin{split}
&\sup\limits_{0\leq t\leq T}\sum\limits_{|\al|=
2\atop{\al_0\leq1}}\left\|\pa^{\al}[n_i,n_e,u,\ta]\right\|^2
\\ &\leq C\sup\limits_{0\leq t\leq T}\sum\limits_{|\al|=
2\atop{\al_0\leq1}}
{\displaystyle\int_{\R\times{\R}^3}}\left|\FM^{-1/2}\partial^\al\FM(t,x,\xi)\right|^2d\xi dx
+C \sup\limits_{0\leq t\leq T}\sum\limits_{|\al'|=1}\left\|(\pa^{\al'}\left[n_i,n_e,u,\ta\right])^2\right\|^2\\
&\leq C\sup\limits_{0\leq t\leq T}\sum\limits_{|\al|=
2\atop{\al_0\leq1}}
{\displaystyle\int_{\R\times{\R}^3}}\left|\FM^{-1/2}\partial^\al\FM(t,x,\xi)\right|^2d\xi dx
+C\eps_0^2 \sup\limits_{0\leq t\leq T}\sum\limits_{1\leq|\al|\leq
2\atop{\al_0\leq1}}\left\|\pa^\al\left[n_i,n_e,u,\ta\right]\right\|^2.
\end{split}
\end{equation}
Therefore \eqref{aps3} together with \eqref{aps2} give \eqref{aps4}. In addition, one can also see that the following a priori bound holds true:
\begin{equation*}%\label{aps.phy}
\sup\limits_{0\leq t\leq T}\sum\limits_{|\al|=
2\atop{\al_0\leq1}}
\left\|\pa^\al\pa_x^2\phi(t)\right\|
\leq C\sup\limits_{0\leq t\leq T}\sum\limits_{|\al|=
2\atop{\al_0\leq1}}
\left\|\pa^\al(q_in_i+q_en_e)\right\|^2\leq C\eps_0^2.
\end{equation*}
This directly follows from the Poisson equation \eqref{tphy} as well as \eqref{aps} and \eqref{aps4}.
%
%In fact \eqref{aps.phy} follows from the standard elliptic estimates for the Poisson equation \eqref{tphy}.

%One can also see that \eqref{aps} and \eqref{aps4} lead to the following a priori estimate
%\begin{equation}\label{aps.phy}
%\sup\limits_{0\leq t\leq T}\left\{
%\left\|\phi (t,x)\right\|^2
%+\sum\limits_{1\leq\ga \leq2}
%\left\|\pa^\al \phi(t,x)\right\|^2_{H^1}\right\}\leq C\sum\limits_{\ga \leq
%2} \sup\limits_{0\leq t\leq T}\left\|\pa^{\al}(v,u,\ta)\right\|^2+\eps\leq C\eps_0^2.
%\end{equation}
%In fact \eqref{aps.phy} follows from the standard elliptic estimates for the Poisson equation \eqref{tphy}.

%
%The subsequent subsections are devoted to deducing the desired energy type estimates based on the
%a priori assumption (\ref{aps}) and the estimates in Lemma \ref{cl.RwRe} on $[n^{r}, u_1^{r},\ta^r]$.
%The first one is concentrated on the energy estimates on the macroscopic part.

%\subsection{Energy estimates on the macroscopic part}
%In this subsection, we discuss the energy estimates on
%$\left[\widetilde{n}_i,\widetilde{n}_e, \widetilde{u}, \widetilde{\ta}, \pa_x\phi \right](t,x)$. For result in this direction, we have

The a priori energy estimates under the assumption \eqref{aps} are divided into two steps. The first step is concerned with the estimates on
$\left[\widetilde{n}_i,\widetilde{n}_e, \widetilde{u}, \widetilde{\ta}, \pa_x\phi \right](t,x)$ basing on equations \eqref{tvi}, \eqref{tve}, \eqref{tu1}, \eqref{tuj}, \eqref{tta}, and \eqref{tphy}.

\begin{proposition}\label{mac.eng.lem.}
Assume that all the conditions in Theorem \ref{main.Res.} hold, and $\FF\in\widetilde{\CE}([0,T])$ for $T>0$. Let $\left[\widetilde{n}_i,\widetilde{n}_e, \widetilde{u}, \widetilde{\ta},  \phi \right](t,x)$
be a smooth solution to the Cauchy problem
\eqref{tvi}, \eqref{tve}, \eqref{tu1}, \eqref{tuj}, \eqref{tta}, \eqref{tphy} and \eqref{tid} on $0\leq t\leq T$
and satisfy \eqref{aps}. Then the following energy estimate holds:
\begin{equation}\label{macro.eng}
\begin{split}
\frac{d}{dt}&\sum\limits_{|\al| =1}\left\|\pa^\al\left[\widetilde{n}_i,\widetilde{n}_e, \widetilde{u}, \widetilde{\ta}\right](t)\right\|^2
+\frac{d}{dt}\left\|\pa_x \phi (t)\right\|^2
+\ka_0\frac{d}{dt}\sum\limits_{|\al| =1}\left(\pa^\al \widetilde{u}_1,\pa_x\pa^\al (\widetilde{n}_i+\widetilde{n}_e)\right)
\\&+\la\left\|\sqrt{\pa_xu^r}\left[\widetilde{n}_i,\widetilde{n}_e, \widetilde{u}, \widetilde{\ta}\right](t)\right\|^2
+\la\sum\limits_{1\leq|\al| \leq2}\left\|\pa^\al\left[\widetilde{n}_i,\widetilde{n}_e,\widetilde{u},\widetilde{\theta}\right](t)\right\|^2
+\la\left\|q_i\widetilde{n}_i+q_e\widetilde{n}_e\right\|^2
\\&+\la\sum\limits_{|\al| \leq1}\left\|\pa^{\al}\left[\pa_x\phi ,\pa^2_x\phi \right]\right\|^2
+\la\sum\limits_{|\al| =2}\left\|\pa^\al \pa^2_x\phi \right\|^2
\\
\lesssim&(1+t)^{-2}\left\|\left[\widetilde{n}_i,\widetilde{n}_e,\widetilde{u},\widetilde{\ta}\right](t)\right\|^2
+\de_r^{1/6}(1+t)^{-6/7}
+\sum\limits_{1\leq|\al| \leq2}\int_{\R\times{\R}^3}(1+|\xi|)\left|\FM^{-1/2}\pa^\al \FG\right|^2 d\xi dx
\\&+\eps_{0}\sum\limits_{|\al| \leq1}\int_{\R\times\R^3}\left|\FM^{-1/2}\pa_{\xi_1}\pa^\al \widetilde{\FG}\right|^2 d\xi dx
+\int_{{\R}\times{\R}^3}(1+|\xi|)\left|\FM_{\sharp}^{-1/2}\widetilde{\FG}\right|^2d\xi dx,
\end{split}
\end{equation}
for all $0\leq t\leq T$, where $\ka_0$ is a small positive constant. Here and in the sequel we also use $\FM_\sharp$ to denote $\FM_*$ or $\FM$ for brevity, whenever there is no confusion.
\end{proposition}

%\subsection{Dissipation of non-fluid component}\label{sec4.2}
The second step is to deduce the energy estimates on the microscopic part $\FG$.
Our method for that is much similar to the one of \cite{DL-VPB}.
%can be outlined as follows. As mentioned in the previous subsection,
%$\|\FM^{-1/2}_\sharp\FG\|^2_{L^2_{x,\xi}}$ is not integrable with respect to the time variable,
%we first perform the zeroth order energy estimate on $\widetilde{\FG}$.
%Then we directly present the higher order energy
%estimates on $F$.
%At last, we deduce the mixed derivative energy estimates on
%$\pa^\al \pa^\be\widetilde{\FG}$ for $|\al|+|\be|\leq2$ and $|\be|\geq1$. It is shown that the above energy estimates only
% with respect to the global Maxwellian
%$\FM_\ast$ or the local Maxwellian $\FM$ can not be closed. To overcome this difficulty,
%one has to use the interplay of these two kinds of weighted energy estimates.
%For results in this direction, we have
\begin{proposition}\label{g.eng.lem.}
Under the conditions listed in Proposition \ref{mac.eng.lem.}, it holds that
\begin{eqnarray}
%\begin{split}
&&\sum\limits_{|\al| \leq1}\left\|\pa^\al \left[\widetilde{n}_i,\widetilde{n}_e, \widetilde{u}, \widetilde{\ta}\right](t)\right\|^2
+\sum\limits_{|\al| \leq2\atop{\al_0\leq1}}
\left\|\pa^\al  \pa_x\phi (t)\right\|^2
+\int_{{\R}\times{\R}^3}\left|\FM^{-1/2}_\ast\widetilde{\FG}\right|^2d\xi dx
\notag\\
&&\quad+\sum\limits_{1\leq|\al| \leq2\atop{\al_0\leq1}}\int_{{\R}\times{\R}^3}\left|\FM_\ast^{-1/2}\pa^\al  \FF\right|^2d\xi dx
+\sum\limits_{|\al|+|\be|\leq 2\atop{|\be|\geq1}}
\int_{{\R}\times{\R}^3}\left|\FM^{-1/2}_\ast\pa^\al \pa^\be \widetilde{\FG}\right|^2d\xi dx
\notag\\
&&\quad+\sum\limits_{1\leq|\al| \leq2}\int_{0}^{T}\int_{\R\times{\R}^3}(1+|\xi|)\left|\FM^{-1/2}_\ast\pa^\al  \FG\right|^2d\xi dxdt
+\int_{0}^{T}\int_{{\R}\times{\R}^3}(1+|\xi|)\left|\FM_\ast^{-1/2}\widetilde{\FG}\right|^2d\xi dxdt
\notag\\
&&\quad+ \sum_{|\al|+|\be|\leq 2\atop{|\be|\geq1}}\int_{0}^{T}\int_{{\R}\times{\R}^3}
(1+|\xi|)\left|\FM^{-1/2}_\ast\pa^\al \pa^\be\widetilde{\FG}\right|^2d\xi dxdt
+\int_{0}^{T}\int_{\R}\left[\widetilde{n}_i,\widetilde{n}_e, \widetilde{u}, \widetilde{\ta}\right]^2\pa_xu^r_{1}\,dxdt
\notag\\
&&\quad+\sum\limits_{1\leq|\al| \leq 2}\int_{0}^{T}\left\|\pa^\al \left[\widetilde{n}_i,\widetilde{n}_e, \widetilde{u}, \widetilde{\ta}\right]\right\|^2dt
+\int_{0}^{T}\|q_i\widetilde{n}_i+q_e\widetilde{n}_e\|^2dt
+\sum\limits_{|\al| \leq1}\int_{0}^{T}\left\|\pa^\al \left[\pa_x\phi ,\pa^2_x\phi \right]\right\|^2dt
\notag\\
&&\leq C_0N^2(0)+C_0\de_r^{1/6},\label{g.eng.}
%\end{split}
\end{eqnarray}
for all $0\leq t\leq T$.
\end{proposition}

The proof of Proposition \ref{mac.eng.lem.} and Proposition \ref{g.eng.lem.} will be given in Section \ref{sec.est.f} and Section \ref{sec.est.nf}, respectively.  Assuming these two propositions, we are now in a position to complete

\begin{proof}[The proof of Theorem \ref{main.Res.}]

The local existence of the solution $\left[\widetilde{n}_i,\widetilde{n}_e,\widetilde{u},\widetilde{\ta},\phi ,\FG\right]$  of the
system \eqref{tvi}, \eqref{tve}, \eqref{tu1}, \eqref{tuj}, \eqref{tta} and \eqref{tphy}
 can be obtained by the standard iteration method, cf.~\cite{DL-VPB}, and its proof is omitted for brevity.

The existence of the solution of \eqref{VPB}, \eqref{id.be}, \eqref{bd.be} and \eqref{bd.phi}
 then follows from the standard continuation argument based on the local existence and the a priori
estimate in Proposition \ref{g.eng.lem.}.
Moreover, one sees that
\begin{eqnarray}
&&\sup\limits_{t\geq0}\sum\limits_{|\al|+|\be|\leq2\atop{\al_0\leq1}}\left\|\pa^\al \pa^\be\left(F_i(t,x,\xi)-\FM_{[-\frac{q_e}{q_i}n^r,u^r,\ta^r](t,x)}(\xi)\right)
\right\|_{L_x^2\left(L^2_\xi\left(\frac{1}{\sqrt{M_{\ast i}(\xi)}}\right)\right)}
\notag\\
&&\quad +\sup\limits_{t\geq0}\sum\limits_{|\al|+|\be|\leq2\atop{\al_0\leq1}}\left\|\pa^\al \pa^\be\left(F_e(t,x,\xi)-\FM_{[n^r,u^r,\ta^r](t,x)}(\xi)\right)
\right\|_{L_x^2\left(L^2_\xi\left(\frac{1}{\sqrt{M_{\ast e}(\xi)}}\right)\right)}+\sup\limits_{t\geq0}\sum\limits_{|\al|\leq1}\|\pa^\al\pa_x\phi(t,x)\|^2_{H^1}
\notag\\
&&\lesssim\sup\limits_{t\geq0}\sum\limits_{|\al| \leq2,\al_0\leq1}
\left\|\pa^\al \left[\widetilde{n}_i, \widetilde{n}_e,\widetilde{u}, \widetilde{\ta}\right](t,x)\right\|^2+
\sup\limits_{t\geq0} \sum\limits_{1\leq|\al| \leq
2\atop{\al_0\leq1}} {\displaystyle\int_{{\R}}\int_{{\R}^3}}\left|\FM_{\ast}^{-1/2}\partial^\al \FF(t,x,\xi)\right|^2d\xi dx
\notag\\
&&\quad +\sup\limits_{t\geq0} {\displaystyle\int_{{\R}}\int_{{\R}^3}}\left|\FM_{\ast}^{-1/2}\widetilde{\FG}(t,x,\xi)\right|^2d\xi dx
+\sup\limits_{t\geq0} \sum\limits_{|\al|+|\be|\leq
2\atop{|\be|\geq1}} {\displaystyle\int_{{\R}}\int_{{\R}^3}}\left|\FM_{\ast}^{-1/2}\pa^\al \partial^\be\widetilde{\FG}(t,x,\xi)\right|^2d\xi dx\notag\\
&&\quad
+\de^{1/6}_r.\label{latm0}
\end{eqnarray}
%\begin{equation}\label{latm0}
%\begin{split}
%\sup\limits_{t\geq0}&\sum\limits_{|\al|+|\be|\leq2\atop{\al_0\leq1}}\left\|\pa^\al \pa^\be\left(F_i(t,x,\xi)-\FM_{[-\frac{q_e}{q_i}n^r,u^r,\ta^r](t,x)}(\xi)\right)
%\right\|_{L_x^2\left(L^2_\xi\left(\frac{1}{\sqrt{M_{\ast i}(\xi)}}\right)\right)}
%\\&+\sup\limits_{t\geq0}\sum\limits_{|\al|+|\be|\leq2\atop{\al_0\leq1}}\left\|\pa^\al \pa^\be\left(F_e(t,x,\xi)-\FM_{[n^r,u^r,\ta^r](t,x)}(\xi)\right)
%\right\|_{L_x^2\left(L^2_\xi\left(\frac{1}{\sqrt{M_{\ast e}(\xi)}}\right)\right)}\\
%&+\sup\limits_{t\geq0}\sum\limits_{|\al|\leq1}\|\pa^\al\pa_x\phi(t,x)\|^2_{H^1}
%\\ \lesssim&\sup\limits_{t\geq0}\sum\limits_{|\al| \leq2,\al_0\leq1}
%\left\|\pa^\al \left[\widetilde{n}_i, \widetilde{n}_e,\widetilde{u}, \widetilde{\ta}\right](t,x)\right\|^2+
%\sup\limits_{t\geq0} \sum\limits_{1\leq|\al| \leq
%2\atop{\al_0\leq1}} {\displaystyle\int_{{\R}}\int_{{\R}^3}}\left|\FM_{\ast}^{-1/2}\partial^\al \FF(t,x,\xi)\right|^2d\xi dx
%\\&+\sup\limits_{t\geq0} {\displaystyle\int_{{\R}}\int_{{\R}^3}}\left|\FM_{\ast}^{-1/2}\widetilde{\FG}(t,x,\xi)\right|^2d\xi dx
%+\sup\limits_{t\geq0} \sum\limits_{|\al|+|\be|\leq
%2\atop{|\be|\geq1}} {\displaystyle\int_{{\R}}\int_{{\R}^3}}\left|\FM_{\ast}^{-1/2}\pa^\al \partial^\be\widetilde{\FG}(t,x,\xi)\right|^2d\xi dx+\de^{1/6}_r.
%\end{split}
%\end{equation}
Then \eqref{VPB.sol.}  follows from \eqref{latm0} provided that $\de\ll \eps_0$.

In order to obtain the large time behavior of solutions as in \eqref{sol.Lab}, one sees
\begin{eqnarray}
&&\frac{d}{dt}\left\|\frac{\pa_x\left(F_{\CA}(t,x,\xi)-M_{[n_{\CA}^r,u^r,\ta^r;m_\CA]}(\xi)\right)}
{\sqrt{M_{\ast\CA}}}\right\|_{L^2_{x,\xi}}^2
\notag\\
&&=2\left(M_{\ast\CA}^{-1}\pa_t\pa_x\left(F_{\CA}(t,x,\xi)-M_{[n_{\CA}^r,u^r,\ta^r;m_{\CA}]}(\xi)\right),
\pa_x\left(F_{\CA}(t,x,\xi)-M_{[n_{\CA}^r,u^r,\ta^r;m_\CA]}(\xi)\right)\right),
\notag
\end{eqnarray}
and thus it follows from the Cauchy-Schwarz inequality that
\begin{eqnarray}
&&\int_{0}^{+\infty}\left|\frac{d}{dt}\left\|\frac{\pa_x\left(F_{\CA}(t,x,\xi)-M_{[n_{\CA}^r,u^r,\ta^r;m_\CA]}(\xi)\right)}
{\sqrt{M_{\ast\CA}}}
\right\|_{L^2_{x,\xi}}^2\right|dt
\notag\\
%&&=2\int_{0}^{+\infty}\left|\left(M_{\ast\CA}^{-1}\pa_t\pa_x\left(F_{\CA}(t,x,\xi)-M_{[n_{\CA}^r,u^r,\ta^r;m_{\CA}]}(\xi)\right),
%\pa_x\left(F_{\CA}(t,x,\xi)-M_{[n_{\CA}^r,u^r,\ta^r;m_\CA]}(\xi)\right)\right)\right|dt
%\notag\\
&&\leq \int_{0}^{+\infty}\left\|M_{\ast\CA}^{-1/2}\pa_t\pa_x\left(F_{\CA}(t,x,\xi)
-M_{[n_{\CA}^r,u^r,\ta^r;m_{\CA}]}(\xi)\right)\right\|_{L^2_{x,\xi}}^2dt
\notag\\
&&\quad+\int_{0}^{+\infty}\left\|M_{\ast\CA}^{-1/2}\pa_x\left(F_{\CA}(t,x,\xi)-M_{[n_{\CA}^r,u^r,\ta^r;m_{\CA}]}(\xi)\right)\right\|_{L^2_{x,\xi}}^2dt
\notag\\
&&\leq C\sum\limits_{1\leq|\al|\leq 2\atop{|\al_0|\leq1}}\int_{0}^{+\infty}\left\|\pa^{\al}\left[\widetilde{n}_{\CA},\widetilde{u},\widetilde{\theta}\right]\right\|^2dt
+C\sum\limits_{1\leq|\al|\leq 2\atop{|\al_0|\leq1}}\int_{0}^{+\infty}\left\|\frac{\pa^{\al}G_{\CA}}{\sqrt{M_{\ast\CA}}}\right\|_{L^2_{x,\xi}}^2dt\notag\\
&&<+\infty. \label{latm1}
\end{eqnarray}
%\begin{equation}\label{latm1}
%\begin{split}
%\int_{0}^{+\infty}&\left|\frac{d}{dt}\left\|\frac{\pa_x\left(F_{\CA}(t,x,\xi)-M_{[n_{\CA}^r,u^r,\ta^r;m_\CA]}(\xi)\right)}
%{\sqrt{M_{\ast\CA}}}
%\right\|_{L^2_{x,\xi}}^2\right|dt
%\\=&2\int_{0}^{+\infty}\left|\left(M_{\ast\CA}^{-1}\pa_t\pa_x\left(F_{\CA}(t,x,\xi)-M_{[n_{\CA}^r,u^r,\ta^r;m_{\CA}]}(\xi)\right),
%\pa_x\left(F_{\CA}(t,x,\xi)-M_{[n_{\CA}^r,u^r,\ta^r;m_\CA]}(\xi)\right)\right)\right|dt
%\\ \leq &\int_{0}^{+\infty}\left\|M_{\ast\CA}^{-1/2}\pa_t\pa_x\left(F_{\CA}(t,x,\xi)
%-M_{[n_{\CA}^r,u^r,\ta^r;m_{\CA}]}(\xi)\right)\right\|_{L^2_{x,\xi}}^2dt
%\\&+\int_{0}^{+\infty}\left\|M_{\ast\CA}^{-1/2}\pa_x\left(F_{\CA}(t,x,\xi)-M_{[n_{\CA}^r,u^r,\ta^r;m_{\CA}]}(\xi)\right)\right\|_{L^2_{x,\xi}}^2dt
%\\ \leq &C\sum\limits_{1\leq|\al|\leq 2\atop{|\al_0|\leq1}}\int_{0}^{+\infty}\left\|\pa^{\al}\left[\widetilde{n}_{\CA},\widetilde{u},\widetilde{\theta}\right]\right\|^2dt
%+C\sum\limits_{1\leq|\al|\leq 2\atop{|\al_0|\leq1}}\int_{0}^{+\infty}\left\|\frac{\pa^{\al}G_{\CA}}{\sqrt{M_{\ast\CA}}}\right\|_{L^2_{x,\xi}}^2dt
%<+\infty.
%\end{split}
%\end{equation}
From \eqref{latm0} and \eqref{latm1}, one sees that
\begin{equation}\label{latm2}
\lim\limits_{t\rightarrow+\infty}\left\|\frac{\pa_x\left(F_{\CA}(t,x,\xi)-M_{[n_{\CA}^r,u^r,\ta^r;m_\CA]}(\xi)\right)}{\sqrt{M_{\ast\CA}}}
\right\|_{L^2_{x,\xi}}^2= 0.
\end{equation}
Then \eqref{sol.Lab} follows from Sobolev's inequality, \eqref{latm2}
and %$(iii)$ in
Lemma \ref{cl.RwRe}.
This completes the proof of Theorem \ref{main.Res.}.
\end{proof}

\section{A priori estimates on the fluid part}\label{sec.est.f}

%\begin{proof} We divide it by the following three steps.

This section is devoted to the proof of Proposition \ref{mac.eng.lem.} on the energy estimates of the fluid part $\FM(t,x,\xi)$.  The proof is divided by three subsections.

%\medskip
%\noindent{\bf Step 1.} {\it Zero-order energy estimates.}
\subsection{Estimate on zero-order energy}\label{sec6.1}
%The proof depends fundamentally upon some cancellations of the critical terms.

We set  $\Phi(y)=y-1-\ln y$, and define
$$
\widetilde{\eta}=\frac{m_{i}n_{i}+m_en_e}{2}\sum\limits_{j=1}^3\widetilde{u}_{j}^{2}+\frac{2}{3}n_{i}\theta^{r}\Phi\left(\frac{n_i^{r}}{n_{i}}\right)
+\frac{2}{3}n_{e}\theta^{r}\Phi\left(\frac{n_e^{r}}{n_{e}}\right)
+(n_{i}+n_e)\theta^{r}\Phi\left(\frac{\theta}{\theta^{r}}\right).
$$
By using \eqref{BE-NS-m}, \eqref{BE-NS-t}, \eqref{BE-NS-mass}, \eqref{Euler3}, \eqref{tu1}, \eqref{tuj} and \eqref{tta}, direct computations give
\begin{equation}\label{basic}
\begin{split}
\pa_{t}\widetilde{\eta}&+3(\mu_i(\ta)+\mu_e(\ta))(\pa_{x}\widetilde{u}_1)^{2}
+\sum\limits_{j=2}^3(\mu_i(\ta)+\mu_e(\ta))(\pa_{x}\widetilde{u}_j)^{2}
+\frac{\kappa_i(\ta)+\kappa_e(\ta)}{\theta}\left(\pa_{x}\widetilde{\ta}\right)^{2}
\\&+\pa_x\CM
+\CN_{1}=\sum\limits_{l=2}^8\CN_{l},
\end{split}
\end{equation}
where
\begin{equation*}%\label{2.1}
\begin{split}
\CM=u_1\widetilde{\eta}+(P-P^{r})\widetilde{u}_1
-\sum\limits_{j=1}^3\mu_j(\ta)\widetilde{u}_{j}\pa_x\widetilde{u}_{j}
-\kappa(\ta)\frac{\pa_x\widetilde{\ta}\,\widetilde{\ta}}{\theta},
\end{split}
\end{equation*}
\begin{equation*}%\label{2.2}
\begin{split}
\CN_{1}=&\pa_{x}u_1^{r}\left[(m_{i}n_{i}+m_en_e)\widetilde{u}_1^{2}+\frac{4}{9}\theta^{r}\left(n_{i}
\Phi\left(\frac{n_i^{r}}{n_{i}}\right)+n_{e}
\Phi\left(\frac{n_e^{r}}{n_{e}}\right)\right)+\frac{2}{3}(n_{i}+n_e)
\theta^{r}\Phi\left(\frac{\theta}{\theta^{r}}\right)\right]
\\[3mm]&+(n_i+n_e)\frac{\pa_x\theta^{r}}{\theta^{r}}\widetilde{u}_1\widetilde{\ta}
-\frac{2}{3}\pa_x\theta^r\frac{m_i\widetilde{n}_i+m_e\widetilde{n}_e}{m_in_i^r+m_en_e^r}(n_i^r+n_e^r)\widetilde{u}_1,
\end{split}
\end{equation*}
\begin{equation*}%\label{2.3}
\begin{split}
\CN_{2}=&\frac{2}{3}\pa_{x}\theta^{r}\widetilde{u}_1\left[n_{i}\Phi\left(\frac{n_i^{r}}{n_{i}}\right)
+n_{e}\Phi\left(\frac{n_e^{r}}{n_{e}}\right)\right]
+(n_{i}+n_e)\pa_x\theta^{r}\widetilde{u}_1\Phi\left(\frac{\theta}{\theta^{r}}\right),\
\\[3mm]
&+\frac{2}{3}\pa_xn_i^r\frac{m_e(n_e^r\widetilde{n}_i-n_i^r\widetilde{n}_e)}{n_i^r(m_in_i^r+m_en_e^r)}\widetilde{u}_1
+\frac{2}{3}\pa_xn_e^r\frac{m_i(n_i^r\widetilde{n}_e-n_e^r\widetilde{n}_i)}{n_e^r(m_in_i^r+m_en_e^r)}\widetilde{u}_1,
\end{split}
\end{equation*}
\begin{equation*}%\label{2.4}
\begin{split}
\CN_{3}=&
\pa_x((\mu_i(\ta)+\mu_e(\ta))\pa_xu_1^{r})\widetilde{u}_1+\sum\limits_{j=1}^3(\mu_i(\ta)+\mu_e(\ta))(\pa_x
u_{j})^{2}\frac{\widetilde{\ta}}{\theta}
\\&+(\kappa_i(\ta)+\kappa_e(\ta))\frac{\widetilde{\ta}(\pa_x\widetilde{\ta})^{2}}{\theta^{2}}
+(\kappa_i(\ta)+\kappa_e(\ta))\frac{\widetilde{\ta}\pa_x\widetilde{\ta}\pa_x\theta^{r}}{\theta^{2}}
+\pa_x\left((\kappa_i(\ta)+\kappa_e(\ta))\pa_x\ta^r\right)\frac{\widetilde{\ta}}{\theta},
\end{split}
\end{equation*}
\begin{equation*}%\label{2.4}
\begin{split}
\CN_{4}=
-\widetilde{u}_1(q_in_i+q_en_e)\pa_x\phi,
\end{split}
\end{equation*}
\begin{equation*}%\label{2.4}
\begin{split}
\CN_5=&
\frac{2}{3}\ta^r\ln \left(\frac{n_i^r}{n_i}\right)\int_{{\R}^3}\xi_1\pa_xG_i \,d\xi
+\frac{2}{3}\ta^r\ln \left(\frac{n_e^r}{n_e}\right)\int_{{\R}^3}\xi_1\pa_xG_e \,d\xi
+\widetilde{\ta}\int_{{\R}^3}[\xi_1,\xi_1]^{\rm T}\cdot\pa_x\FG \,d\xi
\\&-\sum\limits_{j=1}^3\widetilde{u}_j\int_{{\R}^3}\psi_{(j+2)i}\xi_1\pa_x \left(P_0^{M_i}G_i\right) d\xi
-\sum\limits_{j=1}^3\widetilde{u}_j\int_{{\R}^3}\psi_{(j+2)e}\xi_1\pa_x \left(P_0^{M_e}G_e\right) d\xi
\\&
-\frac{\widetilde{\ta}}{\ta}\int_{{\R}^3}\psi_{6i}\xi_1\pa_x\left(P_0^{M_i}G_i\right) d\xi-\frac{\widetilde{\ta}}{\ta}\int_{{\R}^3}\psi_{6e}\xi_1\pa_x\left(P_0^{M_e}G_e\right) d\xi
\\&+\sum\limits_{j=1}^3\frac{u_j\widetilde{\ta}}{\ta}\int_{{\R}^3}\psi_{(j+2)i}\xi_1\pa_x \left(P_0^{M_i}G_i\right) d\xi
+\sum\limits_{j=1}^3\frac{u_j\widetilde{\ta}}{\ta}\int_{{\R}^3}\psi_{(j+2)e}\xi_1\pa_x \left(P_0^{M_e}G_e\right) d\xi,
\end{split}
\end{equation*}
\begin{equation*}%\label{2.4}
\begin{split}
\CN_6=
-\sum\limits_{j=1}^3\widetilde{u}_j\int_{{\R}^3}\xi_1\psi_{j+2}\cdot\pa_x\overline{\FR} \,d\xi
-\frac{\widetilde{\ta}}{\ta}\int_{{\R}^3}\xi_1\left(\psi_{6}-\sum\limits_{j=1}^3u_j \psi_{j+2}\right)\cdot\pa_x\overline{\FR} \,d\xi,
\end{split}
\end{equation*}
\begin{equation*}%\label{2.4}
\begin{split}
\CN_7=
\frac{\widetilde{\ta}}{\ta}\pa_x\phi\int_{{\R}^3}\frac{|\xi|^2}{2}\left[q_i,q_e\right]^{\rm T}\cdot\pa_{\xi_1}\FG \,d\xi.
\end{split}
\end{equation*}
%\begin{equation*}%\label{2.4}
%\begin{split}
%\CN_8=
%\frac{\widetilde{\ta}}{\ta}\pa_x\int_{{\R}^3}\xi_1\psi_{6}\cdot \FL^{-1}_{\FM}\FP_1^{\FM}\left\{\frac{3}{4\ta^2}\pa_x\ta
%\left({\rm diag}(m_iC_i,m_eC_e)\xi_1\FM\right)\right\}d\xi.
%\end{split}
%\end{equation*}
We now get by integrating \eqref{basic} respect to $x$ over $\R$ that
\begin{equation*}%\label{z.eng1}
\begin{split}
\frac{d}{dt}&\int_{\R}\widetilde{\eta}\,dx
+3\int_{\R}(\mu_i(\ta)+\mu_e(\ta))(\pa_{x}\widetilde{u}_1)^{2}\,dx
+\sum\limits_{j=2}^3\int_{\R}(\mu_i(\ta)+\mu_e(\ta))(\pa_{x}\widetilde{u}_j)^{2}\,dx\\&
+\int_{\R}\frac{\kappa_i(\ta)+\kappa_e(\ta)}{\theta}\left(\pa_{x}\widetilde{\ta}\right)^{2}\,dx
+\int_{\R}\CN_{1}\,dx
=\sum\limits_{l=2}^7\int_{\R}\CN_{l}\,dx.
\end{split}
\end{equation*}
To compute ${\int_{\R}}\,\CN_{1}dx$, we first note that
\begin{equation}\label{Bp.ep}
\Phi\left(\frac{n_\CA^{r}}{n_{\CA}}\right)=\frac{\widetilde{n}_\CA^{2}}{n_{\CA}^{2}}+O(1)\widetilde{n}_\CA^{3},\ \CA=i,e,\
\ \
\Phi\left(\frac{\theta}{\theta^{r}}\right)=\frac{\widetilde{\ta}^{2}}{(\theta^{r})^{2}}+O(1)\widetilde{\ta}^3,
\end{equation}
and from \eqref{1-RW.def.},  it follows that
\begin{equation}\label{ptar}
\pa_x\ta^r=\frac{2A}{3B}(n^r)^{1/3}\pa_x u_1^r,\ \ B:=\sqrt{\frac{10A}{9}\frac{q_i-q_e}{m_eq_i-m_iq_e}}.
\end{equation}
With \eqref{Bp.ep} and \eqref{ptar} in hand, we now write
$\CN_1$ as
\begin{equation*}%\label{CN1.2}
\begin{split}
\CN_{1}=&\pa_{x}u_1^{r}\left[(m_{i}n^r_{i}+m_en^r_e)\widetilde{u}_1^{2}+\frac{4}{9}\theta^{r}\left(\frac{1}{n^r_{i}}\widetilde{n}^2_i
+\frac{1}{n^r_{e}}\widetilde{n}^2_e\right)+\frac{2}{3}\frac{(n^r_{i}+n^r_e)}{\ta^r}
\widetilde{\ta}^2\right]
\\[3mm]&+\frac{2}{3B}(n^r_i+n^r_e)(n^r)^{-1/3}\widetilde{u}_1\widetilde{\ta}\pa_xu_1^r
-\frac{4A}{9B}(n^r)^{1/3}\widetilde{u}_1\pa_xu_1^r\frac{q_i-q_e}{m_eq_i-m_iq_e}(m_i\widetilde{n}_i+m_e\widetilde{n}_e)
\\[3mm]&+\pa_{x}u_1^{r}\left[(m_{i}\widetilde{n}_{i}+m_e\widetilde{n}_e)\widetilde{u}_1^{2}
-\frac{4}{9}\theta^{r}\left(\frac{\widetilde{n}_i^3}{n_in^r_{i}}
+\frac{\widetilde{n}_{e}^3}{n_en_e^r}\right)+\frac{2}{3}\frac{(\widetilde{n}_{i}+\widetilde{n}_e)}{\ta^r}
\widetilde{\ta}^2\right]
\\[3mm]&+(\widetilde{n}_i+\widetilde{n}_e)\frac{\pa_x\theta^{r}}{\theta^{r}}\widetilde{u}_1\widetilde{\ta}
+\pa_{x}u_1^{r}\left[O(1)\widetilde{n}_i^{3}+O(1)\widetilde{n}_e^{3}+O(1)\widetilde{\ta}^3\right]
\\[3mm]=&\left[\widetilde{n}_i,\widetilde{n}_e,\widetilde{u}_1,\widetilde{\ta}\right]\mathbb{M}
\left[\widetilde{n}_i,\widetilde{n}_e,\widetilde{u}_1,\widetilde{\ta}\right]^{\rm T}
\\[3mm]&+\pa_{x}u_1^{r}\left[(m_{i}\widetilde{n}_{i}+m_e\widetilde{n}_e)\widetilde{u}_1^{2}
+\frac{4}{9}\theta^{r}\left(\frac{-\widetilde{n}_i^3}{n_in^r_{i}}
+\frac{-\widetilde{n}_{e}^3}{n_en_e^r}\right)+\frac{2}{3}\frac{(\widetilde{n}_{i}+\widetilde{n}_e)}{\ta^r}
\widetilde{\ta}^2\right]
\\[3mm]&+(\widetilde{n}_i+\widetilde{n}_e)\frac{\pa_x\theta^{r}}{\theta^{r}}\widetilde{u}_1\widetilde{\ta}
+\pa_{x}u_1^{r}\left[O(1)\widetilde{n}_i^{3}+O(1)\widetilde{n}_e^{3}+O(1)\widetilde{\ta}^3\right],
\end{split}
\end{equation*}
where $\mathbb{M}$ is a $4\times 4$ symmetric matrix given by
\begin{eqnarray}\label{def.matrM}
\begin{aligned} &\left[\begin{array} {cccc}
-\frac{4}{9}A\frac{q_i}{q_e}(n^r)^{-\frac{1}{3}} \  \  \  & 0\  \  \ & -\frac{2A}{9B}\frac{m_i(q_i-q_e)}{m_eq_i-m_iq_e}(n^r)^{\frac{1}{3}} \  \  \ & 0 \\[3mm]
0 \  \  \ & \frac{4}{9}A(n^r)^{-\frac{1}{3}} \  \  \  & -\frac{2A}{9B}\frac{m_e(q_i-q_e)}{m_eq_i-m_iq_e}(n^r)^{\frac{1}{3}} \  \  \ & 0  \\[3mm]
\ast \  \  \ & \ast \  \  \ & \frac{m_eq_i-m_iq_e}{q_i}n^r\  \  \  & \frac{q_i-q_e}{3q_iB}(n^r)^{\frac{2}{3}} \\[3mm]
0 \  \  \ & 0\  \  \ & \ast \  \  \ & \frac{2(q_i-q_e)}{3q_iA}(n^r)^{\frac{1}{3}}
\end{array} \right].
\end{aligned}
\end{eqnarray}
We claim that $\mathbb{M}$ is positive-definite provided that $\frac{q_i}{|q_e|}\leq 9$ and $m_i\geq m_e$. To see this, we compute its four leading principal minors as follows:
\begin{eqnarray*}
&&\De_{11}>0,\\
&&\De_{22}=-\frac{16}{81}A^2\frac{q_i}{q_e} (n^r)^{-\frac{2}{3}}>0,\\
&&\De_{33}= -\frac{16}{81q_e}A^2(n^r)^{\frac{1}{3}}(m_eq_i-m_iq_e)+\frac{16}{810q_e}A^2(n^r)^{\frac{1}{3}}
\frac{(q_im_e^2-q_em_i^2)}{m_eq_i-m_iq_e}(q_i-q_e)\\
&&\qquad=-\frac{A^2}{q_e}\frac{(n^r)^{1/3}}{m_eq_i-m_iq_e}\left[\frac{16\times9}{810}(m^2_eq_i^2+m_i^2q_e^2)
-\frac{32}{81}m_im_eq_iq_e+\frac{16}{810}(m_i^2+m_e^2)q_iq_e\right],\\
&&\De_{44}= \frac{16}{810q_iq_e}A(n^r)^{\frac{2}{3}}(q_i-q_e)(m_eq_i-m_iq_e)-\frac{2(q_i-q_e)}{3q_iA}(n^r)^{\frac{1}{3}}
\De_{33}\\
&&\qquad=-\frac{q_i-q_e}{q_iq_e}\frac{A(n^r)^{2/3}}{m_eq_i-m_iq_e}\left[\frac{16\times5}{810}(m^2_eq_i^2+m_i^2q_e^2)
-\frac{32}{81}\times(\frac{2}{3}-\frac{1}{10})m_im_eq_iq_e\right.\\
&&\left.\qquad\qquad\qquad\qquad\qquad\qquad\qquad\qquad\qquad\qquad\qquad\qquad\qquad+\frac{16}{810}\times\frac{2}{3}(m_i^2+m_e^2)q_iq_e\right].
\end{eqnarray*}
One sees that whenever one has $\frac{q_i}{|q_e|}\leq9$ and $m_i\geq m_e$, it holds that $\De_{33}>0$ and $\De_{44}>0$ and hence $\mathbb{M}$ is positive-definite. As  $\mathbb{M}$ is positive definite,
we immediately get from \eqref{aps} that there exists $\la>0$ such that %(\Red{Here the condition for $q_i$ and $q_e$ is not good, note that there is no such restriction when working in Lagrangian coordinate, SQ})
\begin{equation*}%\label{key1}
\int_{\R} \CN_1\,dx\geq \la\int_{\R}\pa_xu_1^r\left[\widetilde{n}_i,\widetilde{n}_e,\widetilde{u}_1,\widetilde{\ta}\right]^2dx,
\end{equation*}
where we have also used the Cauchy-Schwarz inequality and Sobolev's inequality
\begin{equation}\label{so.ine}
\|h\|_{L^\infty}\leq \sqrt{2}\|h\|^{1/2}\|\pa_xh\|^{1/2},\ \text{for}\ h\in H^1(\R).
\end{equation}
Next,
by applying \eqref{so.ine} the a priori assumption \eqref{aps} and Lemma \ref{cl.RwRe},
 %as well as Cauchy-Schwarz inequality with $\eta>0$,
we obtain
\begin{equation*}%\label{I4.1}
\begin{split}
\left|\int_{\R}\CN_2\,dx\right|\lesssim& \eps_0\int_{\R}\pa_xu_1^r\left[\widetilde{n}_i,\widetilde{n}_e,\widetilde{u}_1,\widetilde{\ta}\right]^2dx
+\de_r^{1/2}\|q_i\widetilde{n}_i+q_e\widetilde{n}_e\|^2+\de_r^{-1/2}\|\pa_xn^r\|_{L^\infty}^2\|\widetilde{u}_1\|^2\\
\lesssim& \eps_0\int_{\R}\pa_xu_1^r\left[\widetilde{n}_i,\widetilde{n}_e,\widetilde{u}_1,\widetilde{\ta}\right]^2dx+
\de_r^{1/2}\|q_i\widetilde{n}_i+q_e\widetilde{n}_e\|^2+(1+t)^{-3/2}\|\widetilde{u}_1\|^2.
\end{split}
\end{equation*}
%here $0<\si_1<2/3$ and $0<\eta<1$ and we also have used the following estimates to derive the last inequality
%\begin{equation}\label{I4.1}
%\begin{split}
%\int_{\R}\left|\pa_xu_1^r\pa^2_x\phi^r\widetilde{v}_i\right|dx
%\lesssim&\|\widetilde{v}_i\|_{L^\infty}\|\pa_xu_1^r\|_{L^1}\|\pa^2_x\phi^r\|_{L^\infty}
% \lesssim \|\widetilde{v}_i\|^{1/2}\|\pa_x\widetilde{v}_i\|^{1/2}(1+t)^{-1}
%\\ \lesssim& \eps_{0}^{2-\frac{3}{4}(2+\si_1)}\|\pa_x\widetilde{v}_i\|^2+\eps_{0}^{2+\si_1}(1+t)^{-4/3}.
%\end{split}
%\end{equation}
Likewise, one can see that
\begin{equation*}%\label{I6}
\begin{split}
\left|\int_{\R}\CN_4\,dx\right| &\lesssim\int_{\R}\left|q_i\widetilde{n}_i+q_e\widetilde{n}_e\right|
\left|\pa_x\phi \right||\widetilde{u}_1|dx\\
&\lesssim\eps_0\left\|\pa_x\phi \right\|^2+
(\eps_0+\eta)\left\|q_i\widetilde{n}_i+q_e\widetilde{n}_e\right\|^2+C_\eta(1+t)^{-2}\left\|\widetilde{u}_1\right\|^2.
\end{split}
\end{equation*}
%and
%\begin{equation}\label{I8}
%\begin{split}
%|I_8|\lesssim&\int_{\R}\left|\widetilde{v}_i-\widetilde{v}_e\right|
%\left|\pa_t\left[\widetilde{\ta},\ta^r\right]\right||\widetilde{\ta}|dx
%\\ \lesssim&\eps_0\left\|\pa_t\widetilde{\ta}\right\|^2+
%(\eps_0+\eta)\left\|\widetilde{v}_i-\widetilde{v}_e\right\|^2+
%C_\eta(1+t)^{-2}\|\widetilde{\ta}]\|^2.
%\end{split}
%\end{equation}
Utilizing Lemma \ref{cl.RwRe}, Cauchy-Schwarz inequality with $\eta>0$ and Sobolev's inequality \eqref{so.ine} repeatedly, we compute
\begin{equation*}%\label{I7}
\begin{split}
\left|\int_{\R}\CN_3\,dx\right|\lesssim& \int_{\R}\left|\left[ \pa_x[u_1^r,\ta^r]\pa_x\left[\widetilde{n}_i,\widetilde{n}_e,\widetilde{\ta},n^r,\ta^r\right]\right]
\left[\widetilde{u}_1,\widetilde{\ta}\right]\right|dx
\\&+ \int_{\R}\left|\left[\pa_x^2u_1^r,\pa_x^2\ta^r\right]\left[\widetilde{u}_1,\widetilde{\ta}\right]\right|dx
+\int_{\R}\left|\left[\pa_x\widetilde{u},\pa_x\widetilde{\ta}\right]^2\left[\widetilde{u}_1,\widetilde{\ta}\right]\right|dx\\
\lesssim& (\eps_0+\de_r^{1/2})\left\|\pa_x\left[\widetilde{n}_i,\widetilde{n}_e,\widetilde{u}_1,\widetilde{\ta}\right]\right\|^{1/2}+
(1+t)^{-3/2}\left\|\left[\widetilde{u}_1,\widetilde{\ta}\right]\right\|^2
\\&+
\left\|\left[\widetilde{u}_1,\widetilde{\ta}\right]\right\|^{1/2}
\left\|\pa_x\left[\widetilde{u}_1,\widetilde{\ta}\right]\right\|^{1/2}\left\{\left\|\pa_x^2[u_1^r,\ta^r]\right\|_{L^1}
+\left\|\pa_x[u_1^r,\ta^r]\pa_x\left[n^r,\ta^r\right]\right\|_{L^1}\right\}
\\
\lesssim& (\eps_0+\de_r^{1/2})\left\|\pa_x\left[\widetilde{n}_i,\widetilde{n}_e,\widetilde{u}_1,\widetilde{\ta}\right]\right\|^{1/2}+
(1+t)^{-3/2}\left\|\left[\widetilde{u}_1,\widetilde{\ta}\right]\right\|^2+\de^{1/6}_r(1+t)^{-7/6}.
\end{split}
\end{equation*}
We now turn to estimate the terms involving $\CN_{5}$, $\CN_{6}$ and $\CN_{7}$.
Let us first consider $\int_{\R}\CN_{5}\,dx$ and $\int_\R\CN_{7}\,dx$. Recalling $\FG =\widetilde{\FG}+\overline{\FG}$ with $\overline{\FG}$
given by \eqref{def.ng},
%it is easy to see that
%$\left\|\overline{\FG}_\CA\left(\FM_\CA^{(2)}\right)^{-1/2}\right\|^2_{L^2_{x,\xi}}$
%is not integrable with respect to the time  variable $t$.
%Noticing that
%With \eqref{key1} and \eqref{cp.bg} in hand,
we get from integration by parts, Cauchy-Schwarz inequality and Lemma \ref{cl.RwRe} that
\begin{equation*}
\begin{split}
\left|\int_{\R}\CN_5\,dx\right|\leq& (\eps_0+\eta)\left\|\pa_x\left[\widetilde{n}_i,\widetilde{n}_e,\widetilde{\ta}\right]\right\|^2
+\int_{\R}\left|\pa_x\left[\widetilde{n}_i,\widetilde{n}_e,\widetilde{u},\widetilde{\ta}\right]\right|
\left|\pa_x\left[n^r,u_1^r,\ta^r\right]\right|\left|\left[\widetilde{n}_i,\widetilde{n}_e,\widetilde{\ta}\right]\right|dx
\\&+\int_{\R}
\left|\pa_x\left[n^r,u_1^r,\ta^r\right]\right|^2\left|\left[\widetilde{n}_i,\widetilde{n}_e,\widetilde{\ta}\right]\right|dx
+\int_{\R}
\left|\pa^2_x\left[n^r,u_1^r,\ta^r\right]\right|\left|\left[\widetilde{n}_i,\widetilde{n}_e,\widetilde{\ta}\right]\right|dx
\\&+C_\eta\int_{{\R}\times{\R}^3}(1+|\xi|)\left|\FM^{-1/2}_\sharp\widetilde{\FG}\right|^2d\xi dx
\\ \leq& (\eps_0+\eta)\left\|\pa_x\left[\widetilde{n}_i,\widetilde{n}_e,\widetilde{u},\widetilde{\ta}\right]\right\|^2
+C_\eta(1+t)^{-2}\left\|\left[\widetilde{n}_i,\widetilde{n}_e,\widetilde{\ta}\right]\right\|^2
+\de_r^{1/6}(1+t)^{-7/6}
\\&+C_\eta\int_{{\R}\times{\R}^3}(1+|\xi|)\left|\FM^{-1/2}_\sharp\widetilde{\FG}\right|^2d\xi dx.
\end{split}
\end{equation*}
Next, by Cauchy-Schwarz inequality, one has
\begin{equation*}%\label{I14}
\begin{split}
\left|\int_{\R}\CN_7\,dx\right|\lesssim&(\eps_0+\eta)\left\|\pa_x\phi \right\|^2+C_\eta(1+t)^{-2}\left\|\widetilde{\ta}\right\|^2
+\eps_{0}\left\|\pa_x\widetilde{\ta}\right\|^2
\\&+\de_r^{1/6}(1+t)^{-7/6}
+(\eps_0+\eta)
\int_{{\R}\times{\R}^3}(1+|\xi|)\left|\FM^{-1/2}\pa_{\xi_1}\widetilde{\FG}\right|^2d\xi dx.
\end{split}
\end{equation*}
%where we have also used Cauchy-Schwarz inequality with $\de$ and Sobolev's inequality
%\begin{equation}\label{so.ine}
%\|h\|_{L^\infty}\leq \sqrt{2}\|h\|^{1/2}\|\pa_xh\|^{1/2},\ \text{for}\ h\in H^1(\R).
%\end{equation}
It now remains to consider $\int_\R\CN_{6}\,dx$. %In fact, by choosing some special $C_i$ and $C_e$, we let $I_{12}=0$, to confirm this,
Notice that
\begin{equation}\label{d.iL}
\pa^{\al}\pa^\be\left\{L^{-1}_{M_{\CA}}h\right\}=L^{-1}_{M_{\CA}}(\pa^{\al} \pa^\be h)
-\sum\limits_{j=1}^{\al+|\be|}\sum\limits_{|\al'|  +|\be'|=j}C^{\al,\be}_{\al',\be'}
L^{-1}_{M_{\CA}}\mathcal {M}_{\al\be},
\end{equation}
where $\CM_{\al\be}$ is given by
\begin{eqnarray*}
\CM_{\al\be}=
Q_{\CA\CA}\left(\pa^{\al-\al'}\pa^{\be-\be'}\left(L^{-1}_{M_{\CA}}h\right),\pa^{\al'}\pa^{\be'}M_{\CA}\right)
+Q_{\CA\CA}\left(\pa^{\al'}\pa^{\be'}M_{\CA},\pa^{\al-\al'}\pa^{\be-\be'}\left(L^{-1}_{M_{\CA}}h\right)\right).
\end{eqnarray*}
Utilizing \eqref{d.iL}, Lemma \ref{est.nonop}, Lemma \ref{cl.RwRe} and Corollary \ref{inv.L.} again, we now arrive at
\begin{equation*}
\begin{split}
\left|\int_\R\CN_{6}\,dx\right|
\lesssim &
\left|\sum\limits_{j=1}^3\left(\int_{{\R}^3}\xi_1\psi_{j+2}\cdot\pa_xL_{M_{\CA}}^{-1}\overline{Q}_{\CA}(\FG,\FG) \,d\xi,\widetilde{u}_j\right)\right|
\\&+\left|\left(\int_{{\R}^3}\xi_1\left(\psi_{6}-\sum\limits_{j=1}^3u_j\cdot\psi_{j+2}\right)
\cdot\pa_xL_{M_{\CA}}^{-1}\overline{Q}_{\CA}(\FG,\FG) \,d\xi,\ta^{-1}\widetilde{\ta}\right)\right|\\
&+C_\eta\sum\limits_{\CA\in\{i,e\}}\int_{\R\times\R^3}\left|M_\CA^{-1/2}
\left(\overline{R}_{\CA}-L_{M_{\CA}}^{-1}\overline{Q}_{\CA}(\FG,\FG)\right)\right|^2d\xi dx\\
&+{\eta\sum\limits_{j=1}^3\int_{\R\times\R^3}|\xi|^4|\FM||\pa_x\widetilde{u}_j|^2 d\xi dx }
\\&+\eta\int_{\R\times\R^3}|\xi|^6|\FM|\left|\pa_x\left(\frac{\widetilde{\theta}}{\theta}\right)\right|^2 d\xi dx
+\eta\sum\limits_{i=j}^3\int_{\R\times\R^3}|\xi|^4|\FM|
\left|\pa_x\left(\frac{\widetilde{\theta}}{\theta}\widetilde{u}_j\right)\right|^2 d\xi dx
\\&+\eta\int_{\R\times\R^3}|\xi|^4|\FM|\left|\pa_x\left(\frac{\widetilde{\theta}}{\theta}u^r_1\right)\right|^2 d\xi dx,
\end{split}
\end{equation*}
which further gives
\begin{equation*}
\begin{split}
\left|\int_\R\CN_{6}\,dx\right|
\leq&(\eta+\eps_0)\left\|\pa_x\left[\widetilde{n}_i,\widetilde{n}_e,\widetilde{u},\widetilde{\theta}\right]\right\|^2
+C(1+t)^{-2}\left\|\left[\widetilde{n}_i,\widetilde{n}_e,\widetilde{u},\widetilde{\theta}\right]\right\|^2\\
&+\eps_0\|\pa_x\phi\|^2+C_\eta\de^{1/6}_r(1+t)^{-7/6}
\\&+C_\eta\sum\limits_{|\al|=1}\int_{\R\times\R^3}(1+|\xi|)|\FM^{-1/2}\pa^\al\FG|^2 d\xi dx\\
&+C_\eta\int_{\R\times\R^3}|\pa_x\phi|^2|\FM^{-1/2}\pa_{\xi_1}(\widetilde{\FG}+\overline{\FG})|^2d\xi dx
\\&+C_\eta\int_{\R}\left(\int_{\R^3}(1+|\xi|)|\FM^{-1/2}_\sharp\FG|^2d\xi\right)
\left(\int_{\R^3}|\FM^{-1/2}_\sharp\FG|^2 d\xi\right) dx.
\end{split}
\end{equation*}
It then follows that
\begin{equation*}%\label{R.es1}
\begin{split}
\left|\int_\R\CN_{6}\,dx\right|
\leq&(\eta+\eps_0)\left\|\pa_x\left[\widetilde{u},\widetilde{\theta}\right]\right\|^2+C_\eta\de_r\left\|\pa_x\phi \right\|^2
+C(1+t)^{-2}\left\|\left[\widetilde{n}_i,\widetilde{n}_e,\widetilde{u},\widetilde{\theta}\right]\right\|^2+C_\eta\de^{1/6}_r(1+t)^{-7/6}
\\&+C_\eta\eps_0\int_{\R\times\R^3}|\FM^{-1/2}\pa_{\xi_1}\widetilde{\FG}|^2 d\xi dx
+C_\eta\sum\limits_{|\al|=1}\int_{\R\times\R^3}(1+|\xi|)|\FM^{-1/2}\pa^\al\FG|^2 d\xi dx
\\&+C_\eta\int_{{\R}\times{\R}^3}(1+|\xi|)\left|\FM^{-1/2}_\sharp\widetilde{\FG}\right|^2 d\xi dx.
\end{split}
\end{equation*}
Here the following %significant
estimate has been used:
\begin{equation}\label{db.G}
\begin{split}
\int_{\R}&\left(\int_{\R^3}(1+|\xi|)|\FM^{-1/2}_\sharp\FG|^2d\xi\right)
\left(\int_{\R^3}|\FM^{-1/2}_\sharp\FG|^2 d\xi\right) dx
\\ \leq&
C\int_{{\R}}\left(\int_{{\R}^3}(1+|\xi|)\left|\FM_\sharp^{-1/2}(\widetilde{\FG}+\overline{\FG})\right|^2d\xi\right)
\left(\int_{{\R}^3}\left|\FM_\sharp^{-1/2}(\widetilde{\FG}+\overline{\FG})\right|^2d\xi\right) dx\\
\leq& \int_{{\R}\times{\R}^3}(1+|\xi|)\left|\FM^{-1/2}_\sharp\widetilde{\FG}\right|^2 d\xi dx
+C\int_{{\R}}\left|\pa_x[u^r,\ta^r]\right|^4 dx\\
\leq& \int_{{\R}\times{\R}^3}(1+|\xi|)\left|\FM^{-1/2}_\sharp\widetilde{\FG}\right|^2 d\xi dx
+C\de_r(1+t)^{-2}.
\end{split}
\end{equation}
Moreover,
Corollary \ref{inv.L.} and Lemma \ref{est.nonop} as well as similar calculations for obtaining the formula \eqref{d.iL} are also applied
to deduce
\begin{equation*}%\label{d.bG}
\begin{split}
&\int_{\R^3}\left|\FM^{-1/2}_\sharp\pa_{\xi_1}\overline{\FG}\right|^2d\xi
\\ &\leq
C\sum_{|\be|\leq 1}\int_{{\R}^3}(1+|\xi|)\left|\FM^{-1/2}_\sharp \FL^{-1}_{\FM}\pa^\be \left\{{\bf P}_1^{\FM}\left[\left[m_iM_i,m_eM_e\right]^{\rm T}
\xi_1\left(\xi_1\pa_xu_1^r+\frac{|\xi-u|^2}{2\ta}\pa_x\ta^r\right)\right]\right\} \right|^2d\xi
\\&\quad+C\sum_{|\be|\leq 1}\int_{{\R}^3}(1+|\xi|)\left|\FM^{-1/2}_\sharp \FL^{-1}_{\FM}\pa^\be \left\{{\bf P}_1^{\FM}\left[\left[n^{-1}_iM_i\pa_xn_i^r,n^{-1}_eM_e\pa_xn_e^r\right]^{\rm T}
\xi_1\right]\right\} \right|^2d\xi
\\&\quad+C\sum_{|\be|\leq 1}\int_{{\R}^3}(1+|\xi|)\left|\FM^{-1/2}_\sharp \FL^{-1}_{\FM}\pa^\be \left\{{\bf P}_1^{\FM}\left[\left[M_i,M_e\right]^{\rm T}
\xi_1\pa_x\ta^r\right]\right\} \right|^2d\xi
\\ &\leq C\left|\pa_x\left[n^r,u_1^r,\ta^r\right]\right|^2.
\end{split}
\end{equation*}
Finally, by substituting the above estimates for $\int_\R\CN_{l}\,dx$ $(1\leq l\leq 7)$ into \eqref{basic}, we conclude that
\begin{equation}\label{zero.eng}
\begin{split}
\frac{d}{dt}&\widetilde{\eta}
+\la\left\|\pa_x\left[\widetilde{u},\widetilde{\ta}\right]\right\|^2
+\la\int_{\R}\pa_xu_1^r\left|\left[\widetilde{n}_i,\widetilde{n}_e,\widetilde{u}_1,\widetilde{\ta}\right]\right|^2dx
\\
\leq & (\eta+\eps_0)\left\|\pa_t\left[\widetilde{u},\widetilde{\ta}\right]\right\|^2
+(\eta+\eps_0)\left\|\pa_x\left[\widetilde{n}_i,\widetilde{n}_e,\phi \right]\right\|^2
+(\eps_0+\eta)\left\|q_i\widetilde{n}_i+q_e\widetilde{n}_e\right\|^2
\\&+C_\eta(1+t)^{-2}\left\|\left[\widetilde{n}_i,\widetilde{n}_e,\widetilde{u},\widetilde{\ta}\right]\right\|^2
+\de^{1/6}_r(1+t)^{-7/6}
+C_\eta\sum\limits_{|\al|=1}\int_{\R\times\R^3}(1+|\xi|)|\FM^{-1/2}\pa^\al\FG|^2 d\xi dx
\\&+C_\eta\int_{{\R}\times{\R}^3}(1+|\xi|)\left|\FM^{-1/2}_\sharp\widetilde{\FG}\right|^2 d\xi dx
+C_\eta\eps_0\int_{\R\times\R^3}|\FM^{-1/2}\pa_{\xi_1}\widetilde{\FG}|^2 d\xi dx.
\end{split}
\end{equation}

\subsection{Estimate on first-order dissipation}

One has to further consider the dissipation terms involving $\pa_x\left[\widetilde{n}_i,\widetilde{n}_e,\phi ,\pa_x\phi \right]$
and $\pa_t\left[\widetilde{n}_i,\widetilde{n}_e,\widetilde{u},\widetilde{\ta}\right]$. The computations are divided into three steps.

\medskip

\noindent{\it Step 1.} {\it Dissipation of $\pa_x(\widetilde{n}_i+\widetilde{n}_e)$:} We first differentiate \eqref{tvi} and \eqref{tve} with respect to $x$, respectively, to obtain
\begin{equation}\label{d.vi}
\pa_t\pa_x\widetilde{n}_i+\pa^2_x\left(n_iu_1-n_i^ru_1^r\right)
=-\pa^2_x\int_{{\R}^3}\xi_1G_i \,d\xi,
\end{equation}
and
\begin{equation}\label{d.ve}
\pa_t\pa_x\widetilde{n}_e+\pa^2_x\left(n_eu_1-n_e^ru_1^r\right)
=-\pa^2_x\int_{{\R}^3}\xi_1G_e \,d\xi.
\end{equation}
Then taking the inner products of \eqref{tu1}, \eqref{d.vi} and \eqref{d.ve} with terms
$$
\pa_x(\widetilde{n}_i+\widetilde{n}_e),\ \ \frac{3(\mu_i(\theta)+\mu_e(\theta))}{n_i^r}\pa_x\widetilde{n}_i,\ \
\text{and}\ \
\frac{3(\mu_i(\theta)+\mu_e(\theta))}{n_e^r}\pa_x\widetilde{n}_e,
$$
with  respect to $x$ over $\R$, respectively, one has
%\begin{equation}\label{d.phi.ip}
%\begin{split}
%&\left(\frac{v^r}{v}\pa^2_x\phi ,\pa^2_x\phi \right)
%+\left(ve^{-\phi }\pa_x\phi ,\pa_x\phi \right)\\
%&\qquad=\left(\frac{v^r}{v^2}\pa_xv\pa_x\phi ,\pa^2_x\phi \right)+
%\left(\pa_x\widetilde{v},\pa_x\phi \right)-\left(\pa_xv\left(1-e^{-\phi }\right),\pa_x\phi \right)
%\\
%&\qquad\quad+\left(\pa_xv^r\pa_x\left(\frac{\pa_x\phi^r}{v}\right),\pa_x\phi \right)
%+\left(v^r\pa^2_x\left(\frac{\pa_x\phi^r}{v}\right),\pa_x\phi \right),
%\end{split}
%\end{equation}
\begin{equation}\label{tu1.ip}
\begin{split}
&\big((m_in_i+m_en_e)\left(\pa_t \widetilde{u}_1+u_1\pa_x\widetilde{u}_1+\widetilde{u}_1\pa_xu_1^{r}\right),\pa_x(\widetilde{n}_i+\widetilde{n}_e)\big)
+\left(\pa_x P-\pa_xP^r,\pa_x(\widetilde{n}_i+\widetilde{n}_e)\right)
\\
&\quad+\left(\left(1-\frac{m_in_i+m_en_e}{m_in_i^r+m_en_e^r}\right)\pa_xP^r,\pa_x(\widetilde{n}_i+\widetilde{n}_e)\right)
+\left((q_in_i+q_en_e)\pa_x\phi ,
\pa_x(\widetilde{n}_i+\widetilde{n}_e)\right)
\\
&=3\left((\mu_i(\theta)+\mu_e(\theta))\pa^2_x \widetilde{u}_1,\pa_x(\widetilde{n}_i+\widetilde{n}_e)\right)
+3\left((\mu_i(\theta)+\mu_e(\theta))\pa^2_x u^r_1,\pa_x(\widetilde{n}_i+\widetilde{n}_e)\right)\\
&\quad +3\left(\pa_x(\mu_i(\theta)+\mu_e(\theta))\pa_x u_1,\pa_x(\widetilde{n}_i+\widetilde{n}_e)\right)
-\left(\int_{{\R}^3}\xi_1\psi_{3}\cdot\pa_x\overline{\FR} \,d\xi,\pa_x(\widetilde{n}_i+\widetilde{n}_e)\right)
\\&\quad
-\left(\int_{{\R}^3}\psi_{3i}\xi_1\pa_x \left(P_0^{M_i}G_i\right) d\xi,\pa_x(\widetilde{n}_i+\widetilde{n}_e)\right)
-\left(\int_{{\R}^3}\psi_{3e}\xi_1\pa_x \left(P_0^{M_e}G_e\right) d\xi,\pa_x(\widetilde{n}_i+\widetilde{n}_e)\right),
\end{split}
\end{equation}
and
\begin{equation}\label{vi.ip}
\begin{split}
&\left(\pa_t\pa_x\widetilde{n}_i+\pa_x^2(\widetilde{n}_i\widetilde{u}_1+\widetilde{n}_iu_1^r)+\pa_x^2n^r_i\widetilde{u}_1
+\pa_xn^r_i\pa_x\widetilde{u}_1
+n^r_i\pa^2_x\widetilde{u}_1,\frac{3(\mu_i(\theta)+\mu_e(\theta))}{n_i^r}\pa_x\widetilde{n}_i\right)
\\&=
-\left(\pa^2_x\int_{{\R}^3}\xi_1G_i \,d\xi,\frac{3(\mu_i(\theta)+\mu_e(\theta))}{n^r_i}\pa_x\widetilde{n}_i\right),
\end{split}
\end{equation}
and
\begin{equation}\label{ve.ip}
\begin{split}
&\left(\pa_t\pa_x\widetilde{n}_e+\pa_x^2(\widetilde{n}_e\widetilde{u}_1+\widetilde{n}_eu_1^r)+\pa_x^2n^r_e\widetilde{u}_1
+\pa_xn^r_e\pa_x\widetilde{u}_1
+n^r_e\pa^2_x\widetilde{u}_1,\frac{3(\mu_i(\theta)+\mu_e(\theta))}{n_e^r}\pa_x\widetilde{n}_e\right)
\\&=
-\left(\pa^2_x\int_{{\R}^3}\xi_1G_e \,d\xi,\frac{3(\mu_i(\theta)+\mu_e(\theta))}{n^r_e}\pa_x\widetilde{n}_e\right).
\end{split}
\end{equation}
Notice that
\begin{equation}\label{p-pr}
P-P^r=\frac{2}{3}\left((\widetilde{n}_i+\widetilde{n}_e)\widetilde{\ta}+(\widetilde{n}_i+\widetilde{n}_e)\ta^r
+(n^r_i+n^r_e)\widetilde{\ta}\right).
\end{equation}
We get from the summation of \eqref{tu1.ip}, \eqref{vi.ip} and \eqref{ve.ip} that
\begin{equation}\label{sum.d1}
\begin{split}
&-\frac{d}{dt}\left((m_in_i+m_en_e)\widetilde{u}_1,\pa_x(\widetilde{n}_i+\widetilde{n}_e)\right)
+\frac{3}{2}\frac{d}{dt}\left(\pa_x\widetilde{n}_i,\frac{\mu_i(\theta)+\mu_e(\theta)}{n_i^r}\pa_x\widetilde{n}_i\right)
\\&\quad +\frac{3}{2}\frac{d}{dt}\left(\pa_x\widetilde{n}_e,\frac{\mu_i(\theta)+\mu_e(\theta)}{n^r_e}\pa_x\widetilde{n}_e\right)
+\frac{2}{3}\left(\ta^r\pa_x(\widetilde{n}_i+\widetilde{n}_e),
\pa_x(\widetilde{n}_i+\widetilde{n}_e)\right)\\
&=\sum\limits_{l=1}^{9}\CI_{l},
\end{split}
\end{equation}
with
\begin{eqnarray*}
\left\{\begin{array}{rll}
\begin{split}
\CI_1=&-\left((m_in_i+m_en_e)(u_1\pa_x\widetilde{u}_1+\widetilde{u}_1\pa_xu_1^{r}),
\pa_x(\widetilde{n}_i+\widetilde{n}_e)\right),\\
\CI_2=&-\frac{2}{3}\left(\pa_x\ta^r(\widetilde{n}_i+\widetilde{n}_e),
\pa_x(\widetilde{n}_i+\widetilde{n}_e)\right)+\frac{2}{3}\left(\pa_x\left(\widetilde{\ta}\left(\widetilde{n}_i+\widetilde{n}_e\right)\right),
\pa_x(\widetilde{n}_i+\widetilde{n}_e)\right),\\
\CI_3=&\frac{2}{3}\left(\pa_x\left(\widetilde{\ta}(n_i^r+n_e^r)\right),
\pa_x(\widetilde{n}_i+\widetilde{n}_e)\right),\ \
\CI_4=\left((q_in_i+q_en_e)\pa_x\phi,
\pa_x(\widetilde{n}_i+\widetilde{n}_e)\right),\\
\CI_5=&
\left((m_in_i+m_en_e)\widetilde{u}_1,\pa_t\pa_x(\widetilde{n}_i+\widetilde{n}_e)\right)
+\left((m_i\pa_tn_i+m_e\pa_tn_e)\widetilde{u}_1,\pa_x(\widetilde{n}_i+\widetilde{n}_e)\right),\\
\CI_6=&\frac{3}{2}\left(\pa_t\left(\frac{\mu_i(\theta)+\mu_e(\theta)}{n_i^r}\right)\pa_x\widetilde{n}_i,\pa_x\widetilde{n}_i\right)
+
\frac{3}{2}\left(\pa_t\left(\frac{\mu_i(\theta)+\mu_e(\theta)}{n_e^r}\right)\pa_x\widetilde{n}_e,\pa_x\widetilde{n}_e\right),\\
\CI_7=&3\left((\mu_i(\theta)+\mu_e(\theta))\pa^2_x u^r_1,\pa_x(\widetilde{n}_i+\widetilde{n}_e)\right)
+3\left(\pa_x(\mu_i(\theta)+\mu_e(\theta))\pa_x u_1,\pa_x(\widetilde{n}_i+\widetilde{n}_e)\right),
%\\
%\CI_8=&-\left(\pa_x^2(\widetilde{n}_i\widetilde{u}_1+\widetilde{n}_iu_1^r)+\pa_x^2n^r_i\widetilde{u}_1
%+\pa_xn^r_i\pa_x\widetilde{u}_1
%,\frac{3(\mu_i(\theta)+\mu_e(\theta))}{n_i^r}\pa_x\widetilde{n}_i\right)
%\\&-\left(\pa_x^2(\widetilde{n}_e\widetilde{u}_1+\widetilde{n}_eu_1^r)+\pa_x^2n^r_e\widetilde{u}_1
%+\pa_xn^r_e\pa_x\widetilde{u}_1
%,\frac{3(\mu_i(\theta)+\mu_e(\theta))}{n_e^r}\pa_x\widetilde{n}_e\right),\\
%\CI_9=&-\left(\int_{{\R}^3}\xi_1\psi_{3}\cdot\pa_x\overline{\FR} d\xi,\pa_x(\widetilde{n}_i+\widetilde{n}_e)\right)
%-\left(\int_{{\R}^3}\psi_{3i}\xi_1\pa_x \left(P_0^{M_i}G_i\right) d\xi,\pa_x(\widetilde{n}_i+\widetilde{n}_e)\right)
%\\&-\left(\int_{{\R}^3}\psi_{3e}\xi_1\pa_x \left(P_0^{M_e}G_e\right) d\xi,\pa_x(\widetilde{n}_i+\widetilde{n}_e)\right)
%-\left(\pa_x^2\int_{{\R}^3}\xi_1G_i d\xi,\frac{3(\mu_i(\theta)+\mu_e(\theta))}{n_i^r}\pa_x\widetilde{n}_i\right)
%\\&-\left(\pa^2_x\int_{{\R}^3}\xi_1G_e d\xi,\frac{3(\mu_i(\theta)+\mu_e(\theta))}{n_e^r}\pa_x\widetilde{n}_e\right).
\end{split}
\end{array}\right.
\end{eqnarray*}
and
\begin{eqnarray*}
\left\{\begin{array}{rll}
\begin{split}
%\CI_1=&-\left((m_in_i+m_en_e)(u_1\pa_x\widetilde{u}_1+\widetilde{u}_1\pa_xu_1^{r}),
%\pa_x(\widetilde{n}_i+\widetilde{n}_e)\right),\\
%\CI_2=&-\frac{2}{3}\left(\pa_x\ta^r(\widetilde{n}_i+\widetilde{n}_e),
%\pa_x(\widetilde{n}_i+\widetilde{n}_e)\right)+\frac{2}{3}\left(\pa_x\left(\widetilde{\ta}\left(\widetilde{n}_i+\widetilde{n}_e\right)\right),
%\pa_x(\widetilde{n}_i+\widetilde{n}_e)\right),\\
%\CI_3=&\frac{2}{3}\left(\pa_x\left(\widetilde{\ta}(n_i^r+n_e^r)\right),
%\pa_x(\widetilde{n}_i+\widetilde{n}_e)\right),\ \
%\CI_4=\left((q_in_i+q_en_e)\pa_x\phi,
%\pa_x(\widetilde{n}_i+\widetilde{n}_e)\right),\\
%\CI_5=&
%\left((m_in_i+m_en_e)\widetilde{u}_1,\pa_t\pa_x(\widetilde{n}_i+\widetilde{n}_e)\right)
%+\left((m_i\pa_tn_i+m_e\pa_tn_e)\widetilde{u}_1,\pa_x(\widetilde{n}_i+\widetilde{n}_e)\right),\\
%\CI_6=&\frac{3}{2}\left(\pa_t\left(\frac{\mu_i(\theta)+\mu_e(\theta)}{n_i^r}\right)\pa_x\widetilde{n}_i,\pa_x\widetilde{n}_i\right)
%+
%\frac{3}{2}\left(\pa_t\left(\frac{\mu_i(\theta)+\mu_e(\theta)}{n_e^r}\right)\pa_x\widetilde{n}_e,\pa_x\widetilde{n}_e\right),\\
%\CI_7=&3\left((\mu_i(\theta)+\mu_e(\theta))\pa^2_x u^r_1,\pa_x(\widetilde{n}_i+\widetilde{n}_e)\right)
%+3\left(\pa_x(\mu_i(\theta)+\mu_e(\theta))\pa_x u_1,\pa_x(\widetilde{n}_i+\widetilde{n}_e)\right),\\
\CI_8=&-\left(\pa_x^2(\widetilde{n}_i\widetilde{u}_1+\widetilde{n}_iu_1^r)+\pa_x^2n^r_i\widetilde{u}_1
+\pa_xn^r_i\pa_x\widetilde{u}_1
,\frac{3(\mu_i(\theta)+\mu_e(\theta))}{n_i^r}\pa_x\widetilde{n}_i\right)
\\&-\left(\pa_x^2(\widetilde{n}_e\widetilde{u}_1+\widetilde{n}_eu_1^r)+\pa_x^2n^r_e\widetilde{u}_1
+\pa_xn^r_e\pa_x\widetilde{u}_1
,\frac{3(\mu_i(\theta)+\mu_e(\theta))}{n_e^r}\pa_x\widetilde{n}_e\right),\\
\CI_9=&-\left(\int_{{\R}^3}\xi_1\psi_{3}\cdot\pa_x\overline{\FR} \,d\xi,\pa_x(\widetilde{n}_i+\widetilde{n}_e)\right)
-\left(\int_{{\R}^3}\psi_{3i}\xi_1\pa_x \left(P_0^{M_i}G_i\right) d\xi,\pa_x(\widetilde{n}_i+\widetilde{n}_e)\right)
\\&-\left(\int_{{\R}^3}\psi_{3e}\xi_1\pa_x \left(P_0^{M_e}G_e\right) d\xi,\pa_x(\widetilde{n}_i+\widetilde{n}_e)\right)
-\left(\pa_x^2\int_{{\R}^3}\xi_1G_i \,d\xi,\frac{3(\mu_i(\theta)+\mu_e(\theta))}{n_i^r}\pa_x\widetilde{n}_i\right)
\\&-\left(\pa^2_x\int_{{\R}^3}\xi_1G_e \,d\xi,\frac{3(\mu_i(\theta)+\mu_e(\theta))}{n_e^r}\pa_x\widetilde{n}_e\right).
\end{split}
\end{array}\right.
\end{eqnarray*}
We now turn to estimate
$\CI_l$ $(1\leq l\leq 9)$ term by term.
It follows from Lemma \ref{cl.RwRe}, Sobolev's inequality and Cauchy-Schwarz inequality with $0<\eta<1$ that
\begin{equation*}
|\CI_1|,\ |\CI_3|
\lesssim \eta\|\pa_x(\widetilde{n}_i+\widetilde{n}_e)\|^2+C_\eta(1+t)^{-2}\left\|\left[\widetilde{u}_1,\widetilde{\ta}\right]\right\|^2
+C_\eta\left\|\pa_x\left[\widetilde{u}_1,\widetilde{\ta}\right]\right\|^2,
\end{equation*}
\begin{equation*}
|\CI_2|
\lesssim (\eta+\eps_0)\|\pa_x(\widetilde{n}_i+\widetilde{n}_e)\|^2+\eps_0\left\|\pa_x\widetilde{\ta}\right\|^2
+C_\eta(1+t)^{-2}\|[\widetilde{n}_i,\widetilde{n}_e]\|^2,
\end{equation*}
\begin{equation*}
|\CI_4|
\lesssim \eps_0\|q_i\widetilde{n}_i+q_e\widetilde{n}_e\|^2
+\eps_0\|\pa_x(\widetilde{n}_i+\widetilde{n}_e)\|^2.
\end{equation*}
%In view of \eqref{tphy}, applying Sobolev's inequality and the a priori estimates \eqref{aps}, we see that
%\begin{remark}
%Here \eqref{theta2} rather than \eqref{theta} is used to estimate $I_6$, since \eqref{theta} leads to the higher order term $(\pa_x^2\ta, %(\pa_x\widetilde{v})^2)$ which can not be controlled in the context of our basic norm \eqref{aps}.
%\end{remark}
By integration by parts and the a priori estimates \eqref{aps}, we see that
\begin{equation*}
\begin{split}
|\CI_5|
\lesssim& (\eta+\eps_0)\|\pa_x [\widetilde{n}_i,\widetilde{n}_e]\|^2+(\eta+\eps_0)\|\pa_t [\widetilde{n}_i,\widetilde{n}_e]\|^2+C_\eta(1+t)^{-2}\|\widetilde{u}_1\|^2
+C_\eta\left\|\pa_x\widetilde{u}_1\right\|^2.
\end{split}
\end{equation*}
For the estimate on $\CI_6$, from  \eqref{aps}, it follows that
\begin{equation*}
\begin{split}
|\CI_6|
\lesssim \eps_0\|\pa_x[\widetilde{n}_i,\widetilde{n}_e]\|^2.
\end{split}
\end{equation*}
For $\CI_7$, Lemma \ref{cl.RwRe} and Cauchy-Schwarz inequality imply
\begin{equation*}
\begin{split}
|\CI_7|
\lesssim \eta\left\|\pa_x(\widetilde{n}_i+\widetilde{n}_e)\right\|^2
+C_\eta(\de_r+\eps_0)\left\|\pa_x\left[\widetilde{n}_i,\widetilde{n}_e,\widetilde{u}_1,\widetilde{\ta}\right]\right\|^2
+C_\eta\de_r^{1/2}(1+t)^{-3/2}.
\end{split}
\end{equation*}
As to $I_8$, one has
\begin{equation*}
\begin{split}
|\CI_8|\lesssim (\eps_0+\eta)\left\|\pa_x\left[\widetilde{n}_i,\widetilde{n}_e, \pa_x\widetilde{u}_1\right]\right\|^2
+\eta\left\|\pa_x\left[\widetilde{n}_i,\widetilde{n}_e\right]\right\|^2
+C_\eta(1+t)^{-2}\left\|\left[\widetilde{n}_i,\widetilde{n}_e\right]\right\|^2,
\end{split}
\end{equation*}
according to Lemma \ref{cl.RwRe}, Sobolev's inequality and Cauchy-Schwarz inequality with $0<\eta<1$ again.

Finally, for $\CI_9$,
by performing the similar calculations as for $\int_\R\CN_{6}\,dx$ in the previous step,
we have
\begin{equation*}
\begin{split}
\left|\CI_9\right|
 \lesssim& (\eps_0+\eta)\left\|\pa_x\left[\widetilde{n}_i,\widetilde{n}_e,\widetilde{u},\widetilde{\ta}\phi ,\pa_x\phi \right]\right\|^2
+\de_r(1+t)^{-2}
+C_\eta\sum\limits_{1\leq|\al| \leq2}\int_{\R\times\R^3}(1+|\xi|)|\FM^{-1/2}\pa^{\al}\FG|^2 d\xi dx
\\&+C_\eta\int_{{\R}\times{\R}^3}(1+|\xi|)\left|\FM^{-1/2}_\sharp\widetilde{\FG}\right|^2 d\xi dx
+C_\eta\eps_0\sum\limits_{|\al| \leq1}\int_{\R\times\R^3}|\FM^{-1/2}\pa^\al \pa_{\xi_1}\widetilde{\FG}|^2 d\xi dx.
\end{split}
\end{equation*}
We insert the above estimations for $\CI_l$ $(1\leq l\leq 9)$ into \eqref{sum.d1}, choose $\eps_0$ and $\eta$ to be suitably small and then obtain
that
\begin{equation}\label{1s.vive}
\begin{split}
&\frac{d}{dt}\left((m_in_i+m_en_e)\widetilde{u}_1,\pa_x(\widetilde{n}_i+\widetilde{n}_e)\right)
+\frac{3}{2}\frac{d}{dt}\left(\pa_x\widetilde{n}_i,\frac{\mu_i(\theta)+\mu_e(\theta)}{n^r_i}\pa_x\widetilde{n}_i\right)
\\&\quad +\frac{3}{2}\frac{d}{dt}\left(\pa_x\widetilde{n}_e,\frac{\mu_i(\theta)+\mu_e(\theta)}{n^r_e}\pa_x\widetilde{n}_e\right)
+\la\left\|\pa_x(\widetilde{n}_i+\widetilde{n}_e)\right\|^2
\\
&\lesssim
C_\eta\left\|\pa_x\left[\widetilde{u}_1,\widetilde{\ta}\right]\right\|^2+
(\eps_0+\eta)\left\|\pa_x\left[\widetilde{n}_i,\widetilde{n}_e,
\phi,\pa_x\widetilde{u}_1,\pa_x\phi \right]\right\|^2
+\de_r^{1/2}(1+t)^{-3/2}\\
&\quad +(\eps_0+\eta)\left\|\pa_t\left[\widetilde{n}_i,\widetilde{n}_e,\widetilde{u},\widetilde{\ta}\right]\right\|^2
+(\eps_0+\eta)\left\|\pa^2_x\left[\widetilde{n}_i,\widetilde{n}_e\right]\right\|^2
\\&\quad +C_\eta(1+t)^{-2}\left\|\left[\widetilde{n}_i,\widetilde{n}_e,\widetilde{u}_1,\widetilde{\ta}\right]\right\|^2
+C_\eta\sum\limits_{1\leq|\al| \leq2}\int_{\R\times\R^3}(1+|\xi|)|\FM^{-1/2}\pa^{\al}\FG|^2 d\xi dx
\\&\quad +C_\eta\int_{{\R}\times{\R}^3}(1+|\xi|)\left|\FM^{-1/2}_\sharp\widetilde{\FG}\right|^2 d\xi dx
+C_\eta\eps_0\sum\limits_{|\al| \leq1}\int_{\R\times\R^3}|\FM^{-1/2}\pa^\al \pa_{\xi_1}\widetilde{\FG}|^2 d\xi dx.
\end{split}
\end{equation}

\medskip

\noindent{\it Step 2.} {\it Dissipation of $\pa_x\left[\phi ,\pa_x\phi \right]$:}
To do this, we shall make use of the equations
%$\pa_x[\widetilde{v}_i,\widetilde{v}_e]$, we now consider the quantity $\pa_x$
\eqref{cons.law.i}, \eqref{cons.law.e} and \eqref{Euler3}. Specifically,
we get from the summation of the second equation of \eqref{cons.law.i} multiplied by $q_i$ and the second equation of  \eqref{cons.law.e} multiplied by $q_e$ that
%$q_i\times \eqref{cons.law.i}_2$ and $q_e\times \eqref{cons.law.e}_2$
\begin{equation}\label{vi-vr}
\begin{split}
&(q_im_in_i+q_em_en_e)(\pa_t u_1+u_1\pa_xu_1)+\frac{2}{3}\pa_x \left((q_in_i+q_en_e)\ta\right)
+(q^2_in_i+q_e^2n_e)\pa_x\phi
\\&=-q_i\int_{{\R}^3}\psi_{3i}\pa_tG_i \,d\xi-q_i\int_{{\R}^3}\psi_{3i}\xi\pa_xG_i \,d\xi
+q_i\int_{{\R}^3}\psi_{3i}Q_i(\FF,\FF)\,d\xi
+q_iu_1\int_{{\R}^3}\psi_{1i}\xi_1\pa_xG_i \,d\xi
\\&\quad -q_e\int_{{\R}^3}\psi_{3e}\pa_tG_e \,d\xi-q_e\int_{{\R}^3}\psi_{3e}\xi\pa_xG_e \,d\xi
+q_e\int_{{\R}^3}\psi_{3e}Q_e(\FF,\FF)\,d\xi
+q_eu_1\int_{{\R}^3}\psi_{2e}\xi_1\pa_xG_e \,d\xi.
\end{split}
\end{equation}
%and similarly, the subtraction of $\eqref{cons.law.e}_2$ and $\eqref{Euler3}_3$ gives rise to
%\begin{equation}\label{ve-vr}
%\begin{split}
%m_en_e&(\pa_t \widetilde{u}_1+u_1\pa_x\widetilde{u}_1+\widetilde{u}_1\pa_xu_1^{r})+\frac{2}{3}\pa_x \left(n_e\ta-n_e^r\ta^r\right)
%-\frac{2}{3}\frac{\widetilde{n}_e}{n_e^r}\pa_x(n_e^r\ta^r)+q_en_e\pa_x\phi
%\\=&-\int_{{\R}^3}\psi_{3e}\pa_tG_e d\xi-\int_{{\R}^3}\psi_{3e}\xi\pa_xG_e d\xi
%+\int_{{\R}^3}\psi_{3e}Q_e(\FF,\FF)d\xi
%+u_1\int_{{\R}^3}\psi_{2e}\xi_1\pa_xG_e d\xi.
%\end{split}
%\end{equation}
Then the difference of $\eqref{vi-vr}$ and the third equation of \eqref{Euler3} yields
\begin{equation}\label{visbve}
\begin{split}
&(q_im_in_i+q_em_en_e)(\pa_t \widetilde{u}_1+u_1\pa_x\widetilde{u}_1+\widetilde{u}_1\pa_xu_1^{r})
+(q_im_i\widetilde{n}_i+q_em_e\widetilde{n}_e)(\pa_t u^r_1+u^r_1\pa_xu_1^{r})
\\&\quad +\frac{2}{3}\pa_x \left(q_in_i\ta+q_en_e\ta\right)
+(q^2_in_i+q^2_en_e)\pa_x\phi
\\&=-q_i\int_{{\R}^3}\psi_{3i}\pa_tG_i \,d\xi-q_i\int_{{\R}^3}\psi_{3i}\xi\pa_xG_i \,d\xi
+q_i\int_{{\R}^3}\psi_{3i}Q_i(\FF,\FF)\,d\xi
+q_iu_1\int_{{\R}^3}\psi_{1i}\xi_1\pa_xG_i \,d\xi
\\&\quad -q_e\int_{{\R}^3}\psi_{3e}\pa_tG_e \,d\xi-q_e\int_{{\R}^3}\psi_{3e}\xi\pa_xG_e \,d\xi
+q_e\int_{{\R}^3}\psi_{3e}Q_e(\FF,\FF)\,d\xi
+q_eu_1\int_{{\R}^3}\psi_{2e}\xi_1\pa_xG_e \,d\xi
\\&\quad +\frac{2\ta^r}{3}\frac{q_im_in_i^r+q_em_en_e^r}{m_in^r_i+m_en^r_e}\pa_x(n_i^r+n_e^r)
+\frac{2}{3}\pa_x\ta^r\frac{q_im_in_i^r+q_em_en_e^r}{m_in^r_i+m_en^r_e}(n_i^r+n_e^r).
\end{split}
\end{equation}
%where we have used the fact that $q_in_i^r+q_en_e^r=0$.
Taking the inner product of \eqref{visbve} with $\pa_x\phi $ with respect to $x$ over $\R$, one has
\begin{equation}\label{vi-ve.ip1}
\begin{split}
&\left((q_im_in_i+q_em_en_e)(\pa_t \widetilde{u}_1+u_1\pa_x\widetilde{u}_1+\widetilde{u}_1\pa_xu_1^{r}),\pa_x\phi \right)\\
&\quad +\left((q_im_i\widetilde{n}_i+q_em_e\widetilde{n}_e)(\pa_t u^r_1+u^r_1\pa_xu_1^{r}),\pa_x\phi \right)
+\left((q^2_in_i+q^2_en_e)\pa_x\phi ,\pa_x\phi \right)
\\&\quad
-\frac{2}{3}\left(\ta\left(q_in_i+q_en_e\right),\pa^2_x\phi \right)
%\\&\qquad-\left(\frac{2q_i}{3}\frac{\widetilde{n}_i}{n_i^r}\pa_x(n_i^r\ta^r)
%+\frac{2q_e}{3}\frac{\widetilde{n}_e}{n_e^r}\pa_x(n_e^r\ta^r),\pa_x\phi \right)
\\&=-\left(q_i\int_{{\R}^3}\psi_{3i}\pa_tG_i \,d\xi+q_e\int_{{\R}^3}\psi_{3e}\pa_tG_e \,d\xi,\pa_x\phi \right)\\
&\quad -\left(q_i\int_{{\R}^3}\psi_{3i}\xi\pa_xG_i \,d\xi+q_e\int_{{\R}^3}\psi_{3e}\xi\pa_xG_e \,d\xi,\pa_x\phi \right)
\\&\quad+\left(q_iu_1\int_{{\R}^3}\psi_{1i}\xi_1\pa_xG_i \,d\xi+q_eu_1\int_{{\R}^3}\psi_{2e}\xi_1\pa_xG_e \,d\xi,\pa_x\phi \right)
\\&\quad+\left(q_i\int_{{\R}^3}\psi_{3i}Q_i(\FF,\FF)\,d\xi+q_e\int_{{\R}^3}\psi_{3e}Q_e(\FF,\FF)\,d\xi,\pa_x\phi \right)
\\&\quad+\left(\frac{2\ta^r}{3}\frac{q_im_in_i^r+q_em_en_e^r}{m_in^r_i+m_en^r_e}\pa_x(n_i^r+n_e^r),\pa_x\phi \right)\\
&\quad+\left(\frac{2}{3}\pa_x\ta^r\frac{q_im_in_i^r+q_em_en_e^r}{m_in^r_i+m_en^r_e}(n_i^r+n_e^r),\pa_x\phi \right).
\end{split}
\end{equation}
Thanks to \eqref{cons.law.i}, \eqref{cons.law.e} and \eqref{Euler3}, we have
\begin{equation}\label{u1.cons}
\begin{split}
&(m_in_i+m_en_e)(\pa_t \widetilde{u}_1+u_1\pa_x\widetilde{u}_1+\widetilde{u}_1\pa_xu_1^{r})\\
&\quad+\pa_x P-\pa_xP^r+\left(1-\frac{m_in_i+m_en_e}{m_in_i^r+m_en_e^r}\right)\pa_xP^r
+\left(q_in_i+q_en_e\right)\pa_x\phi\\
&=
-\int_{{\R}^3}\xi_1\psi_{3}\cdot\pa_x\FG \,d\xi.
\end{split}
\end{equation}
In view of \eqref{u1.cons} and \eqref{aps}, we get from Cauchy-Schwarz inequality with $\eta>0$ that
\begin{equation*}%\label{u1.phi}
\begin{split}
&\left|\left((q_im_in_i+q_em_en_e)\pa_t \widetilde{u}_1,\pa_x\phi \right)\right|
\\&\lesssim (\eps_0+\de_r+\eta)\left\|\pa_x\phi \right\|^2
+(\eps_0+\de_r)\left\|q_in_i+q_en_e\right\|^2
+C_\eta(1+t)^{-2}\left\|\left[\widetilde{n}_i,\widetilde{n}_e,\widetilde{u}_1,\widetilde{\ta}\right]\right\|^2
\\&\quad+C_\eta\left\{\left\|\pa_x(\widetilde{n}_i+\widetilde{n}_e)\right\|^2
+\left\|\pa_x\left[\widetilde{u}_1,\widetilde{\ta}\right]\right\|^2\right\}\\
&\quad+C_\eta\sum\limits_{|\al|=1}\int_{\R\times\R^3}(1+|\xi|)|\FM^{-1/2}\pa^\al\FG|^2 d\xi dx.
\end{split}
\end{equation*}
Next, by Cauchy-Schwarz inequality with $\eta>0$ and Lemma \ref{cl.RwRe}, one has
\begin{equation*}%\label{u1.phi2}
\begin{split}
&\left|\left((q_im_in_i+q_em_en_e)(u_1\pa_x\widetilde{u}_1+\widetilde{u}_1\pa_xu_1^{r}),\pa_x\phi \right)\right|
\lesssim \eta\left\|\pa_x\phi \right\|^2
+C_\eta(1+t)^{-2}\left\|\widetilde{u}_1\right\|^2
+C_\eta
\left\|\pa_x\widetilde{u}_1\right\|^2,
\end{split}
\end{equation*}
and
\begin{equation*}%\label{p2.phi}
\begin{split}
\left|\left((q_im_i\widetilde{n}_i+q_em_e\widetilde{n}_e)(\pa_t u^r_1+u^r_1\pa_xu_1^{r}),\pa_x\phi \right)\right|
\lesssim
\eta\left\|\pa_x\phi \right\|^2
+C_\eta(1+t)^{-2}\left\|\left[\widetilde{n}_i,\widetilde{n}_e\right]\right\|^2.
\end{split}
\end{equation*}
%By applying \eqref{tphy} and Lemma \ref{cl.RwRe}, we obtain
%\begin{equation*}\label{p2.phi}
%\begin{split}
%-\frac{2}{3}&\left(\ta\left(q_in_i+q_en_e\right),\pa^2_x\phi \right)
%+\left((q^2_i\widetilde{n}_i+q^2_e\widetilde{n}_e)\pa_x\phi^r,\pa^2_x\phi \right)
%\\ \geq& \frac{2}{3}\left(\ta\pa^2_x\phi,\pa^2_x\phi \right)-\eta\left\|\pa^2_x\phi\right\|^2
%-C_\eta(1+t)^{-2}\left\|\left[\widetilde{n}_i,\widetilde{n}_e\right]\right\|^2
%-C_\eta \de^{1/2}_r(1+t)^{-3/2}.
%\end{split}
%\end{equation*}
%Similarly, the second and third line of \eqref{vi-ve.ip1} are bounded up to a generic constant by
%$$
%(\eps_0+\eta)\left\|\pa_x\left[\phi ,\pa_x\phi \right]\right\|^2
%+(\eps_0+\eta)\left\|\widetilde{v}_i-\widetilde{v}_e\right\|^2+C_\eta(1+t)^{-2}\left\|[\widetilde{v}_i,\widetilde{v}_e]\right\|^2.
%$$
For the terms on the right-hand side of \eqref{vi-ve.ip1},
we only present the exact computations of those terms involving $Q_\CA(\FF,\FF)$. %In order to avoid the $t-$derivative, we employ \eqref{micBE}
Note that
$$
Q_{\CA}(\FF,\FF)=(\FL_{\FM}\FG)_{\CA}+Q_{\CA}(\FG,\FG)
=(\FL_{\FM}\widetilde{\FG})_{\CA}+(\FL_{\FM}\overline{\FG})_{\CA}+Q_{\CA}(\FG,\FG).
$$
With this we write
\begin{equation}\label{e.d1}
\begin{split}
&\left(q_i\int_{{\R}^3}\psi_{3i}Q_i(\FF,\FF)\,d\xi+q_e\int_{{\R}^3}\psi_{3e}Q_e(\FF,\FF)\,d\xi,\pa_x\phi \right)
\\ &=\left(q_i\int_{{\R}^3}\psi_{3i}(\FL_{\FM}\widetilde{\FG})_{i}\,d\xi
+q_e\int_{{\R}^3}\psi_{3e}(\FL_{\FM}\widetilde{\FG})_{e}\,d\xi,\pa_x\phi \right)
\\&\quad+\left(q_i\int_{{\R}^3}\psi_{3i}(\FL_{\FM}\overline{\FG})_{i}\,d\xi
+q_e\int_{{\R}^3}\psi_{3e}(\FL_{\FM}\overline{\FG})_{e}\,d\xi,\pa_x\phi \right)
\\&\quad+\left(q_i\int_{{\R}^3}\psi_{3i}Q_{i}(\FG,\FG)\,d\xi+q_e\int_{{\R}^3}\psi_{3e}Q_{e}(\FG,\FG)\,d\xi,\pa_x\phi \right).
\end{split}
\end{equation}
Substituting  \eqref{def.ng} into the second term on the right-hand side of \eqref{e.d1}, we find
\begin{equation*}%\label{e.d2}
\begin{split}
&\left(q_i\int_{{\R}^3}\psi_{3i}(\FL_{\FM}\overline{\FG})_{i}\,d\xi
+q_e\int_{{\R}^3}\psi_{3e}(\FL_{\FM}\overline{\FG})_{e}\,d\xi,\pa_x\phi \right)
\\&\quad+\left(\frac{2\ta^r}{3}\frac{q_im_in_i^r+q_em_en_e^r}{m_in^r_i+m_en^r_e}\pa_x(n_i^r+n_e^r),\pa_x\phi \right)
+\left(\frac{2}{3}\pa_x\ta^r\frac{q_im_in_i^r+q_em_en_e^r}{m_in^r_i+m_en^r_e}(n_i^r+n_e^r),\pa_x\phi \right)
\\&=
\left(\left\{-\frac{2\ta}{3}\frac{q_im_in_i+q_em_en_e}{m_in_i+m_en_e}
+\frac{2\ta^r}{3}\frac{q_im_in_i^r+q_em_en_e^r}{m_in^r_i+m_en^r_e}\right\}\pa_x(n_i^r+n_e^r),\pa_x\phi \right)
\\&\quad+\left(-\frac{2}{3}\pa_x\ta^r\frac{q_im_in_i+q_em_en_e}{m_in_i+m_en_e}(n_i+n_e)
+\frac{2}{3}\pa_x\ta^r\frac{q_im_in_i^r+q_em_en_e^r}{m_in^r_i+m_en^r_e}(n^r_i+n_e^r),\pa_x\phi \right),
\end{split}
\end{equation*}
whose absolute value can be bounded by
$$
\eta\left\|\pa_x\phi \right\|^2
+C_\eta(1+t)^{-2}\left\|\left[\widetilde{n}_i,\widetilde{n}_e,\widetilde{\ta}\right]\right\|^2.
$$
%The quadratic term $e\left(\frac{Z+1}{v}\pa_x\phi^r,\pa_x\phi \right)$ together with \eqref{im.phi} gives
%$$
%-e\left(\frac{Z+1}{v}\pa_x\phi ,\pa_x\phi \right)
%$$
%which in non-positive.
The remaining terms on the right-hand side of \eqref{vi-ve.ip1} and \eqref{e.d1} are bounded by
\begin{equation*}
\begin{split}
(\eps_0+\eta)&\left\|\pa_x\phi \right\|^2
+\eps_0\left\|q_i\widetilde{n}_i+q_e\widetilde{n}_e\right\|^2
+C_\eta\de_r(1+t)^{-2}+C_\eta\int_{{\R}\times{\R}^3}(1+|\xi|)\left|\FM^{-1/2}_\sharp\widetilde{\FG}\right|^2 d\xi dx
\\&+C_\eta\sum\limits_{|\al|=1}\int_{\R\times\R^3}(1+|\xi|)|\FM^{-1/2}\pa^\al\FG|^2 d\xi dx
+\eps_0\sum\limits_{|\al| \leq1}\int_{\R\times\R^3}|\FM^{-1/2}\pa^\al \pa_{\xi_1}\widetilde{\FG}|^2 d\xi dx,
\end{split}
\end{equation*}
according to Cauchy-Schwarz inequality and the estimate \eqref{db.G}.
Finally, substituting the above estimates into \eqref{vi-ve.ip1} and applying \eqref{tphy}, we arrive at
\begin{equation}\label{dis.phi}
\begin{split}
&\la\left(\pa_x\phi ,\pa_x\phi \right)
+\la\left(\pa^2_x\phi ,\pa^2_x\phi \right)
\\& \lesssim
(\eps_0+\eta)\left\|q_i\widetilde{n}_i+q_e\widetilde{n}_e\right\|^2
+C_\eta(1+t)^{-2}\left\|\left[\widetilde{n}_i,\widetilde{n}_e,\widetilde{\ta}\right]\right\|^2
+\de_r^{1/2}(1+t)^{-3/2}
\\&\quad+C_\eta\left\{\left\|\pa_x(\widetilde{n}_i+\widetilde{n}_e)\right\|^2
+\left\|\pa_x\left[\widetilde{u}_1,\widetilde{\ta}\right]\right\|^2\right\}
+C_\eta\sum\limits_{|\al|=1}\int_{\R\times\R^3}(1+|\xi|)|\FM^{-1/2}\pa^\al\FG|^2 d\xi dx
\\&\quad+C_\eta\int_{{\R}\times{\R}^3}(1+|\xi|)\left|\FM^{-1/2}_\sharp\widetilde{\FG}\right|^2 d\xi dx.
\end{split}
\end{equation}

\medskip

\noindent{\it Step 3.} {\it Dissipation of $\pa_x\left(q_i\widetilde{n}_i+q_e\widetilde{n}_e\right)$:}
%We next deduce the dissipation of $\pa_x\left(q_i\widetilde{n}_i+q_e\widetilde{n}_e\right)$, for
To deduce this, we take the inner product of \eqref{visbve} with $\pa_x\left(q_i\widetilde{n}_i+q_e\widetilde{n}_e\right)$ with respect to $x$ over $\R$ to obtain
\begin{equation}\label{vi-ve.ip2}
\begin{split}
&((q_im_in_i+q_em_en_e)(\pa_t \widetilde{u}_1+u_1\pa_x\widetilde{u}_1+\widetilde{u}_1\pa_xu_1^{r}),\pa_x \left(q_i\widetilde{n}_i+q_e\widetilde{n}_e\right))\\
&\quad+\left((q^2_in_i+q^2_en_e)\pa_x\phi ,\pa_x \left(q_i\widetilde{n}_i+q_e\widetilde{n}_e\right)\right)
\\&\quad+\frac{2}{3}\left(\ta\pa_x \left(q_i\widetilde{n}_i+q_e\widetilde{n}_e\right),\pa_x \left(q_i\widetilde{n}_i+q_e\widetilde{n}_e\right)\right)
+\frac{2}{3}\left(\pa_x \ta \left(q_i\widetilde{n}_i+q_e\widetilde{n}_e\right),\pa_x \left(q_i\widetilde{n}_i+q_e\widetilde{n}_e\right)\right)
\\&\quad+\left(q_im_i\widetilde{n}_i+q_em_e\widetilde{n}_e)(\pa_t u^r_1+u^r_1\pa_xu_1^{r}),\pa_x \left(q_i\widetilde{n}_i+q_e\widetilde{n}_e\right)\right)
\\&=-\left(q_i\int_{{\R}^3}\psi_{3i}\pa_tG_i \,d\xi+q_e\int_{{\R}^3}\psi_{3e}\pa_tG_e \,d\xi,\pa_x \left(q_i\widetilde{n}_i+q_e\widetilde{n}_e\right)\right)
\\&\quad-\left(q_i\int_{{\R}^3}\psi_{3i}\xi_1\pa_xG_i \,d\xi+q_e\int_{{\R}^3}\psi_{3e}\xi_1\pa_xG_e \,d\xi,\pa_x \left(q_i\widetilde{n}_i+q_e\widetilde{n}_e\right)\right)
\\&\quad+\left(q_iu_1\int_{{\R}^3}\psi_{1i}\xi_1\pa_xG_i \,d\xi+q_eu_1\int_{{\R}^3}\psi_{2e}\xi_1\pa_xG_e \,d\xi,\pa_x\left(q_i\widetilde{n}_i+q_e\widetilde{n}_e\right) \right)
\\&\quad+\left(q_i\int_{{\R}^3}\psi_{3i}Q_i(\FF,\FF)\,d\xi+q_e\int_{{\R}^3}\psi_{3e}Q_e(\FF,\FF)\,d\xi,\pa_x\left(q_i\widetilde{n}_i+q_e\widetilde{n}_e\right) \right)
\\&\quad+\left(\frac{2\ta^r}{3}\frac{q_im_in_i^r+q_em_en_e^r}{m_in^r_i+m_en^r_e}\pa_x(n_i^r+n_e^r),\pa_x\left(q_i\widetilde{n}_i+q_e\widetilde{n}_e\right) \right)
\\&\quad+\left(\frac{2}{3}\pa_x\ta^r\frac{q_im_in_i^r+q_em_en_e^r}{m_in^r_i+m_en^r_e}(n_i^r+n_e^r),\pa_x\left(q_i\widetilde{n}_i+q_e\widetilde{n}_e\right) \right).
\end{split}
\end{equation}
With \eqref{vi-ve.ip2} in hand,  by performing the similar calculations as for obtaining \eqref{dis.phi}, one has
\begin{eqnarray}
&&\la\left(\pa_x \left(q_i\widetilde{n}_i+q_e\widetilde{n}_e\right),\pa_x \left(q_i\widetilde{n}_i+q_e\widetilde{n}_e\right)\right)
+\la\left(\left(q_i\widetilde{n}_i+q_e\widetilde{n}_e\right),\left(q_i\widetilde{n}_i+q_e\widetilde{n}_e\right)\right)
\notag\\
&& \lesssim
(\eps_0+\eta)\left\|\pa_x\left[\widetilde{n}_i,\widetilde{n}_e,\phi \right]\right\|^2
+C_\eta(1+t)^{-2}\left\|\left[\widetilde{n}_i,\widetilde{n}_e,\widetilde{\ta}\right]\right\|^2
+C_\eta\left\{\left\|\pa_x(\widetilde{n}_i+\widetilde{n}_e)\right\|^2
+\left\|\pa_x\left[\widetilde{u}_1,\widetilde{\ta}\right]\right\|^2\right\}
\notag\\
&&\quad+C_\eta\sum\limits_{|\al|=1}\int_{\R\times\R^3}(1+|\xi|)|\FM^{-1/2}\pa^\al\FG|^2 d\xi dx
+C_\eta\int_{{\R}\times{\R}^3}(1+|\xi|)\left|\FM^{-1/2}_\sharp\widetilde{\FG}\right|^2 d\xi dx\notag\\
&&\quad +\de_r^{1/2}(1+t)^{-3/2}.\label{dis.vi-ve}
\end{eqnarray}
%\begin{equation}\label{dis.vi-ve}
%\begin{split}
%&\la\left(\pa_x \left(q_i\widetilde{n}_i+q_e\widetilde{n}_e\right),\pa_x \left(q_i\widetilde{n}_i+q_e\widetilde{n}_e\right)\right)
%+\la\left(\left(q_i\widetilde{n}_i+q_e\widetilde{n}_e\right),\left(q_i\widetilde{n}_i+q_e\widetilde{n}_e\right)\right)
%\\& \lesssim
%(\eps_0+\eta)\left\|\pa_x\left[\widetilde{n}_i,\widetilde{n}_e,\phi \right]\right\|^2
%+C_\eta(1+t)^{-2}\left\|\left[\widetilde{n}_i,\widetilde{n}_e,\widetilde{\ta}\right]\right\|^2
%+C_\eta\left\{\left\|\pa_x(\widetilde{n}_i+\widetilde{n}_e)\right\|^2
%+\left\|\pa_x\left[\widetilde{u}_1,\widetilde{\ta}\right]\right\|^2\right\}
%\\&\quad+C_\eta\sum\limits_{|\al|=1}\int_{\R\times\R^3}(1+|\xi|)|\FM^{-1/2}\pa^\al\FG|^2 d\xi dx
%+C_\eta\int_{{\R}\times{\R}^3}(1+|\xi|)\left|\FM^{-1/2}_\sharp\widetilde{\FG}\right|^2 d\xi dx\\
%&\quad +\de_r^{1/2}(1+t)^{-3/2}.
%\end{split}
%\end{equation}
We now get from \eqref{1s.vive}, \eqref{dis.phi} and \eqref{dis.vi-ve} that
\begin{eqnarray}
&&\frac{d}{dt}\left((m_in_i+m_en_e)\widetilde{u}_1,\pa_x(\widetilde{n}_i+\widetilde{n}_e)\right)
+\frac{3}{2}\frac{d}{dt}\left(\pa_x\widetilde{n}_i,\frac{\mu_i(\theta)+\mu_e(\theta)}{n^r_i}\pa_x\widetilde{n}_i\right)
\notag\\
&&\quad+\frac{3}{2}\frac{d}{dt}\left(\pa_x\widetilde{n}_e,\frac{\mu_i(\theta)+\mu_e(\theta)}{n^r_e}\pa_x\widetilde{n}_e\right)+\la\left\|\pa_x\left[\widetilde{n}_i,\widetilde{n}_e,\phi ,\pa_x\phi \right]\right\|^2
+\la\left\|q_i\widetilde{n}_i+q_e\widetilde{n}_e\right\|^2
\notag\\
&&\lesssim
C_\eta\left\|\pa_x\left[\widetilde{u}_1,\widetilde{\ta}\right]\right\|^2+(\eps_0+\eta)\|\pa_x^2\widetilde{u}_1\|^2
+\de_r^{1/2}(1+t)^{-3/2}
+C_\eta(1+t)^{-2}\left\|\left[\widetilde{n}_i,\widetilde{n}_e,\widetilde{u}_1,\widetilde{\ta}\right]\right\|^2
\notag\\
&&\quad+C_\eta\sum\limits_{1\leq\al \leq2}\int_{\R\times\R^3}(1+|\xi|)|\FM^{-1/2}\pa^{\al}\FG|^2 d\xi dx
+C_\eta\int_{{\R}\times{\R}^3}(1+|\xi|)\left|\FM^{-1/2}_\sharp\widetilde{\FG}\right|^2 d\xi dx
\notag\\
&&\quad
+C_\eta\eps_0\sum\limits_{|\al| \leq1}\int_{\R\times\R^3}|\FM^{-1/2}\pa^\al \pa_{\xi_1}\widetilde{\FG}|^2 d\xi dx.\label{1st.vp.diss}
\end{eqnarray}
%\begin{equation}\label{1st.vp.diss}
%\begin{split}
%\frac{d}{dt}&\left((m_in_i+m_en_e)\widetilde{u}_1,\pa_x(\widetilde{n}_i+\widetilde{n}_e)\right)
%+\frac{3}{2}\frac{d}{dt}\left(\pa_x\widetilde{n}_i,\frac{\mu_i(\theta)+\mu_e(\theta)}{n^r_i}\pa_x\widetilde{n}_i\right)
%+\frac{3}{2}\frac{d}{dt}\left(\pa_x\widetilde{n}_e,\frac{\mu_i(\theta)+\mu_e(\theta)}{n^r_e}\pa_x\widetilde{n}_e\right)
%\\&+\la\left\|\pa_x\left[\widetilde{n}_i,\widetilde{n}_e,\phi ,\pa_x\phi \right]\right\|^2
%+\la\left\|q_i\widetilde{n}_i+q_e\widetilde{n}_e\right\|^2
%\\
%\lesssim&
%C_\eta\left\|\pa_x\left[\widetilde{u}_1,\widetilde{\ta}\right]\right\|^2+(\eps_0+\eta)\|\pa_x^2\widetilde{u}_1\|^2
%+\de_r^{1/2}(1+t)^{-3/2}
%+C_\eta(1+t)^{-2}\left\|\left[\widetilde{n}_i,\widetilde{n}_e,\widetilde{u}_1,\widetilde{\ta}\right]\right\|^2
%\\&+C_\eta\sum\limits_{1\leq\al \leq2}\int_{\R\times\R^3}(1+|\xi|)|\FM^{-1/2}\pa^{\al}\FG|^2 d\xi dx
%+C_\eta\int_{{\R}\times{\R}^3}(1+|\xi|)\left|\FM^{-1/2}_\sharp\widetilde{\FG}\right|^2 d\xi dx
%\\&
%+C_\eta\eps_0\sum\limits_{|\al| \leq1}\int_{\R\times\R^3}|\FM^{-1/2}\pa^\al \pa_{\xi_1}\widetilde{\FG}|^2 d\xi dx.
%\end{split}
%\end{equation}

Having obtained \eqref{1st.vp.diss}, one can see that $\pa_t\left[\widetilde{n}_i,\widetilde{n}_e,\widetilde{u},\widetilde{\ta}\right]$
also enjoys the similar estimate. In fact, we get from \eqref{cons.law.i}, \eqref{cons.law.e} and \eqref{Euler3} that
\begin{eqnarray}\label{pb.con.}
\left\{\begin{array}{rlll}
\begin{split}
%&\pa_t\widetilde{v}-\pa_x\widetilde{u}_1=0,\label{tv}\\
&\pa_t\widetilde{n}_i+\pa_x(n_iu_1-n_i^ru_1^r)=-\int_{{\R}^3}\xi_1\pa_xG_i \,d\xi,\\
&\pa_t\widetilde{n}_e+\pa_x(n_eu_1-n_e^ru_1^r)=-\int_{{\R}^3}\xi_1\pa_xG_e \,d\xi,\\
&(m_in_i+m_en_e)\pa_t \widetilde{u}_1+
(m_in_i+m_en_e)(u_1\pa_x\widetilde{u}_1+\widetilde{u}_1\pa_xu_1^{r})+\pa_x P-\pa_xP^r+\left(q_i\widetilde{n}_i+q_e\widetilde{n}_e\right)\pa_x\phi
\\&\qquad
+\left(1-\frac{m_in_i+m_en_e}{m_in_i^r+m_en_e^r}\right)\pa_xP^r
=
-\int_{{\R}^3}\xi_1\psi_{3}\cdot\pa_x\FG d\xi,\\
&(m_in_i+m_en_e)\pa_t \widetilde{u}_j+(m_in_i+m_en_e)u_1\pa_x\widetilde{u}_j=-\int_{{\R}^3}\xi_1\psi_{j+2}\pa_x\FG d\xi,\
\ j=2,\ 3,\\
&(n_i+n_e)\pa_t\widetilde{\theta}+(n_i+n_e)(u_1\pa_x\widetilde{\ta}+\widetilde{u}_1\pa_x\ta^r)+P\pa_x u_1-P^r\pa_x u^r_1
+\left(1-\frac{n_i+n_e}{n_i^r+n_e^r}\right)P^r\pa_x u^r
\\&\qquad=-\int_{{\R}^3}\xi_1\left(\psi_{6}-\sum\limits_{j=1}^3u_j\cdot\psi_{j+2}\right)\cdot\pa_x\FG \,d\xi
+\ta\int_{{\R}^3}[\xi_1,\xi_1]^{\rm T}\cdot\pa_x\FG \,d\xi
\\&\qquad\quad+\pa_x\phi\int_{{\R}^3}\frac{|\xi|^2}{2}\left[q_i,q_e\right]^{\rm T}\cdot\pa_{\xi_1}\FG \,d\xi.
\end{split}
\end{array}\right.
\end{eqnarray}
%which yields
This yields that
\begin{equation}\label{sum.d13}
\begin{split}
&\left\|\pa_t\left[\widetilde{n}_i,\widetilde{n}_e,\widetilde{u},\widetilde{\ta}\right]\right\|^2\\
&\lesssim\left\|\pa_x\left[\widetilde{n}_i,\widetilde{n}_e,\widetilde{u},\widetilde{\ta}, \phi \right]\right\|^2
+(1+t)^{-2}\left\|\left[\widetilde{n}_i,\widetilde{n}_e,\widetilde{u},\widetilde{\ta}\right]\right\|^2
+\de_r(1+t)^{-2}
\\&\quad +C_\eta\int_{\R\times\R^3}(1+|\xi|)|\FM^{-1/2}\pa_x\FG|^2 d\xi dx
+C_\eta\int_{{\R}\times{\R}^3}(1+|\xi|)\left|\FM^{-1/2}_\sharp\widetilde{\FG}\right|^2 d\xi dx
\\&\quad +C_\eta\eps_0\sum\limits_{|\al| \leq1}\int_{\R\times\R^3}|\FM^{-1/2}\pa^\al \pa_{\xi_1}\widetilde{\FG}|^2 d\xi dx.
\end{split}
\end{equation}
Letting $1\gg\ka_1\gg\ka_2>0$ and taking the summation
of \eqref{zero.eng}, $\eqref{1st.vp.diss}\times\ka_1$ and $\eqref{sum.d13}\times\ka_2$, one has that for suitably small constants $\eps_0>0$, $\de_r>0$ and $\eta>0$,
% for sufficiently small  that
\begin{eqnarray}
&&\frac{d}{dt}\widetilde{\eta}
+\ka_1\frac{d}{dt}\left((m_in_i+m_en_e)\widetilde{u}_1,\pa_x(\widetilde{n}_i+\widetilde{n}_e)\right)
\notag\\
&&\quad+\ka_1\frac{3}{2}\frac{d}{dt}\left\{\left(\pa_x\widetilde{n}_i,\frac{\mu_i(\theta)+\mu_e(\theta)}{n^r_i}\pa_x\widetilde{n}_i\right)
+\left(\pa_x\widetilde{n}_e,\frac{\mu_i(\theta)+\mu_e(\theta)}{n^r_e}\pa_x\widetilde{n}_e\right)\right\}
\notag\\
&&\quad+\la\sum\limits_{|\al| =1}\left\|\pa^{\al}\left[\widetilde{n}_i,\widetilde{n}_e,\widetilde{u},\widetilde{\ta}\right]\right\|^2
+\la\left\|q_i\widetilde{n}_i+q_e\widetilde{n}_e\right\|^2+\la\left\|\pa_x\left[\phi ,\pa_x\phi \right]\right\|^2
+\la\int_{\R}\pa_xu_1^r\left[\widetilde{n}_i,\widetilde{n}_e,\widetilde{u}_1,\widetilde{\ta}\right]^2dx
\notag\\
&&\leq  (\eps_0+\eta)\left\|\pa_x^2\widetilde{u}_1\right\|^2
+C_\eta(1+t)^{-2}\left\|\left[\widetilde{n}_i,\widetilde{n}_e,\widetilde{u},\widetilde{\ta}\right]\right\|^2
+\de^{1/6}_r(1+t)^{-7/6}
\notag\\
&&\quad+C_\eta\sum\limits_{1\leq|\al| \leq2}\int_{\R\times\R^3}(1+|\xi|)|\FM^{-1/2}\pa^{\al}\FG|^2 d\xi dx
+C_\eta\int_{{\R}\times{\R}^3}(1+|\xi|)\left|\FM^{-1/2}_\sharp\widetilde{\FG}\right|^2 d\xi dx
\notag\\
&&\quad+C_\eta\eps_0\int_{\R\times\R^3}|\FM^{-1/2}\pa_{\xi_1}\widetilde{\FG}|^2 d\xi dx.
\label{1st.diss}
\end{eqnarray}
%\begin{equation}\label{1st.diss}
%\begin{split}
%&\frac{d}{dt}\widetilde{\eta}
%+\ka_1\frac{d}{dt}\left((m_in_i+m_en_e)\widetilde{u}_1,\pa_x(\widetilde{n}_i+\widetilde{n}_e)\right)
%\\&\quad+\ka_1\frac{3}{2}\frac{d}{dt}\left\{\left(\pa_x\widetilde{n}_i,\frac{\mu_i(\theta)+\mu_e(\theta)}{n^r_i}\pa_x\widetilde{n}_i\right)
%+\left(\pa_x\widetilde{n}_e,\frac{\mu_i(\theta)+\mu_e(\theta)}{n^r_e}\pa_x\widetilde{n}_e\right)\right\}
%\\&\quad+\la\sum\limits_{|\al| =1}\left\|\pa^{\al}\left[\widetilde{n}_i,\widetilde{n}_e,\widetilde{u},\widetilde{\ta}\right]\right\|^2
%+\la\left\|q_i\widetilde{n}_i+q_e\widetilde{n}_e\right\|^2+\la\left\|\pa_x\left[\phi ,\pa_x\phi \right]\right\|^2
%+\la\int_{\R}\pa_xu_1^r\left[\widetilde{n}_i,\widetilde{n}_e,\widetilde{u}_1,\widetilde{\ta}\right]^2dx
%\\
%&\leq  (\eps_0+\eta)\left\|\pa_x^2\widetilde{u}_1\right\|^2
%+C_\eta(1+t)^{-2}\left\|\left[\widetilde{n}_i,\widetilde{n}_e,\widetilde{u},\widetilde{\ta}\right]\right\|^2
%+\de^{1/6}_r(1+t)^{-7/6}
%\\&\quad+C_\eta\sum\limits_{1\leq|\al| \leq2}\int_{\R\times\R^3}(1+|\xi|)|\FM^{-1/2}\pa^{\al}\FG|^2 d\xi dx
%+C_\eta\int_{{\R}\times{\R}^3}(1+|\xi|)\left|\FM^{-1/2}_\sharp\widetilde{\FG}\right|^2 d\xi dx
%\\&\quad+C_\eta\eps_0\int_{\R\times\R^3}|\FM^{-1/2}\pa_{\xi_1}\widetilde{\FG}|^2 d\xi dx.
%\end{split}
%\end{equation}

\subsection{Estimate on first-order energy}
%\noindent{\bf Step 3.} {\it The first order energy estimates.}

Let $|\al|=1$. Taking the inner product of $\pa^\al \eqref{tvi}$, $\pa^\al \eqref{tve}$, $\pa^\al \eqref{tu1}$, $\pa^\al \eqref{tuj}$ and $\pa^\al \eqref{tta}$ with $\pa^\al \widetilde{n}_i$, $\pa^\al \widetilde{n}_e$, $\pa^\al \widetilde{u}_1$,
$\pa^\al \widetilde{u}_j$ $(2\leq j\leq3)$ and $\pa^\al \widetilde{\ta}$, respectively, and then taking the summation of the resulting equations,
one has
\begin{equation}\label{p2u1}
\begin{split}
&\frac{1}{2}\frac{d}{dt}\left\{\|\pa^\al \widetilde{n}_i\|^2
+\|\pa^\al \widetilde{n}_e\|^2
+\sum\limits_{j=1}^3\left\|\sqrt{m_in_i+m_en_e}\pa^\al \widetilde{u}_j\right\|^2
+\left\|\sqrt{n_i+n_e}\pa^\al \widetilde{\ta}\right\|^2\right\}\\
&\quad+3\left((\mu_i(\theta)+\mu_e(\theta))\pa_x \pa^\al \widetilde{u}_1,\pa_x\pa^\al \widetilde{u}_1\right)
+\sum\limits_{j=2}^3\left((\mu_i(\theta)+\mu_e(\theta))\pa_x \pa^\al \widetilde{u}_j,\pa_x\pa^\al \widetilde{u}_j\right)
\\&\quad+\left((\kappa_i(\theta)+\kappa_i(\theta))\pa_x \pa^\al \widetilde{\theta},\pa_x\pa^\al \widetilde{\ta}\right)=\sum\limits_{l=1}^{12}J_l,
\end{split}
\end{equation}
where the right-hand terms are given by
\begin{eqnarray*}
\left\{\begin{array}{rll}
\begin{split}
&J_1=-\left(\pa^\al\pa_x(n_iu_1-n_i^ru_1^r), \pa^\al \widetilde{n}_i\right)
+\left(\pa^\al\pa_x(n_eu_1-n_e^ru_1^r), \pa^\al \widetilde{n}_e\right),\\
&J_2=\frac{1}{2}\sum\limits_{j=1}^3\left(\pa_t(m_in_i+m_en_e)\pa^\al\widetilde{u}_j, \pa^\al \widetilde{u}_j\right)
+\frac{1}{2}\left(\pa_t(n_i+n_e)\pa^\al\widetilde{\theta},\pa^\al \widetilde{\theta}\right),\\
&J_3=-\left(\pa^\al\left\{(m_in_i+m_en_e)(u_1\pa_x\widetilde{u}_1+\widetilde{u}_1\pa_xu_1^{r})\right\}, \pa^\al \widetilde{u}_1\right)
-\sum\limits_{j=2}^3\left(\pa^\al\left\{(m_in_i+m_en_e)u_1\pa_x\widetilde{u}_j\right\}, \pa^\al \widetilde{u}_j\right),\\
&J_4=-\left(\pa^\al \left(\left(q_i\widetilde{n}_i+q_e\widetilde{n}_e\right)\pa_x\phi
\right), \pa^\al \widetilde{u}_1\right),\\
&J_5=\left(\pa^\al  (P-P^r),\pa^\al \pa_x\widetilde{u}_1\right)
-\left(\pa^\al \left\{\left(1-\frac{m_in_i+m_en_e}{m_in_i^r+m_en_e^r}\right)\pa_xP^r\right\},\pa^\al \widetilde{u}_1\right),\\
&J_{6}=-\left(\pa^\al \left\{(n_i+n_e)(u_1\pa_x\widetilde{\ta}+\widetilde{u}_1\pa_x\ta^r)\right\},\pa^\al \widetilde{\ta}\right)
-\left(\pa^\al \left(P\pa_x u_1-P^r\pa_x u^r_1\right),\pa^\al \widetilde{\ta}\right),
\\&\qquad-\left(\pa^\al \left\{\left(1-\frac{n_i+n_e}{n_i^r+n_e^r}\right)P^r\pa_x u^r\right\},\pa^\al \widetilde{\ta}\right),
%\\
%&J_{7}=-3\left((\mu_i(\theta)+\mu_e(\theta))\pa^\al \pa_x u^r_1,\pa^\al \pa_x\widetilde{u}_1\right)
%-\left((\ka_i(\theta)+\ka_e(\theta))\pa^\al \pa_x \theta^r,\pa^\al \pa_x\widetilde{\theta}\right),\\
%&J_{8}=-3\left(\pa^\al(\mu_i(\theta)+\mu_e(\theta))\pa_x u^r_1,\pa^\al \pa_x\widetilde{u}_1\right)
%-\sum\limits_{j=2}^3\left(\pa^\al (\mu_i(\theta)+\mu_e(\theta))\pa_x u_j,\pa^\al \pa_x\widetilde{u}_j\right)
%\\&\qquad-\left(\pa^\al (\ka_i(\theta)+\ka_e(\theta))\pa_x \theta,\pa^\al \pa_x\widetilde{\theta}\right),\\
%&J_{9}=3\left(\pa^\al \left((\mu_i(\theta)+\mu_e(\theta))(\pa_x u_1)^2\right),\pa^\al \widetilde{\ta}\right)+
%\sum\limits_{j=2}^3\left(\pa^\al \left((\mu_i(\theta)+\mu_e(\theta))(\pa_x \widetilde{u}_j)^2\right),\pa^\al \widetilde{\ta}\right),
\end{split}
\end{array}\right.
\end{eqnarray*}
%and
\begin{eqnarray*}
\left\{\begin{array}{rll}
\begin{split}
%&J_1=-\left(\pa^\al\pa_x(n_iu_1-n_i^ru_1^r), \pa^\al \widetilde{n}_i\right)
%+\left(\pa^\al\pa_x(n_eu_1-n_e^ru_1^r), \pa^\al \widetilde{n}_e\right),\\
%&J_2=\frac{1}{2}\sum\limits_{j=1}^3\left(\pa_t(m_in_i+m_en_e)\pa^\al\widetilde{u}_j, \pa^\al \widetilde{u}_j\right)
%+\frac{1}{2}\left(\pa_t(n_i+n_e)\pa^\al\widetilde{\theta},\pa^\al \widetilde{\theta}\right),\\
%&J_3=-\left(\pa^\al\left\{(m_in_i+m_en_e)(u_1\pa_x\widetilde{u}_1+\widetilde{u}_1\pa_xu_1^{r})\right\}, \pa^\al \widetilde{u}_1\right)
%-\sum\limits_{j=2}^3\left(\pa^\al\left\{(m_in_i+m_en_e)u_1\pa_x\widetilde{u}_j\right\}, \pa^\al \widetilde{u}_j\right),\\
%&J_4=-\left(\pa^\al \left(\left(q_i\widetilde{n}_i+q_e\widetilde{n}_e\right)\pa_x\phi
%\right), \pa^\al \widetilde{u}_1\right),\\
%&J_5=\left(\pa^\al  (P-P^r),\pa^\al \pa_x\widetilde{u}_1\right)
%-\left(\pa^\al \left\{\left(1-\frac{m_in_i+m_en_e}{m_in_i^r+m_en_e^r}\right)\pa_xP^r\right\},\pa^\al \widetilde{u}_1\right),\\
%&J_{6}=-\left(\pa^\al \left\{(n_i+n_e)(u_1\pa_x\widetilde{\ta}+\widetilde{u}_1\pa_x\ta^r)\right\},\pa^\al \widetilde{\ta}\right)
%-\left(\pa^\al \left(P\pa_x u_1-P^r\pa_x u^r_1\right),\pa^\al \widetilde{\ta}\right),
%\\&\qquad-\left(\pa^\al \left\{\left(1-\frac{n_i+n_e}{n_i^r+n_e^r}\right)P^r\pa_x u^r\right\},\pa^\al \widetilde{\ta}\right),\\
&J_{7}=-3\left((\mu_i(\theta)+\mu_e(\theta))\pa^\al \pa_x u^r_1,\pa^\al \pa_x\widetilde{u}_1\right)
-\left((\ka_i(\theta)+\ka_e(\theta))\pa^\al \pa_x \theta^r,\pa^\al \pa_x\widetilde{\theta}\right),\\
&J_{8}=-3\left(\pa^\al(\mu_i(\theta)+\mu_e(\theta))\pa_x u^r_1,\pa^\al \pa_x\widetilde{u}_1\right)
-\sum\limits_{j=2}^3\left(\pa^\al (\mu_i(\theta)+\mu_e(\theta))\pa_x u_j,\pa^\al \pa_x\widetilde{u}_j\right)
\\&\qquad-\left(\pa^\al (\ka_i(\theta)+\ka_e(\theta))\pa_x \theta,\pa^\al \pa_x\widetilde{\theta}\right),\\
&J_{9}=3\left(\pa^\al \left((\mu_i(\theta)+\mu_e(\theta))(\pa_x u_1)^2\right),\pa^\al \widetilde{\ta}\right)+
\sum\limits_{j=2}^3\left(\pa^\al \left((\mu_i(\theta)+\mu_e(\theta))(\pa_x \widetilde{u}_j)^2\right),\pa^\al \widetilde{\ta}\right),
\end{split}
\end{array}\right.
\end{eqnarray*}
%and
\begin{eqnarray*}%\label{J13-15}
%\left\{
%\begin{array}{rll}
\begin{split}
%&J_{9}=3\left(\pa^\al \left((\mu_i(\theta)+\mu_e(\theta))(\pa_x u_1)^2\right),\pa^\al \widetilde{\ta}\right)+
%\sum\limits_{j=2}^3\left(\pa^\al \left((\mu_i(\theta)+\mu_e(\theta))(\pa_x \widetilde{u}_j)^2\right),\pa^\al \widetilde{\ta}\right),\\
J_{10}&=-\left( \int_{{\R}^3}\xi_1\pa^\al\pa_xG_i \,d\xi,\pa^\al \widetilde{n}_i\right)
-\left( \int_{{\R}^3}\xi_1\pa^\al\pa_xG_e \,d\xi,\pa^\al \widetilde{n}_e\right)
\\&\quad-\sum\limits_{j=1}^3\left(\int_{{\R}^3}\psi_{(j+2)i}\xi_1\pa^\al\pa_x \left(P_0^{M_i}G_i\right) d\xi,\pa^\al \pa_x\widetilde{u}_j\right)\\
&\quad-\sum\limits_{j=1}^3\left(\int_{{\R}^3}\psi_{(j+2)e}\xi_1\pa^\al\pa_x \left(P_0^{M_e}G_e\right) d\xi,\pa^\al \pa_x\widetilde{u}_j\right)
\\&\quad+\sum\limits_{j=1}^3\left(\pa^\al\left\{u_j\int_{{\R}^3}\psi_{(j+2)i}\xi_1\pa_x \left(P_0^{M_i}G_i\right) d\xi\right\},\pa^\al \widetilde{\ta}\right)
\\&\quad+\sum\limits_{j=1}^3\left(\pa^\al\left\{u_j\int_{{\R}^3}\psi_{(j+2)e}\xi_1\pa_x \left(P_0^{M_e}G_e\right) d\xi\right\},\pa^\al \widetilde{\ta}\right),
%&J_{11}=
%\sum\limits_{j=1}^3\left(\int_{{\R}^3}\xi_1\psi_{j+2}\cdot\pa^\al \overline{\FR} d\xi,\pa^\al \pa_x\widetilde{u}_j\right)
%+\left(\int_{{\R}^3}\xi_1\left(\psi_{6}-\sum\limits_{j=1}^3u_j\psi_{j+2}\right)\cdot\pa^\al \overline{\FR} d\xi,\pa_x\pa^\al \widetilde{\ta}\right),\\
%&J_{12}=-\sum\limits_{j=1}^3\left(\int_{{\R}^3}\xi_1\pa_xu_j\psi_{j+2}\cdot\pa^\al \overline{\FR} d\xi,\pa^\al \widetilde{\ta}\right)
%+\sum\limits_{j=1}^3\left(\int_{{\R}^3}\xi_1\pa^\al  u_j\psi_{j+2}\cdot\pa_x\overline{\FR} d\xi,\pa^\al \widetilde{\ta}\right)
%\\&\qquad+\left(\pa^\al \left(\ta\int_{{\R}^3}[\xi_1,\xi_1]^{\rm T}\cdot\pa_x\FG d\xi\right),\pa^\al \widetilde{\ta}\right)
%+\left(\pa^\al \left(\pa_x\phi\int_{{\R}^3}\frac{|\xi|^2}{2}\left[q_i,q_e\right]^{\rm T}\cdot\pa_{\xi_1}\FG d\xi\right),\pa^\al \widetilde{\ta}\right).
\end{split}
%\end{array}\right.
\end{eqnarray*}
and
\begin{eqnarray*}%\label{J13-15}
\left\{\begin{array}{rll}
\begin{split}
%&J_{9}=3\left(\pa^\al \left((\mu_i(\theta)+\mu_e(\theta))(\pa_x u_1)^2\right),\pa^\al \widetilde{\ta}\right)+
%\sum\limits_{j=2}^3\left(\pa^\al \left((\mu_i(\theta)+\mu_e(\theta))(\pa_x \widetilde{u}_j)^2\right),\pa^\al \widetilde{\ta}\right),\\
%&J_{10}=-\left( \int_{{\R}^3}\xi_1\pa^\al\pa_xG_i d\xi,\pa^\al \widetilde{n}_i\right)
%-\left( \int_{{\R}^3}\xi_1\pa^\al\pa_xG_e d\xi,\pa^\al \widetilde{n}_e\right)
%\\&\qquad-\sum\limits_{j=1}^3\left(\int_{{\R}^3}\psi_{(j+2)i}\xi_1\pa^\al\pa_x \left(P_0^{M_i}G_i\right) d\xi,\pa^\al \pa_x\widetilde{u}_j\right)
%-\sum\limits_{j=1}^3\left(\int_{{\R}^3}\psi_{(j+2)e}\xi_1\pa^\al\pa_x \left(P_0^{M_e}G_e\right) d\xi,\pa^\al \pa_x\widetilde{u}_j\right)
%\\&\qquad+\sum\limits_{j=1}^3\left(\pa^\al\left\{u_j\int_{{\R}^3}\psi_{(j+2)i}\xi_1\pa_x \left(P_0^{M_i}G_i\right) d\xi\right\},\pa^\al \widetilde{\ta}\right)
%\\&\qquad+\sum\limits_{j=1}^3\left(\pa^\al\left\{u_j\int_{{\R}^3}\psi_{(j+2)e}\xi_1\pa_x \left(P_0^{M_e}G_e\right) d\xi\right\},\pa^\al \widetilde{\ta}\right),\\
&J_{11}=
\sum\limits_{j=1}^3\left(\int_{{\R}^3}\xi_1\psi_{j+2}\cdot\pa^\al \overline{\FR} \,d\xi,\pa^\al \pa_x\widetilde{u}_j\right)
+\left(\int_{{\R}^3}\xi_1\left(\psi_{6}-\sum\limits_{j=1}^3u_j\psi_{j+2}\right)\cdot\pa^\al \overline{\FR} \,d\xi,\pa_x\pa^\al \widetilde{\ta}\right),\\
&J_{12}=-\sum\limits_{j=1}^3\left(\int_{{\R}^3}\xi_1\pa_xu_j\psi_{j+2}\cdot\pa^\al \overline{\FR} \,d\xi,\pa^\al \widetilde{\ta}\right)
+\sum\limits_{j=1}^3\left(\int_{{\R}^3}\xi_1\pa^\al  u_j\psi_{j+2}\cdot\pa_x\overline{\FR} \,d\xi,\pa^\al \widetilde{\ta}\right)
\\&\qquad+\left(\pa^\al \left(\ta\int_{{\R}^3}[\xi_1,\xi_1]^{\rm T}\cdot\pa_x\FG \,d\xi\right),\pa^\al \widetilde{\ta}\right)
+\left(\pa^\al \left(\pa_x\phi\int_{{\R}^3}\frac{|\xi|^2}{2}\left[q_i,q_e\right]^{\rm T}\cdot\pa_{\xi_1}\FG \,d\xi\right),\pa^\al \widetilde{\ta}\right).
\end{split}
\end{array}\right.
\end{eqnarray*}
We now turn to estimate $J_l$ $(1\leq l\leq12)$ term by term. %It should be noted that $J_1$ is relevant to the energy of $\pa^\al \phi $
%and $\pa^\al \pa_x\phi $. To confirm this,
%we use integration by parts and \eqref{tphy} to obtain
For brevity, we give straightforward calculations as follows:
\begin{equation*}
\begin{split}
|J_1|
\lesssim \eta\sum\limits_{|\al|=1}\left\|\pa_x\pa^\al \left[\widetilde{n}_i,\widetilde{n}_e\right]\right\|^2
+C_\eta\sum\limits_{|\al|=1}\left\|\pa^\al \left[\widetilde{n}_i,\widetilde{n}_e,\widetilde{u}_1\right]\right\|^2
+C_\eta(1+t)^{-2}\left\|\left[\widetilde{n}_i,\widetilde{n}_e,\widetilde{u}_1\right]\right\|^2,
\end{split}
\end{equation*}
\begin{equation*}
\begin{split}
|J_2|
\lesssim \eps_0\sum\limits_{|\al|=1}\left\|\pa^\al\left[\widetilde{u}_1,\widetilde{\ta}\right]\right\|^2,\ \
|J_3|
\lesssim \eta\sum\limits_{|\al|=1}\left\|\pa_x\pa^\al\widetilde{u}\right\|^2+C_\eta\sum\limits_{|\al|=1}\left\|\pa^\al\widetilde{u}\right\|^2+
C_\eta(1+t)^{-2}\left\|\widetilde{u}_1\right\|^2,
\end{split}
\end{equation*}
\begin{equation*}
\begin{split}
|J_4|
\lesssim \eps_0\sum\limits_{|\al|=1}\|\pa^\al \widetilde{u}_1\|^2
+\eps_0\sum\limits_{|\al|\leq1}\|\pa^\al (q_i\widetilde{n}_i+q_e\widetilde{n}_e)\|^2,
\end{split}
\end{equation*}
%As to the remaining terms in \eqref{J8-12} and \eqref{J13-15},
%for the sake of simplicity,
%we present a straightforward calculation as follows:
\begin{equation*}
\begin{split}
|J_5|\leq& \eta\sum\limits_{|\al|=1}\left\|\pa_x\pa^\al \widetilde{u}_1\right\|^2
+C_{\eta}\sum\limits_{|\al|=1}\left\|\pa^\al \left[\widetilde{n}_i,\widetilde{n}_e,\widetilde{\ta}\right]\right\|^2
+C_\eta(1+t)^{-2}\left\|\left[\widetilde{n}_i,\widetilde{n}_e,\widetilde{\ta}\right]\right\|^2,
\end{split}
\end{equation*}
\begin{equation*}
\begin{split}
|J_{6}|\lesssim&
\eta\sum\limits_{|\al|=1}\left\|\pa_x\pa^\al \left[\widetilde{u}_1,\widetilde{\ta}\right]\right\|^2+
C_\eta\sum\limits_{|\al| =1}\left\|\pa^\al \left[\widetilde{n}_i,\widetilde{n}_e,\widetilde{u}_1,\widetilde{\ta}\right]\right\|^2
+C_\eta(1+t)^{-2}\left\|\left[\widetilde{n}_i,\widetilde{n}_e,\widetilde{u}_1,\widetilde{\ta}\right]\right\|^2,
\end{split}
\end{equation*}
\begin{equation*}
\begin{split}
|J_{7}|\lesssim\eta\sum\limits_{|\al| =1}\left\|\pa_x\pa^\al \left[\widetilde{u},\widetilde{\ta}\right]\right\|^2+C_\eta\de^{1/2}_r(1+t)^{-3/2},
\end{split}
\end{equation*}
\begin{equation*}
\begin{split}
|J_{8}|\leq&\eta\sum\limits_{|\al| =1}\left\|\pa_x\pa^\al \left[\widetilde{u},\widetilde{\ta}\right]\right\|^2+
C_\eta\sum\limits_{|\al| =1}\left\|\pa^\al \left[\widetilde{n}_i,\widetilde{n}_e,\widetilde{\ta}\right]\pa_x\left[\widetilde{u},\widetilde{\ta}\right]\right\|^2
+C_\eta\sum\limits_{|\al| =1}\left\|\pa^\al [n^r,\ta^r]\pa_x\left[\widetilde{u},\widetilde{\ta}\right]\right\|^2
\\&+C_\eta\sum\limits_{|\al| =1}\left\|\pa^\al \left[\widetilde{n}_i,\widetilde{n}_e,\widetilde{\ta}\right]\pa_x[u^r,\ta^r]\right\|^2+
C_\eta\sum\limits_{|\al| =1}\|\pa^\al [n^r,\ta^r]\pa_x[u_1^r,\ta^r]\|^2\\
\leq&\eta\sum\limits_{|\al| =1}\left\|\pa_x\pa^\al \left[\widetilde{u},\widetilde{\ta}\right]\right\|^2
+C_\eta\sum\limits_{|\al| =1}\left\|\pa^\al \left[\widetilde{n}_i,\widetilde{n}_e,\widetilde{u},\widetilde{\ta}\right]\right\|^2+
+C_\eta\de_r^2(1+t)^{-2},
\end{split}
\end{equation*}
\begin{equation*}
\begin{split}
|J_{9}|\lesssim&\left|\left(\pa^\al \left[\widetilde{n}_i,\widetilde{n}_e,\widetilde{\ta}\right](\pa_x u)^2,\pa^\al \widetilde{\ta}\right)\right|+
\left|\left(\pa^\al [n^r,\ta^r](\pa_x \widetilde{u})^2,\pa^\al \widetilde{\ta}\right)\right|
+\left|\left(\pa^\al [n^r,\ta^r](\pa_x u_1^r)^2,\pa^\al \widetilde{\ta}\right)\right|
\\&+\left|\left(\pa^\al \pa_x\widetilde{u}\pa_x\widetilde{u},\pa^\al \widetilde{\ta}\right)\right|
+\left|\left(\pa^\al \pa_xu_1^r\pa_xu_1^r,\pa^\al \widetilde{\ta}\right)\right|\\
\lesssim& \sum\limits_{|\al| =1}\left\|\pa^\al \left[\widetilde{n}_i,\widetilde{n}_e,\widetilde{u},\widetilde{\ta}\right]\right\|^2
+\eps_0\sum\limits_{|\al| =1}\|\pa_x\pa^\al \widetilde{u}\|^2+\de^2_r(1+t)^{-2},
\end{split}
\end{equation*}
and
\begin{equation*}
\begin{split}
|J_{10}|+|J_{11}|+|J_{12}|
\leq&\eta\sum\limits_{|\al| =1}\left\|\pa_x\pa^\al \left[\widetilde{u},\widetilde{\ta}\right]\right\|^2
+C_\eta\sum\limits_{|\al| =1}\left\|\pa^\al \left[\widetilde{n}_i,\widetilde{n}_e,\widetilde{u},\widetilde{\ta}\right]\right\|^2
\\&+ C_\eta\sum\limits_{1\leq|\al|\leq2}\int_{\R\times{\R}^3}(1+|\xi|)|\FM^{-1/2}\pa^\al \FG|^2 d\xi dx\\
&+\eps_{0}\sum\limits_{|\al| \leq1}\int_{\R\times\R^3}|\FM^{-1/2}\pa_{\xi_1}\pa^\al \widetilde{\FG}|^2 d\xi dx
\\&+\eps_{0}\int_{{\R}\times{\R}^3}(1+|\xi|)\left|\FM^{-1/2}_{\sharp}\widetilde{\FG}\right|^2d\xi dx
+\de_r^{1/2}(1+t)^{-3/2}.
\end{split}
\end{equation*}
Plugging the above estimates for $J_l$ $(1\leq l\leq12)$ into \eqref{p2u1}, one thus has
\begin{equation}\label{p2u2}
\begin{split}
&\frac{d}{dt}\sum\limits_{|\al=1|}\left\{\|\pa^\al \widetilde{n}_i\|^2
+\|\pa^\al \widetilde{n}_e\|^2
+\sum\limits_{j=1}^3\left\|\sqrt{m_in_i+m_en_e}\pa^\al \widetilde{u}_j\right\|^2
+\left\|\sqrt{n_i+n_e}\pa^\al \widetilde{\ta}\right\|^2\right\}
\\&\quad+\sum\limits_{|\al| =1}\la\left\|\pa_x \pa^\al \left[\widetilde{u},\widetilde{\theta}\right]\right\|^2
\\ &
\lesssim\sum\limits_{|\al| =1}\left\|\pa^\al \left[\widetilde{n}_i,\widetilde{n}_e,\widetilde{u},\widetilde{\ta}\right]\right\|^2
+\|q_i\widetilde{n}_i+q_e\widetilde{n}_e\|^2
+(1+t)^{-2}\left\|\left[\widetilde{n}_i,\widetilde{n}_e,\widetilde{\ta}\right]\right\|^2
+\de_r^{1/2}(1+t)^{-3/2}
\\&\quad+ \sum\limits_{1\leq|\al| \leq2}\int_{\R\times{\R}^3}(1+|\xi|)|\FM^{-1/2}\pa^\al \FG|^2 d\xi dx
+\int_{{\R}\times{\R}^3}(1+|\xi|)\left|\FM^{-1/2}_{\sharp}\widetilde{\FG}\right|^2d\xi dx
\\&\quad+\eps_{0}\sum\limits_{\ga \leq1}\int_{\R\times\R^3}|\FM^{-1/2}\pa_{\xi_1}\pa^\al \widetilde{\FG}|^2 d\xi dx.
\end{split}
\end{equation}

Let us now deduce the second-order dissipation of $\widetilde{n}_i$ and $\widetilde{n}_e$. As it has been shown in the previous subsection, it may not be direct to obtain
the second-order dissipation of $\widetilde{n}_i$ and $\widetilde{n}_e$ in a separate way, and instead one has to consider $\pa^\al \pa_x\left(\widetilde{n}_i+\widetilde{n}_e\right)$
and $\pa^\al \pa_x\left(q_i\widetilde{n}_i+q_e\widetilde{n}_e\right)$ $(|\al| =1)$ in an equivalent way. In what follows,
we shall turn to derive these two kinds of dissipations by using different equations. In fact,
one can first take the inner product of $\pa^\al \eqref{u1.cons}$ with $\pa^\al \pa_x\left(\widetilde{n}_i+\widetilde{n}_e\right)$
$(|\al| =1)$ to obtain
\begin{equation*}%\label{2-vipve}
\begin{split}
&\left((m_in_i+m_en_e)\pa_t \pa^\al \widetilde{u}_1,\pa^\al \pa_x\left(\widetilde{n}_i+\widetilde{n}_e\right)\right)
+\left(\pa^\al(m_in_i+m_en_e)\pa_t  \widetilde{u}_1,\pa^\al \pa_x\left(\widetilde{n}_i+\widetilde{n}_e\right)\right)
\\&\quad+\left(\pa^\al\left\{(m_in_i+m_en_e)(u_1\pa_x\widetilde{u}_1+\widetilde{u}_1\pa_xu_1^{r})\right\},\pa^\al \pa_x\left(\widetilde{n}_i+\widetilde{n}_e\right)\right)
\\&\quad+\left(\pa^\al \pa_x P-\pa^\al \pa_xP^r,\pa^\al \pa_x\left(\widetilde{n}_i+\widetilde{n}_e\right)\right)
+\left(\pa^\al \left(\left(q_i\widetilde{n}_i+q_e\widetilde{n}_e\right)\pa_x\phi \right),
\pa^\al \pa_x\left(\widetilde{n}_i+\widetilde{n}_e\right)\right)
\\&
=-\left(\int_{{\R}^3}\xi_1\psi_{3}\cdot\pa^\al \pa_x\FG \,d\xi,\pa^\al \pa_x\left(\widetilde{n}_i+\widetilde{n}_e\right)\right),
\end{split}
\end{equation*}
from which as well as \eqref{p-pr}, it follows that
\begin{equation}\label{2-vipve.p1}
\begin{split}
&\frac{d}{dt}\left((m_in_i+m_en_e)\pa^\al \widetilde{u}_1,\pa^\al \pa_x\left(\widetilde{n}_i+\widetilde{n}_e\right)\right)+
\frac{2}{3}\left(\ta^r\pa^\al \pa_x\left(\widetilde{n}_i+\widetilde{n}_e\right),
\pa^\al \pa_x\left(\widetilde{n}_i+\widetilde{n}_e\right)\right)
\\&=\left((m_in_i+m_en_e)\pa^\al \widetilde{u}_1,
\pa_t \left\{\pa^\al \pa_x\left(\widetilde{n}_i+\widetilde{n}_e\right)\right\}\right)
+\left(\pa_t (m_in_i+m_en_e)\pa^\al \widetilde{u}_1,
\left\{\pa^\al \pa_x\left(\widetilde{n}_i+\widetilde{n}_e\right)\right\}\right)
\\&\quad-\left(\pa^\al(m_in_i+m_en_e)\pa_t  \widetilde{u}_1,\pa^\al \pa_x\left(\widetilde{n}_i+\widetilde{n}_e\right)\right)
-\frac{2}{3}\left(\pa^\al \ta^r\pa_x(\widetilde{n}_i+\pa_x\widetilde{n}_e),
\pa^\al \pa_x\left(\widetilde{n}_i+\widetilde{n}_e\right)\right)
\\&\quad-\frac{2}{3}\left(\pa^\al \left\{\pa_x\ta^r(\widetilde{n}_i
+\widetilde{n}_e)\right\},
\pa^\al \pa_x\left(\widetilde{n}_i+\widetilde{n}_e\right)\right)
-\frac{2}{3}\left(\pa^\al \pa_x\left\{\widetilde{\ta}(\widetilde{n}_i
+\widetilde{n}_e)\right\},
\pa^\al \pa_x\left(\widetilde{n}_i+\widetilde{n}_e\right)\right)
\\&\quad-\left(\pa^\al \pa_x\left(\widetilde{\ta}(n_i^r+n_e^r)\right),
\pa^\al \pa_x\left(\widetilde{n}_i+\widetilde{n}_e\right)\right)
-\left(\pa^\al \left(\left(q_i\widetilde{n}_i+q_e\widetilde{n}_e\right)\pa_x\phi \right),
\pa^\al \pa_x\left(\widetilde{n}_i+\widetilde{n}_e\right)\right)
\\
&\quad-\left(\int_{{\R}^3}\xi_1\psi_{3}\cdot\pa^\al \pa_x\FG \,d\xi,\pa^\al \pa_x\left(\widetilde{n}_i+\widetilde{n}_e\right)\right).
\end{split}
\end{equation}
Using integration by parts and applying \eqref{tvi} and \eqref{tve}, one can deduce
\begin{equation}\label{th.vipve}
\begin{split}
&\left|\left((m_in_i+m_en_e)\pa^\al \widetilde{u}_1,
\pa_t \pa^\al \pa_x\left(\widetilde{n}_i+\widetilde{n}_e\right)\right)\right|
\\&=\left|\left(\pa_x\left\{(m_in_i+m_en_e)\pa^\al \widetilde{u}_1\right\},
\pa_t \pa^\al \left(\widetilde{n}_i+\widetilde{n}_e\right)\right)\right|
\\&\lesssim \left|\left(\pa_x\left\{(m_in_i+m_en_e)\pa^\al \widetilde{u}_1\right\},
\pa^\al \pa_x\left(n_iu_1-n_i^ru_1^r+n_eu_1-n_e^ru_1^r\right)\right)\right|
\\&\quad+\left|\left(\pa_x\left\{(m_in_i+m_en_e)\pa^\al \widetilde{u}_1\right\},\int_{{\R}^3}\xi_1\pa^\al\pa_x(G_i+G_e) d\xi\right)\right|
\\&\lesssim C_\eta\sum\limits_{|\al| =1}\|\pa^\al \pa_x\widetilde{u}_1\|^2
+\eta\sum\limits_{|\al| =1}\left\|\pa^\al \pa_x\left(\widetilde{n}_i+\widetilde{n}_e\right)\right\|^2
+\sum\limits_{|\al| =1}\left\|\pa^\al \pa_x\left[\widetilde{n}_i,\widetilde{n}_e,\widetilde{u}_1,\widetilde{\ta}\right]\right\|^2
\\&\quad+(1+t)^{-2}\left\|\left[\widetilde{n}_i,\widetilde{n}_e,\widetilde{u}_1\right]\right\|^2
+\sum\limits_{1\leq|\al| \leq 2}\int_{\R\times{\R}^3}(1+|\xi|)|\FM^{-1/2}\pa^\al \FG|^2 d\xi dx.
\end{split}
\end{equation}
Furthermore, it is straightforward to show that the remaining terms on the right-hand side of \eqref{2-vipve.p1} are bounded by
\begin{equation}\label{Rm.2-vipve.p1}
\begin{split}
&(\eps_0+\eta)\left\|\pa^\al \pa_x\left(\widetilde{n}_i+\widetilde{n}_e\right)\right\|^2
+C_\eta\sum\limits_{|\al| =1}\left\|\pa^\al \left[\widetilde{n}_i,\widetilde{n}_e,\widetilde{u},\widetilde{\ta}\right]\right\|^2
+C_\eta\sum\limits_{|\al| =1}\left\|\pa^\al \pa_x\widetilde{\ta}\right\|^2\\
&\quad +(1+t)^{-2}\left\|\left[\widetilde{n}_i,\widetilde{n}_e,\widetilde{u}_1,\widetilde{\ta}\right]\right\|^2+\eps_0\sum\limits_{|\al| \leq1}\|\pa^\al(q_i\widetilde{n}_i+q_e\widetilde{n}_e)\|^2+C_\eta\de_r^{1/2}(1+t)^{-3/2}\\
&\quad+C_\eta\sum\limits_{1\leq|\al| \leq 2}\int_{\R\times{\R}^3}(1+|\xi|)|\FM^{-1/2}\pa^\al \FG|^2 d\xi dx.
\end{split}
\end{equation}
Substituting \eqref{th.vipve} and \eqref{Rm.2-vipve.p1} into \eqref{2-vipve.p1}, we arrive at
\begin{equation}\label{2-vipve.p2}
\begin{split}
&\frac{d}{dt}\sum\limits_{|\al| =1}\left((m_in_i+m_en_e)\pa^\al \widetilde{u}_1,\pa^\al \pa_x\left(\widetilde{n}_i+\widetilde{n}_e\right)\right)+
\la\sum\limits_{|\al| =1}\left(\pa^\al \pa_x\left(\widetilde{n}_i+\widetilde{n}_e\right),
\pa^\al \pa_x\left(\widetilde{n}_i+\widetilde{n}_e\right)\right)
\\& \lesssim
\sum\limits_{|\al| =1}\left\|\pa^\al \left[\widetilde{n}_i,\widetilde{n}_e,\widetilde{u},\widetilde{\ta}\right]\right\|^2
+\sum\limits_{|\al| =1}\left\|\pa^\al \pa_x\left[\widetilde{u}_1,\widetilde{\ta}\right]\right\|^2
+(1+t)^{-2}\left\|\left[\widetilde{n}_i,\widetilde{n}_e,\widetilde{u}_1,\widetilde{\ta}\right]\right\|^2
\\&\quad+\sum\limits_{|\al| \leq1}\|\pa^\al(q_i\widetilde{n}_i+q_e\widetilde{n}_e)\|^2+\de_r^{1/2}(1+t)^{-3/2}\\
&\quad+\sum\limits_{1\leq|\al| \leq 2}\int_{\R\times{\R}^3}(1+|\xi|)|\FM^{-1/2}\pa^\al \FG|^2 d\xi dx.
\end{split}
\end{equation}

One the other hand, taking the inner product of $\pa^\al \eqref{visbve}$ with $\pa^\al \pa_x\left(q_in_i+q_en_e\right)$ with respect to $x$ over $\R$, one has that
\begin{eqnarray*}
&&\left((q_im_in_i+q_em_en_e)\pa_t\pa^\al \widetilde{u}_1,\pa^\al \pa_x \left(q_i\widetilde{n}_i+q_e\widetilde{n}_e\right)\right)
+\left((q^2_in_i+q^2_en_e)\pa^\al \pa_x\phi ,\pa^\al \pa_x \left(q_i\widetilde{n}_i+q_e\widetilde{n}_e\right)\right)
\notag\\
&&\quad
+\frac{2}{3}\left(\ta\pa^\al \pa_x \left(q_i\widetilde{n}_i+q_e\widetilde{n}_e\right),\pa^\al \pa_x \left(q_i\widetilde{n}_i+q_e\widetilde{n}_e\right)\right)
+\left(\pa^\al \left(q^2_in_i+q^2_en_e\right)\pa_x\phi ,\pa^\al \pa_x \left(q_i\widetilde{n}_i+q_e\widetilde{n}_e\right)\right)
\notag\\
&&\quad+\frac{2}{3}\left(\pa^\al \ta\pa_x \left(q_i\widetilde{n}_i+q_e\widetilde{n}_e\right),\pa^\al \pa_x \left(q_i\widetilde{n}_i+q_e\widetilde{n}_e\right)\right)
+\frac{2}{3}\left(\pa^\al \left(\pa_x \ta \left(q_i\widetilde{n}_i+q_e\widetilde{n}_e\right)\right),\pa^\al \pa_x \left(q_i\widetilde{n}_i+q_e\widetilde{n}_e\right)\right)
\notag\\
&&\quad+\left(\pa^\al\left(q_im_i\widetilde{n}_i+q_em_e\widetilde{n}_e)(\pa_t u^r_1+u^r_1\pa_xu_1^{r})\right),\pa^\al\pa_x \left(q_i\widetilde{n}_i+q_e\widetilde{n}_e\right)\right)
%\notag\\
%&&=-\left(q_i\int_{{\R}^3}\psi_{3i}\pa^\al \pa_tG_i d\xi+q_e\int_{{\R}^3}\psi_{3e}\pa^\al \pa_tG_e d\xi,\pa^\al \pa_x \left(q_i\widetilde{n}_i+q_e\widetilde{n}_e\right)\right)
%\notag\\
%&&\quad-\left(q_i\int_{{\R}^3}\psi_{3i}\xi_1\pa^\al \pa_xG_i d\xi+q_e\int_{{\R}^3}\psi_{3e}\pa^\al \xi_1\pa_xG_e d\xi,\pa^\al \pa_x \left(q_i\widetilde{n}_i+q_e\widetilde{n}_e\right)\right)
%\notag\\
%&&\quad
%+\left(q_i\pa^\al\left(u_1\int_{{\R}^3}\psi_{1i}\xi_1\pa_xG_i d\xi\right)+q_e\pa^\al\left(u_1\int_{{\R}^3}\psi_{2e}\xi_1\pa_xG_e d\xi\right),\pa^\al\pa_x\left(q_i\widetilde{n}_i+q_e\widetilde{n}_e\right) \right)
%\notag\\
%&&\quad+\left(q_i\int_{{\R}^3}\psi_{3i}\pa^\al Q_i(\FF,\FF)d\xi+q_e\int_{{\R}^3}\psi_{3e}\pa^\al Q_e(\FF,\FF)d\xi,
%\pa^\al\pa_x\left(q_i\widetilde{n}_i+q_e\widetilde{n}_e\right) \right)
%\notag\\
%&&\quad+\left(\pa^\al\left(\frac{2\ta^r}{3}\frac{q_im_in_i^r+q_em_en_e^r}{m_in^r_i+m_en^r_e}\pa_x(n_i^r+n_e^r)\right),
%\pa^\al\pa_x\left(q_i\widetilde{n}_i+q_e\widetilde{n}_e\right) \right)
%\notag\\
%&&\quad+\left(\pa^\al\left(\frac{2}{3}\pa_x\ta^r\frac{q_im_in_i^r+q_em_en_e^r}{m_in^r_i+m_en^r_e}(n_i^r+n_e^r)\right)
%,\pa^\al\pa_x\left(q_i\widetilde{n}_i+q_e\widetilde{n}_e\right) \right). \label{2-vi-ve.ip2}
\end{eqnarray*}
is equal to
\begin{eqnarray*}
%&&\left((q_im_in_i+q_em_en_e)\pa_t\pa^\al \widetilde{u}_1,\pa^\al \pa_x \left(q_i\widetilde{n}_i+q_e\widetilde{n}_e\right)\right)
%+\left((q^2_in_i+q^2_en_e)\pa^\al \pa_x\phi ,\pa^\al \pa_x \left(q_i\widetilde{n}_i+q_e\widetilde{n}_e\right)\right)
%\notag\\
%&&\quad
%+\frac{2}{3}\left(\ta\pa^\al \pa_x \left(q_i\widetilde{n}_i+q_e\widetilde{n}_e\right),\pa^\al \pa_x \left(q_i\widetilde{n}_i+q_e\widetilde{n}_e\right)\right)
%+\left(\pa^\al \left(q^2_in_i+q^2_en_e\right)\pa_x\phi ,\pa^\al \pa_x \left(q_i\widetilde{n}_i+q_e\widetilde{n}_e\right)\right)
%\notag\\
%&&\quad+\frac{2}{3}\left(\pa^\al \ta\pa_x \left(q_i\widetilde{n}_i+q_e\widetilde{n}_e\right),\pa^\al \pa_x \left(q_i\widetilde{n}_i+q_e\widetilde{n}_e\right)\right)
%+\frac{2}{3}\left(\pa^\al \left(\pa_x \ta \left(q_i\widetilde{n}_i+q_e\widetilde{n}_e\right)\right),\pa^\al \pa_x \left(q_i\widetilde{n}_i+q_e\widetilde{n}_e\right)\right)
%\notag\\
%&&\quad+\left(\pa^\al\left(q_im_i\widetilde{n}_i+q_em_e\widetilde{n}_e)(\pa_t u^r_1+u^r_1\pa_xu_1^{r})\right),\pa^\al\pa_x \left(q_i\widetilde{n}_i+q_e\widetilde{n}_e\right)\right)
%\notag\\
&&-\left(q_i\int_{{\R}^3}\psi_{3i}\pa^\al \pa_tG_i \,d\xi+q_e\int_{{\R}^3}\psi_{3e}\pa^\al \pa_tG_e \,d\xi,\pa^\al \pa_x \left(q_i\widetilde{n}_i+q_e\widetilde{n}_e\right)\right)
\notag\\
&&%\quad
-\left(q_i\int_{{\R}^3}\psi_{3i}\xi_1\pa^\al \pa_xG_i \,d\xi+q_e\int_{{\R}^3}\psi_{3e}\pa^\al \xi_1\pa_xG_e \,d\xi,\pa^\al \pa_x \left(q_i\widetilde{n}_i+q_e\widetilde{n}_e\right)\right)
\notag\\
&&%\quad
+\left(q_i\pa^\al\left(u_1\int_{{\R}^3}\psi_{1i}\xi_1\pa_xG_i \,d\xi\right)+q_e\pa^\al\left(u_1\int_{{\R}^3}\psi_{2e}\xi_1\pa_xG_e \,d\xi\right),\pa^\al\pa_x\left(q_i\widetilde{n}_i+q_e\widetilde{n}_e\right) \right)
\notag\\
&&%\quad
+\left(q_i\int_{{\R}^3}\psi_{3i}\pa^\al Q_i(\FF,\FF)\,d\xi+q_e\int_{{\R}^3}\psi_{3e}\pa^\al Q_e(\FF,\FF)\,d\xi,
\pa^\al\pa_x\left(q_i\widetilde{n}_i+q_e\widetilde{n}_e\right) \right)
\notag\\
&&%\quad
+\left(\pa^\al\left(\frac{2\ta^r}{3}\frac{q_im_in_i^r+q_em_en_e^r}{m_in^r_i+m_en^r_e}\pa_x(n_i^r+n_e^r)\right),
\pa^\al\pa_x\left(q_i\widetilde{n}_i+q_e\widetilde{n}_e\right) \right)
\notag\\
&&%\quad
+\left(\pa^\al\left(\frac{2}{3}\pa_x\ta^r\frac{q_im_in_i^r+q_em_en_e^r}{m_in^r_i+m_en^r_e}(n_i^r+n_e^r)\right)
,\pa^\al\pa_x\left(q_i\widetilde{n}_i+q_e\widetilde{n}_e\right) \right). %\label{2-vi-ve.ip2}
\end{eqnarray*}
Therefore, by the similar argument as for obtaining \eqref{dis.phi}, it follows that
\begin{equation}\label{hdis.vi-ve}
\begin{split}
&\la\sum\limits_{|\al| =1}\left(\pa^\al \pa_x \left(q_i\widetilde{n}_i+q_e\widetilde{n}_e\right),\pa^\al \pa_x \left(q_i\widetilde{n}_i+q_e\widetilde{n}_e\right)\right)
+\la\sum\limits_{|\al| =1}\left(\pa^\al \left(q_i\widetilde{n}_i+q_e\widetilde{n}_e\right),\pa^\al \left(q_i\widetilde{n}_i+q_e\widetilde{n}_e\right)\right)
\\& \lesssim
\sum\limits_{|\al| =1}\left\|\pa^\al \left[\widetilde{n}_i,\widetilde{n}_e,\widetilde{u},\widetilde{\ta}\right]\right\|^2
+\sum\limits_{|\al| =1}\left\|\pa^\al \pa_x\left[\widetilde{u}_1,\widetilde{\ta}\right]\right\|^2
+\eps_0\sum\limits_{|\al| \leq1}\|\pa^\al(q_i\widetilde{n}_i+q_e\widetilde{n}_e)\|^2
\\
&\quad+\eps_{0}\sum\limits_{|\al| \leq1}\left\|\pa^\al \pa_x\phi \right\|^2+\de_r^{1/2}(1+t)^{-3/2}
+(1+t)^{-2}\left\|\left[\widetilde{n}_i,\widetilde{n}_e,\widetilde{u}_1,\widetilde{\ta}\right]\right\|^2
\\
&\quad+\sum\limits_{1\leq|\al| \leq2}\int_{\R\times\R^3}(1+|\xi|)|\FM^{-1/2}\pa^{\al}\FG|^2 d\xi dx
+C_\eta\int_{{\R}\times{\R}^3}(1+|\xi|)\left|\FM^{-1/2}_\sharp\widetilde{\FG}\right|^2 d\xi dx.
\end{split}
\end{equation}
We are now in a position to derive the dissipation of $\pa^\al \left[\pa_x\phi ,\pa^2_x\phi \right]$ with $ |\al| = 1$. For this,
we take the inner product of $\pa^\al \eqref{visbve}$ with $\pa^\al \pa_x\phi $ with respect to $x$ over $\R$ to obtain that
\begin{equation*}
\begin{split}
&\left((q_im_in_i+q_em_en_e)\pa_t\pa^\al \widetilde{u}_1,\pa^\al \pa_x \phi\right)
+\left((q^2_in_i+q^2_en_e)\pa^\al \pa_x\phi ,\pa^\al \pa_x \phi\right)
\\&\quad
+\frac{2}{3}\left(\ta\pa^\al \pa_x \left(q_i\widetilde{n}_i+q_e\widetilde{n}_e\right),\pa^\al \pa_x \phi\right)
+\left(\pa^\al \left(q^2_in_i+q^2_en_e\right)\pa_x\phi ,\pa^\al \pa_x \phi\right)
\\&\quad+\frac{2}{3}\left(\pa^\al \ta\pa_x \left(q_i\widetilde{n}_i+q_e\widetilde{n}_e\right),\pa^\al \pa_x \phi\right)
+\frac{2}{3}\left(\pa^\al \left(\pa_x \ta \left(q_i\widetilde{n}_i+q_e\widetilde{n}_e\right)\right),\pa^\al \pa_x \phi\right)
\\&\quad+\left(\pa^\al\left(q_im_i\widetilde{n}_i+q_em_e\widetilde{n}_e)(\pa_t u^r_1+u^r_1\pa_xu_1^{r})\right),\pa^\al\pa_x \phi\right)
%\\&\qquad=-\left(q_i\int_{{\R}^3}\psi_{3i}\pa^\al \pa_tG_i d\xi+q_e\int_{{\R}^3}\psi_{3e}\pa^\al \pa_tG_e d\xi,\pa^\al \pa_x \phi\right)
%\\&\qquad\quad-\left(q_i\int_{{\R}^3}\psi_{3i}\xi_1\pa^\al \pa_xG_i d\xi+q_e\int_{{\R}^3}\psi_{3e}\pa^\al \xi_1\pa_xG_e d\xi,\pa^\al \pa_x \phi\right)
%\\&\qquad\quad
%+\left(q_i\pa^\al\left(u_1\int_{{\R}^3}\psi_{1i}\xi_1\pa_xG_i d\xi\right)+q_e\pa^\al\left(u_1\int_{{\R}^3}\psi_{2e}\xi_1\pa_xG_e d\xi\right),\pa^\al\pa_x\phi \right)
%\\&\qquad\quad+\left(q_i\int_{{\R}^3}\psi_{3i}\pa^\al Q_i(\FF,\FF)d\xi+q_e\int_{{\R}^3}\psi_{3e}\pa^\al Q_e(\FF,\FF)d\xi,
%\pa^\al\pa_x\phi \right)
%\\&\qquad\quad+\left(\pa^\al\left(\frac{2\ta^r}{3}\frac{q_im_in_i^r+q_em_en_e^r}{m_in^r_i+m_en^r_e}\pa_x(n_i^r+n_e^r)\right),
%\pa^\al\pa_x\phi \right)
%\\&\qquad\quad+\left(\pa^\al\left(\frac{2}{3}\pa_x\ta^r\frac{q_im_in_i^r+q_em_en_e^r}{m_in^r_i+m_en^r_e}(n_i^r+n_e^r)\right),\pa^\al\pa_x\phi \right).
\end{split}
\end{equation*}
is equal to
\begin{equation*}
\begin{split}
%&\left((q_im_in_i+q_em_en_e)\pa_t\pa^\al \widetilde{u}_1,\pa^\al \pa_x \phi\right)
%+\left((q^2_in_i+q^2_en_e)\pa^\al \pa_x\phi ,\pa^\al \pa_x \phi\right)
%\\&\qquad\quad
%+\frac{2}{3}\left(\ta\pa^\al \pa_x \left(q_i\widetilde{n}_i+q_e\widetilde{n}_e\right),\pa^\al \pa_x \phi\right)
%+\left(\pa^\al \left(q^2_in_i+q^2_en_e\right)\pa_x\phi ,\pa^\al \pa_x \phi\right)
%\\&\qquad\quad+\frac{2}{3}\left(\pa^\al \ta\pa_x \left(q_i\widetilde{n}_i+q_e\widetilde{n}_e\right),\pa^\al \pa_x \phi\right)
%+\frac{2}{3}\left(\pa^\al \left(\pa_x \ta \left(q_i\widetilde{n}_i+q_e\widetilde{n}_e\right)\right),\pa^\al \pa_x \phi\right)
%\\&\qquad\quad+\left(\pa^\al\left(q_im_i\widetilde{n}_i+q_em_e\widetilde{n}_e)(\pa_t u^r_1+u^r_1\pa_xu_1^{r})\right),\pa^\al\pa_x \phi\right)
&-\left(q_i\int_{{\R}^3}\psi_{3i}\pa^\al \pa_tG_i \,d\xi+q_e\int_{{\R}^3}\psi_{3e}\pa^\al \pa_tG_e \,d\xi,\pa^\al \pa_x \phi\right)
\\&%\qquad\quad
-\left(q_i\int_{{\R}^3}\psi_{3i}\xi_1\pa^\al \pa_xG_i \,d\xi+q_e\int_{{\R}^3}\psi_{3e}\pa^\al \xi_1\pa_xG_e \,d\xi,\pa^\al \pa_x \phi\right)
\\&%\qquad\quad
+\left(q_i\pa^\al\left(u_1\int_{{\R}^3}\psi_{1i}\xi_1\pa_xG_i \,d\xi\right)+q_e\pa^\al\left(u_1\int_{{\R}^3}\psi_{2e}\xi_1\pa_xG_e \,d\xi\right),\pa^\al\pa_x\phi \right)
\\&%\qquad\quad
+\left(q_i\int_{{\R}^3}\psi_{3i}\pa^\al Q_i(\FF,\FF)\,d\xi+q_e\int_{{\R}^3}\psi_{3e}\pa^\al Q_e(\FF,\FF)\,d\xi,
\pa^\al\pa_x\phi \right)
\\&%\qquad\quad
+\left(\pa^\al\left(\frac{2\ta^r}{3}\frac{q_im_in_i^r+q_em_en_e^r}{m_in^r_i+m_en^r_e}\pa_x(n_i^r+n_e^r)\right),
\pa^\al\pa_x\phi \right)
\\&%\qquad\quad
+\left(\pa^\al\left(\frac{2}{3}\pa_x\ta^r\frac{q_im_in_i^r+q_em_en_e^r}{m_in^r_i+m_en^r_e}(n_i^r+n_e^r)\right),\pa^\al\pa_x\phi \right).
\end{split}
\end{equation*}
In almost the same way as for obtaining \eqref{dis.phi}, one can further derive that
\begin{equation}\label{dis.htphi}
\begin{split}
&\la\sum\limits_{|\al| =1}\left(\pa^\al \pa_x\phi ,\pa^\al \pa_x\phi \right)
+\la\sum\limits_{|\al| =1}\left(\pa^\al \pa^2_x\phi ,\pa^\al \pa^2_x\phi \right)
\\&
\lesssim
\sum\limits_{|\al| \leq1}\left\|q_i\widetilde{n}_i+q_e\widetilde{n}_e\right\|^2
+\sum\limits_{|\al| =1}\left\|\pa^\al \pa_x\left[\widetilde{u}_1,\widetilde{\ta}\right]\right\|^2
+\sum\limits_{|\al| =1}\left\|\pa^\al \left[\widetilde{n}_i,\widetilde{n}_e,\widetilde{u},\widetilde{\ta}\right]\right\|^2
+\|\pa_x\phi\|^2
\\&\quad+(1+t)^{-2}\left\|\left[\widetilde{n}_i,\widetilde{n}_e,\widetilde{u}_1,\widetilde{\ta}\right]\right\|^2+\de^{1/2}_r(1+t)^{-3/2}
\\&\quad+
\int_{{\R}\times{\R}^3}(1+|\xi|)\left|\FM^{-1/2}\widetilde{\FG}\right|^2d\xi dx
+\sum\limits_{1\leq |\al| \leq2}
\int_{{\R}\times{\R}^3}(1+|\xi|)\left|\FM^{-1/2}\FG\right|^2d\xi dx.
\end{split}
\end{equation}
Taking the suitable linear combination of \eqref{2-vipve.p2}, \eqref{hdis.vi-ve} and \eqref{dis.htphi}, we conclude that
\begin{equation}\label{2-dis.viep}
\begin{split}
&\frac{d}{dt}\sum\limits_{|\al| =1}\left((m_in_i+m_en_e)\pa^\al \widetilde{u}_1,\pa_x\pa^\al (\widetilde{n}_i+\widetilde{n}_e)\right)+
\la\sum\limits_{|\al| =1}
\left\|\pa^\al \pa_x\left[\widetilde{n}_i,\widetilde{n}_e,\phi ,\pa_x\phi \right]\right\|^2
\\ &\lesssim
\sum\limits_{|\al| =1}\left\|\pa^\al \left[\widetilde{n}_i,\widetilde{n}_e,\widetilde{u},\widetilde{\ta}\right]\right\|^2
+\|\pa_x\phi\|^2
+\sum\limits_{|\al| =1}\left\|\pa^\al \pa_x\left[\widetilde{u}_1,\widetilde{\ta}\right]\right\|^2
+\sum\limits_{|\al| \leq1}\|q_i\widetilde{n}_i+q_e\widetilde{n}_e\|^2
\\&\quad+(1+t)^{-2}\left\|\left[\widetilde{n}_i,\widetilde{n}_e,\widetilde{u}_1,\widetilde{\ta}\right]\right\|^2+\de_r^{1/2}(1+t)^{-3/2}
\\&\quad+C_\eta\sum\limits_{1\leq|\al| \leq2}\int_{\R\times\R^3}(1+|\xi|)|\FM^{-1/2}\pa^{\al}\FG|^2 d\xi dx\\
&\quad+C_\eta\int_{{\R}\times{\R}^3}(1+|\xi|)\left|\FM^{-1/2}_\sharp\widetilde{\FG}\right|^2 d\xi dx.
\end{split}
\end{equation}
As to the second-order time derivative of $\left[\widetilde{n}_i,\widetilde{n}_e,\widetilde{u},\widetilde{\ta}\right]$, one has by \eqref{pb.con.} that
\begin{equation}\label{sed.t}
\begin{split}
&\left\|\pa^2_t\left[\widetilde{n}_i,\widetilde{n}_e,\widetilde{u},\widetilde{\ta}\right]\right\|^2\\
&\lesssim\sum\limits_{|\al| =1}
\left\|\pa^\al \pa_x\left[\widetilde{n}_i,\widetilde{n}_e,\widetilde{u}_1,\widetilde{\ta},\phi \right]\right\|^2
+(1+t)^{-2}\left\|\left[\widetilde{n}_i,\widetilde{n}_e,\widetilde{u}_1,\widetilde{\ta}\right]\right\|^2
+\de_r^{1/2}(1+t)^{-3/2}
\\&\quad+\sum\limits_{1\leq|\al| \leq2}\int_{\R\times{\R}^3}(1+|\xi|)|\FM^{-1/2}\pa^\al \FG|^2 d\xi dx
+\eps_{0}\sum\limits_{|\al| \leq1}\int_{\R\times\R^3}|\FM^{-1/2}\pa_{\xi_1}\pa^\al \widetilde{\FG}|^2 d\xi dx .
\end{split}
\end{equation}
In addition, in light of \eqref{tphy}, one can see that $\phi $ enjoys much higher order dissipative property, namely,
\begin{equation}\label{ptphyL22}
\begin{split}
\sum\limits_{|\al| =2}\left\|\pa^\al \pa^2_x\phi \right\|^2\lesssim&
\sum\limits_{|\al| =2}\left\|\pa^\al \left[q_i\widetilde{n}_i+q_e\widetilde{n}_e\right]\right\|^2.
\end{split}
\end{equation}

Finally, letting $\ka_2\gg \ka_3\gg \ka_4\gg \ka_5\gg\ka_6>0$, we get from the summation of \eqref{1st.diss}, \eqref{p2u2}$\times\ka_3$, \eqref{2-dis.viep}$\times\ka_4$, \eqref{sed.t}$\times\ka_5$ and \eqref{ptphyL22}$\times\ka_6$ that
\begin{eqnarray}
&&\frac{d}{dt}\widetilde{\eta}
+\ka_1\frac{d}{dt}\left\{\left((m_in_i+m_en_e)\widetilde{u}_1,\pa_x(\widetilde{n}_i+\widetilde{n}_e)\right)
+\frac{3}{2}\left(\pa_x\widetilde{n}_i,\frac{\mu_i(\theta)+\mu_e(\theta)}{n^r_i}\pa_x\widetilde{n}_i\right)\right.
\notag\\
&&\left.\hspace{9cm}+\frac{3}{2}\left(\pa_x\widetilde{n}_e,\frac{\mu_i(\theta)+\mu_e(\theta)}{n^r_e}\pa_x\widetilde{n}_e\right)\right\}
\notag\\
&&\quad+\ka_3\frac{d}{dt}\sum\limits_{|\al| =1}\left\{\|\pa^\al \widetilde{n}_i\|^2
+\|\pa^\al \widetilde{n}_e\|^2
+\sum\limits_{j=1}^3\left\|\sqrt{m_in_i+m_en_e}\pa^\al \widetilde{u}_j\right\|^2
+\left\|\sqrt{n_i+n_e}\pa^\al \widetilde{\ta}\right\|^2\right\}
\notag\\
&&\quad+\ka_4\frac{d}{dt}\sum\limits_{|\al| =1}\left((m_in_i+m_en_e)\pa^\al \widetilde{u}_1,\pa_x\pa^\al (\widetilde{n}_i+\widetilde{n}_e)\right)
\notag\\
&&\quad+\la\sum\limits_{1\leq |\al| \leq2}\left\|\pa^{\al}\left[\widetilde{n}_i,\widetilde{n}_e,\widetilde{u},\widetilde{\ta}\right]\right\|^2
+\la\left\|q_i\widetilde{n}_i+q_e\widetilde{n}_e\right\|^2
+\la\sum\limits_{ |\al| \leq1}\left\|\pa^{\al}\left[\pa_x\phi ,\pa^2_x\phi \right]\right\|^2
\notag\\
&&\quad+\la\sum\limits_{|\al| =2}\left\|\pa^\al \pa^2_x\phi \right\|^2
+\la\int_{\R}\pa_xu_1^r\left[\widetilde{n}_i,\widetilde{n}_e,\widetilde{u},\widetilde{\ta}\right]^2dx
\notag\\
&&\leq
C(1+t)^{-2}\left\|\left[\widetilde{n}_i,\widetilde{n}_e,\widetilde{u},\widetilde{\ta}\right]\right\|^2
+C\de^{1/6}_r(1+t)^{-7/6}
+C_\eta\sum\limits_{1\leq|\al| \leq2}\int_{\R\times\R^3}(1+|\xi|)|\FM^{-1/2}\pa^{\al}\FG|^2 d\xi dx
\notag\\
&&\quad+C\int_{{\R}\times{\R}^3}(1+|\xi|)\left|\FM^{-1/2}_\sharp\widetilde{\FG}\right|^2 d\xi dx
+C\eps_0\int_{\R\times\R^3}|\FM^{-1/2}\pa_{\xi_1}\widetilde{\FG}|^2 d\xi dx. \label{mac.diss}
\end{eqnarray}
Noticing that
$$
\widetilde{\eta}\sim \left\|\left[\widetilde{n}_i,\widetilde{n}_e,\widetilde{u},\widetilde{\ta}\right]\right\|^2,\quad
\left|\frac{d}{dt}\|\pa_x\phi\|^2\right|\lesssim \left|(\pa_t\pa_x\phi,\pa_x\phi)\right|,
$$
we see that
\eqref{macro.eng} follows from \eqref{mac.diss}. This concludes the proof of Proposition \ref{mac.eng.lem.}.\qed

%\begin{remark}
%Note that the above estimates do not include the energy of the second order derivatives of $\left[\widetilde{v}_i,\widetilde{v}_i,\widetilde{u},\widetilde{\theta}\right]$,
%and they will be left to the next subsection, where the dissipation of the microscopic part will be mainly presented. This special treatment coincides with
%the energy method developed in \cite{G06}.
%\end{remark}

%\end{proof}

\section{A priori estimates on the non-fluid part}\label{sec.est.nf}

With estimates on the fluid part in Proposition \ref{mac.eng.lem.}, this section is further devoted to the proof of Proposition \ref{g.eng.lem.} on the non-fluid part. In a way similar to the previous section, the proof is divided by three subsections.

\subsection{Estimate on zero-order dissipation}
The goal of this subsection is to obtain the dissipation of $\FM_\ast^{-1/2}\widetilde{\FG}$.
%\noindent{\bf Step 1.} {\it Dissipation of $\FM_\ast^{-1/2}\widetilde{\FG}$.}
Notice that $\widetilde{\FG}$ solves
\begin{eqnarray}
&&\pa_t\widetilde{\FG}+\frac{3\pa_x\phi(\xi_1-u_1)(q_im_e-q_em_i)}{2\ta(m_in_i+m_en_e)}\left[n_eM_i,-n_iM_e\right]^{\rm T}-\FL_{\FM}\widetilde{\FG}
\notag\\
&&=
-\frac{3}{2\ta}\FP_1^{\FM}\left\{\xi_1\left[m_iM_i,m_eM_e\right]^{\rm T}
\left(\xi\cdot\pa_x\widetilde{u}+\frac{|\xi-u|^2}{2\ta}\pa_x\widetilde{\ta}\right)\right\}
-\FP_1^{\FM}\left\{\xi_1\left[n^{-1}_iM_i\pa_x\widetilde{n}_i,n^{-1}_eM_e\pa_x\widetilde{n}_e\right]^{\rm T}\right\}
\notag\\
&&\quad+\frac{3}{2\ta}\FP_1^{\FM}\left\{\left[M_i,M_e\right]^{\rm T}
\xi_1\right\}\pa_x\widetilde{\ta}
-\FP_1^{\FM}\left(\xi_1\pa_x\FG\right)
-{\bf P}^{\FM}_1\left(q_0\pa_x\phi\pa_{\xi_1}{\bf
G}\right)
+\FQ(\FG,\FG)-\pa_t\overline{\FG},\label{g.eq1.}
\end{eqnarray}
%\begin{equation}\label{g.eq1.}
%\begin{split}
%\pa_t\widetilde{\FG}&+\frac{3\pa_x\phi(\xi_1-u_1)(q_im_e-q_em_i)}{2\ta(m_in_i+m_en_e)}\left[n_eM_i,-n_iM_e\right]^{\rm T}-\FL_{\FM}\widetilde{\FG}\\=&
%-\frac{3}{2\ta}\FP_1^{\FM}\left\{\xi_1\left[m_iM_i,m_eM_e\right]^{\rm T}
%\left(\xi\cdot\pa_x\widetilde{u}+\frac{|\xi-u|^2}{2\ta}\pa_x\widetilde{\ta}\right)\right\}
%-\FP_1^{\FM}\left\{\xi_1\left[n^{-1}_iM_i\pa_x\widetilde{n}_i,n^{-1}_eM_e\pa_x\widetilde{n}_e\right]^{\rm T}\right\}
%\\
%&+\frac{3}{2\ta}\FP_1^{\FM}\left\{\left[M_i,M_e\right]^{\rm T}
%\xi_1\right\}\pa_x\widetilde{\ta}
%-\FP_1^{\FM}\left(\xi_1\pa_x\FG\right)
%-{\bf P}^{\FM}_1\left(q_0\pa_x\phi\pa_{\xi_1}{\bf
%G}\right)
%+\FQ(\FG,\FG)-\pa_t\overline{\FG},
%\end{split}
%\end{equation}
where we have used the fact that
\begin{equation*}
\begin{split}
\FP_1^{\FM}\left(\xi_1\pa_x\FM\right)-L_{\FM}\overline{\FG}
=&\frac{3}{2\ta}\FP_1^{\FM}\left\{\xi_1\left[m_iM_i,m_eM_e\right]^{\rm T}
\left(\xi\cdot\pa_x\widetilde{u}+\frac{|\xi-u|^2}{2\ta}\pa_x\widetilde{\ta}\right)\right\}
\\&+\FP_1^{\FM}\left\{\xi_1\left[n^{-1}_iM_i\pa_x\widetilde{n}_i,n^{-1}_eM_e\pa_x\widetilde{n}_e\right]^{\rm T}\right\}
+\frac{3}{2\ta}\FP_1^{\FM}\left\{\xi_1\left[M_i,M_e\right]^{\rm T}
\right\}\pa_x\widetilde{\ta},
\end{split}
\end{equation*}
and
\begin{equation*}
\begin{split}
{\bf P}^{\FM}_1\left(q_0\pa_x\phi\pa_{\xi_1}{\bf
M}\right)
=\frac{3\pa_x\phi(\xi_1-u_1)(q_im_e-q_em_i)}{2\ta(m_in_i+m_en_e)}\left[n_eM_i,-n_iM_e\right]^{\rm T}.
\end{split}
\end{equation*}

Let $\al_0=0$ or $1$.
%\begin{remark}
%It is worthy to point out that $\overline{\FG}$ defined in \eqref{def.ng} is designed to deal with the
%linear term ${\bf P}_1^{\FM}\left(\xi_1\pa_x\FM\right)$ and $q_0\pa_x\phi\pa_{\xi_1}{\bf
%M}^{(2)}$ which can not be directly controlled.
%\end{remark}
Taking the inner product of $\pa_t^{\al_0}\eqref{g.eq1.}$ with $ (n^r)^{-1}\ta(m_in_i+m_en_e)\FM_\ast^{-1}\pa_t^{\al_0}\widetilde{\FG} $ over ${\R}\times{\R}^3$, one has
\begin{equation}
\label{zero.g.eng.}
\frac{1}{2}\frac{d}{dt}\int_{{\R}\times{\R}^3}(n^r)^{-1}\ta(m_in_i+m_en_e)\pa_t^{\al_0}\widetilde{\FG}\cdot  \left(\FM_\ast^{-1}\pa_t^{\al_0}\widetilde{\FG}\right) d\xi dx
+\CJ_1+\CJ_2=\sum_{l=3}^{10}\CJ_l,
\end{equation}
where $\CJ_l$ $(1\leq \CJ\leq 10)$ are given by
\begin{eqnarray*}
\CJ_1&=&-{\displaystyle\int_{{\R}\times{\R}^3}}(n^r)^{-1}\ta(m_in_i+m_en_e)\pa_t^{\al_0}\widetilde{\FG}\cdot \left(\FM_\ast^{-1}\FL_{\FM}\pa_t^{\al_0}\widetilde{\FG}\right)d\xi dx,\\
\CJ_2&=&\frac{3}{2}\left(\pa_t^{\al_0}\pa_x\phi(\xi_1-u_1)(q_im_e-q_em_i)\left[n_eM_i,-n_iM_e\right]^{\rm T}, (n^r)^{-1}\FM_\ast^{-1}\pa_t^{\al_0}\widetilde{\FG} \right),\\
\CJ_3&=&-\chi_{\al_0}\left(\pa_t\widetilde{\FG}\pa_t^{\al_0}\left((n^r)^{-1}\ta(m_in_i+m_en_e)\right), \FM_\ast^{-1}\pa_t^{\al_0}\widetilde{\FG} \right)\\
&&+\frac{1}{2}\left(\pa_t^{\al_0}\widetilde{\FG}\pa_t\left((n^r)^{-1}\ta(m_in_i+m_en_e)\right), \FM_\ast^{-1}\pa_t^{\al_0}\widetilde{\FG} \right),\\
\CJ_4&=&-\frac{3}{2}\left(\pa_x\phi(q_im_e-q_em_i)\pa_t^{\al_0}
\left\{\frac{\xi_1-u_1}{m_in_i+m_en_e}\left[n_eM_i,-n_iM_e\right]^{\rm T}\right\}, \frac{\ta}{n^r}(m_in_i+m_en_e)\FM_\ast^{-1}\pa_t^{\al_0}\widetilde{\FG} \right),\\
\CJ_5&=&-\frac{3}{2}\left(\pa_t^{\al_0}\left\{\frac{1}{\ta}\FP_1^{\FM}\left[\xi_1\left[m_iM_i,m_eM_e\right]^{\rm T}
\left(\xi\cdot\pa_x\widetilde{u}+\frac{|\xi-u|^2}{2\ta}\pa_x\widetilde{\ta}\right)\right]\right\}, \frac{\ta}{n^r}(m_in_i+m_en_e)\FM_\ast^{-1}\pa_t^{\al_0}\widetilde{\FG} \right),\\
\CJ_6&=&\left(\pa_t^{\al_0}\left\{\FP_1^{\FM}\left[\xi_1\left[\frac{M_i}{n_i}\pa_x\widetilde{n}_i,n^{-1}_eM_e\pa_x\widetilde{n}_e\right]^{\rm T}
+\frac{3}{2\ta}\left[M_i,M_e\right]^{\rm T}
\xi_1\pa_x\widetilde{\ta}\right]\right\}, \frac{\ta}{n^r}(m_in_i+m_en_e)\FM_\ast^{-1}\pa_t^{\al_0}\widetilde{\FG} \right),\\
\CJ_7&=&-\left(\pa_t^{\al_0}\FP_1^{\FM}\left[q_0\pa_x\phi\pa_{\xi_1}\FG\right], (n^r)^{-1}\ta(m_in_i+m_en_e)\FM_\ast^{-1}\pa_t^{\al_0}\widetilde{\FG} \right),\\
\CJ_8&=&-\left(\pa_t^{\al_0}\pa_t\overline{\FG}, (n^r)^{-1}\ta(m_in_i+m_en_e)\FM_\ast^{-1}\pa_t^{\al_0}\widetilde{\FG} \right)\\
&&-\left(\pa_t^{\al_0}\FP_1^{\FM}\left[\xi_1\pa_x\FG\right], (n^r)^{-1}\ta(m_in_i+m_en_e)\FM_\ast^{-1}\pa_t^{\al_0}\widetilde{\FG} \right),\\
\CJ_9&=&\chi_{\al_0}\left(\FQ(\pa_t \FM, \FG)+\FQ( \FG, \pa_t\FM), (n^r)^{-1}\ta(m_in_i+m_en_e)\FM_\ast^{-1}\pa_t\widetilde{\FG}\right),\\
\CJ_{10}&=&\left(\pa_t^{\al_0}\FQ(\FG,\FG), (n^r)^{-1}\ta(m_in_i+m_en_e)\FM_\ast^{-1}\pa_t^{\al_0}\widetilde{\FG} \right).
\end{eqnarray*}
Here we have used the notation
$$
\chi_{\al_0}=\left\{\begin{array}{rll}
0,& \ \ \al_0=0,\\
1,& \ \ \al_0>0.
\end{array}\right.
$$
From Lemma \ref{co.est.}, we see that
\begin{eqnarray*}%\label{J1}
\begin{array}{rl}
\CJ_1\gtrsim \de{\displaystyle
\int_{{\R}\times{\R}^3}}(1+|\xi|)\left|\FM_\ast^{-1/2}\pa_t^{\al_0}\widetilde{\FG}\right|^2d\xi dx.
\end{array}
\end{eqnarray*}
For $\CJ_2$, if $\al_0=0$, it is bounded by
$$
\eta\int_{{\R}\times{\R}^3}(1+|\xi|)\left|\FM_\ast^{-1/2}\widetilde{\FG}\right|^2d\xi+C_\eta\|\pa_x\phi\|^2.
$$
If $\al_0=1$,
we first rewrite $\CJ_2$ as
\begin{equation*}
\begin{split}
\CJ_2=&\frac{3}{2}\left(\pa_t\pa_x\phi\xi_1(q_im_e-q_em_i)\left[n_eM_i,-n_iM_e\right]^{\rm T},~ (n^r)^{-1}\FM^{-1}\pa_t\widetilde{\FG} \right)
\\&+\frac{3}{2}\left(\pa_t\pa_x\phi\xi_1(q_im_e-q_em_i)\left[n_eM_i,-n_iM_e\right]^{\rm T},~ (n^r)^{-1}\left(\FM_\ast^{-1}-\FM^{-1}\right)\pa_t\widetilde{\FG} \right)\\
=&\frac{3}{2}\left(\pa_t\pa_x\phi\xi_1(q_im_e-q_em_i)\left[(n_e-n_e^r)M_i,-(n_i-n_i^r)M_e\right]^{\rm T},~ (n^r)^{-1}\FM^{-1}\pa_t\widetilde{\FG} \right)
\\&+\frac{3}{2}\left(\pa_t\pa_x\phi\xi_1(q_im_e-q_em_i)\left[M_i,\frac{q_e}{q_i}M_e\right]^{\rm T},~ \FM^{-1}\pa_t\widetilde{\FG} \right)
\\&+\frac{3}{2}\left(\pa_t\pa_x\phi\xi_1(q_im_e-q_em_i)\left[n_eM_i,-n_iM_e\right]^{\rm T},~ (n^r)^{-1}\left(\FM_\ast^{-1}-\FM^{-1}\right)\pa_t\widetilde{\FG} \right)
\\:= & \CJ_{2,1}+\CJ_{2,2}+\CJ_{2,3}.
\end{split}
\end{equation*}
Notice that
\begin{equation*}
\begin{split}
&|n_i(t,x)-n_{*i}|+|n_e(t,x)-n_{*e}|+|u(t,x)-u_*|+|\ta(t,x)-\ta_*|\\
&\leq
|n_i^r(t,x)-n_{i}(t,x)|+|n_e^r(t,x)-n_{e}(t,x)|+|u^r(t,x)-u(t,x)|+|\ta^r(t,x)-\ta(t,x)|
\\&\quad+
|n_i^r(t,x)-n_{*i}|+|n_e^r(t,x)-n_{*e}|+|u^r(t,x)-u_*|+|\ta^r(t,x)-\ta_*|\\
&
\lesssim\left|\left[\widetilde{n}_i,\widetilde{n}_e,\widetilde{u},\widetilde{\ta}\right]\right|+\eta_0.
\end{split}
\end{equation*}
From this together with the Cauchy-Schwarz inequality, it follows that
\begin{equation*}
\begin{split}
|\CJ_{2,3}|\lesssim& (\eps_0+\eta_0)\left\|\pa_t\pa_x\phi\right\|^2
+(\eps_0+\eta_0)\int_{{\R}\times{\R}^3}(1+|\xi|)\left|\FM_\ast^{-1/2}\pa_t\widetilde{\FG}\right|^2d\xi.
\end{split}
\end{equation*}
Moreover, one can see that $\CJ_{2,1}$ also enjoys the same upper bound as $\CJ_{2,3}$.

As to $\CJ_{2,2}$, from integration by parts and using the first equations of \eqref{cons.law.i} and \eqref{cons.law.e} as well as \eqref{tphy}, one has
\begin{equation*}
\begin{split}
\CJ_{2,1}=&-\frac{3}{2q_i}\left(\pa_t\phi\xi_1(q_im_e-q_em_i)\left[q_i,q_e\right]^{\rm T},~ \pa_t\pa_x\widetilde{\FG} \right)
\\=&-\frac{3}{2q_i}\left(\pa_t\phi\xi_1(q_im_e-q_em_i)\left[q_i,q_e\right]^{\rm T},~ \pa_t\pa_x\FG \right)
\\&+\frac{3}{2q_i}\left(\pa_x\pa_t\phi\xi_1(q_im_e-q_em_i)\left[q_i,q_e\right]^{\rm T},~ \pa_t\overline{\FG} \right)
\\=&\frac{3(q_im_e-q_em_i)}{2q_i}\left(\pa_t\phi,(\pa_t(q_in_i+q_en_e)+\pa_x((q_in_i+q_en_e)u_1)\right)
\\&+\frac{3}{2q_i}\left(\pa_x\pa_t\phi\xi_1(q_im_e-q_em_i)\left[q_i,q_e\right]^{\rm T},~ \pa_t\overline{\FG} \right)
\\=&-\frac{3(q_im_e-q_em_i)}{4q_i}\frac{d}{dt}\|\pa_t\pa_x\phi\|^2
+\frac{3(q_im_e-q_em_i)}{2q_i}\left(\pa_x\pa_t\phi,(q_in_i+q_en_e)u_1\right)
\\&+\frac{3}{2q_i}\left(\pa_x\pa_t\phi\xi_1(q_im_e-q_em_i)\left[q_i,q_e\right]^{\rm T},~ \pa_t\overline{\FG} \right).
\end{split}
\end{equation*}
Thus it holds that
\begin{equation*}
\begin{split}
\left|\CJ_{2,1}+\frac{3(q_im_e-q_em_i)}{4q_i}\frac{d}{dt}\|\pa_t\pa_x\phi\|^2\right|\lesssim \left\|\pa_x\pa_t\phi\right\|^2
+\|q_in_i+q_en_e\|^2+\de_r^{1/2}(1+t)^{-3/2}.
\end{split}
\end{equation*}
Next,  we get from %the a priori assumption
\eqref{aps} that
%Cauchy-Schwarz inequality with $\eta$ and the assumption $\frac{\ta}{2}< \ta_*<\ta$ that
\begin{eqnarray*}
\begin{array}{rl}
|\CJ_3|\leq& \eps_0\sum\limits_{\al_0\leq1}{\displaystyle
\int_{{\R}\times{\R}^3}}(1+|\xi|)\left|\FM_\ast^{-1/2}\pa_t^{\al_0}\widetilde{\FG}\right|^2d\xi dx.
%+C_\eta\sum\limits_{\al_0\leq1}\left\|\pa_x\pa_t^{\al_0}\left[\widetilde{n}_i,\widetilde{n}_e,\widetilde{u},\widetilde{\ta}\right]\right\|^2.
\end{array}
\end{eqnarray*}
By applying Lemma \ref{cl.RwRe} %and the a priori assumption
together with \eqref{aps}, one can see that
$\CJ_4$, $\CJ_5$  and $\CJ_6$ can be bounded as follows:
\begin{eqnarray*}
\begin{array}{rl}
|\CJ_4|+|\CJ_5|+|\CJ_6|\leq& \eta\sum\limits_{\al_0\leq1}{\displaystyle
\int_{{\R}\times{\R}^3}}(1+|\xi|)\left|\FM_\ast^{-1/2}\pa_t^{\al_0}\widetilde{\FG}\right|^2d\xi dx
+C_\eta\sum\limits_{\al_0\leq1}\left\|\pa_x\pa_t^{\al_0}
\left[\widetilde{n}_i,\widetilde{n}_e,\widetilde{u},\widetilde{\ta},\phi\right]\right\|^2,
\end{array}
\end{eqnarray*}
%\begin{equation*}
%\begin{split}
%|\CI_3|\leq& \left|\left(\frac{u_1}{v}\pa_x\widetilde{\FG}, \FM^{-1}\widetilde{\FG} \right)\right|+
%\left|\left(\frac{u_1}{v}\pa_x\overline{\FG}, \FM^{-1}\widetilde{\FG} \right)\right|\\
%\leq& (\eps_{0}+\eta)\int_{{\R}\times{\R}^3}(1+|\xi|)\left|\FM_\ast^{-1/2}\widetilde{\FG}\right|^2d\xi dx
%+C_\eta\eps(1+t)^{-2},
%\end{split}
%\end{equation*}
\begin{equation*}
\begin{split}
|\CJ_7|\leq& \sum\limits_{\al_0\leq1}\left|\left(\pa_x\pa_t^{\al_0}\phi\pa_{\xi_1}\widetilde{\FG}, \FM^{-1}\pa_t^{\al_0}\widetilde{\FG} \right)\right|+\sum\limits_{\al_0\leq1}
\left|\left(\pa_x\pa_t^{\al_0}\phi\pa_{\xi_1}\overline{\FG}, \FM^{-1}\pa_t^{\al_0}\widetilde{\FG} \right)\right|
\\&+\sum\limits_{\al_0\leq1}\left|\left(\pa_x\phi\pa_{\xi_1}\widetilde{\FG}, \FM^{-1}\pa_t^{\al_0}\widetilde{\FG} \right)\right|+
\left|\left(\pa_x\phi\pa_{\xi_1}\overline{\FG}, \FM^{-1}\pa_t^{\al_0}\widetilde{\FG} \right)\right|
\\&+\sum\limits_{\al_0\leq1}\eta\int_{{\R}\times{\R}^3}(1+|\xi|)\left|\FM_\ast^{-1/2}\pa_t^{\al_0}\widetilde{\FG}\right|^2d\xi dx
+C_\eta\sum\limits_{\al_0\leq1}\left\|\pa_x\pa_t^{\al_0}\phi\right\|^2
\\
\leq& (\eps_{0}+\eta)\sum\limits_{\al_0\leq1}\int_{{\R}\times{\R}^3}(1+|\xi|)\left|\FM_\ast^{-1/2}\widetilde{\FG}\right|^2d\xi dx
+C_\eta\sum\limits_{\al_0\leq1}\left\|\pa_x\pa_t^{\al_0}\phi\right\|^2,
\end{split}
\end{equation*}
and
\begin{equation*}
\begin{split}
|\CJ_8|
\leq& \eta\sum\limits_{\al_0\leq1}\int_{{\R}\times{\R}^3}(1+|\xi|)\left|\FM_\ast^{-1/2}\pa_t^{\al_0}\widetilde{\FG}\right|^2d\xi dx
+C_\eta\sum\limits_{\al_0\leq1}\int_{{\R}\times{\R}^3}(1+|\xi|)\left|\FM_\ast^{-1/2}\pa_x\pa_t^{\al_0}\widetilde{\FG}\right|^2d\xi dx
\\&
+C_\eta\de^{1/2}_r(1+t)^{-3/2}.
\end{split}
\end{equation*}
As to $\CJ_{9}$ and $\CJ_{10}$, it follows from \eqref{db.G}, Lemma \ref{est.nonop} and Cauchy-Schwarz's inequality with $\eta$ that
\begin{equation*}
\begin{split}
|\CJ_{9}|\leq& \eta\sum\limits_{\al_0\leq1}\int_{{\R}\times{\R}^3}(1+|\xi|)\left|\FM_\ast^{-1/2}\pa_t\widetilde{\FG}\right|^2d\xi dx
\\&+C_\eta\sum\limits_{\al_0\leq1}\int_{\R}\left(\int_{\R^3}(1+|\xi|)\left|\FM_\ast^{-1/2}\pa_t\FM\right|^2d\xi\right) \left(\int_{\R^3}(1+|\xi|)\left|\FM_\ast^{-1/2}\widetilde{\FG}\right|^2d\xi\right) dx
\\&+C_\eta\sum\limits_{\al_0\leq1}\int_{\R}\left(\int_{\R^3}(1+|\xi|\left|\pa_t\FM_\ast^{-1/2}\FM\right|^2d\xi\right) \left(\int_{\R^3}(1+|\xi|)\left|\FM_\ast^{-1/2}\overline{\FG}\right|^2d\xi\right) dx
\\
\leq& (\eps_{0}+\eta)\sum\limits_{\al_0\leq1}\int_{{\R}\times{\R}^3}(1+|\xi|)\left|\FM_\ast^{-1/2}\pa_t^{\al_0}\widetilde{\FG}\right|^2d\xi dx\\
&+C_\eta\de_r\left\|\pa_t\left[\widetilde{n}_i,\widetilde{n}_e,\widetilde{u},\widetilde{\ta}\right]\right\|^2
+C_\eta\de^{1/2}_r(1+t)^{-3/2},
\end{split}
\end{equation*}
and
\begin{equation*}
\begin{split}
|\CJ_{10}|\leq& \eta\sum\limits_{\al_0\leq1}\int_{{\R}\times{\R}^3}(1+|\xi|)\left|\FM_\ast^{-1/2}\pa_t^{\al_0}\widetilde{\FG}\right|^2d\xi dx
\\&+C_\eta\sum\limits_{\al_0\leq1}\int_{\R}\left(\int_{\R^3}(1+|\xi|)\left|\FM_\ast^{-1/2}\pa_t^{\al_0}\FG\right|^2d\xi\right) \left(\int_{\R^3}\left|\FM_\ast^{-1/2}\FG\right|^2d\xi\right) dx
\\&+C_\eta\sum\limits_{\al_0\leq1}\int_{\R}\left(\int_{\R^3}(1+|\xi|\left|\FM_\ast^{-1/2}\FG\right|^2d\xi\right) \left(\int_{\R^3})\left|\FM_\ast^{-1/2}\pa_t^{\al_0}\FG\right|^2d\xi\right) dx
\\
\leq& (\eps_{0}+\eta)\sum\limits_{\al_0\leq1}\int_{{\R}\times{\R}^3}(1+|\xi|)\left|\FM_\ast^{-1/2}\pa_t^{\al_0}\widetilde{\FG}\right|^2d\xi dx
+C\de_r (1+t)^{-2}.
\end{split}
\end{equation*}
Now substituting all the above estimates into \eqref{zero.g.eng.}, we arrive at
\begin{equation}\label{zero.g.eng1.}
\begin{split}
&\frac{d}{dt}\sum\limits_{\al_0\leq1}\int_{{\R}\times{\R}^3}\left|\FM^{-1/2}_\ast\pa_t^{\al_0}\widetilde{\FG}\right|^2d\xi dx
+\frac{d}{dt}\|\pa_t\pa_x\phi\|^2\\
&\quad+ \la\sum\limits_{\al_0\leq1}{\displaystyle
\int_{{\R}\times{\R}^3}}(1+|\xi|)\left|\FM_\ast^{-1/2}\pa_t^{\al_0}\widetilde{\FG}\right|^2d\xi dx\\
&\lesssim \sum\limits_{\al_0\leq1}\int_{{\R}\times{\R}^3}\frac{(1+|\xi|)\left|\pa_x\pa_t^{\al_0}\FG\right|^2}{\FM_{*}}d\xi dx
+\sum\limits_{1\leq|\al|\leq2}\left\|\pa^{\al}\left[\widetilde{n}_i,\widetilde{n}_e,\widetilde{u},\widetilde{\ta}\right]\right\|^2
\\&\quad+\sum\limits_{|\al|\leq1}\left\|\pa^{\al}\left[\pa_x\phi,\pa^2_x\phi\right]\right\|^2
+\de_r^{1/2}(1+t)^{-3/2}.
\end{split}
\end{equation}

Furthermore, it follows from \eqref{g.eq1.} that
\begin{equation}\label{g.eq2.}
\begin{split}
\pa^2_t\widetilde{\FG}=&\pa_t\left\{\frac{3\pa_x\phi(\xi_1-u_1)(q_im_e-q_em_i)}{2\ta(m_in_i+m_en_e)}\left[n_eM_i,-n_iM_e\right]^{\rm T}\right\}
\\&-\pa_t\left\{\frac{3}{2\ta}\FP_1^{\FM}\left\{\xi_1\left[m_iM_i,m_eM_e\right]^{\rm T}
\left(\xi\cdot\pa_x\widetilde{u}+\frac{|\xi-u|^2}{2\ta}\pa_x\widetilde{\ta}\right)\right\}\right\}
\\&-\pa_t\left\{\FP_1^{\FM}\left\{\xi_1\left[n^{-1}_iM_i\pa_x\widetilde{n}_i,n^{-1}_eM_e\pa_x\widetilde{n}_e\right]^{\rm T}\right\}\right\}
\\
&+\pa_t\left\{\frac{3}{2\ta}\FP_1^{\FM}\left\{\left[M_i,M_e\right]^{\rm T}
\xi_1\right\}\pa_x\widetilde{\ta}\right\}
-\pa_t\left\{\FP_1^{\FM}\left(\xi_1\pa_x\FG\right)
+{\bf P}^{\FM}_1\left(q_0\pa_x\phi\pa_{\xi_1}{\bf
G}\right)\right\}
\\&+\pa_t(\FQ(\FM,\FG)+\FQ(\FG,\FM))+\pa_t\FQ(\FG,\FG)-\pa^2_t\overline{\FG}.
\end{split}
\end{equation}
Then  \eqref{zero.g.eng1.} and \eqref{g.eq2.} give rise to
\begin{equation}\label{zero.g.eng2.}
\begin{split}
&\frac{d}{dt}\sum\limits_{\al_0\leq1}\int_{{\R}\times{\R}^3}\left|\FM^{-1/2}_\ast\pa_t^{\al_0}\widetilde{\FG}\right|^2d\xi dx
+\frac{d}{dt}\left\|\pa_t\pa_x\phi\right\|^2\\
&\quad+ \la\sum\limits_{\al_0\leq2}{\displaystyle
\int_{{\R}\times{\R}^3}}(1+|\xi|)\left|\FM_\ast^{-1/2}\pa_t^{\al_0}\widetilde{\FG}\right|^2d\xi dx\\
&\lesssim \sum\limits_{\al_0\leq1}\int_{{\R}\times{\R}^3}\frac{(1+|\xi|)\left|\pa_x\pa_t^{\al_0}\FG\right|^2}{\FM_{*}}d\xi dx
+\sum\limits_{1\leq|\al|\leq2}\left\|\pa^{\al}\left[\widetilde{n}_i,\widetilde{n}_e,\widetilde{u},\widetilde{\ta}\right]\right\|^2
\\&\quad+\sum\limits_{|\al|\leq1}\left\|\pa^{\al}\left[\pa_x\phi,\pa^2_x\phi\right]\right\|^2
+\de_r^{1/2}(1+t)^{-3/2}.
\end{split}
\end{equation}

\subsection{Estimate on high-order energy}
%\noindent{\bf Step 2.} {\it Higher order energy.}
In this subsection, let us now deduce estimates on the higher order energy of $\FF$.
%Note that even for $|\al|\geq1$, one can not directly obtain the dissipation of $\FM^{-1/2}\pa^\al  \FG$ with the aid of \eqref{micBE},
%since the linear term $\left(\pa^\al  {\bf P}^{{\bf M}}_1\left(\xi_1\pa_x{\bf
%M}\right), \FM^{-1}\pa^\al  \FG \right)$
%makes big trouble. To overcome this difficulty, we first deduce the energy estimates on $\pa_x\pa^\al  \FF$ $(|\al|\leq1)$ by using the original equation $\eqref{v.F}$ with respect to the local Maxwellian $\FM$,
%in this case, the corresponding term becomes $\left(\pa_x\pa^\al \left(\xi_1\pa_x\FF\right), \FM^{-1}\pa_x\pa^\al  \FF \right)$, which is under control.
%Then we tend to obtain another estimates based on the global Maxwellian $\FM_*$.
The desired estimates will be obtained by the interplay of two kinds of weighted energy estimates. Let $|\al| \leq1$. Taking the $L^2\times L^2$ inner product of \eqref{v.F} with $k_B\ta \FM^{-1}\pa_x\pa^\al  \FF  $
with respect to $x$ and $\xi$ over $\R\times\R^3$, one has
\begin{equation}
\label{2d.F}
\frac{1}{2}\frac{d}{dt}\int_{{\R}\times{\R}^3}k_B\ta\pa_x\pa^\al \FF \cdot  \left(\FM^{-1}\pa_x\pa^\al  \FF\right)  d\xi dx+\CK_1+\CK_2=\sum_{l=3}^9\CK_l,
\end{equation}
where all terms $\CK_l$ $(1\leq l\leq 9)$ are given by
\begin{eqnarray*}
\CK_1&=&-\left(\FL_{\FM}\pa_x\pa^\al \FG, k_B\ta\FM^{-1}\pa_x\pa^\al  \FG \right),\\
\CK_2&=&\left(q_0\pa^\al \pa^2_x\phi\pa_{\xi_1}\FM, k_B\ta\FM^{-1}\pa_x\pa^\al  \FF  \right),
\end{eqnarray*}
and
\begin{eqnarray*}
\CK_3&=&\frac{1}{2}\left(\pa_x\pa^\al  \FF, k_B\ta\pa_t\left(\FM^{-1}\right)\pa_x\pa^\al  \FF\right)+\frac{1}{2}\left(\pa_x\pa^\al  \FF,k_B\pa_t\ta\FM^{-1}\pa_x\pa^\al  \FF\right),\\
\CK_4&=&\sum\limits_{\al'\leq\al }\left(\FQ(\pa^{\al'} \pa_x\FM, \pa^{\al-\al'}\FG)+\FQ( \pa^{\al-\al'}\FG, \pa^{\al'}\pa_x\FM), k_B\ta\FM^{-1}\pa_x\pa^\al  \FF  \right),\\
\CK_5&=&\left(\FL_{\FM}\pa_x\pa^\al \FG, k_B\ta\FP_1^{\FM}\left(\FM^{-1}\pa_x\pa^\al  \FM\right)\right),\\
\CK_6&=&-\left(\xi_1\pa^2_x\pa^\al \FF, k_B\ta\FM^{-1}\pa_x\pa^\al  \FF  \right),\\
\CK_7&=&-\sum\limits_{\al'\leq\al }
C_{\al'}^{\al}\left(q_0\pa^{\al-\al'} \pa_x\phi \pa^{\al'}\pa_x\pa_{\xi_1}\FF, k_B\ta\FM^{-1}\pa_x\pa^\al  \FF  \right),\\
\CK_8&=&-\left(q_0\pa^\al \pa^2_x\phi\pa_{\xi_1}\FG, k_B\ta\FM^{-1}\pa_x\pa^\al  \FF  \right),\\
\CK_9&=&\left(\pa_x\pa^\al  \FQ(\FG,\FG), k_B\ta\FM^{-1}\pa_x\pa^\al  \FF  \right).
\end{eqnarray*}
Here we have used the decomposition $\FF=\FM+\FG$.
%and the fact that
%$$
%\left(\frac{\pa^\al \pa_x\phi}{v}\pa_{\xi_1}\FM, \FM^{-1}\pa^\al  \FG \right)=0.
%$$
First of all, for $\CK_1$, Lemma \ref{co.est.} implies that
\begin{equation*}
\begin{split}
\CK_{1}\geq \de\int_{{\R}\times{\R}^3}(1+|\xi|)\left|\FM^{-1/2}\pa_x\pa^\al \FG\right|^2d\xi dx.
\end{split}
\end{equation*}

For $\CK_{2}$,
%noticing that
%\begin{equation}\label{con.mie}
%\pa_t\left(\frac{1}{v_i}\right)+Z\int_{\R^3}\frac{\xi_1-u_1}{v}\pa_x F_id\xi=0,\ \ \
%\pa_t\left(\frac{1}{v_e}\right)+\int_{\R^3}\frac{\xi_1-u_1}{v}\pa_x F_ed\xi=0.
%\end{equation}
from the first equations of \eqref{cons.law.i} and \eqref{cons.law.e}, we claim that
\begin{equation}\label{CK2}
\begin{split}
\left|\CK_{2}
-\frac{1}{2}\frac{d}{dt}\sum\limits_{|\al|\leq1}\left\|\pa^\al \pa^2_x\phi\right\|^2
\right|
\lesssim
\eps_{0}\sum\limits_{|\al| \leq1}\left\|\pa^\al \left[\pa_x\phi,\pa_x^2\phi\right]\right\|^2.
\end{split}
\end{equation}
In fact, to show \eqref{CK2},
%The proof of \eqref{CK2} will be given in the Appendix.
%We now turn to complete
%\begin{proof}[The proof of \eqref{CK2}]
we notice
\begin{equation*}
\begin{split}
\CK_2=&-\left(\pa^\al \pa^2_x\phi[q_i,q_e]^{\rm T},(\xi_1-u_1)\pa^\al\pa_x\FM\right)-\left(\pa^\al \pa^2_x\phi[q_i,q_e]^{\rm T},(\xi_1-u_1)\pa^\al\pa_x\FG\right)\\[3mm]
&-\left(\pa^\al \pa^2_x\phi[q_i,q_e]^{\rm T},(\xi_1-u_1)\pa^\al\pa_x\FM\right)-\left(\pa^\al \pa^2_x\phi[q_i,q_e]^{\rm T},\xi_1\pa^\al\pa_x\FG\right),
%\quad |\al|\leq 1.
\end{split}
\end{equation*}
with  $|\al|\leq 1$.
%Denote $\pa_j\in\{\pa_t,\pa_x\}$ and $\pa_k\in\{\pa_t,\pa_x\}$.
Here, by direct computations, it holds that
\begin{eqnarray*}
\pa_x M_{\CA} =\frac{\pa_x n_{\CA}}{n_{\CA}}M_{\CA} +\frac{\xi-u}{k_{\CA}\theta}\cdot \pa_x uM_{\CA} +\left(\frac{|\xi-u|^2}{2k_{\CA}\theta}-{\frac{3}{2}}\right)\frac{\pa_x \theta}{\theta}M_{\CA},
\end{eqnarray*}
for $\CA=i,e$. Then, from the first equations of \eqref{cons.law.i} and \eqref{cons.law.e}, it follows that for $|\al|=0$,
%from which, we see that when $\al=\pa_i$ (or $\al=\pa_j$),
\begin{equation*}
\begin{split}
\CK_{2}=&-\left(\pa^2_x\phi,(q_in_i+q_en_e)\pa_x u_1+[q_i,q_e]^{\rm T}\cdot\xi_1\pa_x\FG\right)
\\=&\left(\pa^2_x\phi,\pa_t(q_in_i+q_en_e)+\pa_x(q_in_i+q_en_e)u_1\right)
\\=&-\left(\pa^2_x\phi,\pa_t\pa^2_x\phi\right)-\left(\pa^2_x\phi,\pa^3_x\phi u_1\right),
\end{split}
\end{equation*}
and hence one has
%Thus, for $\al=0$, we get
from integration by parts and \eqref{aps} that
\begin{equation}\label{CK20}
\begin{split}
\left|\CK_{2}+\frac{d}{d}\|\pa^2_x\phi\|^2\right|\lesssim \eps_0\|\pa^2_x\phi\|^2,
\end{split}
\end{equation}
for $|\al=0|$. Furthermore, for $|\al|=1$, one can also obtain from direct calculations that
\begin{equation*}
\begin{split}
\CK_{2}=&-\left(\pa^{\al}\pa^2_x\phi, (q_in_i+q_en_e)\pa^{\al} \pa_xu_1\right)\\
&-\left(\pa^{\al}\pa^2_x\phi, (q_i\pa^\al n_i+q_e\pa^\al n_e)\pa_x u_1+(q_i\pa_xn_i+q_e\pa_xn_e)\pa^\al u_1\right)
\\&-\left(\pa^\al\pa^2_x\phi,[q_i,q_e]^{\rm T}\cdot\xi_1\pa^\al\pa_x\FG\right)
\\=&\left(\pa^\al\pa^2_x\phi,\pa^\al\pa_t(q_in_i+q_en_e)+(q_i\pa^\al\pa_xn_i+q_e\pa^\al\pa_xn_e)u_1\right)
\\=&-\left(\pa^\al\pa^2_x\phi,\pa_t\pa^\al\pa^2_x\phi+\pa^\al\pa^3_x\phi u_1\right),
\end{split}
\end{equation*}
which implies
\begin{equation}\label{CK21}
\begin{split}
\left|\CK_{2}+\frac{d}{d}\|\pa^\al\pa^2_x\phi\|^2\right|\lesssim \eps_0\|\pa^\al\pa^2_x\phi\|^2,
%\ \ |\al|=1.
\end{split}
\end{equation}
for $|\al|=1$.
Therefore \eqref{CK2} follows from \eqref{CK20} and \eqref{CK21}. This completes the estimate on $\CK_2$.
%\end{proof}

For the remaining terms in \eqref{2d.F},
we only give estimates in the case of $|\al| =0$ as the proof in the case of $|\al| =1$ is similar.
For this, by applying Lemma \ref{cl.RwRe}, Sobolev's inequality and Cauchy Schwarz inequality,  we have that for $|\al|=0$,
\begin{equation*}
\begin{split}
|\CK_{3}|\lesssim&
\int_{\R\times{\R}^3}(1+|\xi|)|\pa_t[n_i,n_e,u,\ta]|\left(\left|\FM_\ast^{-1/2}\pa_x
\FM\right|^2+\left|\FM_\ast^{-1/2}\pa_x  \FG\right|^2\right)d\xi dx
\\
\lesssim& \int_{\R}\left|\pa_x \left[\widetilde{n}_i,\widetilde{n}_e,\widetilde{u},\widetilde{\ta}\right]\right|^2
|\pa_t\left[n_i,n_e,u,\ta \right]|\,dx+
\int_{\R}|\pa_x [n^r,u^r,\ta^r]|^2
|\pa_t[n^r,u^r,\ta^r]|\,dx
\\&+\int_{\R}|\pa_x [n^r,u^r,\ta^r]|^2
\left|\pa_t\left[\widetilde{n}_i,\widetilde{n}_e,\widetilde{u},\widetilde{\ta}\right]\right|dx
+\eps_{0}\int_{\R\times{\R}^3}(1+|\xi|)\left|\FM_\ast^{-1/2}\pa_x \FG\right|^2d\xi dx\\
\lesssim&
\left\|\pa_x \left[\widetilde{n}_i,\widetilde{n}_e,\widetilde{u},\widetilde{\ta}\right]\right\|^2\|\pa_t[n^r,u^r,\ta^r]\|_{L^\infty}
+
\|\pa_t[n^r,u^r,\ta^r]\|_{L^\infty}\|\pa_x [n^r,u^r,\ta^r]\|^2
\\&+\left\|\pa_t\left[\widetilde{n}_i,\widetilde{n}_e,\widetilde{u},\widetilde{\ta}\right]\right\|^{1/2}
\left\|\pa_x\pa_t\left[\widetilde{n}_i,\widetilde{n}_e,\widetilde{u},\widetilde{\ta}\right]\right\|^{1/2}
\|\pa_x [n^r,u^r,\ta^r]\|^2
%\\&+\sum\limits_{|\al| =1}(\eps_0+\eta)\|\pa^\al [\widetilde{v},\widetilde{u},\widetilde{\ta}]\|^2
\\&+\eps_{0}\int_{\R\times{\R}^3}(1+|\xi|)\left|\FM_\ast^{-1/2}\pa_x \FG\right|^2d\xi dx
\\ \lesssim &\left(\eps_0+\de_r\right)\sum\limits_{|\al|=1}
\left\|\pa^\al \left[\widetilde{n}_i,\widetilde{n}_e,\widetilde{u},\widetilde{\ta}\right]\right\|^2
+\eps_{0}\int_{\R\times{\R}^3}(1+|\xi|)\left|\FM_\ast^{-1/2}\pa_x \FG\right|^2d\xi dx
\\&+\de_r^{1/6}(1+t)^{-7/6}.
\end{split}
\end{equation*}
%Therefore
%\begin{equation*}
%\begin{split}
%\sum\limits_{1\leq\ga \leq2}|\CI_{18}|
%\lesssim &\sum\limits_{1\leq\ga \leq2}(\eps_0+\eta)\|\pa^\al [\widetilde{v},\widetilde{u},\widetilde{\ta}]\|^2
%+\eps_{0}\sum\limits_{|\al| =1}\int_{\R\times{\R}^3}(1+|\xi|)\left|\FM_\ast^{-1/2}\pa^\al \FG\right|^2d\xi dx
%\\&+\eps^{3/4}(1+t)^{-5/4}.
%\end{split}
%\end{equation*}
For $\CK_4$, one sees that for $|\al|=0$, $\CK_{4}$ reduces to
$$
\left(\FQ(\pa_x \FM, \FG)+\FQ( \FG, \pa_x\FM), \FM^{-1}\pa_x  \FG \right),
$$
%therefore, for $\al=0$,
and hence we have
\begin{equation*}
\begin{split}
|\CK_{4}|\lesssim& \eta\int_{\R\times{\R}^3}(1+|\xi|)\left|\FM^{-1/2}\pa_x \FG\right|^2d\xi dx
\\&+C_\eta\int_{\R}\left(\int_{\R^3}(1+|\xi|)|\FM^{-1/2}\pa_x \FM|^2d\xi\right)
\left(\int_{\R^3}|\FM^{-1/2}(\widetilde{\FG}+\overline{\FG})|^2d\xi\right) dx
\\&+C_\eta\int_{\R}\left(\int_{\R^3}|\FM^{-1/2}\pa_x \FM|^2d\xi\right)
\left(\int_{\R^3}(1+|\xi|)|\FM^{-1/2}\widetilde{\FG}+\overline{\FG}|^2d\xi\right) dx
\\ \lesssim& \eta\int_{\R\times{\R}^3}(1+|\xi|)\left|\FM^{-1/2}\pa_x \FG\right|^2d\xi dx+\eps_{0}\int_{{\R}\times{\R}^3}(1+|\xi|)\left|\FM_\ast^{-1/2}\widetilde{\FG}\right|^2d\xi dx
\\&+\eps_{0}\left\|\pa_x \left[\widetilde{n}_i,\widetilde{n}_e,\widetilde{u},\widetilde{\ta}\right]\right\|^2
+\de_r(1+t)^{-2}.
\end{split}
\end{equation*}
For $\CK_{5}$, it should vanish for $|\al| =0$. For $\CK_6$, by using integration by parts and performing the similar calculations as for $\CK_{3}$, one sees that for $|\al| =0$,
$|\CK_{6}|$ is bounded by
\begin{equation*}
\begin{split}
&\int_{\R\times{\R}^3}(1+|\xi|)\left|\pa_x \left[n_i,n_e,u,\ta\right]\right|\left(\left|\FM_\ast^{-1/2}\pa_x  \FM\right|^2+\left|\FM_\ast^{-1/2}\pa_x  \FG\right|^2\right)d\xi dx
\\ &\lesssim \sum\limits_{|\al| =1}\left(\eps_0+\de_r\right)
\left\|\pa_x \left[\widetilde{n}_i,\widetilde{n}_e,\widetilde{u},\widetilde{\ta}\right]\right\|^2
+\eps_{0}\int_{\R\times{\R}^3}(1+|\xi|)\left|\FM_\ast^{-1/2}\pa_x \FG\right|^2d\xi dx
+\de_r^{1/6}(1+t)^{-7/6}.
\end{split}
\end{equation*}
For $\CK_{7}$ with $|\al| =0$, using $\FF=\FM+\FG$ again, one has
\begin{equation*}
\begin{split}
|\CK_{7}|\lesssim&\left|
\left(q_0\pa_x\phi\pa_x \pa_{\xi_1}\FM,\FM^{-1}\pa_x \FM  \right)\right|
+\left|\left(q_0\pa_x\phi\pa_x \pa_{\xi_1}\FG,\FM^{-1}\pa_x \FG \right)\right|
\\[3mm]&+\left|\left(q_0\pa_x\phi\pa_x \pa_{\xi_1}\FM,\FM^{-1}\pa_x \FG \right)\right|
\\
\lesssim &\int_{\R}|\pa_x\phi||\pa_x  [n_i,n_e,u,\ta]|^2 dx
+\int_{\R\times{\R}^3}|\pa_x\phi|
|\FM^{-1/2}\pa_x\pa_{\xi_1} \overline{\FG}|^2d\xi dx
\\&+\int_{\R\times{\R}^3}|\pa_x\phi|
|\FM^{-1/2}\pa_x\pa_{\xi_1} \widetilde{\FG}|^2d\xi dx
+\sum\limits_{|\al| =1}\int_{\R\times{\R}^3}|\pa_x\phi|
\left|\FM^{-1/2}\pa_x  \FG\right|^2d\xi dx
\\&+C_\eta\sum\limits_{|\al| =1}\left\|\pa_x\phi\pa_x  [n_i,n_e,u,\ta]\right\|^2+\eta\sum\limits_{|\al| =1}\int_{\R\times{\R}^3}(1+|\xi|)\left|\FM^{-1/2}\pa_x \FG\right|^2d\xi dx
\\ \lesssim &\sum\limits_{|\al| =1}(\eps_0+\eta)\left\|\pa_x \left[\widetilde{n}_i,\widetilde{n}_e,\widetilde{u},\widetilde{\ta}\right]\right\|^2
+\eps_0\left\|\pa_x\phi\right\|^2+\eps_{0}\int_{\R\times{\R}^3}(1+|\xi|)\left|\FM^{-1/2}\pa_x\pa_{\xi_1}\widetilde{\FG}\right|^2d\xi dx
\\&+(\eps_{0}+\eta)\sum\limits_{|\al| =1}\int_{\R\times{\R}^3}(1+|\xi|)\left|\FM^{-1/2}\pa_x \FG\right|^2d\xi dx
+C_\eta\de^{1/6}_r(1+t)^{-7/6}.
\end{split}
\end{equation*}

%\begin{equation*}
%\begin{split}
%|\CJ_{15}|=&\Bigg|-\left(q_0\left(\frac{\pa_x\phi}{v}\right)\pa^\al \pa_{\xi_1}\FG,\FM^{-1}\pa^\al \FG \right)
%-\left(q_0\left(\frac{\pa_x\phi}{v}\right)\pa^\al \pa_{\xi_1}\FG,\FM^{-1}\pa^\al \FM  \right)
%\\&-\left(q_0\left(\frac{\pa_x\phi}{v}\right)\pa^\al \pa_{\xi_1}\FM,\FM^{-1}\pa^\al \FM  \right)
%-\left(q_0\left(\frac{\pa_x\phi}{v}\right)\pa^\al \pa_{\xi_1}\FM,\FM^{-1}\pa^\al \FG \right)\Bigg|\\
%\lesssim &C_\eta\sum\limits_{|\al| =1}\int_{\R}(|\pa_x\phi|+|\pa_x\phi|^2)|\pa^\al  (n_i,n_e,u,\ta)|^2 dx
%+(\eps_0+\eta)\sum\limits_{|\al| =1}\int_{\R\times{\R}^3}\left|\FM^{-1/2}\pa^\al \pa_{\xi_1}\overline{\FG}\right|^2d\xi dx
%\\&+(\eps_{0}+\eta)\sum\limits_{|\al| =1}\int_{\R\times{\R}^3}(1+|\xi|)
%\left|\FM^{-1/2}\pa^\al \pa_{\xi_1}\widetilde{\FG}\right|^2d\xi dx
%\\&+(\eps_{0}+\eta)\sum\limits_{|\al| =1}\int_{\R\times{\R}^3}(1+|\xi|)\left|\FM^{-1/2}\pa^\al \FG\right|^2d\xi dx
%\\ \lesssim &\sum\limits_{|\al| =1}(\eps_0+\eta)
%\left\|\pa^\al \left[\widetilde{v}_i,\widetilde{v}_e,\widetilde{u},\widetilde{\ta}\right]\right\|^2
%+(\eps_0+\eta)\left\|\pa_x\phi \right\|^2
%\\&+(\eps_0+\eta)\sum\limits_{|\al| =1}\int_{\R\times{\R}^3}(1+|\xi|)
%\left|\FM^{-1/2}\pa^\al \pa_{\xi_1}\widetilde{\FG}\right|^2d\xi dx
%\\&+(\eps_{0}+\eta)\sum\limits_{|\al| =1}\int_{\R\times{\R}^3}(1+|\xi|)\left|\FM^{-1/2}\pa^\al \FG\right|^2d\xi dx
%+\de_r^{1/6}(1+t)^{-7/6},
%\end{split}
%\end{equation*}
Likewise, for $\CK_8$ with $|\al|=0$, it follows that
\begin{equation*}
\begin{split}
|\CK_{8}|\leq &\left|-\left(q_0 \pa^2_x\phi\pa_{\xi_1}\FG,\FM^{-1}\pa_x \FG \right)\right|
+\left|-\left(q_0\pa^2_x\phi\pa_{\xi_1}\FG,\FM^{-1}\pa_x \FM  \right)\right|\\
\lesssim &C_\eta\int_{\R\times{\R}^3}|\pa^2_x\phi|^2\left|\FM^{-1/2}\pa_{\xi_1}\overline{\FG}\right|^2d\xi dx
+(\eps_{0}+\eta)\int_{\R\times{\R}^3}(1+|\xi|)
\left|\FM^{-1/2}\pa_{\xi_1}\widetilde{\FG}\right|^2d\xi dx
\\&+(\eps_{0}+\eta)\int_{\R\times{\R}^3}(1+|\xi|)\left|\FM^{-1/2}\pa_x \FG\right|^2d\xi dx
+C_\eta\int_{\R}| \pa^2_x\phi|^2|\pa_x  [n_i,n_e,u,\ta]|^2 dx
\\&+\int_{\R\times{\R}^3}|\pa^2_x\phi||\pa_x  [n_i,n_e,u,\ta](1+|\xi|)|\pa_{\xi_1}\overline{\FG}|d\xi dx
\\ \lesssim &(\eps_0+\eta)
\left\|\pa_x \left[\widetilde{n}_i,\widetilde{n}_e,\widetilde{u},\widetilde{\ta},\pa_x\phi\right]\right\|^2
+\eps_{0}\int_{\R\times{\R}^3}(1+|\xi|)
\left|\FM^{-1/2}\pa_x \pa_{\xi_1}\widetilde{\FG}\right|^2d\xi dx
\\&+(\eps_{0}+\eta)\int_{\R\times{\R}^3}(1+|\xi|)\left|\FM^{-1/2}\pa_x \FG\right|^2d\xi dx
+\de_r(1+t)^{-2}.
\end{split}
\end{equation*}
As to the last term $\CK_{9}$ with $\al=0$, we get from Lemma \ref{est.nonop} and
Cauchy Schwarz inequality that
\begin{equation*}
\begin{split}
\CK_{9}\lesssim& \eta\int_{\R\times{\R}^3}(1+|\xi|)\left|\FM^{-1/2}\pa_x \FG\right|^2d\xi dx
\\&+\int_{\R}\left(\int_{\R^3}(1+|\xi|)\left|\FM^{-1/2}\pa_x \FG\right|^2d\xi\right)
\left(\int_{\R^3}\left|\FM^{-1/2}\FG\right|^2d\xi\right) dx
\\&+\int_{\R}\left(\int_{\R^3}\left|\FM^{-1/2}\pa_x  \FG\right|^2d\xi\right)
\left(\int_{\R^3}(1+|\xi|)\left|\FM^{-1/2}\FG\right|^2d\xi\right) dx
\\ \lesssim&(\eta+\eps_0)\int_{\R\times{\R}^3}(1+|\xi|)\left|\FM^{-1/2}\pa_x \FG\right|^2d\xi dx
+\eps_0\int_{\R\times{\R}^3}(1+|\xi|)\left|\FM^{-1/2} \widetilde{\FG}\right|^2d\xi dx.
\end{split}
\end{equation*}
Substituting all the above estimates for $\CK_{l}$ $(1\leq l\leq 9)$ into \eqref{2d.F} and performing the similar calculation as above for the case $|\al|=1$, one sees that
\begin{equation}\label{2d.F.sum}
\begin{split}
&\frac{d}{dt}\left\{\sum\limits_{|\al|\leq1}\int_{{\R}\times{\R}^3}\left|\FM^{-1/2}\pa_x\pa^\al  \FF\right|^2d\xi dx
+\sum\limits_{|\al|\leq1}\left\|\pa^\al \pa^2_x\phi \right\|\right\}
\\&\quad+\la\sum\limits_{|\al|\leq1}\int_{\R\times{\R}^3}(1+|\xi|)\left|\FM^{-1/2}\pa_x\pa^\al \FG\right|^2d\xi dx\\
&\lesssim(\eps_{0}+\eta)\sum\limits_{|\al|\leq1}\int_{\R\times{\R}^3}(1+|\xi|)\left|\FM_\ast^{-1/2}\pa_x\pa^\al \FG\right|^2d\xi dx
+\eps_{0}\int_{{\R}\times{\R}^3}(1+|\xi|)\left|\FM^{-1/2}\widetilde{\FG}\right|^2d\xi dx
\\&
\quad+(\eps_{0}+\eta)\sum\limits_{|\al| \leq1}\int_{\R\times{\R}^3}(1+|\xi|)\left|\FM^{-1/2}\pa^\al \pa_{\xi_1} \widetilde{\FG}\right|^2d\xi dx
+(\eps_{0}+\eta)\sum\limits_{1\leq\ga \leq2}
\left\|\pa^\al \left[\widetilde{n}_i,\widetilde{n}_e,\widetilde{u},\widetilde{\ta}\right]\right\|^2
\\&\quad+(\eps_{0}+\eta)\sum\limits_{|\al|\leq1}
\left\|\pa^\al\left[\pa_x\phi,\pa^2_x\phi\right]\right\|^2+\de_r^{1/6}(1+t)^{-7/6}.
\end{split}
\end{equation}
%where we have used the fact that
%\begin{multline*}
%\sum\limits_{1\leq\ga \leq2}\Bigg\{
%\int_{\R}\frac{e^{-\phi }}{v}\left|\pa^\al \phi \right|^2dx
%+\int_{\R}\frac{v^r}{v^3}\left|\pa^\al \pa_x\phi \right|^2dx
%\\+\int_{\R}\frac{v^r}{v^5}\pa_x\phi
%\pa^\al  \widetilde{v}\pa_x\phi \pa^\al \widetilde{v}dx
%+\frac{d}{dt}\int_{\R}\frac{\pa^\al \phi^r}{v^2}\pa^\al  \widetilde{v}dx\Bigg\}
%\sim\sum\limits_{1\leq\ga \leq2}\left\|\pa^\al \phi \right\|_{H^1}^2+\eps.
%\end{multline*}
Similarly, one can obtain the following energy estimates for $\pa_x\pa^\al  \FF$ $(|\al|\leq1)$ with respect to the global Maxwellian $\FM_*$:
\begin{equation}\label{2d.F.sum2}
\begin{split}
\frac{d}{dt}&\sum\limits_{|\al|\leq1}\int_{{\R}\times{\R}^3}\left|\FM_\ast^{-1/2}\pa_x\pa^\al  F\right|^2d\xi dx
+\la\sum\limits_{|\al|\leq1}\int_{\R\times{\R}^3}(1+|\xi|)\left|\FM^{-1/2}_\ast\pa_x\pa^\al  \FG\right|^2d\xi dx\\
\lesssim&
(\eps_{0}+\eta)\sum\limits_{|\al|\leq1}\int_{\R\times{\R}^3}(1+|\xi|)\left|\FM^{-1/2}_\ast\pa^\al \pa_{\xi_1} \widetilde{\FG}\right|^2d\xi dx+\eps_{0}\int_{{\R}\times{\R}^3}(1+|\xi|)\left|\FM_\ast^{-1/2}\widetilde{\FG}\right|^2d\xi dx
\\&+C_\eta\sum\limits_{1\leq|\al| \leq2}
\left\|\pa^\al \left[\widetilde{n}_i,\widetilde{n}_e,\widetilde{u},\widetilde{\ta}\right]\right\|^2
+C_\eta\sum\limits_{|\al|\leq1}
\left\|\pa^\al\left[\pa_x\phi,\pa^2_x\phi\right]\right\|^2
+\de_r^{1/6}(1+t)^{-7/6}.
\end{split}
\end{equation}
Note that one may not require the smallness of the coefficient of
$$
\left\|\pa^\al\left[\widetilde{n}_i,\widetilde{n}_e,\widetilde{u},\widetilde{\ta},\pa_x\phi,\pa^2_x\phi\right]\right\|^2,
$$
and this implies that the derivation of  \eqref{2d.F.sum2} is much simpler than the one of \eqref{2d.F.sum}. %, in this case,
Due to this we would omit details of %its justification here
the proof of \eqref{2d.F.sum2} for brevity.

With \eqref{2d.F.sum} in hand, by letting $1\gg\ka_7>0$, we get from the summation of \eqref{2d.F.sum} and $\eqref{macro.eng}\times\ka_7$ that
\begin{equation}\label{mi.diss2}
\begin{split}
&\ka_7\frac{d}{dt}\left\{\sum\limits_{\ga \leq1}\left\|\pa^\al \left[\widetilde{n}_i,\widetilde{n}_e, \widetilde{u}, \widetilde{\ta}\right](t)\right\|^2
+\left\| \pa_x\phi (t)\right\|^2\right\}
-\ka_7\ka_0\frac{d}{dt}\sum\limits_{|\al| =1}\left(\pa^\al \widetilde{u}_1,\pa^\al \pa_x\widetilde{v}_i+\pa^\al \pa_x\widetilde{v}_e\right)
\\
&\quad+\frac{d}{dt}\sum\limits_{|\al| \leq1}\left\{\int_{{\R}\times{\R}^3}\left|\FM^{-1/2}\pa_x\pa^\al  \FF\right|^2d\xi dx
+\left\|\pa^\al \pa^2_x\phi \right\|^2\right\}
\\&\quad+\la \left\|\sqrt{\pa_xu^r}\left[\widetilde{n}_i,\widetilde{n}_e,\widetilde{u},\widetilde{\theta}\right](t)\right\|^2
+\la\sum\limits_{1\leq\ga \leq 2}\left\|\pa^\al \left[\widetilde{n}_i,\widetilde{n}_e,\widetilde{u},\widetilde{\theta}\right](t)\right\|^2
+\la\left\|q_i\widetilde{n}_i+q_e\widetilde{n}_e\right\|^2
\\&\quad+\la\sum\limits_{ |\al|\leq1}\left\|\pa^{\al}\left[\pa_x\phi,\pa^2_x\phi \right]\right\|^2
+\la \sum\limits_{|\al|\leq1}\int_{{\R}\times{\R}^3}(1+|\xi|)\left|\FM^{-1/2}\pa_x\pa^\al \FG\right|^2d\xi dx
\\
&\lesssim(1+t)^{-2}\left\|\left[\widetilde{n}_i,\widetilde{n}_e,\widetilde{u},\widetilde{\ta}\right]\right\|^2
+\de_r^{1/6}(1+t)^{-7/6}
+\sum\limits_{|\al| \leq1}\eps_{0}\int_{\R\times\R^3}|\FM_\ast^{-1/2}\pa_{\xi_1}\pa^\al \widetilde{\FG}|^2 d\xi dx
\\&\quad+\ka_7\sum\limits_{1\leq|\al_0|\leq2}\int_{\R\times{\R}^3}(1+|\xi|)\left|\FM_\ast^{-1/2}\pa_t^{\al_0} \FG\right|^2d\xi dx
+\ka_7\int_{{\R}\times{\R}^3}(1+|\xi|)\left|\FM_\ast^{-1/2}\widetilde{\FG}\right|^2d\xi dx.
\end{split}
\end{equation}
%where we have also used the fact that $\sigma_0<3/4$ and $1<\al_0<5/4$.
On the other hand, by choosing $1\gg\ka_{8}\gg\ka_{9}>0$, it follows from the summation of $\eqref{zero.g.eng2.}\times\ka_{9}$
 and $\eqref{2d.F.sum2}\times\ka_{8}$ that
\begin{eqnarray}
&&\ka_{9}\frac{d}{dt}\sum\limits_{\al_0\leq1}\int_{{\R}\times{\R}^3}\left|\FM^{-1/2}_\ast\pa_t^{\al_0}\widetilde{\FG}\right|^2d\xi dx
+\ka_{9}\frac{d}{dt}\left\|\pa_t\pa_x\phi\right\|^2
+\ka_{8}\frac{d}{dt}\sum\limits_{1\leq|\al| \leq2}\int_{{\R}\times{\R}^3}\left|\FM_\ast^{-1/2}\pa^\al  \FF\right|^2d\xi dx
\notag\\
&&\quad+\la\sum\limits_{1\leq|\al| \leq2}\int_{\R\times{\R}^3}(1+|\xi|)\left|\FM^{-1/2}_\ast\pa^\al  \FG\right|^2d\xi dx
+\la{\displaystyle
\int_{{\R}\times{\R}^3}}(1+|\xi|)\left|\FM_\ast^{-1/2}\widetilde{\FG}\right|^2d\xi dx
\notag\\
&&\lesssim (\ka_{8}+\ka_{9})\sum\limits_{1\leq\al \leq2}
\left\|\pa^\al \left[\widetilde{n}_i,\widetilde{n}_e,\widetilde{u},\widetilde{\ta}\right]\right\|^2
+(\ka_{8}+\ka_{9})\sum\limits_{ |\al|\leq1}\left\|\pa^{\al}\left[\pa_x\phi ,\pa^2_x\phi \right]\right\|^2
\notag\\
&&\quad+\de^{1/6}_r(1+t)^{-7/6}+\sum\limits_{|\al| \leq1}\eps_{0}\int_{\R\times\R^3}|\FM^{-1/2}\pa_{\xi_1}\pa^\al \widetilde{\FG}|^2 d\xi dx, \label{mi.diss3}
\end{eqnarray}
%\begin{equation}\label{mi.diss3}
%\begin{split}
%&\ka_{9}\frac{d}{dt}\sum\limits_{\al_0\leq1}\int_{{\R}\times{\R}^3}\left|\FM^{-1/2}_\ast\pa_t^{\al_0}\widetilde{\FG}\right|^2d\xi dx
%+\ka_{9}\frac{d}{dt}\left\|\pa_t\pa_x\phi\right\|^2
%+\ka_{8}\frac{d}{dt}\sum\limits_{1\leq|\al| \leq2}\int_{{\R}\times{\R}^3}\left|\FM_\ast^{-1/2}\pa^\al  \FF\right|^2d\xi dx
%\\
%&\quad+\la\sum\limits_{1\leq|\al| \leq2}\int_{\R\times{\R}^3}(1+|\xi|)\left|\FM^{-1/2}_\ast\pa^\al  \FG\right|^2d\xi dx
%+\la{\displaystyle
%\int_{{\R}\times{\R}^3}}(1+|\xi|)\left|\FM_\ast^{-1/2}\widetilde{\FG}\right|^2d\xi dx
%\\
%&\lesssim (\ka_{8}+\ka_{9})\sum\limits_{1\leq\al \leq2}
%\left\|\pa^\al \left[\widetilde{n}_i,\widetilde{n}_e,\widetilde{u},\widetilde{\ta}\right]\right\|^2
%+(\ka_{8}+\ka_{9})\sum\limits_{ |\al|\leq1}\left\|\pa^{\al}\left[\pa_x\phi ,\pa^2_x\phi \right]\right\|^2
%\\
%&\quad+\de^{1/6}_r(1+t)^{-7/6}+\sum\limits_{|\al| \leq1}\eps_{0}\int_{\R\times\R^3}|\FM^{-1/2}\pa_{\xi_1}\pa^\al \widetilde{\FG}|^2 d\xi dx,
%\end{split}
%\end{equation}
where we have  used the fact that
\begin{equation*}
\begin{split}
&\sum\limits_{1\leq|\al|\leq2}{\displaystyle
\int_{{\R}\times{\R}^3}}(1+|\xi|)\left|\FM_\ast^{-1/2}\pa^{\al}\FG\right|^2d\xi dx
\\
&\lesssim \sum\limits_{1\leq|\al|\leq2}{\displaystyle
\int_{{\R}\times{\R}^3}}(1+|\xi|)\left|\FM_\ast^{-1/2}\pa^{\al}\widetilde{\FG}\right|^2d\xi dx
+\de_r\sum\limits_{1\leq|\al|\leq2}\left\|\pa^\al\left[\widetilde{n}_i,\widetilde{n}_e,\widetilde{u},\widetilde{\ta}\right]\right\|^2
+\de_r^{1/2}(1+t)^{-3/2}.
\end{split}
\end{equation*}

\subsection{Estimate on energy with mixed derivatives}
%{\bf Step 3.} {\it Energy estimates with mixed derivatives.}
In what follows, we deduce the energy estimates on the mixed derivative terms $\pa^\al \pa^\be \widetilde{\FG}$. To do so, let $|\be|\geq1$ and $|\al|+|\be|\leq 2$. Acting $\pa^\al \pa^\be $ to \eqref{g.eq1.} and taking the inner product of the resulting equation with $\FM^{-1}_\ast\pa^\al \pa^\be \widetilde{\FG}$  over ${\R}\times{\R}^3$, one has
\begin{equation*}%\label{mixd.ip}
\begin{split}
&\frac{1}{2}\frac{d}{dt}\int_{{\R}\times{\R}^3}\left|\FM^{-1/2}_\ast\pa^\al \pa^\be \widetilde{\FG}\right|^2d\xi dx
-{\displaystyle\int_{{\R}\times{\R}^3}}\pa^\al \pa^\be \widetilde{\FG} \cdot\left(\FM_\ast^{-1}\FL_{\FM}\pa^\al \pa^\be \widetilde{\FG}\right)d\xi dx
\\
&=\sum\limits_{1\leq|\al'|+|\be'|\atop{\al'\leq\al,\be'\leq \be}}\left(\FQ(\pa^{\al'}\pa^{\be'}\FM,\pa^{\al-\al'}\pa^{\be-\be'}\widetilde{\FG}),
\FM^{-1}_\ast\pa^\al \pa^\be\widetilde{\FG}\right)
\\&\quad +\sum\limits_{1\leq|\al'|+|\be'|\atop{\al'\leq\al,\be'\leq \be}}\left(\FQ(\pa^{\al'}\pa^{\be'}\widetilde{\FG}
,\pa^{\al-\al'}\pa^{\be-\be'}\FM),
\FM^{-1}_\ast\pa^\al \pa^\be\widetilde{\FG}\right)\\
&\quad -\left(\pa^\al \pa^\be\FP_1^\FM\left(\frac{3}{2\ta}\xi_1\left[m_iM_i,m_eM_e\right]^{\rm T}
\left(\xi\cdot\pa_x\widetilde{u}+\frac{|\xi-u|^2}{2\ta}\pa_x\widetilde{\ta}\right)\right)
,\FM^{-1}_\ast\pa^\al \pa^\be\widetilde{\FG}\right)
\\&\quad +\left(\pa^\al \pa^\be\FP_1^\FM\left(\xi_1
\left[n_i^{-1}M_i\pa_x\widetilde{n}_i,n_e^{-1}M_e\pa_x\widetilde{n}_e\right]^{\rm T}
+\frac{3}{2\ta}\left[\left[M_i,M_e\right]^{\rm T}\xi_1\pa_x\widetilde{\ta}\right]\right)
,\FM^{-1}_\ast\pa^\al \pa^\be\widetilde{\FG}\right)
\\&\quad
-\left(\pa^\al \pa^\be\left(\FP_1^{\FM}\left(\xi_1\pa_x\FG\right)\right),\FM^{-1}_\ast\pa^\al \pa^\be\widetilde{\FG}\right)
-\left(\pa^\al \pa^\be\left(q_0\pa_x\phi \pa_{\xi_1}{\FM}\right),\FM^{-1}_\ast\pa^\al \pa^\be\widetilde{\FG}\right)
\\
&\quad -\left(\pa^\al \pa^\be\left(q_0\pa_x\phi\pa_{\xi_1}\FG\right),\FM^{-1}_\ast\pa^\al \pa^\be\widetilde{\FG}\right)
+\left(\pa^\al \pa^\be \FQ(\FG,\FG),\FM^{-1}_\ast\pa^\al \pa^\be\widetilde{\FG}\right)\\
&\quad -\left(\pa_t\pa^\al \pa^\be\overline{\FG},\FM^{-1}_\ast\pa^\al \pa^\be\widetilde{\FG}\right).
\end{split}
\end{equation*}
%Here the summation $\sum$ is over $\ga+|\be|\leq 2$ and $|\be|\geq1$.
Similar to those calculations in the previous subsection, %to the discussions in above Step 2,
it holds that
\begin{equation}\label{mi.diss4}
\begin{split}
\frac{d}{dt}&\sum\limits_{|\al|+|\be|\leq 2\atop{|\be|\geq1}}\int_{{\R}\times{\R}^3}
\left|\FM^{-1/2}_\ast\pa^\al \pa^\be \widetilde{\FG}\right|^2d\xi dx
+\la \sum_{|\al|+|\be|\leq 2\atop{|\be|\geq1}}\int_{{\R}\times{\R}^3}(1+|\xi|)
\left|\FM_\ast^{-1/2}\pa^\al \pa^\be\widetilde{\FG}\right|^2d\xi dx\\
\lesssim&\sum\limits_{1\leq|\al| \leq2}\int_{{\R}\times{\R}^3}(1+|\xi|)\left|\FM_\ast^{-1/2}\pa^\al \FG\right|^2d\xi dx
+\int_{{\R}\times{\R}^3}(1+|\xi|)\left|\FM_\ast^{-1/2}\widetilde{\FG}\right|^2d\xi dx\\&+ \sum\limits_{1\leq|\al| \leq2}
\left\|\pa^\al \left[\widetilde{n}_i,\widetilde{n}_e,\widetilde{u},\widetilde{\ta}\right]\right\|^2
+\sum\limits_{|\al| \leq1}\left\|\pa_x\phi\right\|^2+\de_r^{1/6}(1+t)^{-7/6}.
\end{split}
\end{equation}
Consequently, it follows from \eqref{mi.diss2}, \eqref{mi.diss3} and \eqref{mi.diss4} that
\begin{eqnarray}
&&K_0\ka_7\frac{d}{dt}\left\{\sum\limits_{|\al| \leq1}\left\|\pa^\al \left[\widetilde{n}_i, \widetilde{n}_e, \widetilde{u}, \widetilde{\ta}\right](t)\right\|^2
+\left\|\pa_x \phi (t)\right\|^2\right\}
+K_0\ka_7\ka_0\frac{d}{dt}\sum\limits_{|\al| =1}\left(\pa^\al \widetilde{u}_1,\pa^\al \pa_x\widetilde{n}_i+\pa^\al \pa_x\widetilde{n}_e\right)
\notag\\
&&\quad+K_0\frac{d}{dt}\sum\limits_{|\al| \leq1}\left\{\int_{{\R}\times{\R}^3}\left|\FM^{-1/2}\pa_x\pa^\al  \FF\right|^2d\xi dx
+\left\|\pa^\al \pa^2_x\phi \right\|^2\right\}
\notag\\&&\quad+\ka_{9}\frac{d}{dt}\sum\limits_{\al_0\leq1}\int_{{\R}\times{\R}^3}\left|\FM^{-1/2}_\ast\pa_t^{\al_0}\widetilde{\FG}\right|^2d\xi dx
+\ka_{9}\frac{d}{dt}\left\|\pa_t\pa_x\phi\right\|^2
\notag\\&&
\quad+\ka_{8}\frac{d}{dt}\sum\limits_{|\al| \leq1}\int_{{\R}\times{\R}^3}\left|\FM_\ast^{-1/2}\pa_x\pa^\al  \FF\right|^2d\xi dx
+\ka_{10}\frac{d}{dt}\sum\limits_{|\al|+|\be|\leq 2\atop{|\be|\geq1}}
\int_{{\R}\times{\R}^3}\left|\FM^{-1/2}_\ast\pa^\al \pa^\be \widetilde{\FG}\right|^2d\xi dx
\notag\\&&\quad+\la\sum\limits_{1\leq|\al| \leq2}\int_{\R\times{\R}^3}(1+|\xi|)\left|\FM^{-1/2}_\ast\pa^\al  \FG\right|^2d\xi dx
+\la{\displaystyle
\int_{{\R}\times{\R}^3}}(1+|\xi|)\left|\FM_\ast^{-1/2}\widetilde{\FG}\right|^2d\xi dx
\notag\\&&\quad+\la \sum\limits_{1\leq|\al| \leq2}\int_{{\R}\times{\R}^3}(1+|\xi|)\left|\FM^{-1/2}\pa^\al \FG\right|^2d\xi dx
+\la \sum_{|\al|+|\be|\leq 2\atop{|\be|\geq1}}
\int_{{\R}\times{\R}^3}(1+|\xi|)\left|\FM^{-1/2}_\ast\pa^\al \pa^\be\widetilde{\FG}\right|^2d\xi dx
\notag\\&&\quad+\la \left\|\sqrt{\pa_xu^r}\left[\widetilde{n}_i,\widetilde{n}_e,\widetilde{u},\widetilde{\theta}\right](t)\right\|^2
+\la\sum\limits_{1\leq|\al| \leq 2}\left\|\pa^\al \left[\widetilde{n}_i,\widetilde{n}_e,\widetilde{u},\widetilde{\theta}\right](t)\right\|^2
\notag\\&&\quad+\la\left\|q_i\widetilde{n}_i+q_e\widetilde{n}_e\right\|^2+\la\sum\limits_{ |\al| \leq1}\left\|\pa^{\al}\left[\pa_x\phi ,\pa^2_x\phi \right]\right\|^2\notag\\
&&\lesssim (1+t)^{-2}\left\|\left[\widetilde{n}_i,\widetilde{n}_e,\widetilde{u},\widetilde{\ta}\right]\right\|^2
+\de_r^{1/6}(1+t)^{-7/6},\label{mi.diss5}
\end{eqnarray}
where $K_0$ is a positive large constant and $\ka_{10}$ is also a positive but suitably small constant. Therefore \eqref{g.eng.} follows from
\eqref{mi.diss5} with the help of the Gronwall's inequality. This completes the proof of Proposition \ref{g.eng.lem.}.\qed

%\vskip 4mm
%\section{Appendix}

%\medskip
%\noindent {\bf Acknowledgements:} RJD was supported by the General Research Fund (Project No.~409913) from RGC of Hong Kong. SQL was
%supported by grants from the National Natural Science Foundation of China under contracts 11471142 and 11271160.

%has received research grants from 

\medskip
\noindent {\bf Conflict of Interest:}  Renjun Duan has received the General Research Fund (Project No.~409913) from RGC of Hong Kong. Shuangqian Liu has received grants from the National Natural Science Foundation of China under contracts 11471142 and 11271160.

%\newpage

\end{document}